\documentclass[11pt, leqno]{amsart}

\usepackage{amsxtra}
\usepackage{amscd}
\usepackage{longtable}
\usepackage{bigstrut}
\usepackage{color}

\usepackage[curve,matrix,arrow,frame,tips]{xy}
\usepackage{mathdots}
\usepackage{verbatim}

\usepackage{graphicx}
\usepackage{xspace}
\usepackage{ifthen}

\newsavebox{\lineone}
\newsavebox{\linetwo}
\newsavebox{\linethree}
\newlength{\lineonelen}
\newlength{\linetwolen}
\newlength{\linethreelen}
\newlength{\biggerlen}
\newcommand{\twolinestight}[2]{%
   \sbox{\lineone}{#1}%
   \sbox{\linetwo}{#2}%
   \settowidth{\lineonelen}{\usebox{\lineone}}%
   \settowidth{\linetwolen}{\usebox{\linetwo}}%
   \ifthenelse{\lengthtest{\lineonelen > \linetwolen}}%
      {\setlength{\biggerlen}{\the\lineonelen}}%
      {\setlength{\biggerlen}{\the\linetwolen}}%
   \begin{minipage}{\biggerlen}%
      \centering
      \usebox\lineone\\%
      \usebox\linetwo%
   \end{minipage}%
   }
\newcommand{\threelinestight}[3]{%
   \sbox{\lineone}{#1}%
   \sbox{\linetwo}{#2}%
   \sbox{\linethree}{#3}%
   \settowidth{\lineonelen}{\usebox{\lineone}}%
   \settowidth{\linetwolen}{\usebox{\linetwo}}%
   \settowidth{\linethreelen}{\usebox{\linethree}}%
   \ifthenelse{\lengthtest{\lineonelen > \linetwolen}}%
      {\ifthenelse{\lengthtest{\lineonelen > \linethreelen}}%
         {\setlength{\biggerlen}{\the\lineonelen}}%
         {\setlength{\biggerlen}{\the\linethreelen}}%
      }%
      {\ifthenelse{\lengthtest{\linetwolen > \linethreelen}}%
         {\setlength{\biggerlen}{\the\linetwolen}}%
         {\setlength{\biggerlen}{\the\linethreelen}}%
      }%
   \begin{minipage}{\biggerlen}%
      \centering
      \usebox\lineone\\%
      \usebox\linetwo\\%
      \usebox\linethree\\
   \end{minipage}%
   }
\usepackage{amsthm, amsmath}
\usepackage{amssymb}
\usepackage{amscd}

\newtheorem{thm}{Theorem}[section]
\newtheorem{prop}[thm]{Proposition}
\newtheorem{lemma}[thm]{Lemma}

\newtheorem{Definition}[thm]{Definition}

\newtheorem{Remark}[thm]{Remark}

\newtheorem{cor}[thm]{Corollary}
\newtheorem{example}[thm]{Example}

\newcounter{ex}[section]

\newcommand{\quash}[1]{}  
\newcommand{\nc}{\newcommand}
\nc{\on}{\operatorname}

\newcommand{\frakO}{{\mathfrak O}}
\newcommand{\frakP}{{\mathfrak P}}

\newcommand{\frakR}{{\mathfrak R}}

\newcommand{\frakX}{{\mathfrak X}}
\newcommand{\frakY}{{\mathfrak Y}}

\newcommand{\bbA}{{\mathbb A}}

\newcommand{\bbF}{{\mathbb F}}
\newcommand{\bbG}{{\mathbb G}}

\newcommand{\bbQ}{{\mathbb Q}}
\newcommand{\bbR}{{\mathbb R}}

\newcommand{\bbZ}{{\mathbb Z}}
\newcommand{\calA}{{\mathcal A}}
\newcommand{\calB}{{\mathcal B}}

\newcommand{\calD}{{\mathcal D}}
\newcommand{\calE}{{\mathcal E}}
\newcommand{\calF}{{\mathcal F}}
\newcommand{\calG}{{\mathcal G}}
\newcommand{\calH}{{\mathcal H}}

\newcommand{\calL}{{\mathcal L}}

\newcommand{\calP}{{\mathcal P}}
\newcommand{\calQ}{{\mathcal Q}}

\newcommand{\calS}{{\mathcal S}}
\newcommand{\calT}{{\mathcal T}}
\newcommand{\calU}{{\mathcal U}}
\newcommand{\calV}{{\mathcal V}}

\newcommand{\calZ}{{\mathcal Z}}

\newcommand{\mA}{\ensuremath{\mathbb{A}}\xspace}
\newcommand{\mC}{\ensuremath{\mathbb{C}}\xspace}
\newcommand{\mF}{\ensuremath{\mathbb{F}}\xspace}

\newcommand{\mP}{\ensuremath{\mathbb{P}}\xspace}
\newcommand{\mQ}{\ensuremath{\mathbb{Q}}\xspace}

\newcommand{\mZ}{\ensuremath{\mathbb{Z}}\xspace}
\newcommand{\mE}{\ensuremath{\mathbb{E}}\xspace}

\nc{\al}{{\alpha}} \nc{\be}{{\beta}}

\nc{\ve}{{\varepsilon}} \nc{\Ga}{{\Gamma}}
\newcommand{\la}{{\lambda}}
\nc{\La}{{\Lambda}}

\def\xcoch{\mathbb{X}_\bullet}
\def\xch{\mathbb{X}^\bullet}

\newcommand{\boV}{{\bold V}}
\newcommand{\boF}{{\bold F}}
\newcommand{\boG}{{\bold G}}
\newcommand{\boE}{{\bold E}}
\newcommand{\boB}{{\bold B}}
\newcommand{\boS}{{\bold S}}

\newcommand{\cal}{\mathcal}

\newcommand{\E}{{\mathcal E}}
\newcommand{\A}{{\mathcal A}}

\renewcommand{\AA}{{\mathbb A}}

\newcommand{\C}{{\mathbb C}}
\newcommand{\R}{{\mathbb R}}
\newcommand{\Q}{{\mathbb Q}}

\newcommand{\ep}{\epsilon}
\newcommand{\B}{{\mathcal B}}

\newcommand{\Ff}{{\mathbb F}}

\newcommand{\Gg}{{\mathcal G}}

\newcommand{\Hh}{{\mathcal H}}
\newcommand{\Gm}{{{\mathbb G}_{\rm m}}}
\newcommand{\Z}{{\mathbb Z}}
\newcommand{\F}{{\mathcal F}}
\newcommand{\ti}{\tilde}
\newcommand{\Spec}{{\rm Spec } }
 \renewcommand{\O}{{\mathcal O}}

\newcommand{\Gr}{{\rm Gr}}
\newcommand{\UU}{{\mathcal U}}

\newcommand{\GL}{{\rm GL}}

\newcommand{\und}{\underline}

\newcommand{\Gu}{\underline{G}}

\renewcommand{\L}{{\mathcal L}}

\newcommand{\Res}{{\rm Res}}

\newcommand{\Tr}{{\rm Tr}}

\newcommand{\PP}{{\mathcal P}}

\newcommand{\Fsm}{\breve F}

\newcommand{\rH}{{\mathrm H}}

\def\thfill{\null\nobreak\hfill}

\def\endproof{\thfill\vbox{\hrule
  \hbox{\vrule\hbox to 5pt{\vbox to 5pt{\vfil}\hfil}\vrule}\hrule}}

\newcommand{\twprod}{\,\widetilde{\times}\, }

\renewcommand{\P}{{\cal P}}

 \newcommand{\ke}{k_E}

\newcommand{\Gal}{{\rm Gal}}

\newcommand{\Adm}{{\rm Adm}}

\begin{document}

\title[ ]{Local models of Shimura varieties\\ and a conjecture of Kottwitz}

\author[G. Pappas]{G. Pappas }
\thanks{G. P partially supported by  NSF grant DMS11-02208.}
\thanks{X. Z partially supported by NSF grant DMS10-01280.}

\address{Dept. of
Mathematics\\
Michigan State
University\\
E. Lansing\\
MI 48824 \\
USA} \email{pappas@math.msu.edu}
\author[X. Zhu]{X. Zhu}
\address{Dept. of
Mathematics\\ Northwestern University\\
Evanston\\ IL 60208\\ USA} \email{xinwenz@math.northwestern.edu}

\begin{abstract}
{We give a group theoretic definition  of  ``local models"
 as sought after in the theory of  Shimura varieties. These are projective schemes
 over the integers of a $p$-adic local field
 that are expected to model the singularities of
 integral models of Shimura varieties with  parahoric level structure. Our local models are certain
 mixed characteristic degenerations of Grassmannian varieties; they are obtained by extending
 constructions
 of  Beilinson, Drinfeld, Gaitsgory and the
 second-named author to mixed characteristics and to the case of general (tamely ramified) reductive
 groups. We study the singularities
of local models and hence also of the corresponding integral models of Shimura varieties.
In particular, we study the monodromy (inertia) action and show a commutativity
property for the sheaves of nearby cycles.
As a result, we prove a conjecture of Kottwitz
 which asserts that the semi-simple trace of Frobenius
 on the nearby cycles gives a function which is
 central in the parahoric Hecke algebra. }
 \end{abstract}
\date{\today}

\maketitle


{\footnotesize Key words: Shimura variety, affine flag variety,
local model, group scheme, nearby cycles. }
\bigskip

\bigskip
 
\centerline{\sc Contents}
\medskip

\noindent  \ \ \  \ \ \ Introduction\\
\S 1.  \ Preliminaries\\
\S 2.  \ Reductive groups over $\O[u, u^{-1}]$\\
\S 3.  \ Parahoric group schemes over $\O[u]$ \\
\S 4.  \ Classical groups\\
\S 5.  \ Loop groups and affine Grassmannians\\
\S 6.  \ Local models\\
\S 7. \ Shimura varieties and   local models\\
\S 8. \ The special fibers of the local models\\
\S 9. \ Nearby cycles and the conjecture of Kottwitz\\
\S 10. Appendix: Homogeneous spaces\\
 \phantom{aa} \  \ References\\

\bigskip
\medskip
\vfill\eject

 \medskip
 \section*{Introduction}

The ``local models" of this paper are  projective schemes
 over the integers of a $p$-adic local field
 that are expected to model the singularities of
 integral models of Shimura varieties at places of (tame, parahoric)
bad reduction. This is meant in the  sense that each point on the
integral model of the Shimura variety should have an \'etale
neighborhood which
 is isomorphic to an \'etale neighborhood of a corresponding point on the local model.
 The simplest example is for the classical modular curve $X_0(p)$ with $\Gamma_0 (p)$-level structure.
In this case, the local model is obtained by blowing up the
projective line $\mP_{\mZ_p}^1$ over $\Spec(\mZ_p)$ at the origin
$0$ of the special fiber $\mP_{\mF_p}^1$. More generally, local
models for Shimura varieties of PEL type with parahoric level
structure were given by Rapoport and Zink in \cite{RapZinkBook}.
Their construction was tied to the description of the Shimura
variety as a moduli space of abelian schemes with additional
structures; the fact that they capture the singularities of this
moduli space is a consequence of the Grothendieck-Messing
deformation theory of abelian schemes. However, it was soon realized
that the Rapoport-Zink construction is not adequate when the group
of the Shimura variety is ramified at $p$  and in many cases of
orthogonal groups. Indeed, then the corresponding integral models
are often not flat (\cite{PaJAG}). In the ramified PEL case,
corrected integral models were considered and studied in
\cite{PaJAG}, \cite{PappasRaI}, \cite{PappasRaII},
\cite{PappasRaIII}. These are flat by definition but they do not
always have a neat moduli space interpretation.  Unfortunately, the
constructions in these works are somewhat ad-hoc and are mostly done
case-by-case: Ultimately, they are based on representing the
corresponding reductive group explicitly as the neutral component of
the group of automorphisms of a suitable bilinear (symmetric,
alternating, hermitian etc.) form on a vector space over a division
ring. Then, its parahoric subgroups are the connected stabilizers of
a self-dual chain of lattices as explained in \cite{BTclassI},
\cite{BTclassII} (see also \cite[Appendix to Ch. 3]{RapZinkBook}) and the
local models are given as flat closures in certain corresponding
linked Grassmannians. We refer the reader to the survey \cite{PRS}
for precise definitions and more references.

In this paper, we provide a general group theoretic definition of local
models that is not tied to a particular representation of the group.
This approach allows us to make significant progress and resolve several
open questions. Our local models are constructed starting
from the ``local Shimura data", i.e triples $(G, K, \{\mu\})$, where
$G$ is a (connected) reductive group over $\Q_p$, $ K\subset
G(\Q_p)$ a parahoric ``level" subgroup, and $\{\mu\}$ a geometric
conjugacy class of one parameter subgroups $\mu: \Gm_{/\overline
\Q_p}\to G_{\overline\Q_p}$  over an algebraic closure
$\overline\Q_p$. Here, we assume that $\mu$ is minuscule.  Denote by
$E$ the field of definition of the conjugacy class $\{\mu\}$. This
is the {\sl local reflex field}; it is a finite extension of $\Q_p$
and is contained in a splitting field for $G$.  We  will
assume\footnote{This is an important assumption that we keep
throughout the paper.}   that $G$ splits over a tamely ramified
extension of $\Q_p$. Let $\O_E$ be the ring of integers of $E$ and
$\ke$ its residue field. By definition, the local model ${\rm
M}^{\rm loc}$ is a projective scheme over $\Spec(\O_E)$ with generic
fiber the homogeneous space for $G_E$ that corresponds to $\mu$. Our
first main result is the following (a weaker version of a combination of Theorems \ref{CMfiber} and \ref{special fiber}):

\begin{thm}\label{thm01}
Suppose  that the prime $p$ does not divide the order of the
fundamental group $\pi_1(G({\overline \Q_p})_{\rm der})$.
Then   ${\rm M}^{\rm loc} $ is normal with reduced special fiber;
all the   irreducible components of the geometric special fiber ${\rm M}^{\rm loc}\otimes_{\O_E}\bar
{\mathbb F}_p$ are normal and
Cohen-Macaulay. In fact, ${\rm M}^{\rm loc}\otimes_{\O_E}\bar
{\mathbb F}_p$ can be identified with the reduced union of a finite
set of affine Schubert varieties in the affine flag variety ${\rm
Gr}_{\Gg, {\mathbb F}_p}\otimes_{{\mathbb F}_p}\bar {\mathbb F}_p$;
the set parametrizing this union is the ``$\mu$-admissible set"
defined by Kottwitz and Rapoport.
\end{thm}

The definition of ${\rm M}^{\rm loc}$ and of the affine flag variety
${\rm Gr}_{\Gg, {\mathbb F}_p}\otimes_{{\mathbb F}_p}\bar {\mathbb
F}_p$ will be explained below; the rest is given in \S \ref{ss8d}. The main ingredients in the proof of
this theorem are results of \cite{FaltingsLoops, PappasRaTwisted} on
the structure of Schubert varieties in  affine flag varieties and
the coherence conjecture of \cite{PappasRaTwisted} which was shown
by the second-named author in \cite{ZhuCoherence}.

In the case of Shimura data of PEL type, we show that the local
models ${\rm M}^{\rm loc}$ agree with the ``corrected local models"
which are obtained (in most cases at least) by taking the flat
closures of the ``naive local models" of Rapoport-Zink; these
corrected local models do describe the \'etale local structure of
corresponding integral models of Shimura varieties as we  discussed
in the beginning of the introduction.   For  PEL types the result in
the above theorem was conjectured and verified for a few special
cases in several papers (\cite{GortzFlatGLn},
\cite{GortzSymplectic}, \cite{PappasRaI}, \cite{PappasRaII},
\cite{PappasRaIII}, \cite{PappasRaTwisted}, see also \cite{PRS}). In
\cite{ZhuCoherence}, the theorem is proven for the ramified unitary
similitude groups.    Our approach here allows a unified treatment
in almost all cases.

 For example, let us explain a result that follows by combining the above with the work of Rapoport and Zink.
Suppose that ${\mathfrak D}=({\bold B}, \O_{\bold B}, $*$, {\bold
V}, (\ ,\ ), \mu, \{\L \}, K^p)$ give PEL data as in \S
\ref{PELremark}, \S \ref{sss8c4}, with corresponding group $\bold G$
over $\Q$ and reflex field $\boE$. Assume that $p$ is odd, that
$K^p\subset {\bold G}({\mathbb A}^{p}_f)$ is sufficiently small and
that the subgroup $ K=K_p$ of ${\bold G}(\Q_p)$ that stabilizes the
lattice chain $\{\L\}$
 is parahoric.  Suppose  that  $G={\bold G}_{\Q_p}$ splits over a tamely ramified extension of $\Q_p$
 and, in addition,  that ${\bold G} $ is connected. Let $\mathfrak P$ be a prime of $\boE$ over $p$. Under these assumptions, we obtain:

\begin{thm}\label{thmPEL}
The Shimura variety
$Sh_{\bold K}$ defined by the PEL data $\mathfrak D$
affords a flat integral model $\calS_{\bold K}$ over
$\O_{{\boE}_{\mathfrak P}}$ which is, locally for the \'etale topology,
isomorphic to the local model ${\rm M}^{\rm loc}$ for $(G, K,
\{\mu\})$. The scheme $\calS_{\bold K} $ is normal with reduced special fiber;
the geometric special fiber $\calS_{\bold K}\otimes_{\O_E}\bar
{\mathbb F}_p$ admits a stratification with locally closed
strata parametrized by the $\mu$-admissible set;
the closure of each stratum is normal and Cohen-Macaulay.
\end{thm}

In fact, the result is   more precise: We show the existence of a
``local model diagram" (see (\ref{locmoddiagram}));
also, the model
$\calS_{\bold K}$ is the flat closure of the generic fiber of the
corresponding Rapoport-Zink integral model and thus supports a
natural morphism to a Siegel moduli scheme. See \S  \ref{rem8c4}
  for the proof and for more similar results, in
particular for a discussion of cases in which $\bold G$ is not
connected.  (Note  that, in general, $Sh_{\bold K}$ above is  equal to a disjoint union of
Shimura varieties in the sense of Deligne; this is related to the failure of the Hasse principle,
see \cite[\S 8]{KottJAMS} and Remark \ref{sss8d5}.)
 In general, we
conjecture that the general Shimura variety with local data $(G, K,
\{\mu\})$ has an integral model which affords a local model diagram
and hence is locally for the \'etale
topology isomorphic to ${\rm M}^{\rm loc}$. Showing this, in   cases
of Shimura varieties of abelian type, is the subject of joint work in
preparation of the first author with M. Kisin \cite{K-P}.  This
combined with Theorem \ref{thm01} will then imply that the conclusion
of 
Theorem \ref{thmPEL} also holds for such Shimura varieties.

Before considering Kottwitz's conjecture, we will   explain our definition of   local
models. This uses the construction of certain
group schemes over a two-dimensional base and of their corresponding
(affine) flag varieties. We believe that these objects are of
independent interest and we begin by discussing them
in some detail.

We start by discussing the group schemes. Let $F$ be a $p$-adic
field with ring of integers $\O$ and residue field $k$. Suppose $G$
is a reductive group over  $F$, $K$ a parahoric subgroup of $G(F)$.
By definition, $K$ is the (connected) stabilizer of a point $x$ in
the Bruhat-Tits building ${\calB}(G, F)$ of $G(F)$. In \cite{BTII},
Bruhat-Tits construct a smooth group scheme $\P_x$ over the discrete
valuation ring $\O$ such that $K=\P_x(\O)$. Assume that $G$ splits
over a tamely ramified extension of $F$.  Choose a uniformizer
$\varpi$ of $\O$. In the first part of the paper, we construct a
smooth affine group scheme $\calG$ over the affine line ${\mathbb
A}^1_{\O }=\Spec(\O [u])$ which has connected fibers, is reductive
over the complement of $u=0$ and specializes to $\P_x$ along the
section $\Spec(\O )\to {\mathbb A}^1_{\O} $ given by $u\mapsto
\varpi$. In addition, the base changes of $\Gg$ by $\O [u]\to
F[[u]]$ and $\O[u]\to k[[u]]$ give corresponding Bruhat-Tits group
schemes over these two discrete valuation rings.

Having given $\Gg$, we can now define various mixed characteristic
versions of
 the familiar (from the theory of geometric Langlands correspondence)
 global and local versions of the affine Grasmannian and the affine flag variety.
 For simplicity, we set $X={\mathbb A}^1_\O=\Spec(\O[u])$.
 The main actor is the global affine Grassmannian ${\rm Gr}_{\Gg, X}$;
 this is the moduli functor on schemes over $X$ which to $y: S\to X$
 associates the set of isomorphism classes of
 $\Gg$-bundles over $X\times_\O S$ together with a trivialization on the complement of
  the graph of $y$. When $\O$ is replaced by a field and $\Gg$
  is a constant split reductive group, this is the global affine Grassmannian (over the affine line)
  of Beilinson and Drinfeld. We show that
  ${\rm Gr}_{\Gg, X}$ is represented by an   ind-scheme
  which is ind-proper over $X$.  Denote   by ${\rm Gr}_{\Gg, \O}\to \Spec(\O)$ the base change
 of ${\rm Gr}_{\Gg, X}\to X$ by $\Spec(\O)\to X$
 given by $u\mapsto\varpi$. We can easily see, using
  the descent lemma of Beauville and Laszlo, that the generic fiber
  ${\rm Gr}_{\Gg, \O}\otimes_\O F$ is
 isomorphic to the   affine Grassmannian ${\rm Gr}_{G, F}$
 for the loop group $G\otimes_F F((t))$, $t=u-\varpi$. Recall that ${\rm Gr}_{G, F}$ represents
 the fpqc sheaf given by $R\mapsto G(R((t)))/G(R[[t]])$. Similarly, the special fiber ${\rm Gr}_{\Gg, \O}\otimes_\O k$ is
 isomorphic to an affine flag variety ${\rm Gr}_{\Gg, k}$
for the group $\Gg(k((u)))$ over the local field $k((u))$ and its
parahoric subgroup $\Gg(k[[u]])$.\footnote{In
\cite{PappasRaTwisted}, affine flag/Grassmannian varieties for
groups that are not necessarily constant are referred to as
``twisted". Here we omit this adjective.}  Here ${\rm Gr}_{\Gg, k}$
represents $R\mapsto \Gg(R((u)))/\Gg(R[[u]])$.

 We are now ready to give our definition of the local model ${\rm M}^{\rm loc}$.
For this we take $F=\Q_p$ in the above and use the ``local Shimura
data" $(G, K, \{\mu\})$ with local reflex field $E$.  As above,
 we need to assume that $G$ splits over a tamely ramified extension of $\Q_p$.
Since $\Gm=\Spec(\Q_p[t, t^{-1}])$, the coweight $\mu$ provides
a $\overline\Q_p((t))$-valued
 point $s_\mu$ of $G$ and hence  a $\overline\Q_p$-valued
 point $[s_{\mu}]$ of ${\rm Gr}_{G, \Q_p}$. Because $\mu$ is minuscule, the (left) $G(\overline\Q_p[[t]])$-orbit
 of $[s_\mu]$ is a smooth projective variety $X_\mu$ (actually a homogeneous space for $G_E$)
   defined over the local reflex field $E$; $X_\mu$  is a closed subvariety of
   the affine Grassmannian ${\rm Gr}_{G, \Q_p}\otimes_{\Q_p} E=
   {\rm Gr}_{\Gg,\O}\otimes_\O E$. By definition, the local model ${\rm M}^{\rm loc}:=M_{\Gg, \mu}$ is the reduced projective scheme over $\Spec(\O_E)$
  given by the Zariski closure of $X_\mu\subset {\rm Gr}_{\Gg,\O}\otimes_\O E$ in the ind-scheme
  ${\rm Gr}_{\Gg, \O}\otimes_\O\O_E$.

  By construction, the special fiber ${\rm M}^{\rm
loc}\otimes_{\O_E}k_E$ is a closed subscheme of the affine flag
variety ${\rm Gr}_{\Gg, {\mathbb F}_p}\otimes_{{\mathbb F}_p} k_E$.
Let us remark here that the local models ${\rm M}^{\rm loc}$ are
given by taking a Zariski closure and, as a result, they do not
always have a neat moduli space interpretation. This issue does not
concern us in this paper. Indeed,  the close relation of ${\rm M}^{\rm loc}$ to the affine
Grassmannians allows us to show directly many favorable properties
as in Theorem \ref{thm01}. These then imply nice properties of
corresponding integral models for Shimura varieties as in Theorem
\ref{thmPEL}.

The same connection with the theory of  affine Grassmannians also
allows us to obtain results about the sheaf of nearby cycles ${\rm
R}\Psi(\overline\Q_\ell)$ of
 the scheme ${\rm M}^{\rm loc} \to \Spec(\O_E)$. (Here $\ell$ is a prime different from $p$.)
 We will describe these below. Recall that we conjecture  that ${\rm M}^{\rm loc}$ describes the \'etale local
structure of an integral model of the Shimura variety. Therefore,
the  nearby cycles of the local models should also be determining
the  nearby cycles for integral models of Shimura varieties with
parahoric level structure. (As follows from the above, this is
indeed the case for most PEL types.) Our results will
  be useful in expressing the local factor of the
Hasse-Weil zeta function of the Shimura variety at  places of (tame)
parahoric reduction as a product of local factors of automorphic
L-functions. The strategy of using the local model to determine the
(semi-simple) Hasse-Weil zeta function was first suggested by
Rapoport \cite{RapoportGuide}. It has since been advanced by
Kottwitz, Haines-Ng\^o (\cite{HainesNgoNearby}),  Haines  and
others. A centerpiece of this approach is a conjecture of Kottwitz
that states that, in the case that $G$ is split, the semi-simple
trace of Frobenius on the sheaf of nearby cycles gives a central
function in the parahoric Hecke algebra. This was proven by Haines
and Ng\^o (\cite{HainesNgoNearby}) in (split) types A and C by
following an argument of Gaitsgory \cite{GaitsgoryInv}. Gaitsgory
proved a stronger statement for general split groups in the function
field case; he showed that the perverse sheaf of nearby cycles
satisfies a commutativity constraint with respect to the convolution
product. This implies Kottwitz's conjecture in the function field
case. The main tools in Gaitsgory's approach are various versions of
the global affine Grassmannian of Beilinson-Drinfeld. In this paper,
we are able to generalize and simplify the approaches of both
Gaitsgory and Haines-Ng\^o by using the mixed characteristic affine
Grassmannians ${\rm Gr}_{\Gg, X}$ and various other related
ind-schemes. In particular, we  obtain a general result even for
non-split groups as follows:

Our  construction of the group scheme $\Gg$ over $X=\Spec(\Z_p[u])$ also provides us with a
reductive group $G'=\Gg\times_X\Spec({\mathbb F}_p((u)))$ over
${\mathbb F}_p((u))$ and a parahoric subgroup $K'$ which correspond
to $G$ and $K$ respectively. By definition, if $ {\mathbb
F}_q\supset k_E$, we have an equivariant embedding ${\rm M}^{\rm
loc}\otimes_{\O_E}{\mathbb F}_q\subset {\rm Gr}_{\Gg, {\mathbb
F}_p}\otimes_{{\mathbb F}_p} {\mathbb F}_q={\rm Gr}_{G', {\mathbb
F}_p}\otimes_{{\mathbb F}_p} {\mathbb F}_q$. This allows us to view
the Frobenius trace function $\tau^{\on{ss}}_{{\rm
R}\Psi}(x)=\on{tr}^{\rm ss}({\rm Frob}_x, {\rm
R}\Psi(\overline{\Q}_\ell)_{\bar x})$, $x\in {\rm M}^{\rm
loc}({\mathbb F}_q)$, as an element of the parahoric Hecke algebra
$\calH_q(G', K')$
 of bi-$K'({\mathbb F}_q[[u]])$-invariant, compactly supported locally constant
$\overline{\Q}_{\ell}$-valued functions on $G'(\bbF_q((u)))$.

\begin{thm}(Kottwitz's conjecture)\label{thm02}
The  semi-simple trace of Frobenius on the sheaf of nearby cycles
${\rm R}\Psi(\overline\Q_\ell)$ of ${\rm M}^{\rm loc}\to\Spec(\O_E)$
gives a central function $\tau^{\on{ss}}_{{\rm R}\Psi}$ in the
parahoric Hecke algebra $\calH_q(G', K')$.
\end{thm}

See \S \ref{sstraceSect} for more details and in particular
Theorem \ref{thm9.13} for the precise statement which is more general. In the split Iwahori case, as
a corollary of this theorem, one can give an explicit formula for
the semi-simple Frobenius trace function using Bernstein's
presentation of the Iwahori-Hecke algebra, as was explained in
\cite{HainesNgoNearby}. In fact, more generally,
 when $G$ is unramified and the subgroup an arbitrary parahoric,
 we show that the semi-simple trace can be
expressed as a Bernstein function in the parahoric Hecke algebra as
was also conjectured by  Kottwitz and by Haines. Let us mention here
that Kottwitz's conjecture for quasi-split (but not split)
unramified unitary groups was also shown independently by S.~Rostami
in his thesis \cite{RostamiThesis}.

We also obtain results for quasi-split ramified groups when the
level subgroup is {\sl very special}. Here and elsewhere, we say that $K\subset G(\Q_p)$ is very special, if the corresponding
parahoric subgroup of $G(\Q_p^{\rm unr})$ is special in the sense of
Bruhat-Tits; then $K$ itself is also special. In this case, we give a
characterization of the   function $\tau^{\on{ss}}_{{\rm R}\Psi}$ by
identifying its trace on unramified representations; this also
agrees with the prediction of a conjecture of Haines and Kottwitz.
Corresponding to $\mu$ we have a minuscule coweight  for $G'$ which
for simplicity we   still denote by $\mu$. The conjugacy class
$\{\mu\}$ defines an algebraic $\overline\Q_\ell$-representation
$V_{\mu}$ of the  Langlands dual group ${}^LG'=H^\vee\rtimes {\rm
Gal}_{{\mathbb F}_q((u))}$ of $G'$ over $\bbF_q((u))$. (Here $H$ is
the Chevalley split form of $G'$). The inertia invariants $V_{\mu
}^{I}$ give a representation of $(H^\vee)^I\rtimes {\rm
Gal}(\overline {\mathbb F}_q/{\mathbb F}_q)$. If $\pi$ is an
irreducible smooth representation of $G'({\mathbb F}_q((u)))$ with a
$K'({\mathbb F}_q[[u]])$-fixed vector, one can define its
Langlands-Satake parameter $\on{Sat}(\pi): W\to {}^LG'$: Among other
properties, one can show that if $\Phi_q\in W $ is a (geometric)
Frobenius element in the Weil group, then $\on{Sat}(\pi)(\Phi_q)$ is
a semi-simple element in $(H^\vee)^{I }\times  {\rm Frob}_q^{-1} $,
which is well-defined up to $(H^\vee)^{I }$-conjugacy and completely
determines $\on{Sat}(\pi)$. Our characterization of the Frobenius
trace function is the identity
\begin{equation}\label{eq0.1}
\on{tr}(\pi(\tau^{\on{ss}}_{{\rm
R}\Psi}))=\on{tr}(\on{Sat}(\pi)(\Phi_q), V_{\mu} ^{I })
\end{equation}
for all $\pi$ as above. This is shown by combining our constructions
with the results in \cite{ZhuSatake}, \cite{RiZh}. See the last
section of the paper for more details.

As we mentioned above, when
the Shimura data are of PEL type, the local model does indeed
describe the \'etale local structure of   an integral model of the Shimura
variety. Therefore, in this case, our
results also give the semi-simple trace of Frobenius on the nearby cycles of an
integral model of the Shimura variety. One can then apply them to
the calculation of the local factor of the Hasse-Weil zeta function
of the Shimura variety following the arguments of Kottwitz,
Rapoport and Haines (e.g \cite{RapoportGuide}, \cite{HainesSurvey}), but we will not go
into this here. (We also expect that this approach will
 be extended to many Shimura varieties
 with parahoric level which are not of PEL type using \cite{K-P}.)
 Here we should also mention very recent results of Scholze \cite{ScholzeLK} and Scholze and Shin
 \cite{ScholzeShin} that make progress towards this calculation without having to
 explicitly identify the
 semi-simple trace.

Finally, we give several results about the action of monodromy (i.e
of inertia) on the sheaves of nearby cycles. These also imply corresponding results for
Shimura varieties. Here is an example:

\begin{thm}\label{thm03}
Assume that $G$ is   split over the (tamely ramified) extension $F/\Q_p$
and that
 $K$ is a   very special 
subgroup of $G(\Q_p)$. Then the
 inertia subgroup $I_F={\rm Gal}(\overline\Q_p/ \Q_p^{\rm
unr}F)$ acts trivially
  on the sheaf of nearby cycles ${\rm R}\Psi(\overline\Q_{ \ell})$ of
${\rm M}^{\rm loc}\to\Spec(\O_E)$.
\end{thm}

 More generally, without the assumption that $K$ is very special,
we show that the action of $I_F$ on the sheaf of nearby cycles ${\rm
R}\Psi(\overline\Q_{ \ell})$ is unipotent. In this  case, we
can describe the action of the full inertia group $I_E$ in terms of
the geometric Satake equivalence for ramified groups of
\cite{ZhuSatake}. This is obtained by comparing  nearby cycles along
two directions over the two-dimensional base $\O[u]$ and is
 used in proving the identity
(\ref{eq0.1}) above. In the case that
 $G={\rm Res}_{F/\Q_p}{\rm GL}_n$ with $F/\Q_p$ tame, this description of the inertia action
confirms a conjecture in \cite{PappasRaI}.

Recall that,  throughout the paper, we have restricted to the case that the
group $G$ splits over a tamely ramified extension of the base field.
This assumption is important for the construction of the group
scheme $\Gg$ over $\O[u]$; it is always satisfied when $p\geq 5$ and $G$ is absolutely simple
and is either adjoint or simply connected. A combination of our methods 
with the idea of  splitting models from \cite{PappasRaII} can be used  to also deal with groups
that are Weil restrictions of scalars of tame groups down 
from a (possibly) wildly ramified extension. In particular, one can define local models
and show Theorem \ref{thm01} in some such cases.\footnote{
See the forthcoming thesis of B. Levin.} Of course, a general reductive group is
isogenous to a product of Weil restrictions of absolutely simple adjoint groups
and so this would allow us to handle most cases.

Let us now give an overview of the various sections of the
paper and explain some aspects of our constructions.
 In \S 1 we discuss preliminaries on reductive groups over local fields; the main emphasis
 is on obtaining forms of a split group by descent
 as inner twists of a quasi-split form.
In \S 2 we generalize this approach and give  a certain reductive group
 scheme $\und G$ over $\O[u, u^{-1}]$ which
 specializes to $G$ under the base change $\O[u, u^{-1}]\to F$ given by
 $u\mapsto \varpi$. Here, a crucial observation   is
 that $\O[u, u^{-1}]\to F$ identifies the \'etale fundamental group of $\O[u,u^{-1}]$
 with the tame quotient of the Galois group of $F$; our tameness hypothesis enters this way. Our construction of the ``parahoric type" smooth affine
 group schemes $\Gg$ over $\O[u]$ is given in \S 3 (Theorem \ref{grpschemeThm}).
 Here is a brief outline of the construction of $\Gg$: When $G$ is
split over $F$ the existence of such a smooth group
 scheme $\Gg$ follows from \cite{BTII}. However,  showing that this is affine is not so straightforward.
 We do this in two stages. In the case that $K$ is contained in a hyperspecial
 subgroup, we realize $\Gg$ as a   dilatation of the
 corresponding Chevalley group scheme over $\O[u]$ along a subgroup supported along $u=0$.
 In general, still for $G$ split, we reduce to the above case by using the fact that there is always a finite ramified extension $L/F$
 such that the  stabilizer of $x$ in $G(L)$ is contained in a hyperspecial subgroup. Next, we construct $\Gg$
 when $G$ is quasi-split (and splits over a tamely ramified extension) by taking fixed points of
 a corresponding group scheme for the split form. Finally, the general case is obtained by descent
 along an \'etale extension of $\O[u]$; this descent resembles a corresponding argument in \cite{BTII}.
  In \S 4, we give several examples and eventually explain how,
 when $G$ is a classical group and $K$ is the connected stabilizer of a self-dual lattice chain (cf.
  \cite{BTclassI}, \cite{BTclassII}), we can realize concretely
the group schemes $\Gg$ as (the neutral components of) automorphisms of certain self-dual $\O[u]$-lattice
chains.

In \S 5, we prove that the global affine Grassmannian
  ${\rm Gr}_{\Gg, X}$ is represented by an   ind-scheme
  which is eventually shown to be ind-proper over $X$. The strategy here is to first demonstrate that we can
  find a faithful linear representation $\Gg\rightarrow {\rm GL}_n$ such that the quotient ${\rm GL}_n/\Gg$
  is represented by a quasi-affine scheme. Given this we can  reduce
  the proof of representability of ${\rm Gr}_{\Gg, X}$ to
  the standard case of $\Gg={\rm GL}_n$.
Here we are dealing with objects over the two-dimensional base
  $X={\mathbb A}^1_\O$ and we need to work harder than in the usual
situation in which the base is a smooth curve over a field. For example, it is not trivial to show that a smooth group scheme $\Gg$ with connected fibers over a two-dimensional regular base can be embedded as a closed subgroup scheme of ${\rm GL}_n$, as was proven by Thomason \cite{ThomasonEqRes}.
We can upgrade this to also show that there is such an embedding
with ${\rm GL}_n/\Gg$ a quasi-affine scheme; this is done in the appendix (\S 10).
Also, general
homogeneous spaces over a (regular) two dimensional base are not
always represented by schemes; this complicates our proof.
 Raynaud
asked whether the quotients of the form $\Hh/\Gg$, where
  $\Gg$ is a closed subgroup scheme of $\Hh$ and both $\Gg$ and $\Hh$
  are smooth affine with connected fibers over a normal base, are represented by schemes.
  In the appendix, we also give an affirmative answer in the case
  when the base is excellent regular Noetherian and of Krull dimension two.

Our
definition of the local models and certain generalizations is then given in
\S 6. In \S 7, we describe the relation of local models to integral
models of Shimura varieties and explain why in the PEL case our
local models coincide with the (corrected) local models of
Rapoport-Zink, Pappas-Rapoport and others. This uses our description
of the group schemes $\Gg$ for classical groups given in \S 4 and the work of Rapoport and Zink.
 In
  \S 8 we show  Theorem \ref{thm01} and other related results.
Finally, \S 9 is devoted to the study of nearby cycles and of the
semi-simple trace of Frobenius; we show Theorem \ref{thm02}, Theorem
\ref{thm03} and
 the identity (\ref{eq0.1}).
 \smallskip

\noindent{\bf Notations:}  In this paper, $\O$ denotes a discrete
valuation ring with fraction field $F$ and perfect residue field $k$
of characteristic $p>0$. Most of the time $F$ will be a finite
extension of the field of $p$-adic numbers. Usually, $G$ will denote
a (connected) reductive group over $F$ and $\calB(G, F)$ will be the
Bruhat-Tits building of $G(F)$ (here by this we mean the ``enlarged"
building). 
As usual, we will denote by $G_{\rm der}$, resp. $G_{\rm ad}$,  the
derived subgroup, resp. the adjoint quotient of $G$, and by $G_{\rm
sc}$ the simply-connected cover of $G_{\rm der}$. If $A$ is an
affine algebraic group over a field $L$, we will use  $\xch(A)$
(resp. $\xcoch(A)$) to denote the character group (resp. cocharacter
group) over a separable closure $\bar L^s$. Then the Galois group
${\rm Gal}(\bar L^s/L)$ acts on $\xch(A)$ and $\xcoch(A)$. We will
often write $X=\Spec(\O[u])$ for the affine line over $\Spec(\O)$.
We will denote by $R[[u]]$ the ring of formal power series ring in
the variable $u$ with coefficients in $R$, $R((u))=R[[u]][u^{-1}]$
is the corresponding ring of formal Laurent power series. If $Z$ is
a set with an action of a group $\Gamma$, we will write $Z^\Gamma$
for the subset of elements $z$ of $Z$ which are fixed by the action:
$\gamma\cdot z=z$ for all $\gamma\in \Gamma$.
\smallskip

\noindent{\bf Acknowledgements:} The authors would like to warmly
thank M.~Rapoport, B.~Conrad, T.~Haines and  B. Levin for useful
discussions and comments.

\bigskip

\section{Preliminaries}\label{Preliminaries}

\setcounter{equation}{0}

\subsection{} In all this chapter, unless we say
to the contrary, we suppose that $F$ is either a $p$-adic field (i.e
a finite extension of $\Q_p$), or the field of Laurent power series
$k((t))$ with $k$ finite or algebraically closed of characteristic
$p$. Recall $\O$ is the valuation ring of $F$. We fix a separable
closure $\bar F^s$ of $F$ and denote by $\breve{F}$ the maximal
unramified extension of $F$ in $\bar F^s$.

\subsection{Pinnings and quasi-split forms}

\subsubsection{}\label{sss1a1} We refer to \cite{SGA3}  or \cite{ConradNotes} for background on
reductive group schemes over a general base. Recall that a pinned Chevalley group over $\O$ is the data $(H,
T_H, B_H, e)$ where $H$ is a Chevalley (reductive, connected, split)
group scheme over $\O$, $B_H$ is a Borel subgroup scheme of $H$,
$T_H$ a maximal split torus contained in $B_H$ and $e=\sum_{ a\in
\Delta} e_{ a}\in {\rm Lie}(B_H)$, where $\Delta$ is the set of
simple roots, $e_{ a}$ is a generator of the rank $1$ $\O$-module $
{\rm Lie}({  U}_{ a})$. Here, $ U_{ a}$ is the root group
corresponding to $ a$. The group of automorphisms $\Xi_H={\rm
Aut}(\Sigma_H)$ of the based root datum $\Sigma_H=(\xch(T_H), \Delta,
\xcoch(T_H), \Delta^\vee)$ is canonically isomorphic to the group of
automorphisms of $(H, T_H, B_H, e)$. We will call an element
$\gamma$ of $\Xi_H$ a diagram automorphism of $(H, T_H, B_H, e)$.

\subsubsection{}\label{sssPQS}  Let $S$ be a (finite type, separated) scheme over $\Spec(\O)$.

\begin{Definition}\label{defPQS}
A pinned quasi-split (isotrivial) form of the group $H$ over $S$ is
a quadruple $(
 \und G, \und T, \und B, \und e)$, where $\und G$ is a reductive group scheme  over
$S$ (see \cite{SGA3}), $ \und T\subset  \und B$ are closed subgroup
schemes of $\und G$ and $\und e\in{\rm Lie} (\und B)$ a section such
that locally for the finite \'etale topology of $S$,  $(\und G,\und
T,\und B,\und e) \simeq (H,T_H,B_H, e)\times_\O S$. We will denote
the groupoid of pinned quasi-split (isotrivial) forms of the group
$H$ over $S$ by ${\rm PQ}(H, S)$.
\end{Definition}

\begin{Remark}
{\rm In this paper we only need isotrivial forms, i.e forms that
split after a finite \'etale extension. We could also consider forms
that split over a general \'etale extension but we choose not to do
this here.
}
\end{Remark}

We will give a more combinatorial description of ${\rm PQ}(H, S)$. We
say a scheme $\eta$ locally of finite presentation over $S$ is an
isotrivial $\Xi_H$-torsor if it admits a right action of $\Xi_H$ and
can be trivialized, i.e $\eta\times_S S'\simeq \Xi_H\times S'$, after a finite \'etale 
base change $S'\to S$.

\begin{prop}\label{propPQS} The category ${\rm PQ}(H, S)$ is equivalent to the
category given by quadruples $(T_H,B_H, e, \eta)$, where
$(T_H,B_H, e)$ is a pinning of $H$ and $\eta \to S$ is an isotrivial
$\Xi_H$-torsor.
\end{prop}
\begin{proof}
Indeed, given such a $\Xi_H$-torsor $\eta$ we can construct a
reductive group scheme $\und G$ over $S$ together with subgroup
schemes $\und T, \und B$ and $\und e\in {\rm Lie}(\und B)$ via
twisting. Conversely, let $(\und G,\und T, \und B, \und e)$ be a
quadruple corresponding to an object of ${\rm PQ}(H, S)$. Let
\[\eta(S')={\rm Isom}((\und G_{S'},\und T_{S'},\und B_{S'},\und e_{S'}), (H ,T_{H },B_{H },e )\times_\O {S'}).\]
This is represented by a closed subscheme of ${\rm Isom}(\und G,
H_S)$; this isomorphism scheme is a $\und{\rm Aut}(H)$-torsor, and
the latter group scheme is separated and smooth over $S$. Therefore,
the morphism $\eta \to S$ is also separated and locally of finite
presentation;  the rest follows from the above.
\end{proof}

\subsubsection{} \label{groupPQS}
Now let $s$ be a geometric point of $S$. Recall that there is an
equivalence between the category of isotrivial $\Xi_H$-torsors on
$S$ and the following groupoid: objects are continuous group
homomorphisms $\rho:\pi_1(S,s)\to \Xi_H$, and the morphisms between
$\rho_1$ and $\rho_2$ are elements $h\in \Xi_H$ such that
$\rho_2=h\rho_1h^{-1}$. Therefore, we can describe ${\rm PQ}(H, S)$ as
the category of quadruples $(T,B, e, \rho: \pi_1(S,s)\to \Xi_H)$.

This description of ${\rm PQ}(H,S)$ admits an immediate generalization
as follows. Let $\Gamma$ be a profinite group. We let ${\rm
PQ}(H,\Gamma)$ be the category of quadruples $(T,B, e,\rho:\Gamma
\to \Xi_H)$. If $\pi:\pi_1(S,s)\to \Gamma$ is a continuous group
homomorphism, there is a functor ${\rm PQ}(H,\Gamma)\to {\rm PQ}(H,S)$
which is a full embedding if $\pi$ is surjective. We denote the
image of this functor by ${\rm PQ}^\Gamma(H, S)$.

Now, suppose $S_1\to S_2$ is a morphism of schemes, and $s_1$, $s_2$
corresponding geometric points of $S_1$ and $S_2$ so that we have
$\pi_1(S_1,s_1)\to \pi_1(S_2,s_2)$. Suppose that there is a
surjective map $\pi_1(S_2,s_2)\to \Gamma$ such that the composition
$\pi_1(S_1,s_1)\to \Gamma$ is also surjective. Then pullback along
$S_1\to S_2$ induces an equivalence of categories
\begin{equation}\label{ssstame}
{\rm PQ}^\Gamma(H, S_2)\xrightarrow{\sim} {\rm PQ}^{\Gamma}(H, S_1).
\end{equation}

\subsection{Fixed points}
Here, we give some useful statements about the fixed points of an
automorphism of a Chevalley group scheme.

\subsubsection{}

If $\gamma$ is an automorphism of a separated scheme $Z\to S$, we will denote by
$Z^\gamma$ the closed subscheme (\cite[Prop. 3.1]{EdixhTame}) of fixed points for the action of
$\gamma $ on $Z$ so that,  by definition,  we have
$Z^\gamma(R)=Z(R)^\gamma:=\{z\in Z(R)\ |\ \gamma\cdot z=z\}$. We
start with the useful:

\begin{prop}\label{locconstant}
Let $H$ be a Chevalley group scheme (affine connected, reductive and split)
over $\O$ with a pair  $(T, B)$ of a maximal split torus $T$ and a
Borel subgroup scheme $B$ that contains $T$. Suppose that $\gamma$
is an automorphism of $H$ of order $e$ prime to $p$ that preserves both
$T$ and $B$. Denote by $\Gamma=\langle \gamma\rangle$ the group
generated by $\gamma$. Suppose that $E$ is a separably closed field
which is an $\O$-algebra. Then the  group of connected components
$\pi_0(H_E^\gamma)$ of $H_E^\gamma$ is commutative with order that
divides $e$ and is independent of $E$.
 \end{prop}

\begin{proof} We start with the case of tori.
When $E=\C$ the isomorphisms in the Lemma below can be found in
\cite[Lemma 2.2]{KottCuspidal}:

\begin{lemma}\label{torus}
Suppose that $T$ is a split torus over $\O$ which supports an action
of the cyclic group $\Gamma=\langle \gamma\rangle$ of order $e$
prime to $p$. Let $E$ be a separably closed field which is an
$\O$-algebra. We have
$$
\pi_0((T_E)^\gamma)\simeq {\rm H}^1(\Gamma, \xcoch(T)), \quad {\rm
H}^1(\Gamma, T(E))\simeq {\rm H}^2(\Gamma, \xcoch(T)),
$$
and so in particular these groups are finite of order annihilated by
$e$ and are  independent of $E$. The fppf sheaf over $\Spec(\O)$
associated to $R\mapsto \rH^1(\Gamma, T(R))$ is represented by a
finite \'etale commutative group scheme of rank equal to the order
of ${\rm H}^2(\Gamma, \xcoch(T))$.
\end{lemma}

\begin{proof}
Consider the norm homomorphism $N=\prod_{i=0}^{e-1}\gamma^i: T \to T
$. By comparing Lie algebras, we can see that the image $N(T_E)$ is
the connected component $(T^\gamma_E)^0$ and so
$\pi_0(T^\gamma_E)=T^\gamma(E)/N(T(E))\simeq \rH^2(\Gamma,
T(E))=\rH^2(\Gamma, \xcoch(T)\otimes_\Z E^*)$. Now consider the
group $\mu_{e^\infty}(E)$ of roots of unity of order a power of $e$
in $E$. The quotient $E^*/\mu_{e^\infty}$ is uniquely  divisible by
powers of $e$ and we can conclude that $\rH^2(\Gamma,
\xcoch(T)\otimes_\Z E^*)\simeq \rH^2(\Gamma, \xcoch(T)\otimes_\Z
\mu_{e^\infty})$. Since $\mu_{e^\infty}(E)\simeq \prod_{l|e}\Q_l/\Z_l$ and
$\Q_l$ is similarly uniquely divisible, we obtain $\rH^2(\Gamma,
\xcoch(T)\otimes_\Z \mu_{e^\infty})\simeq \rH^1(\Gamma,
\xcoch(T)\otimes_\Z \prod_{l|e}\Z_l)=\rH^1(\Gamma, \xcoch(T))$. The proof of
the second isomorphism is similar. Now, we can see that the fppf
sheaf associated to the presheaf $R\mapsto \rH^1(\Gamma, T(R))$ is
given by the  quotient group scheme ${\rm ker}(N)/T^{\gamma-1}$ and
so the last statement also follows from the above.
\end{proof}

\begin{Remark}\label{flasque}
{\rm Following \cite{ColliotFlasques}, we  call a split torus $T$
over $\O$ with $\Gamma$-action $\Gamma$-{\sl quasi-trivial}, if
$\xcoch(T)$ is a permutation $\Gamma$-module, i.e if $\xcoch(T)$ has
a $\Z$-basis which is stable under the action of $\Gamma$. We will
call the split torus $T$ over $\O$  with $\Gamma$-action
$\Gamma$-{\sl flasque} if for all subgroups $\Gamma'\subset\Gamma$,
we have ${\rm H}^1(\Gamma', \xcoch(T))=(0)$. Notice that by
Shapiro's Lemma, if $T$ is $\Gamma$-quasi-trivial, then $T$ is also
$\Gamma$-flasque. By the above lemma, if $T$ is $\Gamma$-flasque,
then $T^\gamma$ is connected.  }
\end{Remark}

We now continue with the proof of Proposition \ref{locconstant}.

We will first discuss the cohomology set $\rH^1(\Gamma, H(E))$. We
will show that  $\rH^1(\Gamma, H(E))$ is finite of cardinality
independent of $E$. For simplicity, we will omit $E$ from the
notation and write $H$ instead of $H(E)$ etc. By \cite[Lemma
7.3]{SteinbergMemoirs}, every element of $H$ is $\gamma$-conjugate
to an element of $B$. This implies that the natural map
$\rH^1(\Gamma, B)\to \rH^1(\Gamma, H)$ is surjective. Now notice
that $\rH^1(\Gamma, T)=\rH^1(\Gamma, B)$. (Indeed, $\rH^1(\Gamma,
U)=(0)$ because $\Gamma$ has order prime to the characteristic of
$E$. In fact, for the same reason, $\rH^1(\Gamma, {}_cU)=(0)$ where
${}_cU$ is $U$ with $\gamma$-action twisted by a  cocycle $c$ in
$Z^1(\Gamma, B)$. Since $T=B/U$,  the long exact sequence of
cohomology and \cite[Cor. 2, I. \S 5]{SerreGaloisCoh} implies that
$\rH^1(\Gamma, B)\to \rH^1(\Gamma, T)$ is injective. The splitting
$T\to B$ shows that this map is also surjective.) Now suppose that
$t$, $t'\in T$
 give cohomologous $1$-cochains in $H$, i.e
$$
t'=x t \gamma(x)^{-1}, \ \ \hbox{\rm for some $x\in H$}.
$$
Using $H = \sqcup_{n\in N} UnU$, where $N$ is the normalizer of $T$,
we write $x=u_1nu_2$. Then we get
$$
t'=u_1nu_2 t\gamma(u_2)^{-1}\gamma(n)^{-1}\gamma(u_1)^{-1}, \quad
\hbox{\rm or}\quad t'\gamma(u_1)\gamma(n)\gamma(u_2)=u_1nu_2t.
$$
Since $T$ normalizes $U$, this implies $t'\gamma(n)=nt$, so
$t'=nt\gamma(n)^{-1}$. This shows that two classes $[t']$, $[t]$ in
$\rH^1(\Gamma, T)$ are identified in $\rH^1(\Gamma, H)$ if and only
if they are identified in $\rH^1(\Gamma, N)$. Now use the exact
sequence
$$
1\to T^\gamma\to N^\gamma\to W^\gamma\to \rH^1(\Gamma, T)\to
\rH^1(\Gamma, N),
$$
\cite[Prop. 39 (ii), I. \S 5]{SerreGaloisCoh}, and the above to
conclude that $\rH^1(\Gamma, H)$ can be identified with the set of
orbits of a right action of $W^\gamma$ on $\rH^1(\Gamma, T)$. This
action is given as follows (\cite[I. \S 5, 5.5]{SerreGaloisCoh}).
Suppose $w$ is in $W^\gamma$; lift $w$ to $n\in N$ and consider
$t_w:=n^{-1}\gamma(n)\in T$. Then we set $[t]\cdot w=[n^{-1}tn\cdot
t_w]=[n^{-1}t\gamma(n)]$. By Lemma \ref{torus}, $\rH^1(\Gamma,
T)\simeq \rH^2(\Gamma, X_*(T))$ is independent of $E$. We can now
easily see, by picking the lift $n$ of $w\in W^\gamma$ in $N(\O)$,
that the set of orbits $\rH^1(\Gamma, T)/W^\gamma$ is also
independent of the field $E$.

Let us now  consider the group of connected components
$\pi_0(H^\gamma_E)$.

Write $T_{\rm der}$, $T_{\rm sc}$, resp. $B_{\rm der}$, $B_{\rm
sc}$, for the preimages of $T$, $B$, in the derived group $H_{\rm
der}$, simply-connected cover $H_{\rm sc}$ of the derived group of
$H$. The automorphism $\gamma$ gives corresponding automorphisms of
$H_{\rm der}$, $H_{\rm sc}$ that preserve the pairs $(T_{\rm der},
B_{\rm der})$, $(T_{\rm sc}, B_{\rm sc})$. By \cite[Theorem
8.2]{SteinbergMemoirs}, $H^\gamma_{\rm sc}$ is connected. Following
the arguments in the proof of \cite[Prop.-Def. 3.1]{ColliotFlasques}
mutatis-mutandis, we see that we can find a central extension of
reductive split groups with $\Gamma$-action over $\O$
\begin{equation}\label{flasRes1}
1\to S\to H'\to H\to 1
\end{equation}
where $S$ is a $\Gamma$-flasque torus, $H'_{\rm der}$ is simply
connected and $H'/H'_{\rm der}=D$ is a $\Gamma$-quasi-trivial torus.
In fact, we can make sure that $\gamma$ preserves a corresponding
pair $(T', B')$ of $H'$ with $1\to S\to T'\to T\to 1$. As before, we
will now use the same letters to denote the base changes of these
groups to $E$. We see that the exact sequence
$$
1\to H'_{\rm der}\to H'\to D\to 1
$$
gives
$$
1\to H'^\gamma_{\rm der}\to H'^\gamma\to D^\gamma\to \rH^1(\Gamma,
H'_{\rm der}).
$$
By \cite[Theorem 8.2]{SteinbergMemoirs},  $H'^\gamma_{\rm der}$ is
connected; similarly $D^\gamma$ is connected by Lemma \ref{torus},
cf. Remark \ref{flasque}. We can conclude that $H'^\gamma$ is
connected. Now (\ref{flasRes1}) gives
$$
1\to S^\gamma\to H'^\gamma\to H^\gamma\to \rH^1(\Gamma, S)\to
\rH^1(\Gamma, H').
$$
Since both $S^\gamma$ and $H'^\gamma$ are connected (by Remark
\ref{flasque} and the above respectively), we obtain an exact sequence
$$
1\to \pi_0(H^\gamma)\to \rH^1(\Gamma, S)\to \rH^1(\Gamma, H').
$$
Since $S$ is central in $H'$, the connecting map $\pi_0(H^\gamma)\to
\rH^1(\Gamma, S)$ is a group homomorphism and $\pi_0(H^\gamma)$ is
identified with the subgroup of elements of $\rH^1(\Gamma, S)$ that
map to the trivial class in $\rH^1(\Gamma, H')$. We can now conclude
using the above results on $\rH^1$ applied to $H=S$, $H=H'$.
\end{proof}

\subsection{Reductive groups over local fields.}\label{ss1a}
In this section, we give some preliminaries on reductive groups over
the local field $F$. In particular, we explain how we can present,
using descent, such a group $G$ as a form of a split group $H$.

\subsubsection{}\label{sss1a2a} Let $G$ be a connective reductive group over $F$.
Suppose that $G$ splits over the finite Galois extension $\tilde
F/F$ with Galois group $\Gamma={\rm Gal}(\ti F/F)$. Denote by $H$
the Chevalley group scheme over $\Z$ which is the split form of $G$,
i.e we have $G\otimes_F \ti F\simeq H\otimes_\Z \ti F$. In what
follows, we fix a pinning $(T_H, B_H, e)$ of $H$ over $\O$.

\subsubsection{}\label{sss1a3} Let $A$ be a maximal $F$-split torus of $G$.
Also, let $S$ be a maximal $\Fsm$-split torus in $G$ which contains
$A$ and is defined over $F$. Such a torus exists by
\cite[5.1.12]{BTII}. Let $T=Z_G(S)$, $M=Z_G(A)$. Since $G_{\Fsm}$ is
quasi-split, $T$ is a maximal torus of $G$ which is defined over $F$
and splits over $\ti F$.

As in \cite[16.4.7]{SpringerBook}, by adjusting the isomorphism
$G\otimes_F \ti F\simeq H\otimes_\Z \ti F$, we may identify $T$ with
the maximal split torus $T_H\otimes_\O\ti F$ of the split form $H$
of $G$ given by the pinning. We now represent the indexed root datum
(\cite[16.2]{SpringerBook}) of $G$ over $F$ as a collection
$(\xch(T), \Delta, \xcoch(T), \Delta^\vee, \Delta_0, \tau)$ where
$\tau:  \Gamma\to \Xi_H$ is a group homomorphism and $\Delta_0$ is a
subset of $\Delta$ which is stable under the  action of the group
$\Gamma$ via $\tau$. Then using Proposition \ref{propPQS} we see
that $\tau$ together with the pinning $(T_H, B_H, e)$ of $H$ gives a
pinned quasi-split form $(G^*, T^*, B^*, e^*)$ of $H$ over $F$.
Concretely,
$$
G^*=({\rm Res}_{\ti F/F}(H\otimes_F\ti F))^\Gamma,
$$
where $\gamma\in \Gamma$ acts on $H\otimes_F\ti F$ as
$\tau(\gamma)\otimes \gamma$. This is the quasi-split form of $G$
over $F$.

Let $B_H\subset P_0$ be the standard parabolic subgroup of $H$ that
corresponds to $\Delta_0$ and let $M_0$ be the Levi subgroup of
$P_0$. Since $\tau$ leaves $\Delta_0$ stable, we can also similarly
consider $M^*=M^*_0=({\rm Res}_{\ti F/F}(M_0\otimes_\O\ti
F))^\Gamma$ which is also quasi-split over $F$. Then  $T^*=({\rm
Res}_{\ti F/F}(T_H\otimes_\O\ti F))^\Gamma\subset M^*$. Denote by
$A^*$ the maximal $F$-split subtorus of $T^*$ and by $S^*$ the
maximal $\Fsm$-split subtorus of $T^*$ which is clearly defined over
$F$ and which is a maximal $\Fsm$-split torus of $G^*$.

By \cite[16.4.8]{SpringerBook}, $G$ is an inner twist of $G^*$ which
is given by a class $\rH^1(\Gamma, G^*_{\rm ad}(\bar F^s))$. Then
there is a ${\rm Gal}(\bar F^s/F)$-stable $G^*_{\rm ad}(\bar
F^s)$-conjugacy class of isomorphisms $\psi: G_{\bar
F^s}\xrightarrow{\sim } G^*_{\bar F^s}$ such that
\begin{equation}
\gamma\mapsto \on{Int}(g_\gamma)=\psi\cdot
\gamma(\psi)^{-1}=\psi\gamma\psi^{-1}\gamma^{-1}\in G^*_{\rm
ad}(\bar F^s)
\end{equation}
gives a $1$-cocycle that represents the corresponding class in
$\rH^1(\Gamma, G^*_{\rm ad}(\bar F^s))$.

Notice that $G_{\Fsm}\simeq G^*_{\Fsm}$ since they are both
quasi-split and they are inner forms of each other. This implies
that we can realize the inner twist $G$ over $\breve{F}$. (In fact,
over a finite extension of $F$ contained in $\breve{F}$.) More
precisely,  any inner twist  $G_{\Fsm}\otimes_{\Fsm}\bar
F^s\xrightarrow{\sim} G^*_{\breve{F}} \otimes_{\Fsm}\bar F^s$ is
conjugate over  $\bar F^s$ to an isomorphism $\psi:
G_{\Fsm}\xrightarrow{\sim} G^*_{\Fsm}$ (cf.
\cite[11.2]{HainesRostami}). Then
\begin{equation}
\on{Int}(g)=\psi\cdot
\sigma(\psi)^{-1}=\psi\sigma\psi^{-1}\sigma^{-1}\in G^*_{\rm
ad}(\Fsm)
\end{equation}
is an inner automorphism of $G^*_{\Fsm}$, where $\sigma$ is the
topological generator of $\Gal(\Fsm/F)\simeq\hat{\bbZ}$ given by the
lift of the (arithmetic) Frobenius. In what follows, we will choose
this isomorphism $\psi$ more carefully. For this purpose, it is
convenient to formulate the following:

\begin{Definition}\label{rigid}
A {\rm rigidification} of $G$ is a triple $(A,S,P)$, where $A$ is a
maximal $F$-split torus of $G$, $S\supset A$ a maximal $\Fsm$-split
torus of $G$ defined over $F$, and $P\supset M=Z_G(A)$ a minimal
parabolic subgroup defined over $F$. A rigidified group over $F$ is
a reductive group over $F$ together with a rigidification. The
groupoid of rigidified groups over $F$ will be denoted by ${\rm RG}(F)$.
\end{Definition}

\begin{lemma}\label{transitive}
The group $G_{\rm ad}(F)$ acts transitively on the set of
rigidifications of $G$.
\end{lemma}
\begin{proof}Let $(A_1,S_1,P_1)$ and $(A_2,S_2,P_2)$ be two
rigidifications. After conjugation, we can assume that $A_1=A_2$,
$P_1=P_2$. Let $M=Z_G(A)$. Then we can replace $G$ by $M$. In
addition, we can assume that it is of adjoint type; hence we reduce
to the case that $G$ is adjoint anisotropic. We need to show that
$S_1$ and $S_2$ are conjugate in this case.
As $G$ is anisotropic, there is a unique parahoric group scheme of
$G$ (the Iwahori group scheme, see
\cite[5.2.7]{BTII}); let us  denote it by $\calP$. By the construction of $\calP$
 in \cite{BTII}, we see that the N\'{e}ron models $\calS_1$
and $\calS_2$ of $S_1$ and $S_2$ map naturally into $\calP$ as
closed subgroup schemes. Therefore, $\calS_1(\breve\O)$,
$\calS_2(\breve\O)\subset \calP(\breve\O)$. We can therefore choose
$g$ in $\calP(\breve\O)$ such that $gS_1g^{-1}=S_2$. Since $\calP$
is smooth and has connected fibers,
$\rH^1(\hat\Z,\calP(\breve\O))=(1)$ (see \cite[3.4 Lemme 2]{BTIII}), and a standard argument shows
that we can choose $g\in \calP(\O)\subset G(F)$.
\end{proof}

\begin{prop}\label{propopsi}
Denote by ${M'}^*:= N_{G^*_{\on{ad}}}(M^*)\cap P^*_{\on{ad}}$ the
Levi subgroup of  $G^*_{\on{ad}}$ that corresponds to $M^*\subset
G^*$. The inner twist $ G_{\bar F^s}\xrightarrow{\sim } G^*_{\bar
F^s}$ is conjugate over $\bar F^s$ to an isomorphism $\psi:
G_{\Fsm}\xrightarrow{\sim} G^*_{\Fsm}$ which is such that ${\rm
Int}(g)=\psi\cdot \sigma(\psi)^{-1}$ lies in the normalizer
$N_{{M'}^*}(S^*_{\rm ad})$ of the torus $S^*_{\rm ad}$ in ${M'}^*$.
\end{prop}

\begin{proof} We pick up a rigidification $(A,S,P)$ of $G$. Recall $M= Z_G(A)$. Then
$M^*$ is the unique standard Levi of $G^*$ (i.e. $T^*\subset M^*$)
corresponding to $M$.  The group $M^*$ is the quasi-split inner form
of $M$ (\cite[16.4.7]{SpringerBook}). Indeed, let $P^*=B^*M^*=({\rm
Res}_{\tilde F/F}(P_0\otimes_F\ti F))^\Gamma$ (a standard parabolic
subgroup in $G^*$) correspond to $P$. Observe that there is an inner
twist $\psi:G_{\Fsm}\xrightarrow{\sim} G^*_{\Fsm}$ sending
$P_{\Fsm}$ to $P^*_{\Fsm}$ and $M_{\Fsm}$ to $M^*_{\Fsm}$. Now
$\on{Int}(g)= \psi \sigma \psi^{-1}\sigma^{-1}$  preserves
$M^*_{\Fsm}\subset P^*_{\Fsm}$. In particular, $g\in {M'}^*(\Fsm)$,
where ${M'}^*:= N_{G^*_{\on{ad}}}(M^*)\cap P^*_{\on{ad}}$ is the
corresponding Levi in $G^*_{\on{ad}}$. Therefore,
$\on{Int}(g):M^*_{\Fsm}\to M^*_{\Fsm}$ is inner. In particular, when
restricted to the (connected) centers of $M$ and $M^*$, $\psi$ is an
isomorphism of $F$-groups. Then it automatically sends $A= Z(M)^0_s$
(the maximal split torus in $Z(M)^0$) to $ Z(M^*)^0_s\subset A^*$.
By composing by an inner automorphism of $G^*(\Fsm)$ induced by an
element in $M^*(\Fsm)$, we can further assume that $\psi:M_{\Fsm}\to
M^*_{\Fsm}$ sends $S_{\Fsm}\to S^*_{\Fsm}$.
\end{proof}

Let us denote
\begin{equation}\label{N'}
{N'}^*=N_{{M'}^*}(T^*_{\rm ad})=N_{{M'}^*}(S^*_{\rm ad}), \quad
N^*_{\rm ad}:={\rm Im}({N'}^*\subset {M'}^*\to M^*_{\rm ad}).
\end{equation}

\begin{cor}\label{coropsi}
With the above notations, there is a unique class $[c^{\rm rig}]\in
\rH^1(\hat\Z,{N'}^*(\Fsm))$ whose image $[c]$ under
\[\rH^1(\hat\Z,{N'}^*(\Fsm))\to \rH^1(\hat\Z, {M'}^*(\Fsm))\to \rH^1(\hat\Z, G^*_{\on{ad}}(\Fsm))\subset \rH^1(F,G^*_{\on{ad}}),\]
gives $G$ as an inner twist of its quasi-split form $G^*$, where we
identify $\hat\Z={\rm Gal}(\Fsm/F)$ by sending $1$ to the Frobenius
$\sigma$. \endproof
\end{cor}
\begin{proof}The existence is given by Proposition
\ref{propopsi} after choosing a rigidification $(G,A,S,P)$. Let
$[c']$ be another class in $\rH^1(\hat\Z,{N'}^*(\Fsm))$ that also
maps to $[c]$. Then via twisting $(G^*,S^*,M^*,P^*)$ by $c'$, we
obtain $(G',S',M',P')$, where $S'$ is an $F$-torus of $G'$,
$M'\supset T'$ is a $F$-Levi subgroup of $G'$ and $P'\supset M'$ is
an $F$-parabolic subgroup of $G'$. Let $A'=Z(M')^0_s$ be an
$F$-split torus of $G'$. Since $G$ and $G'$ are isomorphic as
$F$-groups, $P'$ is a minimal $F$-parabolic subgroup of $G'$,
$M'\subset P'$ is a minimal $F$-Levi subgroup of $G'$, and therefore
$A'$ is a maximal $F$-split torus. In addition, $S'$ is a maximal
$\Fsm$-split torus of $G'$ defined over $F$. Therefore, there is an
$F$-isomorphism $(G,A,S,P)\simeq (G',A',S',P')$ by Lemma
\ref{transitive}. We can now conclude that  $[c']$ coincides with
$[c]$ given by Proposition \ref{propopsi}.
\end{proof}

\begin{Remark}{\rm Note that we do not claim that the map
$\rH^1(\hat\Z,{N'}^*(\Fsm))\to \rH^1(F,G^*_{\on{ad}})$ is injective.
In fact, the group ${N'}^*$ depends on $[c]$.}
\end{Remark}

\begin{Remark}{\rm Let $(G,A,S,P)$ be a rigidified group over $F$; then
$M=Z_G(A)$ is a minimal $F$-Levi of $G$ and $T=Z_G(S)$  a maximal
torus of $G$. Let $M'$, $T_{\rm ad}$ be the images of $M$, $T$ in
$G_{\rm ad}$. We can see that the elements in $G_{\rm ad}$ that fix
the triple $(A,S,P)$ are those elements in $M'$ that fix $T$ (or
equivalently fix $T_{\rm ad}$). Therefore, we obtain the following
exact sequence of $F$-groups
\[1\to N'\to {\rm Aut}(G,A,S,P)\to {\rm Out}(G)\to 1\]
where $N'=N_{M'}(T_{\rm ad})$.}
\end{Remark}

\subsubsection{}\label{1d3}
Now observe that the center of ${M'}^*=N_{G^*_{\on{ad}}}(M^*)\cap
P^*_{\on{ad}}$ is connected. To see this, recall that
$\xch(T^*_{\on{ad}})=Q(G^*_{\on{ad}})$, and $Q(M^*_{\on{ad}})$ is
just a direct factor of $Q(G^*_{\on{ad}})$, where for a connected
reductive group $L$, $Q(L)$ denote its absolute root lattice.
Indeed, to specify $M^*_{\on{ad}}$ is the same as to choose a
$\Gal(\bar{F}/F)$-stable subset $\Delta_{M^*}\subset \Delta_{G^*}$
of simple roots for $G^*_{\on{ad}}$, and $Q(M^*_{\on{ad}})$ is the
lattice generated by the simple roots in $\Delta_{M^*}$. Then the
center of ${M'}^*$, which is the subgroup of $T^*_{\on{ad}}$ defined
as the intersection of the kernel of all $a$ for $a\in\Delta_{M^*}$,
is indeed an induced torus\footnote{The torus $Z^*$ is induced
because the Galois group permutes
$\Delta_{G^*}\setminus\Delta_{M^*}$}, denoted by $Z^*$. We have
\begin{equation}\label{barM}
1\to Z^*\to {N'}^*\to N^*_{\rm ad}\to 1.
\end{equation}
Since $Z^*$ is induced, Hilbert theorem 90 and Shapiro's Lemma
implies that ${N'}^*(\Fsm)\to N^*_{\on{ad}}(\Fsm)$ is surjective.
Then, by \cite[I. \S 5.7]{SerreGaloisCoh} we obtain an exact
sequence of pointed sets
\begin{equation}\label{barMexact}
 \rH^1(\hat\Z, {N'}^*(\Fsm))\hookrightarrow  \rH^1(\hat\Z,  {N}_{\rm ad}^*(\Fsm))\to \rH^2(\hat\Z, Z^*(\Fsm))
\end{equation}
with the first map injective. Write $Z^*=\prod_j{\rm
Res}_{F_j/F}\Gm$. Then using Shapiro's Lemma we see
\begin{equation}\label{BrF}
\rH^2(\hat\Z, Z^*(\Fsm))=\prod_j {\rm Br}(F_j)\simeq \prod_j \Q/\Z.
\end{equation}

In the next subsection, we recall an explicit cocycle representing
the image of $[c^{\rm rig}]$ under $\rH^1(\hat\Z,{N'}^*(\Fsm))\to
\rH^1(\hat\Z,N^*_{\rm ad}(\Fsm))$.

\subsubsection{}\label{1b4} Here we assume in addition that $F$ is a $p$-adic field, i.e a finite extension of
$\Q_p$. Then we can choose the inner twist $\psi$ of Proposition
\ref{propopsi} even more carefully: Indeed, recall that  since
$M=Z_G(A)$, the group $M_{\rm ad}$ is anisotropic. Therefore, by
\cite[Satz 3]{KneserII}, $M_{\on{ad}}$ is isomorphic to $\prod_i
{\rm Res}_{E_i/F}(B_i^\times/E_i^\times)$, where $B_i$ are central
division algebras of degree $m_i$ over $E_i$, and $E_i$ are finite
extensions of $F$. Therefore, we have an isomorphism
\begin{equation}\label{isoMad}
\rho: M^*_{\on{ad}}\xrightarrow{\sim}\prod
\Res_{E_i/F}\on{PGL}(m_i),
\end{equation}
sending $S^*_{\on{ad}}$ to $\prod\Res_{E_i/F}\on{D}_{m_i}$, where
$\mathrm D_n$ is the standard torus in $\on{PGL}(n,E)$.

\begin{lemma}\label{coxeter}
Suppose that $B$ is a central division algebra over the local field
$E$ of degree $n^2$ and with  Brauer group invariant $r/n$ with $0<
r<n$ and ${\rm gcd}(r, n)=1$. Denote by $\tau$ the permutation
$(12\cdots n)$. If $\varpi$ is a uniformizer of $E$ we let $\und
n:=\und {n}_r(\varpi)$ be the element of $ \on{GL}(n,\breve{E})$
given by $\und {n}(e_i)=e_{\tau^r(i)}$, if $i\neq n$, and $\und
{n}(e_n)=\varpi\cdot e_{\tau^r(n)}$. Then there is an isomorphism
\[
\psi:B\otimes_E\breve{E}\xrightarrow{\sim} M_{n\times n}(\breve{E})
\]
so that $\psi\sigma\psi^{-1}\sigma^{-1}=\on{Int}(\und n)$.
\end{lemma}

\begin{proof}
Suppose that $E_n/E$ is an unramified extension of degree $n$ and
$\sigma$ a generator of ${\rm Gal}(E_n/E)$.
 Then we can represent $B$ as the  associative  $E_n$-algebra with generator $\Pi$
 and relations $\Pi^n=\varpi$,  $a\cdot \Pi=  \Pi\cdot \sigma(a)^r$, i.e
$$
B=\{\oplus_{i=0}^{n-1} E_n\cdot \Pi^i\ |\ \Pi^n=\varpi, \ a\cdot
\Pi=\Pi\cdot \sigma(a)^r\}.
$$
Sending $\Pi$ to $\underline n=\und {n}_r(\varpi)$ as above and
$a\in E_n$ to the diagonal matrix   $(a, \sigma(a),\ldots,
\sigma^{n-1}(a))$, gives an isomorphism  $\psi:
B\otimes_EE_n\xrightarrow{\sim} M_{n\times n}(E_n)$. The result now
follows from an explicit calculation.
\end{proof}

Suppose that the Brauer group invariant of $B_i$ is $r_i/m_i$ with
$1\leq r_i< m_i$, ${\rm gcd}(r_i, m_i)=1$. By the above lemma we can
always choose an inner twist
\begin{equation}\label{psi'}
\psi': (M_{\on{ad}})_{\Fsm}\xrightarrow{\sim}
(\prod\Res_{E_i/F}\on{PGL}(m_i))_{\Fsm},
\end{equation}  such that the element
$g\in \prod\on{PGL}(m_i,E_i\otimes_F\Fsm)$ given by
$\on{Int}(g)=\psi'\sigma\psi'^{-1}\sigma^{-1}$ satisfies $g=\prod
g_i$, and $g_i=\und {n}_{r_i}(\varpi_i)$ where $\varpi_i$ is a
uniformizer of $E_i$. Combining with the discussion in \ref{1d3}, we
obtain that

\begin{prop}\label{propopsi2}
Assume that $F$ is a $p$-adic field and choose an isomorphism
(\ref{isoMad}). Then we can find $\psi : G_{\Fsm}\xrightarrow{\sim}
G^*_{\Fsm}$ as in Proposition \ref{propopsi}, such that in addition
the image of ${\rm Int}(g)$ under ${M'}^*(\Fsm)\to M^*_{\rm
ad}(\Fsm)$ is given by the element $ {g}=\prod_i g_i $ with $g_i$ as
above.\qed
\end{prop}

\bigskip

\section{Reductive groups over  $\O[u, u^{-1}]$}\label{reductive
group}

\setcounter{equation}{0}

 Recall $G$ is a connected reductive group over $F$.
   We will assume that:
\medskip

 ({\sl Tameness hypothesis}) \ {\sl $G$ splits over a finite tamely ramified extension $\ti F/F$.}

\medskip
Let $\varpi$ be a uniformizer of $\O$. Our goal in this section is
to construct, when $F$ is a $p$-adic field, a reductive group scheme $\und G$ over $\Spec(\O[u^{\pm
1}])$ which extends $G$ in the sense that its base change
\begin{equation}
\und G\otimes_{\O[u^{\pm 1}]}F,\quad  u\mapsto \varpi ,
\end{equation}
is isomorphic to $G$. This is done by following the procedure of the construction of
$G$ from its split form $H$ which was described in \S \ref{ss1a}; we
first extend the quasi-split form $G^*$ to $\und G^*$ and then give
$\und G$ by an appropriate descent. (We write $\O[u^{\pm 1}]$
instead of $\O[u, u^{-1}]$ for brevity.)

\subsection{The tame splitting field}\label{sss1a2}  We assume that $F$ is either a $p$-adic field, or $F=k((t))$ with $k$ finite and that in either case the
residue field $k$ has cardinality $q=p^m$.
 Denote by $\ti F_0$ the maximal unramified extension of $F$ that is contained in $\ti F$
 and by $\ti \O_0$, $\ti \O$ the valuation rings of $\ti F_0$, $\ti F$ respectively.
 Set $e=[\ti F:\ti F_0]$  (which is then prime to $p$) and let $\gamma_0$ be a generator of ${\rm Gal}(\ti F/\ti F_0)$.
 Recall that by Steinberg's theorem, the group $G_{\Fsm}:=G\otimes_F\Fsm$
is quasi-split. By possibly enlarging the splitting field $\ti F$,
we can now assume that
\begin{itemize}

\item{}  $G_{\ti F_0}$ is quasi-split,

\item{}  $\ti F/F$ is Galois with group $\Gamma={\rm Gal}(\ti F/F)=\langle \sigma\rangle\rtimes \langle\gamma_0\rangle$
which is the semi-direct product of $\langle \sigma\rangle \simeq
\Z/(r)$, where $\sigma$ is a lift of the (arithmetic) Frobenius
${\rm Frob}_q\in {\rm Gal}(\ti F_0/F)$, with the normal inertia
subgroup $I:={\rm Gal}(\ti F/\ti F_0)=\langle\gamma_0\rangle\simeq
\Z/(e)$, with relation $\sigma\gamma_0\sigma^{-1}=\gamma_0^q$,

\item{}  there is a uniformizer $\ti\varpi$ of $\ti \O$ such that $\ti\varpi^e=\varpi$.

\end{itemize}

Without further mention, we will assume that the extension $\ti F/F$
is as above. Then we also have $\ti\O=\ti \O_0[\ti\varpi]\simeq
\ti\O_0[x]/(x^e-\varpi)$ and $\tilde\O_0$ contains a primitive
$e$-th root of unity $\zeta=\gamma_0(\ti\varpi)\ti\varpi^{-1}$.

\subsection{Covers of $\O[u]$ and $\O[u^{\pm 1}]$.}\label{ssCovers}

Suppose that $F$ is a $p$-adic field with ring of integers $\O$ and
residue field $k=\mathbb F_q$. Let $\ti F/F$ be a finite tamely
ramified Galois extension of $F$ as in \S \ref{sss1a2} with Galois
group $\Gamma=\langle\gamma_0, \sigma |\ \sigma^r=\gamma_0^e=1,
\sigma\gamma_0\sigma^{-1}=\gamma_0^q\rangle$. The maximal tamely
ramified extension $F^t$ of $F$ in $\bar F$ is the union of fields
$\tilde F$ as above and its Galois group
$$
\Gamma^t:={\rm Gal}(F^t/F)\simeq \prod_{l\neq p}\Z_l(1)\rtimes\hat\Z
$$
is the projective limit of the corresponding groups $\Gamma$.

\subsubsection{}\label{sss2a1} Consider the affine line $\AA^1_\O=\Spec(\O[u])$ and its cover
$$
\pi: \AA^1_{\ti\O_0}=\Spec(\ti\O_0[v])\to \AA^1_\O=\Spec(\O[u])
$$
given by $u\mapsto v^e$. The (abstract) group $\Gamma$ described as
above acts on $\ti\O_0[v]$ by
$$
\sigma(\sum_i a_i v^i)=\sum_i \sigma(a_i)v^i, \quad \gamma_0(\sum_i
a_i v^i)=\sum_i a_i \zeta^i v^i,
$$
where $\zeta$ is the primitive $e$-th root of unity
$\gamma_0(\ti\varpi)\ti\varpi^{-1}$ in $\ti\O_0$. We have
$\ti\O_0[v]^\Gamma=\O[u]$ and $\pi$ is a $\Gamma$-cover ramified
over $u=0$. The restriction of $\pi$ over the open subscheme $u\neq
0$ gives a $\Gamma$-torsor
$$
\pi_0: \Spec(\ti\O_0[v^{\pm 1}])\to \Spec(\O[u^{\pm 1}])=\bbG_{m\O}.
$$
Notice that base changing the $\Gamma$-cover  $\pi$ via the map
$\O[u]\to F$ given by $u\mapsto \varpi $ gives $\Spec(\ti F)\to
\Spec(F)$ with its Galois action. In this way, we realize $\Gamma^t$
as a quotient of the fundamental group of $\bbG_{m\O}$ such that the
composed map $\Gal(\bar{F}/F)\to \Gamma^t$ coincides with the tame
quotient of $\Gal(\bar{F}/F)$. In fact, we can easily see that
$\pi_1(\Gm_\O, \Spec(\bar F))\to \Gamma^t$ is an isomorphism.

\subsubsection{} Suppose that we take $F=\Q_p$ and that $L$ is a tamely
ramified finite extension of $\Q_p$. Let $\ti F$ be a Galois
extension of $F=\Q_p$ that contains $L$ and satisfies the
assumptions of \S \ref{sss1a2} with $\Gamma={\rm Gal}(\ti F/\Q_p)$.
In particular, $\ti F_0$ is the maximal unramified extension of
$\Q_p$ contained in $\ti F$. Denote by $L_0$ the maximal unramified
extension of $\Q_p$ contained in $L$ and suppose $L\simeq
L_0[x]/(x^{e_L}-p\cdot c)$, $c\in \O_{L_0}^*$. Let $\Gamma_L$ the
subgroup of $\Gamma$ that fixes $L$. We can then see that there is a
$\Z_p[u]$-algebra homomorphism
\begin{equation}\label{invarRing}
(\ti\O_0[v])^{\Gamma_L}\simeq \O_{L_0}[w]
\end{equation}
in which the right hand side is a $\Z_p[u]$-algebra via $u\mapsto
w^{e_L}\cdot c^{-1}$.

\medskip

\subsection{Groups over  $\O[u^{\pm 1}]$}\label{ss2b}
We will now give the construction of the group schemes $\und G$ over
$\O[u^{\pm 1}]$.

\subsubsection{}\label{sss2b1}
Recall that $G$ is a connected reductive group over the $p$-adic
field $F$ which splits over the tame extension $\ti F/F$ as in \S
\ref{sss1a2}. We will use the notations of \S \ref{Preliminaries}.
In particular, $H$ over $\O$ is the Chevalley (split) form of $G$;
  we pick a pinning on $H$ as in \S \ref{sss1a1}. Then the
indexed root datum for $G$ gives a group homomorphism $\tau:
\Gamma={\rm Gal}(\ti F/F)\to {\rm Aut}_\O(H)$ and using  
\S \ref{groupPQS} we obtain the pinned quasi-split group $(G^*, T^*,
B^*, e)$ over $F$.

We apply the equivalence of categories \eqref{ssstame} to the
following case: Let $S_2=\bbG_{m\O}=\Spec(\O[u^{\pm 1}])$ and
$s=S_1=\Spec(F)\to S_2$ be the point given by $u=\varpi$. Let
$\bar{s}=\Spec (\bar F)$ be a geometric  point over $s$. Then we
obtain an
 equivalence
\begin{equation}\label{ssstameeq}
{\rm PQ}(H, \bbG_{m\O})\simeq {\rm
PQ}^{\Gamma^t}(H, \bbG_{m\O})\xrightarrow{\sim}{\rm PQ}^{\Gamma^t}(H, F).
\end{equation}
We choose a quasi-inverse of this functor and therefore, for any
pinned quasi-split group $(G^*,T^*,B^*,e^*)$ over $F$, we have
$(\und G^*,\und T^*,\und B^*,\und e^*)$ over $\O[u^{\pm 1}]$
together with an isomorphism $(\und G^*,\und T^*,\und B^*,\und
e^*)\otimes_{\O[u, u^{-1}]}F\simeq (G^*,T^*,B^*,e^*)$.

In particular, $\und G^*$ is the reductive group scheme over
$\O[u^{\pm 1}]$ which is the twisted form of $\und
H=H\otimes_\O\O[u^{\pm 1}]$ obtained from the $\Gamma$-torsor
$\pi_0: \Spec(\ti\O_0[v^{\pm 1}])\to \Spec(\O[u^{\pm 1}])$ using
$\tau$. More concretely,
\begin{equation}\label{defG*}
\und G^*=({\rm Res}_{\ti\O_0[v^{\pm 1}]/\O[u^{\pm
1}]}(H\otimes_\O\ti\O_0[v^{\pm 1}]))^\Gamma
\end{equation}
where $\gamma\in \Gamma$ acts diagonally via $\tau(\gamma)\otimes
\gamma$. The same construction applies to other groups that inherit
a pinning from $G^*$; for example, it applies to the Levi subgroup
$M_0$ that corresponds to the $\Gamma$-stable subset
$\Delta_0\subset \Delta$. We obtain
$$
 \underline{M}^*=({\rm Res}_{\ti\O_0[v^{\pm 1}]/\O[u^{\pm
1}]}(M_0\otimes_\O\ti\O_0[v^{\pm 1}]))^\Gamma.
$$
Again base changing along $u\mapsto \varpi$ gives a group
canonically isomorphic to $M^*\subset G^*$.

\subsubsection{}\label{sss2c2}
The goal of the rest of this chapter is to explain how to construct
a quadruple of group schemes $(\und G,\und A,\und S,\und P)$, whose
specialization along $u=\varpi$ gives rise to $G$ together with a
rigidification $(A, S, P)$.

First as above, we obtain adjoint quasi-split forms $\und M^*_{\rm
ad}$, $\und G^*_{\rm ad}$ over $\O[u^{\pm 1}]$ that specialize to
$M^*_{\rm ad}$, $G^*_{\rm ad}$ over $F$ after the base change
$u\mapsto \varpi$. Then let $\und {N'}^*=N_{\und {M'}^*}(\und
T^*_{\rm ad})$ and $\und N_{\rm ad}^*={\rm Im}(\und {N'}^*\to \und
M^*_{\rm ad})$.
 We also obtain a short exact (central) sequence
\begin{equation}\label{barMu}
1\to \und Z^*\to \und {N'}^*\to \und N^*_{\rm ad}\to 1,
\end{equation}
which specializes to (\ref{barM}). We will consider the base change
of these groups after the \'etale extension $\O[u^{\pm 1}]\to
\breve{\O}[u^{\pm 1}]$. For simplicity, we will sometimes omit this
base change from the notation. Here $\und Z^*$ is an induced torus
and by (\ref{invarRing}) we have
\begin{equation}\label{ZZ}
\und Z^*\otimes_\O\breve{\O}\simeq\prod_j {\rm
Res}_{\breve{\O}[w^{\pm 1}]/\breve{\O}[u^{\pm 1}]}\Gm
\end{equation}
where $w^{d_j}=uc_j^{-1}$ for $p\nmid d_j$. By applying ${\rm
Pic}(\breve{\O}[w^{\pm 1}])=(1)$ and Hilbert's theorem 90 we see
that $\rH^1_{\rm et}(\Spec(\breve{\O}[u^{\pm 1}]), \und Z^*)=(1)$.
This gives that
\begin{equation}
\und {N'}^*(\breve{\O}[u^{\pm 1}])\to \und N^*_{\rm
ad}(\breve{\O}[u^{\pm 1}])
\end{equation}
is surjective. The exact cohomology sequence for a central quotient
now gives that
\begin{equation}\label{barMuexact}
 \rH^1(\hat\Z, \und {N'}^*(\breve{\O}[u^{\pm 1}]))\hookrightarrow
\rH^1(\hat\Z, \und N^*_{\rm ad}(\breve{\O}[u^{\pm 1}]))\to
\rH^2(\hat\Z, \und Z^*(\breve{\O}[u^{\pm 1}]))
\end{equation}
is an exact sequence of pointed sets. There are natural maps between
the exact sequences (\ref{barMuexact}) and (\ref{barMexact})
obtained by sending $\breve{\O}[u^{\pm 1}]\to \breve F$ via
$u\mapsto \varpi$.

Recall the description of $M_{\rm ad}$ and the isomorphism
(\ref{isoMad}) that follow from Kneser's theorem. Using
(\ref{ssstameeq}) we obtain that there is an
 isomorphism
\begin{equation}\label{218}
\und \rho: \und M^*_{\rm ad} \xrightarrow{ \sim\ } \prod_i {\rm
Res}_{\ti\O_0[v^{\pm 1}]^{\Gamma_i}/\O[u^{\pm 1}]}{\rm PGL}(m_i)
\end{equation}
where $\Gamma_i\subset \Gamma$ are subgroups of finite index in the
Galois group $\Gamma$ of $\ti\O_0[v^{\pm 1}]/\O[u^{\pm 1}]$ which
specializes under $u\mapsto \varpi$, to  the isomorphism
(\ref{isoMad}).

\subsubsection{}\label{sss2c3}

Now recall that the anisotropic kernel $M$ (an inner form of $M^*$)
gives a well-defined class $[c(M)]$ in $\rH^1(\hat\Z, M^*_{\rm
ad}(\Fsm))$. Applying Proposition \ref{coropsi} to $M$ and the exact
sequence \eqref{barMexact}, we see that $[c(M)]$ has the following
properties:

\begin{itemize}
\item[1)] It is the image of a unique class $[c^{\rm rig}]$ in $\rH^1(\hat\Z, {N'}^*(\Fsm))$, and,

\item[2)] The image of $[c^{\rm rig}]$ under $\rH^1(\hat\Z, {N'}^*(\Fsm))\to \rH^1(\hat\Z, G^*_{\rm ad}(\Fsm))$
gives the inner twist $G$ of $G^*$.
\end{itemize}

We have $\rH^2(\hat\Z, \breve\O^*)=(1)$, hence $\rH^2(\hat\Z,
\breve{\O}[u^{\pm 1}]^*)= \rH^2(\hat\Z, \breve{\O}^*\times
u^{\Z})\simeq \rH^2(\hat\Z, \Z)\simeq \Q/\Z$, and similarly ${\rm
Br}(F)=\rH^2(\hat\Z, \breve{F}^*)=\rH^2(\hat\Z,
\breve{\O}^*\times\varpi^\Z)\simeq \Q/\Z$. Therefore, $u\mapsto
\varpi$ gives an isomorphism $\rH^2(\hat\Z, \breve{\O}[u^{\pm
1}]^*)\xrightarrow{\sim} \rH^2(\hat\Z, \breve{F}^*)$. Similarly,
using (\ref{ZZ}), and Shapiro's Lemma we obtain that
$u\mapsto\varpi$ gives an isomorphism
\begin{equation}\label{219}
 \rH^2(\hat\Z, \und Z^*(\breve{\O}[u^{\pm 1}]))\xrightarrow{\sim}  \rH^2(\hat\Z,   Z^*(\Fsm)).
\end{equation}
Note that using (\ref{218}), (\ref{invarRing})  and Shapiro's lemma
we obtain an isomorphism
\begin{equation}\label{210}
\rH^1(\hat\Z, \und M^*_{\rm ad}(\breve{\O}[u^{\pm 1}])
\xrightarrow{\sim} \prod_i \rH^1(\hat\Z, {\rm
PGL}_{m_i}(\breve{\O}[w^{\pm 1}_i])).
\end{equation}
The constructions explained in \ref{exDivision} and \ref{sss4b3}
produce Azumaya algebras ${\mathcal B}_i$ over $\O[w^{\pm 1}_i]$
that under $w\mapsto \varpi_i$  specialize to the central division
algebras $B_i$ that appear in the decomposition of the anisotropic
kernel $M_{\rm ad}$. These Azumaya algebras split over
$\breve{\O}[w^{\pm 1}_i]$; hence, by using the isomorphism
(\ref{210}), we obtain a well-defined class  $[\und c]\in
\rH^1(\hat\Z, \und M^*_{\rm ad}(\breve{\O}[u^{\pm 1}]))$ that
specializes to $[c(M)]\in \rH^1(\hat\Z, {M}^*_{\rm ad}(\Fsm))$.

Now, using Lemma \ref{coxeter}, the explicit constructions of \S
\ref{exDivision}, \S \ref{sss4b3} and our discussion in \S \ref{ss2b},
the class $[\und c]$ corresponding to the product of Azumaya
algebras in fact comes from a canonical class $[\und c^{\rm rig}]$
in $\rH^1(\hat\Z,\und N^*_{\rm ad}(\breve{\O}[u^{\pm 1}]))$, whose
specialization to $\rH^1(\hat\Z,N^*_{\rm ad}(\breve{F}))$ is the
image of $[c^{\rm rig}]$ in Corollary \ref{coropsi} under the map
$\rH^1(\hat\Z,{N'}^*(\Fsm))\to \rH^1(\hat\Z,N^*_{\rm ad}(\Fsm))$.

By (\ref{barMexact}), $[c]$ maps to the trivial class in
$\rH^2(\hat\Z,   Z^*(\Fsm))$. Using (\ref{barMuexact}) and the
isomorphism (\ref{219}) we see that $[\und c^{\rm rig}]$ belongs to
$\rH^1(\hat\Z, \und {N'}^*(\breve{\O}[u^{\pm 1}]))$ and maps to the
class $[c^{\rm rig}]$ in $\rH^1(\hat\Z, {N'}^*(\breve{F}))$ under
the specialization $u\mapsto\varpi$.

\begin{Remark}\label{unique lift}
{\rm In fact,  one can show that  specialization $u\mapsto \varpi$
induces isomorphisms
\[\rH^1(\hat\Z, \und {N'}^*(\breve{\O}[u^{\pm 1}]))\xrightarrow{\sim} \rH^1(\hat\Z, {N'}^*(\Fsm)), \quad \rH^1(\hat\Z, \und N_{\rm ad}^*(\breve{\O}[u^{\pm 1}]))\xrightarrow{\sim}\rH^1(\hat\Z, N_{\rm ad}^*(\Fsm)).\]
We  do not use this fact here. }
\end{Remark}

\subsubsection{}\label{sss2d3}
The class $[\und c^{\rm rig}]$ allows us to define an inner twist
${\Gu}$ of ${\Gu}^*$: Let $\und c^{\rm rig}$ be a 1-cocycle
representing this class and denote by ${\rm Int}(\bold{g})$ the
element in $\und {M'}^*(\breve{\O}[u^{\pm 1}])\subset \und G^*_{\rm
ad}(\breve{\O}[u^{\pm 1}])$ which is the value at $1$ of this
cocycle. We define
\begin{equation}\label{inner}
{\Gu}(R)={\Gu}^*(\breve{\O}[u^{\pm 1}]\otimes_{\O[u^{\pm
1}]}R)^{\hat\Z}
\end{equation}
where the action of the topological generator $1$ of $\hat\Z$ is
given by ${\rm Int}(\bold{g})\cdot \sigma$. (In other words, ${\Gu}$
is given by the Weil descent datum obtained from the automorphism
${\rm Int}(\bold{g})\cdot\sigma$.)
By descent, ${\Gu}$ is a reductive group over $\O[u^{\pm 1}]$ with
base change to $\breve{\O}[u^{\pm 1}]$ isomorphic to
${\Gu}^*\otimes_\O\breve{\O}$.

In fact, we obtain more. Namely, we
have
\[(\und G,\und A,\und S,\und T,\und M, \und P),
\]
where $(\und G,\und T,\und M,\und P)$ is obtained from $(\und
G^*,\und T^*,\und M^*,\und P^*)$ via twisting by the cocycle $\und
c^{\rm rig}$, $\und S$ is the maximal $\breve\O[u^{\pm 1}]$-split
subtorus of $\und T$, and $\und A$ is the maximal $\O[u^{\pm
1}]$-split subtorus of $\und T$. In addition, the specialization of
$(\und G,\und A,\und S,\und P)$ by $u\mapsto \varpi$ is a
rigidified reductive group
\begin{equation}
(G_0,A_0,S_0,P_0):=(\und G,\und A,\und S,\und P)\otimes_{\O[u^{\pm
1}]}F.
\end{equation}
By our construction, we have an isomorphism $G\simeq G_0$.

\subsubsection{}\label{sss2c5}
Similarly to Definition \ref{rigid}, let us define the groupoid ${\rm
RG}(\O[u^{\pm 1}])$ of rigidified groups over $\O[u^{\pm 1}]$. An
object is a quadruple $(\und G,\und A,\und S,\und P)$, where $\und
G$ is a connected reductive group over $\O[u^{\pm 1}]$, $\und A$ is
a maximal torus of $\und G$, $\und S\supset\und A$ is a torus, which
is maximal $\breve\O[u^{\pm 1}]$-split, and $\und P$ is a parabolic
subgroup of $\und G$ containing $\und M=Z_{\und G}(\und A)$ as a
Levi factor, such that: (i) $(\und G,\und A,\und S,\und
P)\otimes_{\O[u^{\pm 1}]}F$ is a rigidified group as in Definition
\ref{rigid}; (ii) there exists a pinned quasi-split form $(\und
G^*,\und T^*,\und B^*,\und e^*)$ over $\O[u^{\pm 1}]$ and an
\emph{inner} twist
\[\und \psi:(\und G,\und S,\und P)\otimes_{\O[u^{\pm 1}]}\breve{\O}[u^{\pm 1}]\xrightarrow{\sim} (\und G^*,\und S^*,\und P^*)\otimes_{\O[u^{\pm 1}]}\breve{\O}[u^{\pm 1}],\]
where $\und S^*$ is the maximal $\breve\O[u^{\pm 1}]$-split subtorus
of $\und T^*$ and $\und P^*$ is a parabolic subgroup of $\und G^*$
containing $\und B^*$.

Observe that an inner twist $\und \psi$ defines a cocycle $\und
c=\und \psi\cdot\sigma(\und \psi)^{-1}$ with values in $\und
{N'}^*$, where as before $\und {N'}^*$ is the normalizer of $\und
T^*$ in $\und {M'}^*$, and $\und {M'}^*$ is the standard Levi in
$\und G^*_{\rm ad}$ given by $\und P^*_{\rm ad}$.

By definition, we have the specialization functor ${\rm
RG}(\O[u^{\pm 1}])\to {\rm RG}(F)$. Our construction shows that the
essential image of this functor consists of all rigidified groups
over $F$ that are split over $F^t$ (we denote this latter
subgroupoid by ${\rm RG}(F^t/F)$). In fact, assuming Remark \ref{unique
lift}, it is easy to see that the isomorphism classes in ${\rm
RG}(\O[u^{\pm 1}])$ biject to the isomorphism classes in ${\rm
RG}(F^t/F)$.

\begin{Remark}\label{Piano}{\rm
Let $H$ be a Chevalley group over $\bbZ$. Recall that given a base
scheme $S$, the set of isomorphism classes of forms $\cal H$ of $H$ over $S$
are classified by the (\'etale) cohomology set $\rH^1(S, \und{\rm Aut}(H))$. As
every reductive group over $F$ admits a rigidification, a corollary
of the above discussion is that the natural specialization map
\begin{equation}\label{sp}
\rH^1(\Spec(\O[u^{\pm 1}]), \und{\rm Aut}(H))\to \rH^1(F^t/F, \und{\rm Aut}(H))
\end{equation}
is surjective, where $\rH^1(F^t/F, \und{\rm Aut}(H))$ denotes the set of
isomorphism classes of forms of $H$ over $F$ that  split over
$F^t$. 

It is natural to ask whether the map (\ref{sp}) is also injective.
If this is the case, then each form of $H$ over $\O[u^{\pm 1}]$, whose specialization along $u\mapsto \varpi$ is isomorphic to $G$, is isomorphic to $\und G$. This would characterize $\und G$ uniquely up to isomorphism. So far, we can prove a weaker result:
If a form $\mathcal H$ of $H$ over
$\O[u^{\pm 1}]$ contains a maximal torus and splits over a finite
Galois extension (as in \S \ref{sss2a1}) with Galois group $\Gamma$ of order
prime to $p$, then $\Hh$ is isomorphic to $\und G$ where
$G=\Hh\otimes_{\O[u^{\pm 1}]}F$ is the specialization of $\Hh$ along
$u\mapsto\varpi$. This fact implies that (\ref{sp})
is injective when restricted to the subset which corresponds to such
forms and provides,  in the case that the Galois closure of the splitting field of $G$ has order prime to $p$, a characterization of the reductive group $\und G$
as the unique up to isomorphism form of $H$ over $\O[u^{\pm 1}]$ which contains a maximal torus, splits after 
 a base change as above, and extends $G$.
This result is obtained by using ideas of
Pianzola and Gille; they have developed a theory of ``loop reductive
groups" over Laurent polynomial rings $K[u_1^{\pm 1}, \ldots, u^{\pm
1}_n]$ with $K$ a field (see for example \cite{PianzolaGille},
\cite{ChernPianzolaGille}). Essentially, these are reductive groups
that afford a  maximal torus over this base. The fibers $\und
G\otimes_{\O[u^{\pm 1}]}F[u^{\pm 1}]$, $\und G\otimes_{\O[u^{\pm
1}]}k[u^{\pm 1}]$, of the groups $\und G$ defined above are loop
reductive groups in their sense.  
We intend to undertake a more comprehensive study
of reductive groups over $\O[u^{\pm 1}]$ and report on this
in another paper.}
\end{Remark}

\bigskip

\section{Parahoric group schemes over $\O[u]$}\label{groupscheme}
\setcounter{equation}{0}

In this chapter, we give our  construction of the ``parahoric" group
schemes $\calG$ over $\O[u]$ which are among the main tools of the
paper. The main results are Theorem \ref{grpschemeThm} and Corollary \ref{application}
which describes
the properties of these group schemes.

\subsection{Preliminaries}\label{3a}

We start with some preliminaries on Bruhat-Tits buildings. Suppose
that $K$ is a discretely valued quasi-local field
(\cite[1.1.1]{LandvogtCrelle}) with valuation ring $\O_K$ and with
perfect residue field. If $G$ is a connected reductive group over
$K$, we can consider the Bruhat-Tits building ${\mathcal B}(G, K)$.
If $S$ is a maximal split torus of $G$, we denote by $\A(G, S, K)$
the corresponding apartment in ${\mathcal B}(G, K)$.

Note that if $\hat K$ is the completion of $K$, the natural maps
give identifications $\A(G, S, K)=\A(G_{\hat K}, S_{\hat K}, \hat
K)$, ${\mathcal B}(G, K)={\mathcal B}(G_{\hat K}, \hat K)$
(\cite[Prop. 2.1.3]{LandvogtCrelle}).

\subsubsection{} Suppose now that $(H, T_H, B_H, e)$ is a pinned Chevalley
group scheme $H$ over $\Z$ and consider $H_K =H\otimes_\Z K$ and
$T_K=T_H\otimes_\Z K$ over the field $K$. We will sometimes omit the
subscripts when the choice of the field $K$ is clear. After fixing
an identification of the value groups $v(K^*)=\Z$, we can identify
the apartments $\A(H, T_H, K)\subset {\mathcal B}(H, K)$ for all
such fields $K$. In particular, we can identify $\A(H, T_H,
\kappa((u)))$ for $\kappa=F$, $k$ and also $\A(H, T_H, F)$ such that, under
this identification, the hyperspecial points of $\A(H, T_H, K)$ with
stabilizer $H(\O_K)$ correspond to each other. Observe that we can also
identify the Iwahori-Weyl groups
$\widetilde{W}=N_H(T_H)(K)/T_H(\O_K)$ for different $K$ canonically
such that the above identification of the apartments is compatible with the
action of $\widetilde{W}$.

\subsubsection{}The above generalizes to
the quasi-split case as follows. Let $(\und G^*,\und T^*,\und B^*,
\und e^*)$ be a pinned quasi-split group over $\O[u^{\pm 1}]$. Let
$\und S^*$ be the maximal $\breve\O[u^{\pm 1}]$-split subtorus of
$\und T^*$ and $\und A^*$ be the maximal $\O[u^{\pm 1}]$-split torus
of $\und T^*$. Then as the centralizer of $\und S^*$ is $\und T^*$,
we see that $\und S^*_{ \kappa'((u))}:=\und S^*\otimes_{\O[u^{\pm
1}]} \kappa'((u))$ is a maximal $\kappa'((u))$-split torus of $\und
G^*_{\kappa'((u))}:=\und G^*\otimes_{\O[u^{\pm 1}]} \kappa'((u))$
for $\kappa'=\bar{k}$, $\Fsm$. Similarly, $\und S^*_{\Fsm}:=\und
S^*\otimes_{\O[u^{\pm 1}]} \Fsm$ is a maximal $\Fsm$-split torus of
$\und G^*_{\Fsm}:=\und G^*\otimes_{\O[u^{\pm 1}]} \Fsm$. We have
canonical identifications
\begin{equation}\label{aptident1}
\A(\und G^*_{\kappa'((u))},\und
S^*_{\kappa'((u))},\kappa'((u)))=\A(\und G^*_{\Fsm},\und
S^*_{\Fsm},\Fsm).
\end{equation}
Indeed,  these apartments are identified with
$\calA(H,T_H,K)^{\gamma}$ (using the tameness assumption, as in
\cite{BTII}, or \cite[Ch. IV, \S 10]{LandvogtLNM}, see also
\cite{PrasadYuInv}), where $K=\kappa'((u))$ or $\Fsm$, and $\gamma$
is a generator of the inertia subgroup. Similary, we can identify
the Iwahori-Weyl group for $(\und G^*_K,\und S^*_K)$ canonically and
compatibly with the identification of the apartments.

\subsubsection{}\label{sss3a3} We also have the following further generalization. Let
$(\und G,\und A,\und S,\und P)$ be a rigidified group over
$\O[u^{\pm 1}]$ as defined in \S \ref{sss2c5}. Then $\und A_F$ is a
maximal $F$-split torus of $\und G_F$. Let $x\in \A(\und G_F,\und
A_F,F)\subset {\mathcal B}(\und G_F, F)$ be a point in the
apartment. The identification \eqref{aptident1} induces
\begin{multline}\label{aptident2}
\A(\und G_F,\und A_F, F)=\A(\und G_{\Fsm},\und
S_{\Fsm},\Fsm)^\sigma=\A(\und G^*_{\Fsm},\und S^*_{\Fsm},\Fsm)^{{\rm
Int}(\mathbf g)\sigma}\\ =\A(\und G^*_{\kappa'((u))}, \und
S^*_{\kappa'((u))},\kappa'((u)))^{{\rm Int}(\mathbf
g)\sigma}=\A(\und G_{\kappa'((u))}, \und
S_{\kappa'((u))},\kappa'((u)))^{\sigma} ,
\end{multline}
and therefore we obtain $x_{\kappa((u))}\in \A(\und
G_{\kappa'((u))},\und S_{\kappa'((u))},\kappa'((u)))^\sigma$.
Therefore, using again $\calB(\und
G_{\kappa'((u))},\kappa'((u)))^\sigma=\calB(\und
G_{\kappa((u))},\kappa((u)))$, we see that $x_{\kappa((u))}$ can be
regarded as a point in the building $\calB(\und
G_{\kappa((u))},\kappa((u)))$.

\subsection{The main construction}\label{3b}

We now state the main result of this section. Let $(\und G,\und
A,\und S, \und P)$ be a rigidified group over $\O[u^{\pm 1}]$ as
defined in \S \ref{sss2c5}. Let $x\in \A(\und G_F,\und A_F,F)$. By
\ref{sss3a3}, we also obtain from $x$ points $x_{\kappa((u))}$ in
${\mathcal B}(\und G_{\kappa((u))}, \kappa((u)))$ for $\kappa=F$ or
$k$.

\smallskip

\begin{thm}\label{grpschemeThm}
There is a unique   smooth, affine group scheme $\Gg=\Gg_x\to
{\mathbb A}^1_{\O}=\Spec(\O[u])$ (called a Bruhat-Tits group
scheme for $\und G$) with connected fibers and with the following
properties:

1) The group scheme $\Gg_{|\O[u, u^{-1}]}$ is  the group scheme
$\underline{G}$.

2) The base change of $\Gg$ under $\Spec(\O)\to {\mathbb A}^1_{\O}$
given by $u\mapsto \varpi$ is the parahoric group scheme $\PP_{x}$ for
$\und G_F$ (as defined in \cite{BTII}).

3) The base change of $\Gg$ under $\Spec(\kappa[[u]])\to
{\mathbb A}^1_{\O}$ given by $\O[u]\to \kappa[[u]]$ is the parahoric group scheme
$\PP_{x_{\kappa((u))}}$ for $\und G_{\kappa((u))}$.
\end{thm}

\smallskip

\subsubsection{}\label{unique} We first prove the uniqueness statement in
Theorem \ref{grpschemeThm}. Recall that the parahoric group scheme $\PP_x$
is a group scheme (smooth, affine, connected) over $\Spec(\O)$ with
 generic fiber $\und G_F$ such that $\PP_x(\O)\subset \und G(F)$
is the connected stabilizer of $x$. Similarly for
$\PP_{x_{\kappa((u))}}$.

We will show that if $\Gg'$ is a smooth affine connected group scheme over
$\O[u]$ with $\Gg[u^{-1}]=\und G$ and which is such that
$\Gg'(F[[u]])\subset \und G(F((u)))$ is the connected stabilizer of
$x_{F((u))}$ in the building of $\und G_{F((u))}$ then
$\Gg'=\Gg_x$. In particular, this will show the uniqueness. To see this
set $\Gg_x=\Spec(B)$, $\Gg'=\Spec(B')$. Our assumptions imply
$B[u^{-1}]=B'[u^{-1}]$ and that $\Gg'(F[[u]])=\Gg_x(F[[u]])$. Since
$F$ is infinite and perfect, $F[[u]]$ is henselian and both the group
schemes $\Gg'$ and $\Gg_x$ are smooth, condition (ET 2) of \cite[1.7.2]{BTII} is
satisfied; hence the second identity implies that
$B'\otimes_{\O[u]}F[[u]]=B\otimes_{\O[u]}F[[u]]$. Observe now that
since $B$ is smooth over $\O[u]$,
$$
B=B[u^{-1}]\cap (B\otimes_{\O[u]}F[[u]])
$$
and similarly for $B'$. Hence $B=B'$ which gives $\Gg'=\Gg_x$. This
actually shows that $\Gg_x$ only depends on $\und G$ and
$x_{F((u))}$.

\subsubsection{} Here we show the existence part of Theorem \ref{grpschemeThm}.

a) First suppose that $\und G=H\otimes_\Z \O[u^{\pm 1}]$
is split; we will then consider more general convex subsets
of the apartment (and not just points). Denote by $\Phi=\Phi(H, T_H)$ the corresponding root
system. The pinning $(H, T_H, B_H, e)$ gives a hyperspecial vertex
$x_0$ of the apartment $\A(H, T_H, K)$ of $T_H$. This determines a
filtration $\{U_a(K)_{x_0, r}\}_{r\in \R}$ of the corresponding root
subgroups for any local field $K$. Let $f: \Phi\to \R$ be a concave
function; there is an associated convex subset $\Omega=\Omega_f$ of
the apartment $A$ given by
$$
\Omega=\{x\in A\ |\ a\cdot (x-x_0)+f(a)\geq 0\}.
$$
Conversely, to a convex bounded subset $\Omega\subset A$, we can
associate the concave function $f_\Omega: \Phi\to \R$ given by
$$
f_\Omega(a)={\rm inf}\{\lambda\in \R \ |\ a\cdot (x-x_0)+\lambda\geq
0, \forall x\in \Omega\}
$$
Notice that $x_0\in \Omega_f$ if and only if $f\geq 0$. Now denote
by $H(K)_{x_0, f}$ the subgroup of $H(K)$ generated by $U_a(K)_{x_0,
f(a)}$ and $T_H(\O_K)$. By \cite{BTII}, there is an associated
connected affine smooth group scheme $\PP_{x_0, f, K}$ over
$\Spec(\O_K)$. This only depends on $\Omega$ and we can denote it by
$\PP_\Omega$.

We will start by explaining how to construct a smooth affine group
scheme $\Gg_{x_0, f}$ over $\O[u]$ that lifts $\PP_{x_0, f, K}$ for
$K=F$ and $K=k((u))$ as it is required in the statement of the
theorem. We first consider the additive group schemes ${\mathcal
U}_{a, x_0, f}=u^{\langle f(a)\rangle }{U}_{a}\otimes_\Z \O[u]\simeq
{\mathbb G}_{a}\otimes_\Z \O[u]$ over $\O[u]$ and the torus
${\mathcal T}=T_H\otimes_{\Z}\O[u]$. (Here $\langle r \rangle$ is
the smallest integer that is larger or equal to $r$).  Notice here
that by definition $u{U}_{a}\otimes_\Z \O[u]=\Spec(\O[u, x/u])$ is
the dilatation of $U_a\otimes_\Z \O[u]\simeq {\mathbb
G}_{a}\otimes_\Z \O[u]$ along the zero section $\Spec(\O)\to
{\mathbb G}_a\otimes_\Z\O\hookrightarrow {\mathbb G}_{a}\otimes_\Z
\O[u]$ over $u=0$. Similarly, $u^{i}{U}_{a}\otimes_\Z
\O[u]=\Spec(\O[u, u^{-i}x])$. Since $f$ is concave, these group
schemes give as in \cite[3.1.1]{BTII}, or \cite{LandvogtLNM},
schematic root data over $\O[u]$. (In particular, see
\cite[3.2]{BTII} for this split case.) By using \cite{BTII} (3.9.4
together with 2.2.10) we obtain a (smooth) group scheme $\Gg_{x_0,
f}$ over $\O[u]$ which has a fiberwise dense open subscheme given by
$$
{\mathcal V}_{x_0,f}=\prod_{a\in \Phi^-}{\mathcal U}_{a,
x_0,f}\times {\mathcal T}\times \prod_{a\in \Phi^+}{\mathcal U}_{a,
x_0, f}.
$$
By \cite[1.2.13, 1.2.14]{BTII} this group scheme is uniquely
determined from the schematic root data. By \cite{BTII} or
\cite{LandvogtLNM} the group schemes $\Gg_{x_0, f, K}$ are also
given by the schematic root data obtained by base changing the
schematic root data on ${\mathcal V}_{x_0, f}$ via $\O[u]\to \O_K$.
This implies that $\Gg_{x_0, f}$ specializes to $\PP_{x_0, f}$ as in
(2) and (3). Property (1) also follows easily. It remains to show
that $\Gg_{x_0,f}$ is affine.
\medskip

(I) We first assume that $ \Omega$ contains  the  hyperspecial
vertex $x_0$,  i.e  we can take $f\geq 0$.
\medskip

When $f=0$, $\Gg_{x_0, f}\simeq H\otimes_\Z \O[u]$ which is affine.
We will build on this, showing that, in general, $\Gg_{x_0, f}$ can
be obtained from $\Gg_{x_0, 0}$ by a series of dilatations. This
follows an argument of Yu (\cite{YuModels}) and provides   an
alternative construction of the group schemes. Since such
dilatations are affine, we will obtain the desired conclusion.

(i) Assume first that $0\leq f\leq 1$. Consider the parabolic
subgroup $P_f\subset H$ over $\Z$ containing $T=T_H$ which
corresponds to the set of roots $a\in \Phi$ with $f(a)=0$: For any
field $k$, $P_f(k)$ is generated by $T(k)$ and $U_a(k)$ with
$f(a)=0$.  In this case, we
 can consider the dilatation of $H\otimes_\Z\O[u]$ along
the closed subgroup scheme given by  $P_f\otimes_\Z \O \hookrightarrow
H\otimes_\Z \O[u]$ over $u=0$: Suppose that $H=\Spec(A)$ and
$P_f=\Spec(A/I)$. Set $J=\O\cdot I +(u)$ for the ideal generated by
$I$ and $u$ in $A\otimes \O[u]$. Then we consider
$$
\displaystyle{B=(A\otimes \O[u])[ ju^{-1} \ |\ j\in J]\subset
A\otimes\O[u, u^{-1}]}
$$
and set
$$
\Hh_{x_0,f}=\Spec(B).
$$
(We refer the reader to \cite{Waterhouse} for basic properties of
dilatations of group schemes.) If $R$ is an $\O[u]$-algebra we set
$\bar R=R/(u)=R\otimes_{\O[u]}\O$ and we denote by  $\bar h  \in
H(\bar R)$ the reduction of $h\in H(R)$. Then
$\Hh_{x_0, f}$ is the unique, up to isomorphism, group scheme with $u$ not a zero divisor in its structure sheaf, that
supports a homomorphism $\Psi: \Hh_{x_0, f}\to H\otimes_{\Z}\O[u]$
with the following properties:
\begin{itemize}
\item  $\Psi$ is an isomorphism away from $u=0$,
\item for any $\O[u]$-algebra $R$ with $u$ not a zero divisor in $R$,
$\Psi$ identifies $\Hh_{x_0, f}(R)$ with the subset $\{ h\in H(R)\
|\ \bar h\in P_f(\bar R)\}$ of $H(R)$.
\end{itemize}
We can see that $\Hh_{x_0,f}$ has connected fibers and is given by
the same group germ ${\mathcal V}_{x_0,f}$ as $\Gg_{x_0,f}$.
Therefore, $\Gg_{x_0,f}\simeq \Hh_{x_0, f}=\Spec(B)$ and is also
affine.

(ii) In general, we can find a finite sequence $0=t_0<t_1<\cdots <
t_n=1$ such that if $f_i=t_i\cdot f$, we have $f_n=f$, and
$f_{i+1}\leq f_i+1$. We will use induction on $n$. The case $n=1$ is
given by (i) above. Note that by the construction of
\cite[3.9.4]{BTII}  we have:

a) There are closed group scheme embeddings $\UU_{a,
i}\hookrightarrow \Hh_{i}$ and ${\mathcal T}\hookrightarrow \Hh_{i}$
that extend the standard embeddings of the root subgroups and the
maximal torus over $\O[u, u^{-1}]$. (Here and in the rest of the
proof, for simplicity, we omit some subscripts and write for example $\calU_{a,
i}$, $\Hh_i$, instead of $\calU_{a, x_0, f_i}$, $\Hh_{x_0, f_i}$.)

b) The embeddings in (a) combined with the multiplication morphism
induce an open immersion
$$
j: {\mathcal V}_{i}=\prod_{a\in \Phi^-}{\mathcal U}_{a, i}\times
{\mathcal T}\times \prod_{a\in \Phi^+}{\mathcal U}_{a,
i}\hookrightarrow \Gg_{i}
$$
onto a fiberwise dense open subscheme of $\Gg_i$ which makes the
schematic root data ${\mathcal D}_{i}=({\mathcal T}, (\UU_{a,
i})_{a\in \Phi})$ compatible with
 $\Hh_{i}$ (in the sense of \cite{BTII}, 3.1.3).

c)  The group scheme   multiplication morphism
$$
{\mathcal V}_{i}\times {\mathcal V}_{i}\to \Gg_{i}
$$
is surjective on each fiber.

Our induction hypothesis is:  The smooth group scheme
$\Gg_n:=\Gg_{x_0,f_n}$ is affine and supports a group scheme
extension
\begin{equation}\label{extension}
1\to R_n\to \overline {\Gg}_{n}\to \overline {\Gg}_{ n}^{\rm red}\to
1
\end{equation}
of a (connected) reductive group scheme by a smooth affine unipotent
group scheme (both over $\O$). Here the bar indicates base change
via $\O[u]\to \O$ given by $u\mapsto 0$. The case $n=1$ follows from
(i) above. Then, $\overline {\Gg}_{1}^{\rm red}$ is the Levi
component of the parabolic $ P$.

 Let us consider $\Gg_{n+1}$. There is a natural morphism ${\calU}_{a, x_0, f_{n+1}}\to \calU_{a,x_0, f_n}$.
This is the identity when $f_{n+1}(a)=f_n(a)$ and is given by
dilatation of the zero section over $u=0$ when $f_{n+1}(a)>f_n(a)$.
These morphisms combine to give $f_n: \calV_{n+1}\xrightarrow{\ }
\calV_n\subset \Gg_n$. Since $\Gg_{n+1}[u^{-1}]=\Gg_n[u^{-1}]$, by
\cite[1.2.13]{BTII} there is a unique group scheme homomorphism $\ti
f_n: \Gg_{n+1}\to \Gg_n$ that extends $f_n$.

We will now show that $\ti f_n$ identifies $\Gg_{n+1}$ with a
dilatation of $\Gg_n$. Consider the (set-theoretic image)
 $Q_n$ of $\bar\Gg_{n+1}\to \bar\Gg_n$ which is a constructible set. We will show that $Q_n$ is closed in $\bar\Gg_n$ and is underlying a smooth
 group scheme which we will also denote by $Q_n$. (This will imply that $\Gg_{n+1}\to\Gg_n$ factors through
 $\Gg_{n+1}\to \Gg_n'\to\Gg_n$ where $\Gg'_n\to \Gg_n$ is the dilatation of $\Gg_n$ along $Q_n\subset \bar\Gg_n$.)
Using property (c) of $\Gg_{n+1}$, we see that the image $Q_n$ can
be identified with the image of
$\bar{\calV}_{n+1}\times\bar\calV_{n+1}$
 under the product map in $\bar\Gg_n$. By construction, we have morphisms
$\overline{\UU}_{a, {n+1} }\to Q_{n}$.
 Suppose $\kappa$  is either $F$ or   $k$. Denote by $V_{a,n}(\kappa)$ the image of the corresponding
$\overline{\UU}_{a, n+1 }(\kappa)\to Q_{n}(\kappa)$. Then by the
above, we see that the $\kappa$-valued points $Q_{n}(\kappa)$ is the
subgroup of $\overline {\Gg_n}(\kappa)$ generated by the groups
 $V_{a, n}(\kappa)$  and $\overline {\mathcal T}(\kappa)$. By the argument in \cite[8.3.2]{YuModels},
we see that the fibers of $Q_n$ over $\kappa=k$ and $F$, are closed
in the corresponding fibers of $\Gg_n$. Now consider the extension
(\ref{extension}). In this, $\overline {\Gg}_{ n}^{\rm red}$
contains the maximal torus $\overline{\mathcal T}$ and it
corresponds to the root datum given by $(\xch(\overline{\mathcal
T}), \Phi_{f_n}, \xcoch(\overline{\mathcal T}), \Phi^\vee_{f_n}))$
(with $\Phi_{f}=\{a\in \Phi \ |\ f(a)+f(-a)=0\}$). On the other
hand, as a scheme, $R_n$ is an affine space:
$$
R_n=\prod_{\{a\in \Phi\, |\, f_n(a)+f_n(-a)>0\}} \overline{\mathcal
U}_{a, n}.
$$
We now consider the Zariski closure $\tilde Q_n$ of the generic
fiber $Q_n(F)$ in $\overline\Gg_n$. This agrees with the Zariski
closure of $\overline{\calV}_{n+1}\hookrightarrow \overline{\Gg}_n$.
We have of course $Q_n\subset \tilde Q_n$. We can see that $\tilde
Q_n\subset \overline{\Gg}_n$ maps in $ \overline {\Gg}_{ n}^{\rm
red}$ onto the closed parabolic subgroup scheme $P_n$ of $\overline
{\Gg}_{ n}^{\rm red}$ generated by $\overline{\mathcal T}$ and
$\overline{\mathcal U}_{a, n}$ with $f_{n+1}(a)=f_{n}(a)=0$. On the
other hand, the intersection of $\tilde Q_n$ with $R_n$ can be
identified as the closed subscheme of  $R_n= \prod_{a,
f_n(a)+f_n(-a)>0} \overline{\mathcal U}_{a, n}$ given by the product
of those $\overline{\mathcal U}_{a, n}$ for which
$f_{n+1}(a)=f_n(a)$ (with $f_n(a)+f_n(-a)>0$). This and the above
allows us to conclude that $\tilde Q_n$ is an affine fibration over
$P_n$ and so all fibers of $\tilde Q_n$ are geometrically connected.
It follows that $Q_n=\tilde Q_n$ and so $Q_n$ is a smooth closed
subgroup scheme of $\overline\Gg_n$. We can see following the
argument in \cite[8.3.3]{YuModels},  that we have
\begin{equation}\label{bigcell}
Q_{n}\cap \overline {{\mathcal V}}_{n} ={\rm Im}(\overline
{{\mathcal V}}_{n+1}\rightarrow \overline{\calV}_n).
\end{equation}
As above $\Gg_{n+1}\to \Gg_n$ factors $\Gg_{n+1}\to \Gg'_n\to
\Gg_n$, where $\Gg'_n$ is the dilatation of $\Gg_n$ along $Q_n\subset
\overline\Gg_n$. Then $\Gg'_n$ is a smooth affine  group scheme over
$\Spec(\O[u])$ with connected fibers. Observe that  $\UU_{a,
f_{n+1}}$ is by definition isomorphic to the dilatation of $\UU_{a,
f_n}$ along the image of the morphism $\overline{\UU}_{a,
f_{n+1}}\to \overline{\UU}_{a, f_n}$. As a result, the dilatation of
${\mathcal V}_{n}$ along the image of $\overline {{\mathcal
V}}_{n+1}\to \overline {{\mathcal V}}_{n}$ is isomorphic to
$\overline {{\mathcal V}}_{n+1}$. It now follows from the
functoriality  of the dilatation construction and (\ref{bigcell})
that  the dilatation  $\Gg'_{n}$ of $\Gg_n$ along $Q_{n}$ has an
open subscheme isomorphic to ${\mathcal V}_{n+1}\subset \Gg_{n+1}$.
Since ${\mathcal V}_{n+1}$ is fiberwise dense in $\Gg_{n+1}$ it
follows that $\Gg_{n+1}=\Gg_n'$ and hence $\Gg_{n+1}$ is also
affine. The rest of the induction hypothesis for $\Gg_{n+1}$ also
follow. Again, $\overline{\Gg}^{\rm red}_{n+1}$ is the Levi
component of the parabolic of $\overline {\Gg}_{ n}^{\rm red}$ that
corresponds to $a\in \Phi_{f_n}$ for which $f_{n+1}(a)=f_n(a)$.

\smallskip

II) We continue to assume that $\und G=H\otimes_\Z \O[u^{\pm 1}]$ is
split but now we consider the general case in which
$\Omega=\Omega_f$ does not contain $x_0$. The argument in the proof
of \cite[Lemma 2.2]{GilleTorsors} (see also \cite{LarsenDuke})
implies that there is an integer $\delta\geq 1$ such that the subset
$\Omega_{\delta\cdot f}$ of the apartment contains a hyperspecial
vertex $x'_0$ which is the translate $x'_0=x_0+t$ of $x_0$ by an
element $t\in \xcoch(T)$. Consider the homomorphism $\O[u]\to \O[v]$
given by $u\mapsto v^\delta$. Our previous arguments allow us to
construct, using successive dilatations of
$t(H\otimes_\Z\O[v])t^{-1}\simeq H\otimes_\Z\O[v]$, a smooth affine
group scheme $\Gg'_{\Omega}=\Gg_{x'_0, f'}$ over $\O[v]$. (Here
$f'=\delta\cdot f+t$ which is positive.) We can however see that
base changing the schematic root data for $\Gg_{\Omega}$ by
$\O[u]\to \O[v]$ gives schematic root data for $\Gg'_{\Omega}$. As
above, using \cite[1.2.13, 1.2.14]{BTII}, this implies that
$\Gg_{\Omega}\otimes_{\O[u]}\O[v]\simeq \Gg'_{\Omega}$. By faithful
flat descent, $\Gg_{\Omega}$ is then affine.

\smallskip

b)  We now consider the more general case of a quasi-split group
$\und G$ that splits over an extension $\O_0[v^{\pm 1}]/\O[u^{\pm
1}]$ as in \ref{sss2a1}. In particular, $\O$ has residue field
$k={\mathbb F}_q$, $u\mapsto v^e$, and $\Gamma^t$ acts on
$\ti\O_0[v]$ by $\gamma_0(v)=\zeta\cdot v$, $\sigma(\sum a'_i
v^i)=\sum_i \sigma(a'_i)v^i$, with invariants
$\ti\O_0[v]^\Gamma=\O[u]$. Let $\ti F/F$ be the base change of
$\O_0[v^{\pm 1}]/\O[u^{\pm 1}]$ along $u=\varpi$, with the maximal
unramified extension $\ti F_0/F$ of degree $r$ and $e=[\ti F:\ti
F_0]$.

As $\und G$ is quasi-split, $\und G=\und G^*$. We consider
$\Omega\subset \A(\und G_F, \und A_F, F)=\A(H, T_H,  \ti
F)^\Gamma\subset \A(H, T_H, \ti F)$ and let $ \Hh_\Omega$ be the
smooth affine group scheme over $ {\mathbb
A}^1_{\ti\O_0}=\Spec(\ti\O_0[v])$ constructed for the split group
$H$ as in (a). The group $\Gamma$ acts on the apartment $\A(H, T_H,
\ti F)$ via its action on $H$ and $\ti F$. Since $\Omega$ is fixed
by $\Gamma$, we can see that $\Hh_\Omega$ supports a $\Gamma$-action
that lifts the action on $\Gamma$ on ${\mathbb A}^1_{\ti\O_0}$.
Notice that the Weil restriction of scalars $ {\rm
Res}_{\ti\O_0[v]/\O[u]}\Hh_\Omega $ is also a smooth affine group
scheme over $\O[u]$ (\cite{BTII}, \cite[2.2]{EdixhTame}). By the
above and \cite[2.4]{EdixhTame}, this supports a $\Gamma$-action
over $\Spec(\O[u])$; the  inertia  groups for this action are always
subgroups of  $\langle\gamma_0\rangle$. Since, $\gamma_0$ has order
prime to $p$,  by \cite[3.4]{EdixhTame}, the fixed point scheme
$\Gg'_\Omega=({\rm Res}_{\ti\O_0[v]/\O[u]}\Hh_\Omega)^{\Gamma}$
 is a smooth closed subscheme of the smooth
affine ${\rm Res}_{\ti\O_0[v]/\O[u]}\Hh_\Omega$. Hence, it is also
flat over $\O[u]$. Consider the base change
$\Gg'_\Omega\otimes_{\O[u]}\O$ by $u\mapsto 0$. Since this is also
smooth over $\O$, it is the disjoint union
$$
\Gg'_\Omega\otimes_{\O[u]}\O=Z^0\sqcup(\sqcup_i Z_i)
$$
of its smooth irreducible components where $Z^0$ contains the
identity section. By flatness of $\Gg'_\Omega\to \AA^1_\O$, all the
components $Z^0$, $Z_i$ are divisors in $\Gg'_\Omega$. We will set
$$
\Gg''_\Omega=\Gg'_\Omega-\sqcup_i Z_i
$$
(i.e the complement of the union of those components that do not
contain the identity.) Observe that $\Gg''_\Omega$  is affine since
it can also be obtained as the dilatation of the affine
$\Gg'_\Omega$ along the affine and smooth closed subscheme $Z^0$ of
its fiber over $u=0$. We will now show   that $\Gg''_\Omega$ is the
connected component $ \Gg'^0_\Omega$ of $\Gg'_\Omega$. By
\cite[1.2.12]{BTII} we have to show that each fiber of
$\Gg''_\Omega\to \AA^1_\O$ is connected. First observe that the
geometric fibers at points where $u\neq 0$ are isomorphic to the
split form $H$ and so they are connected. By construction, $
\Gg''_\Omega \otimes_{\O[u]}\O=Z^0$; therefore the fiber
$\Gg''_\Omega \otimes_{\O[u]}F$ is connected and is  the connected
component of $\Gg'_\Omega\otimes_{\O[u]}F$. In general for
$\kappa=F$ or $k$, let $\kappa'$  be   $\ti F_0$, resp. the residue
field of $\ti \O_0$. We can consider the fiber  over $\O[u]\to
\O\to\kappa'$
$$
\Gg'_\Omega\otimes_{\O[u]}\kappa'=({\rm
Res}_{\O'[v]/\O[u]}(\Hh_\Omega)\otimes_{\O[u]}\kappa')^\Gamma.
$$
Since $\Gamma$ is an extension of ${\rm Gal}(\kappa'/\kappa)$ by
$\langle\gamma_0\rangle$ this is
$$
({\rm
Res}_{\kappa'[v]/(v^e)/\kappa'}(\Hh_\Omega\otimes_{\ti\O_0[v]}\kappa'[v]/(v^e)))^{\gamma_0}.
$$
We have
$$
1\to U\to {\rm
Res}_{\kappa'[v]/(v^e)/\kappa'}(\Hh_\Omega\otimes_{\ti\O_0[v]}\kappa'[v]/(v^e))
\to
 \overline\Hh^{ \rm red}_{\Omega,\kappa'}\to 1
$$
with $U$ unipotent and $\overline\Hh^{\rm red}_{\Omega,\kappa'}$
(split) reductive over $\kappa'$. Now the maximal reductive quotient
$M:=\overline\Hh^{\rm red}_{\Omega}$ of
$\overline\Hh_{\Omega}=\Hh_{\Omega}\otimes_{\ti\O_0[v]}\ti\O_0$ is a
Chevalley (reductive) group scheme  over $\ti\O_0$. Since
$\gamma_0(\Omega)=\Omega$ we have an action of $\gamma_0$ on $M$.
Since $\gamma_0$ has order prime to $p$, we can see  that
$\rH^1(\langle\gamma_0\rangle, U)=(0)$. Also, by filtering $U$ by vector
groups we can see that $U^{\gamma_0}$ is connected (see
\cite[4.7.2]{YuModels}). It follows that the group of connected
components of $\Gg'_\Omega\otimes_{\O[u]}\kappa'$ is identified with
that of $(\overline\Hh^{\rm red}_{
\Omega,\kappa'})^{\gamma_0}=(\overline\Hh^{ \rm
red}_{\Omega})^{\gamma_0}\otimes_{\O'}\kappa'=M^{\gamma_0}\otimes_{\ti\O_0}\kappa'$.
We can now see that the $\gamma_0$-action on $M$ satisfies the
assumptions of Proposition \ref{locconstant}, i.e preserves a pair
of maximal split torus and a Borel subgroup that contains it:
Indeed,  by construction, $\gamma_0$ preserves the maximal torus
given by $T_H$. Now, as before, consider the subset $\Phi_\Omega$ of
the set of roots $\Phi$ such that there is an affine root with
vector part $a$ and defining a hyperplane containing $\Omega$. The
set $\Phi_\Omega$ can be identified with the roots of $M$ with
respect to the maximal torus given by  $T_H$. The group
$\langle\gamma_0\rangle$ acts on $\Phi_\Omega$. The intersection
$\Phi^+_\Omega:=\Phi^+\cap \Phi_\Omega$ is a system of positive
roots in $\Phi_\Omega$ which is stable under $\gamma_0$. Let $C$ be
the affine chamber containing $\Omega$ in its closure which is given
by $\Phi^+_\Omega$; this provides us with a $\gamma_0$-stable Borel
subgroup in $M$ which contains  $T_H$. We can now  apply Proposition
\ref{locconstant} to the Chevalley group $M=\overline\Hh^{\rm
red}_{\Omega}$ over $\ti\O_0$ and the automorphism induced by
$\gamma_0$ as above. We obtain that the group scheme of connected
components of $\Gg'_\Omega\otimes_{\O[u]}\kappa'$ is given by the
fibers of a finite \'etale commutative group scheme of order
annihilated by $e$ and this order is the same for $\kappa=F$ or
$\kappa=k$. Since $\Gg_\Omega''\otimes_{\O[u]}F$ is the neutral
component of $\Gg_\Omega'\otimes_{\O[u]}F$, the base change
$\Gg_\Omega''\otimes_{\O[u]}\ti F_0$ is also connected (\cite[Exp.
${\rm VI}_{\rm A}$, 2.1.1]{SGA3}). The above now implies   that
$\Gg''_\Omega\otimes_{\O[u]}\ti k$ and hence
 $\Gg''_\Omega\otimes_{\O[u]}k$ is also connected. Therefore, $\Gg''_\Omega=(\Gg'_\Omega)^0$.
 We set $\Gg_\Omega:=\Gg''_\Omega=(\Gg'_\Omega)^0$. It remains to show
 that the base changes of $\Gg_\Omega$ by $\O[u]\to \O$, $u\mapsto \varpi$, resp.
 $\O[u]\to \kappa[[u]]$, are isomorphic to the parahoric group schemes $\PP_x$,
 resp. $\PP_{x_{\kappa((u))}}$.

 For simplicity, set $L=\bar\kappa((u))$, $L'=\bar\kappa((v))$,
 $R=\bar\kappa[[u]]$, $R'=\bar\kappa[[v]]$ and denote by $H(L')_\Omega\subset H(L')$ the
 stabilizer of $\Omega\subset {\mathcal B}(H, L)={\mathcal B}(H, L')^{\gamma_0}$.
 We also set $\Gg(L)=(H(L'))^{\gamma_0}$ (the points of a connected reductive group over $L$).
 Notice that by construction, $\Gg'_\Omega(R)=\Hh_\Omega(R')^{\gamma_0}$ while the
 result in the split case together with \cite[4.6]{BTII}
 gives $\Hh_\Omega(R')\subset H(L')_\Omega$ with finite index. Hence, we have
 $$
 \Gg'_\Omega(R)=\Hh_\Omega(R')^{\gamma_0}\subset (H(L')_\Omega)^{\gamma_0}=(H(L')^{\gamma_0})_\Omega=
 \Gg(L)_\Omega
 $$
and we see that $\Gg'_\Omega(R)$ is of finite index in the
stabilizer $\Gg(L)_\Omega$ of $\Omega\subset  {\mathcal B}(\Gg, L)$
 in $\Gg(L)$. Since $\Gg_\Omega$ is the neutral component of $\Gg'_\Omega$
 we conclude that $\Gg_\Omega\otimes_{\O[u]}R$ is the smooth connected stabilizer of $\Omega$,
 i.e we have $\Gg_\Omega\otimes_{\O[u]}R=\PP_{x_{\kappa((u))}}\otimes_{\kappa[[u]]}\bar\kappa[[u]]$.
 By \cite[1.7]{BTII}, this shows the desired result
 for the base change $\O[u]\to \kappa[[u]]$. The case $\O[u]\to \O$, $u\to \varpi$
 is similar. This
concludes the proof of Theorem \ref{grpschemeThm} in the quasi-split
case.

\smallskip

(c) Finally, we consider the general case. Recall our notations and
in particular the choice of ${\rm Int}(\bold{g})$ in $\und
{N'}^*(\breve{\O}[u^{\pm 1}])$. This gives the semilinear
$$
{}^*\sigma:={\rm Int}(\bold g)\cdot \sigma: \und
G^*\otimes_{\O}\breve{\O}\to \und G^*\otimes_{\O}\breve{\O}
$$
which covers the Frobenius $\breve{\O}[u^{\pm 1}]\to
\breve{\O}[u^{\pm 1}]$. We also have the inner twist $\und G$ of
$\und G^*$ over $\O[u^{\pm 1}]$ defined by taking ${}^*\sigma$ fixed
points of $\und G\otimes_{\O}\breve{\O}$ as in (\ref{inner}). Our
construction applied to the quasi-split $G^*_{\Fsm}$ and the point
$x^*$ given as $\psi_*(x)$ provides with a group scheme
$\Gg^*_{x^*}$ over $\breve{\O}[u]$ which satisfies the conclusions
of the Theorem. In particular, we have
$$
\Gg^*_{x^*}{|_{\breve{\O}[u^{\pm 1}]}}=\und
G^*\otimes_{\O}\breve{\O}.
$$
We will now show that ${}^*\sigma:  \und
G^*\otimes_{\O}\breve{\O}\to \und G^*\otimes_{\O}\breve{\O}$ extends
to an $\sigma$-semilinear
$$
{}^*\sigma: \Gg^*_{x^*}\to \Gg^*_{x^*}.
$$
We first verify that it is enough to check that the base change of
${}^*\sigma$ over $\Fsm((u))$ extends to $\Fsm[[u]]$:   Indeed,
since $\Gg^*_{x^*}=\Spec(A)$ is affine and smooth over
$\breve{\O}[u]$ we can write $A=A[u^{-1}]\cap
(A\otimes_{\breve{\O}[u]}\Fsm[[u]])$. Since ${}^*\sigma$ is defined
over $\O[u, u^{-1}]$, it remains to check that ${}^*\sigma$
preserves $A\otimes_{\breve{\O}[u]}\Fsm[[u]]\subset
A\otimes_{\breve{\O}[u]}\Fsm((u))$. Now let check that the base
change of ${}^*\sigma$ over $\Fsm((u))$ extends to $\Fsm[[u]]$:
Consider $x^*_{\breve{F}((u))}$ which by our construction is fixed
by ${\rm Int}(\bold{g})\cdot \sigma$. This implies that
\begin{equation}\label{eq3.34}
{}^*\sigma (\calP_{x^*_{\Fsm((u))}}(\Fsm[[u]]))\subset
\calP_{x^*_{\Fsm((u))}}(\Fsm[[u]]).
\end{equation}
Since $F[[u]]$ is   henselian and $F$ infinite and perfect,
condition (ET 2) of  \cite[1.7.2]{BTII} is satisfied. Therefore,
(\ref{eq3.34}) implies that ${}^*\sigma$ extends to
$\calP_{x^*_{\Fsm((u))}}$ which, by our construction, is the base
change of the group scheme $\Gg^*_{x^*}$ to $\Fsm[[u]]$.

We now define $\Gg_x$ to be the group scheme over $\O[u]$ given by
the Weil descent datum provided by the action of ${}^*\sigma={\rm
Int}(\bold{g})\cdot \sigma$ on $\Gg^*_{x^*}$ over $\breve{\O}[u]$.
Since $\Gg^*_{x^*}$ is affine, we can indeed see that $\Gg_x$ is
represented by an affine group scheme over $\O[u]$, which then
satisfies all the requirements in the Theorem.
\endproof

\subsection{}\label{ss3c}  In this section we start with a connected reductive group $G$ over the
$p$-adic field $F$ that splits over the tamely ramified $\ti F/F$ and $x\in \calB(G,F)$   a point in the building of $G$. Let $\calP_x$ be the corresponding parahoric group
scheme over $\O$. We will apply Theorem \ref{grpschemeThm} to construct a suitable
group scheme $\Gg=\Gg_x$ over $\O[u]$ (see Corollary \ref{application} below).

\subsubsection{}\label{sss3a4}  
Starting from the above data, we choose a rigidification $(G,A,S,P)$ such that $x\in\A(G,A,F)$. Choose also a pinned split form $H$ of $G$
over $\O$ and let $(\und G,\und A,\und S,\und P)$ be the corresponding rigidified group over
$\O[u^{\pm 1}]$ as constructed in \S \ref{sss2c2}. Let $\und \psi:(\und G,\und
S,\und P)\otimes\breve\O[u^{\pm 1}]\xrightarrow{\sim} (\und G^*,\und S^*,\und
P^*)\otimes\breve\O[u^{\pm 1}]$ be an inner twist and $\und c^{\rm
rig}$ be the corresponding cocycle. Let $(G_0,A_0,S_0,P_0): =(\und
G_F, \und A_F, \und S_F, \und P_F)$ be the base change via
$\O[u^{\pm 1}]\to F$ given by  $u\mapsto \varpi$. The specialization
of the inner twist by  $u\mapsto \varpi$ is denoted by $\psi_0$.
The set of inner twists
\[\psi: (G,S,P)\otimes_F \Fsm\xrightarrow{\sim} (G^*,S^*,P^*)\otimes_F \Fsm\]
such that $\psi\cdot\sigma(\psi)^{-1}=\und c^{\rm
rig}(1)|_{u=\varpi}$ forms an $N'(F)$-torsor, where $N'=N_{M'}(T)$.
As the cocycles corresponding to $\psi$ and $\psi_0$ are the same,
the morphism $\psi_0\psi^{-1}$ is defined over $F$. Therefore, the choice
of a rigidification $(G,A,S,P)$ of $G$ produces an
isomorphism $\alpha: G\xrightarrow{\sim} G_0=\und G_F$, which is well-defined
up to the action of $N'(F)$. Since $N'$ centralizes $A$, $N'(F)$
acts trivially on $\A(G,A,F)$; hence, $x$ corresponds to a
well-defined point $x_0$ in $\A(\und G_F, \und A_F,F)$. Note that, however,
$x_0$ depends on the choice of the rigidification $(G, A,S,P)$. 

Let $(G, A',S',P')$ be another choice of rigidification of $G$ and
let $x'_0\in \A(\und G_F, \und A_F,F)$ be the corresponding point. By Lemma
\ref{transitive}, there is $g\in G_{\rm ad}(F)$ sending $(A,S,P)$ to
$(A',S',P')$. Therefore, as points in the building of $G_0=\und G_F$, $x_0$
and $x'_0$ are in the same $G_{0,{\rm ad}}(F)$-orbit. Therefore,
there is an element $n$ in $N_{G_{0,\rm ad}}(A_0)$ that sends $x_0$
to $x'_0$.
 
Now apply Theorem \ref{grpschemeThm} to $(\und G, \und A, \und S, \und P)$
and $x_0\in \A(\und G_F, \und A_F,F)$. We obtain a  group scheme
$\calG_{x_0}$ whose specialization along
$u=\varpi$ gives back $\calP_x$. Observe that $x_0$ is not
uniquely determined by $x$, but also depends on the rigidification
of $G$. However, as explained before, different $x_0$'s are in the
same $N_{G_{0,{\rm ad}}}(A_0)(F)$-orbit. It is easy to see that an
$N_{G_{0,{\rm ad}}}(A_0)(F)$-orbit on $\A(G_0,A_0,F)$ is the same as
an ${\rm Im}(N_{\und G_{\rm ad}}(\und A)(\O[u^{\pm
1}])\stackrel{u=\varpi}{\to} N_{G_{0,{\rm ad}}}(A_0)(F))$-orbit.
Therefore, different $\calG_{x_0}$'s are isomorphic to each other via
conjugation by an element in $N_{\und G_{\rm ad}}(\und A)(\O[u^{\pm
1}])$ and so the isomorphism class of this group scheme
is independent of choices; we will denote the group scheme by $\Gg_x$. In particular,   we obtain:

\begin{cor}\label{application} Starting with $G$ and $x$ as in the beginning of \S \ref{ss3c}, there exists a smooth, affine group scheme $\Gg=\Gg_x\to {\mathbb
A}^1_{\O}=\Spec(\O[u])$ with connected fibers such that

1) The group scheme $\Gg_{|\O[u, u^{-1}]}$ is $\und G$;

2) The base change of $\Gg$ under $\Spec(\O)\to {\mathbb A}^1_{\O}$
given by $u\mapsto \varpi$ is $\calP_x$ and the base change of $\Gg$ under $\Spec(\kappa[[u]])\to
{\mathbb A}^1_{\O}$ given by $\O[u]\to \kappa[[u]]$ is  
$\PP_{x_{\kappa((u))}}$.\endproof
\end{cor}

\bigskip

\section{Classical groups}
\setcounter{equation}{0}

Recall that when $G$ is a classical group over the local field $F$,
Bruhat and Tits have given a description of the building ${\mathcal
B}(G, F)$ as a set of certain norms on the space of the ``natural
representation" (\cite{BTclassI}, \cite{BTclassII}). At least when
$p$ is odd, this produces a description of the facets of the
building in terms of self-dual lattice chains in this space. The
corresponding parahoric group scheme can also be explicitly
described as the neutral component of the polarized automorphisms of
the lattice chain. In this chapter, we extend some of this picture
to the group schemes over $\O[u]$ constructed in Theorem
\ref{grpschemeThm}.

\subsection{Lattice chains}\label{ss4a}
First we recall the set-up of lattice chains over $\O$ (cf.
\cite{BTclassI}, \cite{BTclassII}, \cite{RapZinkBook}).

\subsubsection{}\label{sss4a1}

Suppose first that $D$ is a central division $F$-algebra
 of degree $d$ and Brauer invariant $s/d$ with $0< s<d$
and ${\rm gcd}(s,d)=1$. Recall $\O$ has residue field ${\mathbb
F}_q$, $q=p^m$. Let $F_d=F\Q_{p^{md}}$ which is then an unramified
extension of $F$ of degree $d$ with integers $\O_d$. Set
$\sigma={\rm Frob}_{p^m}$ for the generator of the Galois group of
$F_d/F$. We can   write
$$
D=F_d\oplus F_d\cdot \Pi\oplus \cdots \oplus F_d\cdot \Pi^{d-1}
$$
with relations $ \Pi^d=\varpi$, $a\cdot \Pi = \Pi\cdot \sigma^s(a)$
for all $a\in F_d$. This contains the maximal order
$\O_D=\O_{F_d}\oplus \O_{F_d}\cdot \Pi\oplus \cdots
\oplus\O_{F_d}\cdot \Pi^{d-1}$.

Consider $V=D^n$ as a left $D$-module and let $G=\GL_n(D)={\rm
Aut}_D(V)$   identified by sending the matrix $A\in \GL_n(D)$ to the
automorphism $x\mapsto x\cdot A^{-1}$. A lattice $\L$  in $V=D^n$ is
a finitely generated $\O_D$-submodule of $V$ that contains a
$D$-basis of $V$; then $\L$ is $\O_D$-free of rank $n$. Recall that
a lattice chain in $V$ is a totally ordered (non-empty) set
$\L_\bullet$ of $\O_D$-lattices in $V$ which is stable under
homotheties. It can be represented as:
\begin{equation}\label{eq435}
\cdots\subset\Pi\L_0\subset\L_{r-1}\subset\cdots \subset \L_1\subset
\L_0\subset\cdots .
\end{equation}
By \cite{BTclassI}, the facets $\Omega$ in $\calB(G, F)$ correspond
bijectively to  $\O_D$-lattice chains (cf. \cite{RapZinkBook}) in
$V=D^n$ (Bruhat and Tits consider right modules but this is
equivalent). Then the parahoric group scheme $\PP_x$ ($x\in\Omega$)
is the group scheme over $\O$ given by the $\O_D$-linear
automorphisms of the corresponding chain, i.e
$$
\PP_x(R)={\rm Aut}_{\O_D\otimes_\O R}(\{\L_\bullet\otimes_\O R\}).
$$

\subsubsection{}\label{sss4a2} More generally suppose that $D$ is a central
simple $L$-algebra with an involution $\tau$ such that $\tau(\O_D)=\O_D$;
assume that $F$ is the fixed field of the involution on $L$. Let
$\ep=\pm 1$ and let $h$ be an $\epsilon$-hermitian $D$-valued form
on $V=D^n$ with respect to $\tau$. If $\L$ is an $\O_D$-lattice in
$V$, we can consider its dual $\L^\vee=\{x\in \L\ |\ h(x,
\lambda)\in \O_D, \forall \lambda\in\L\}$. A lattice chain
$\L_\bullet$ is called self dual if $\L\in \L_\bullet$ if and only
if $\L^\vee\in \L_\bullet$. The form $h$ defines an involution $*$
on $\GL_n(D)$ by $ h(xA, y)=h(x, yA^*)$. Consider the unitary group
${\rm U}(V, h)=\{A\in \GL_n(D)\ |\ (A^*)^{-1}=A\}$ given by elements
of $\GL_n(D)$ that respect $h$ and let $G$ be its neutral component.
Recall that we assume $p$ is odd. By \cite{PrasadYuInv}, the
building $\calB(G, F)$ can be identified with the fixed points of
the action induced by $A\mapsto (A^*)^{-1}$ on $\calB(\GL_n(D), L)$. Using the above, we
now see that facets $\Omega$ in $\calB(G, F)$ correspond to
self-dual $\O_D$-lattice chains $\L_\bullet$  in $V$. (This also
follows from the explicit description in \cite{BTclassII}, noting
that when $p\neq 2$, the maximinorante norms of loc. cit. can be
described via self-dual graded lattice chains). It then also follows
that the parahoric Bruhat-Tits $\PP_x$ ($x$ a generic point in
$\Omega$)
 is  the neutral component of the group scheme over $\O$ given by
$\O_D$-linear automorphisms of the chain $\L_\bullet$ that respect
the perfect forms $\L_i\times\L_j\to \O_D$ obtained from $h$.
There is also a similar description for unitary similitude groups.
\smallskip

We now extend most of this picture to the group schemes over
$\O[u]$. We first start by describing some cases of split groups.

\subsection{Some split classical groups}

\subsubsection{}\label{exGL} {\sl The case of $\GL_{N}$.} Suppose $W=\O[u]^N$ and set $\overline
W=W\otimes_{\O[u], u\mapsto 0}\O$. Write $\overline W=\oplus_{j=0}^{r-1}V_i$ and
consider  the parabolic subgroup $Q\subset \GL(\overline W)$ which is the
stabilizer of the flag $F_i=\oplus_{j\geq i}V_j$ given by the $V_i$.
Denote by $d_i$ the $\O$-rank of $V_i$. For $0\leq i\leq r-1$, set
$W_i$ for the preimage of $F_i$ under $W\to \overline W$ so that
$$
uW \subset W_{r-1}\subset\cdots\subset  W_1\subset W_0=W.
$$
Extend the index set   by requiring $W_{i+k\cdot r}=u^kW_i$ for
$k\in \Z$. Denote by $\iota_i: W_{i+1}\to W_{i}$ the inclusion. We
have a natural identification $\GL(W_i)=\GL(W_{i+r})$ given by
conjugating by $u$.

The dilatation $\Gg=\GL(W)_Q$ of $\GL(W)$ along  $Q$ is isomorphic to
the closed subgroup scheme in
$H=\prod_{i=0}^{r-1}\GL(W_i)=\prod_{i\in \Z/r\Z}\GL(W_i)$ of tuples
that commute with the maps $\iota_i: W_{i}\to W_{i+1}$. It is
isomorphic to the group scheme over $\O[u]$ obtained by Corollary
\ref{application} applied to $G=\GL_N$ and a point $x$ 
corresponding to the lattice chain
$\{W_i\otimes_{\O[u], u\mapsto \varpi}\O\}_i$.
(The lattice chains
$\{W_i\otimes_{\O[u]}\kappa((u))\}_i$ correspond to $x_{\kappa((u))}$.)

\subsubsection{}\label{exGSp}

{\sl The case of ${\rm GSp}_{2n}$.}  Consider
$W=\oplus_{i=1}^{2n}\O[u]\cdot e_i$
 with the perfect $\O[u]$-bilinear alternating form $h: W\times  W\to \O[u]$ determined by $h(e_i, e_{2n+1-j})=\delta_{ij}$,
 $h(e_i, e_j)=h(e_{2n+1-i}, e_{2n+1-j})=0$
 for $1\leq i, j\leq n$. Let us fix a chain of $\O[u]$-submodules
 \begin{equation}
 uW\subset W_{r-1}\subset \cdots \subset W_1\subset W_0\subset W
 \end{equation}
 such that

 (i) $W_{i}^\vee=u^{-1}W_{r-i-1}$, for $0\leq i\leq r-1$,

 (ii) $W_i/W_{i+1}\simeq \O^{d_i}$.

 Again, extend the index set by periodicity by setting $W_{i+k\cdot r}=u^kW_i$
 so that the form $h$ gives $W_i^\vee=W_{-i-a}$, with $a=0$ or $1$, and set
 $\iota_i: W_{i+1}\to W_{i}$ as before. Consider the group
 scheme $\Gg$ over $\O[u]$ of similitude automorphisms of the ``polarized"
system $(W_i, h_i)_{i\in \Z}$;
 more precisely this is the subgroup scheme of $\Gm\times\prod_{i=0}^{r-1}\GL(W_i)=\Gm\times\prod_{i\in \Z/r\Z}\GL(W_i)$
 consisting of $(c, (g_i))$ such that
 $$
 h  (g_i(x), g_{-i-a}(y))=c\cdot h(x,y), \quad \hbox{\rm for all $i\in \Z$}.
 $$
As in \cite[Appendix to Ch. 3]{RapZinkBook}, we can see that $\Gg$ is smooth
over $\O[u]$; it is isomorphic to the group scheme obtained by
Corollary \ref{application} applied to $G={\rm GSp}_{2n}$ and  
a point $x$ 
corresponding to the self-dual lattice chain
$\{W_i\otimes_{\O[u], u\mapsto \varpi}\O\}_i$.

\subsection{Non-split classical groups}\label{exClassical}

We now extend this to (essentially) the general classical case. When
in the sections below we consider symmetric or hermitian forms, we
will assume that the prime $p$ is odd. We first mostly concentrate
on describing explicitly the group schemes $\und G$.

\subsubsection{Division algebras}\label{exDivision}

With the notations of \ref{sss4a1}, consider the associative
(central) $\O[u]$-algebra given by
\begin{equation}\label{order1}
 \O(\calD)=\O_d[u]\oplus \O_d[u]\cdot X\oplus\cdots \oplus \O_d[u]\cdot X^{d-1},
\end{equation}
with relations $X^d=u$, $f\cdot X=X\cdot \sigma(f)^s$ for $f\in
\O_d[u]$ with $\sigma(\sum a_iv^i)=\sum\sigma(a_i)v^i$. Notice that
$\calD:=\O(\calD)[u^{-1}]$ is an Azumaya algebra over $\O[u^{\pm 1}]$ which splits after the unramified extension $ \O[u^{\pm
1}]\to \O_d[u^{\pm 1}]$; then $\O(\calD)$ is a maximal order in this
Azumaya algebra. We have isomorphisms
\begin{equation}
\calD\otimes_{\O[u^{\pm 1}]} F\simeq D,\qquad
\O(\calD)\otimes_{\O[u]}\O_F\simeq \O_D
\end{equation}
where the ring homomorphisms are given by $u\mapsto\varpi$. In
addition, reducing $\O(\calD)$ modulo $\varpi$ followed by
completing at $(u)$ also produces a maximal order in a central
division  algebra of degree $d$ and  invariant $s/d$ over the local
field ${\mathbb F}_{p^m}((u))$.

For $n\geq 1$, we can consider the affine  group scheme  over
$\O[u^{\pm1}]$ given by
$$
R\mapsto {\rm Aut}_{\calD\otimes_{\O[u^{\pm 1}]}R} (
\calD^n\otimes_{\O[u^{\pm 1}]}R).
$$
We can see directly from the construction of \S \ref{ss2b} that this
group scheme is isomorphic to $\und G$ for $G=\GL_n(D)$.

We can also consider the affine  group scheme $\Gg$ over $\O[u]$
given by
$$
\Gg(R):={\rm Aut}_{\O(\calD)\otimes_{\O[u]}R} (
\O(\calD)^n\otimes_{\O[u]}R).
$$
The group $\Gg$ is smooth over $\O[u]$ and is isomorphic to the
group scheme of Corollary \ref{application} for the vertex corresponding to the lattice chain
given by the multiples of $\O_D$.

\subsubsection{} In what follows, the base is $\Q_p$ and we will be discussing group schemes
over $\Z_p[u^{\pm 1}]$. We assume $F$ is a tame finite extension of
$\Q_p$ and let $\Q_{p^r}={\rm Frac}(W({\mF}_{p^r}))$ the maximal
unramified extension of $\Q_p$ contained in $F$. Denote by
$\Z_{p^r}$ the ring of integers $W({\mF}_{p^r})$ of $\Q_{p^r}$. When
$r$ is clear, we will simply write $W$ for
$\Z_{p^r}=W({\mF}_{p^r})$. We will then denote by $W_d$ the integers
of the unique unramified extension $W({\mF}_{p^{rd}})$ of $W$ of
degree $d\geq 2$.

\subsubsection{}\label{sss4b3} We can now explicitly construct the group scheme $\und G$
over $\Z_p[u^{\pm 1}]$ associated to the restriction of scalars
$G={\rm Res}_{F/\Q}\GL_m(D)$ with $D$ a division algebra over $F$ as
above.
We choose a $W$-algebra isomorphism $j:
W[x]/(x^e-pc)\xrightarrow{\sim} \O_F$ where
 $p\nmid e$ and $c\in W^*$;   
i.e a uniformizer $\varpi$ of $F$ such that $\varpi^e$ is in $W$.

As above, we construct an associative (central) $W[v]$-algebra given
by
\begin{equation}\label{order}
 \O(\calD)=W_d[v]\oplus W_d[v]\cdot X\oplus\cdots \oplus W_d[v]\cdot X^{d-1},
\end{equation}
with relations $X^d=v$, $f\cdot X=X\cdot \sigma(f)^s$ for $f\in
W_d[v]$ with $\sigma(\sum a_iv^i)=\sum\sigma(a_i)v^i$. (After the
base change $W[v]\to \O[u]$, $v=u$, this produces the algebra
denoted by the same symbol in the previous paragraph.) Again,
$\calD=\O(\calD)[v^{-1}]$ is an Azumaya algebra over $W[v^{\pm 1}]$.
We have isomorphisms
\begin{equation}
\calD\otimes_{W[v^{\pm 1}]} F\simeq D,\qquad
\O(\calD)\otimes_{W[v]}\O_F\simeq \O_D
\end{equation}
where  $W[v^{\pm 1}]\to F$, $W[v]\to \O_F$, are given by
$v\mapsto\varpi$. In addition, reducing $\O(\calD)$ modulo $p$
followed by completing at $(v)$ also produces a maximal order in a
central division  algebra of degree $d$ and  invariant $s/d$ over
the local field ${\mathbb F}_{p^r}((v))$.

Define $\phi: \Z_p[u]\to \Z_{p^r}[v]$ by $u\mapsto v^e\cdot c^{-1}$
with $c=\varpi^e \cdot p^{-1}$. For $m\geq 1$, we set
$M=\O(\calD)^m$. We consider
$$
 \Gg'(R)={\rm Aut}_{\O(\calD)\otimes_{\Z_p[u]}R} ( M\otimes_{\Z_p[u]}R).
$$
This  defines a smooth affine group scheme over $\Z_p[u]$ such that
$$
\Gg'\otimes_{\Z_p[u], u\mapsto p}\Q_p\simeq {\rm Res}_{F/\Q_p}(
\GL_m(D)).
$$
Suppose we choose another uniformizer $\varpi_1$ with
$\varpi_1^e=pc_1$ and denote by $\phi_1$ the corresponding map as
above. Let $y=\varpi_1/\varpi\in \O^*_F$. Since $y^e\in W^*$ and
$p\nmid e$, the extension $\Q_{p^r}(y)/\Q_{p^r}$ is unramified and
therefore $y$ is in $W^*=\Z_{p^r}^*$. Sending $v$ to $y\cdot v$ then
gives an isomorphism $\alpha: W[v]\xrightarrow{\sim}W[v]$ that maps
$(v^e-pc)$ to $(v^e-pc_1)$ and commutes with $\phi$, $\phi_1$.
Find $z\in W_d^*$ such that $N_{\Q_{p^{rd}}/\Q_{p^r}}(z)=y$; sending
$X\mapsto X\cdot z$ gives $\O(\calD) \xrightarrow{\sim} \O(\calD)
\otimes_{W[v],\alpha }W[v]$. This implies that $\Gg'$ is independent
from the choice of $\varpi$ with $\varpi^e\in W$. The group scheme
$\Gg'_{|\Z_p[u^{\pm 1}]}$ is isomorphic to the group scheme $\und G$
obtained from $G={\rm Res}_{F/\Q_p}(\GL_m(D))$ as above; this
follows directly from the construction of \S \ref{ss2b} using
(\ref{invarRing}). The restriction $\Gg'_{|\Z_p[u^{\pm 1}]}\to
\Spec(\Z_p[u^{\pm 1}])$ is the Weil restriction of scalars from
$W[v^{\pm 1}]$ of a twisted form of ${\rm GL}_{md}$ over $W[v^{\pm
1}]$; this twisted form is the group of automorphisms of the module
$\calD^m=\O(\calD) [v^{-1}]^m$ for the Azumaya  algebra $\calD$ over
$W[v^{\pm 1}]$.

\subsubsection{}\label{sss4c4}  Here again $W$ is $\Z_{p^r}=W({\mF}_{p^r})$
and $W_d=W({\mF}_{p^{rd}})$ as above. Write
\begin{equation}
 W[v^{\pm 1}]^*/(W[v^{\pm 1}]^*)^2=\{1, \alpha, v, \alpha v\}
\end{equation}
where $\alpha$ is an element of $W^*$ which is not a square.

We consider a $W[v]$-algebra $\frakR$ given as
$$
\frakR=W_2[v],\quad {\rm or}\quad \frakR=W[v'], \ v\mapsto
v'^2,\quad {\rm or}\quad \frakR=W[v'], \ v\mapsto \alpha^{-1}v'^2.
$$
We will refer to the first possibility  as the  {\sl unramified}
case. The other two possibilities are the {\sl ramified} case. We
have a $W[v]$-symmetric bilinear form
$$
h_\frakR: \frakR\times\frakR\to W[v];  \quad h_\frakR(x,
y)=\frac{1}{2}{\rm Tr}_{\frakR/W[v]}(x\bar y)
$$
where $r\mapsto \bar r$ is the order two automorphism of $\frakR$
over $W[v]$. The form $h_{\frakR}$ is perfect in the unramified
case;  $h_\frakR[v^{-1}]$ on $\frakR[v^{-1}]$ is always perfect.

We also consider the central $W[v]$-algebra
$$
\O(\calQ)=W_2[v]\oplus W_2[v]\cdot X
$$
with relations $X^2=v$, $f\cdot X=X\cdot \sigma(f)$ for $f\in
W_2[v]$; this corresponds to the quaternion case $(s,d)=(1,2)$ as
above. Denote by $x\mapsto \bar x$ the main involution of
$\O(\calQ)$ which is $\sigma$ on $W_2[v]$ and maps $X$ to $-X$. Let
$\zeta$ be a root of unity that generates $W_2$ over $W$,
$W_2=W(\zeta)$. Then $\bar\zeta=-\zeta$. The reduced norm ${\rm
Norm}(r)=r\cdot \bar r$  defines a $W[v]$-linear quadratic form on
$\O(\calQ)$. Denote by $h_{\O(\calQ)}: \O(\calQ)\times\O(\calQ)\to
W[v]$ the corresponding $W[v]$-bilinear symmetric form.

We will use the symbol $\frakO$ to denote one of $W[v]$, $\frakR$,
or $\O(\calQ)$; each of these $W[v]$-algebras supports an involution
$x\mapsto \bar x$ as above (this is trivial in the case of $W[v]$).
In the following paragraph we will give each time a free (left)
module $M$ over $\frakO$ which is equipped with a certain form $h$
(alternating, symmetric, hermitian, etc.). All the forms below are
``perfect" after we invert $v$, i.e over $W[v^{\pm 1}]$. We will
consider the group scheme $\Gg'$ over $ W[v]$
 given by the $\frakO$-module automorphisms of $M$ that respect the corresponding  form $h$.
 Suppose that $F$ is a totally tamely ramified extension of $\Q_{p^r}$ of degree $e$.
 Choose a uniformizer $\varpi$ of $F$ with
$\varpi^e\in W$ as above and consider the base change $W[v^{\pm
1}]\to F$ given by $v\to \varpi$. In the list below, we mention the
type of the isogeny class of the group $\Gg'|_{W[v^{\pm
1}]}\otimes_{W[v^{\pm 1}]}F$ over $F$ according to the tables 4.2
and 4.3 of \cite[p. 60-65]{TitsCorvallis}. The determination of
these types follows \cite[4.4, 4.5]{TitsCorvallis}. The
corresponding symbol is read from the first column of these tables.

\subsubsection{Alternating forms}\label{exAlt}

\begin{itemize}

\item{} $M=W[v]^{2n}=\oplus_{i=1}^{2n}W[v]\cdot e_i$ with the alternating $W[v]$-bilinear form $h$
determined by $h(e_i, e_{2n+1-j})=\delta_{ij}$, $1\leq i\leq n$.
(cf. \S \ref{exGSp}). (For $n\geq 2$, the type is $C_n$.)

\end{itemize}

\subsubsection{Symmetric forms}\label{exSymm}
 (Set $n=2m+1$, or $n=2m$.)

\begin{itemize}

\item{\sl Split:}  $M=W[v]^{n}=\oplus_{i=1}^{n}W[v]\cdot e_i$ with the symmetric $W[v]$-bilinear form $h=h(n)$
determined by $h(n)(e_i, e_{n+1-j})=\delta_{ij}$.  (For $n\geq 6$,
the type is $B_m$, or $D_m$ respectively.)

\smallskip

\item{\sl Quasi-split, even case:}   Here $n$ is even and
$M=W[v]^{n-2}\oplus \frakR$ with the symmetric $W[v]$-bilinear form
$h$ given as the direct sum $h(n-2)+h_{\frakR}$. (The types are
${}^2D_m$ if $\frakR=W_2[v]$ (unramified) and $n\geq 8$ and
$C-B_{m-1}$ if $\frakR$ is ramified and $n\geq 6$.)
\smallskip

\item{\sl  Non quasi-split, even case:} Here $n$ is even and
 $M=W[v]^{n-4}\oplus \O(\calQ) $
with   the symmetric $W[v]$-bilinear form $h$ given as the direct
sum $h(n-4)+h_{\O(\calQ)}$. (For $n\geq 6$, the type is ${}^2D'_m$.)
\smallskip

\item{\sl Non quasi-split, odd case:}  Here $n$ is odd and $M=W[v]^{n-3}\oplus \O(\calQ)^0 $
with the  symmetric $W[v]$-bilinear form $h$ given as the direct sum
$h(n-3)+h_{\O(\calQ)^0}$.  We denote by $\O(\calQ)^0$ the submodule
of elements $r$ for which $r+\bar r=0$ and by $h_{\O(\calQ)^0}$  the
restriction of $h_{\O(\calQ)}$ to this submodule. (For $n\geq 6$,
the type is ${}^2B'_m$.)

\end{itemize}

\subsubsection{Hermitian forms}\label{exHerm}
 (Set $n=2m+1$, or $n=2m$.)

\begin{itemize}

\item{\sl quasi-split:} $M=\frakR^{n}$ with hermitian form $h=H(n)$ given by
$$
H(n)(x, y)=x^t\cdot K_n\cdot\bar y
$$
where $K_n$ is the antidiagonal $n\times n$ unit matrix. There are
subcases here according to the choice of $\frakR$. (Suppose $n\geq
3$. In the unramified case, the type is ${}^2A'_{n-1}$. In the
ramified case, the type is $C-BC_m$ if $n=2m+1$, or $B-C_m$ if
$n=2m$.)

\item{\sl non quasi-split, ramified, even case:} Here $n=2m$ is even, $M=\frakR^{n-2}\oplus \frakR^2$, $\frakR $
ramified with hermitian form $H$ given as the direct sum
$h=H(n-2)\oplus H_\alpha$ with
$$
H_\alpha((x_1, x_2), (y_1, y_2) )= x_1\bar y_1- \alpha\cdot
x_{2}\bar y_{2}
$$
and $\alpha\in W^*$ which is not in $(W^*)^2$. (If $n\geq 3$, the
type is ${}^2B-C_m$.)

\item{\sl non quasi-split,  unramified, even case:} Here again $n=2m$ is even, $M=\frakR^{n-2}\oplus \frakR^2$,   $\frakR=W_2[v]$,
with hermitian form $h$ given as the direct sum $h=H(n-2)\oplus H_u$
with
$$
H_u((x_1, x_2), (y_1, y_2) )= x_1\bar y_1- v\cdot x_{2}\bar y_{2}.
$$
(If $n\geq 3$, the type is ${}^2A''_{n-1}$.)

\end{itemize}

\subsubsection{Quaternionic $\epsilon$-hermitian forms}\label{exQHermitian}

Let $\epsilon=\pm 1$. If $M$ is a left $\O(\calQ)$-module, then a
$W[u]$-bilinear $H: M\times M\to \O(\calQ)$ is called an
$\epsilon$-hermitian (i.e hermitian if $\epsilon=1$, anti-hermitian
if $\epsilon=-1$) form, for the main involution $d\mapsto \bar d$,
if it satisfies: $H(dx, y)=dH(x, y)$, $\overline{ H(x,y)}=\epsilon
H(y,x)$ for $d\in\O(\calQ)$, $x$, $y\in M$. Choose a unit $\xi\in
W^*_2$ such that ${\rm Norm}(\xi)=-{\rm Norm}(\zeta)=\zeta^2$.

\begin{itemize}


\item{\sl  Quaternionic hermitian:}  $M=\O(\calQ)^n$, with hermitian form $h=H(n): M\times M\to \O(\calQ)$
given by
$$
H(n)(x, y)= x^t\cdot K_n\cdot  \bar y.
$$
(If $n\geq 2$, the type is ${}^2C_n$.)

\item{\sl  Quaternionic anti-hermitian:}  $M=\O(\calQ)^{m}\oplus  \O(\calQ)^m\oplus M_0$ where $M_0=\O(\calQ)^r$,
$n=2m+r$. The anti-hermitian form $h='H: M\times M\to \O(\calQ)$ is
the direct sum $'H(2m)\oplus 'H_0$ where
$$
'H(2m)(x, y)= x^t\cdot \left(\begin{matrix}0& I_m\\
-I_m&0\end{matrix}\right)\cdot \bar y
$$
  is the standard anti-hermitian hyperbolic form and $M_0$ with its form $'H_0$ is given as in the one of the following four cases:

\begin{itemize}

\item{(a)} $M_0=(0)$. (Here $n$ is even. If $n\geq 6$, the type is ${}^2D''_n$.)

\item{(b)} $M_0=\O(\calQ)$ with form $xc\bar y$ with $c=X$, $c=\zeta$, or $c=X\xi$.
(Here $n$ is odd.  The type is either ${}^2D''_n$ if $c=\zeta$ and
$n\geq 5$, or for $n\geq 3$, ${}^2C-B_{n-1}$ otherwise.)

\item{(c)} $M_0=\O(\calQ)^2$ with form $x_1a_1\bar y_1+x_2a_2\bar y_2$ with $a_1$, $a_2$
two distinct elements of the set $\{X, \zeta, X\xi\}$. (Here $n$ is
even. If $n\geq 4$, the type is ${}^4D_n$.)

\item{(d)} $M_0=\O(\calQ)^3$ with  form $x_1X\bar y_1+x_2\zeta\bar y_2+x_3X\xi\bar y_3$.
(Here $n$ is odd. If $n\geq 5$, the type is ${}^4D_n$.)
\end{itemize}

\end{itemize}

\smallskip

\subsubsection{}\label{exhaustive}
Our list of cases above is exhaustive in the following sense: Choose
once and for all the uniformizer $\varpi$ of $F$. The connected
components of the specializations $\Gg'|_{W[v^{\pm
1}]}\otimes_{W[v^{\pm 1}]}F$, together with the groups ${\rm
SL}_m(D)$ for $F$-central division algebras $D$ (these groups are of
type ${}^dA_{md-1}$ with $d$ the degree of $D$), give exactly all
the isogeny classes of absolutely almost simple groups over $F$
which are of classical type. (More precisely, if we avoid
exceptional isomorphisms by obeying the listed restrictions on $n$, we obtain each isogeny class exactly once.) This
follows from the above, the discussion in \cite[4.5]{TitsCorvallis},
and classical results on the classification of quadratic and
(quaternionic) hermitian forms over local fields (e.g.
\cite{Jacobson}, \cite{Tsukamoto}). For example, to deal with the
quasi-split case for symmetric forms, we  notice that, since the
residue characteristic is odd, $v\mapsto\varpi$ gives an isomorphism
\begin{equation}
 W[v^{\pm 1}]^*/(W[v^{\pm 1}]^*)^2\xrightarrow{\sim} F^*/(F^*)^2.
\end{equation}
This allows us to realize any quadratic extension $L/F$ as a
specialization of a uniquely specified $\frakR/W[v]$ at $v\mapsto
\varpi$. As a result, the trace form $\frac{1}{2}{\rm Tr}_{L/F}( \
)$ can be obtained by specializing $\frac{1}{2}{\rm
Tr}_{\frakR[v^{\pm 1}]/W[v^{\pm}]}( \ )$ by $v\mapsto \varpi$.

\subsubsection{}
In what follows, the symbol $\frakO$ will denote either $\O(\calD)$
as in (\ref{order}), or $W[v]$, $\frakR$, $\O(\calQ)$ as in the
previous sections. Recall that we denote by $\Gg'$ the group scheme
over $W[v]$ of $\frakO$-automorphisms of $M$ that also preserve the
form $h$ if applicable. All the above forms are ``perfect" after we
invert $v$, i.e over $W[v^{\pm 1}]$, and we can see that $\Gg'|_{
W[v^{\pm 1}]}$ is reductive. Denote by $G'$ the specialization of
the neutral component $(\Gg'|_{W[v^{\pm 1}]})^\circ\otimes_{W[v^{\pm
1}]}F$ and consider the Weil restriction of scalars $G={\rm
Res}_{F/\Q_p}G'$. Regard $W[v^{\pm 1}]$ as a $\Z_p[u^{\pm
1}]$-algebra via $u\mapsto v^e\cdot p\varpi^{-e}$ as  before.

\begin{prop}\label{coincide}
The group scheme $\und G$, as constructed  in \S \ref{ss2b} from $G$
above, is isomorphic to the neutral component of the group scheme
over $\Z_p[u^{\pm 1}]$ with $R$-valued points the
$\frakO\otimes_{\Z_p[u^{\pm 1}]}R$-linear automorphisms of
$M\otimes_{\Z_p[u^{\pm 1}]}R$ that also respect the form
$h\otimes_{\Z_p[u^{\pm 1}]}R$ if applicable.
\end{prop}

\begin{proof}
As above, consider
 the neutral component $\und J:=(\Gg'|_{W[v^{\pm 1}]})^\circ$ of the group scheme over
 $W[v^{\pm 1}]$
of $\frakO$-automorphisms of $M$ that also preserve the form $h$ if
applicable. Then the group scheme in the statement of the
Proposition is isomorphic to ${\rm Res}_{W[v^{\pm 1}]/\Z_p[u^{\pm
1}]}\und J$ and is enough to check that $\und J$ is isomorphic to
the group scheme $\und G'$ which is obtained by $G'$ using our
constructions in the previous chapters. This can be shown by a
case-by-case verification and we will leave some of the work to the
reader:
 First, in the case of inner forms of type $A_n$ where $\frakO=\O(\calD)$,
the result is in \S \ref{sss4b3}  and follows directly from the
construction of \S \ref{ss2b}. Second, suppose we consider the rest
of the cases of the previous section; then we assume that $p$ is
odd. The group
 $G'$ contains a standard split torus and we can  compute the
quasi-split forms and the corresponding anisotropic kernel in the
explicit descriptions of the cases above. The adjoint groups of
these anisotropic kernels are inner forms of products of
 ${\rm PGL}_2$ or ${\rm PGL}_4$ (the latter occuring only in case
\ref{exQHermitian} (d)).  We  first check that $\und
J\otimes_{W[u^{\pm 1}]}\breve\Z_p[v^{\pm 1}]$ is quasi-split over
$\breve\Z_p[v^{\pm 1}]$. Then the rigidity of quasi-split forms (\S
\ref{ssstameeq}) shows the base change  $\und J\otimes_{W[u^{\pm
1}]}\breve\Z_p[v^{\pm 1}]$  is isomorphic to
 $\und {G'}^*\otimes_{W[u^{\pm 1}]}\breve\Z_p[v^{\pm 1}]$ and
it remains to verify that the inner twists of $\und {G'}^*$ that
define $\und J$ and $\und G'$ agree. This can be done case-by-case;
we leave the details to the reader. (Alternatively, since the
reductive groups $\und J$ always split over a degree $4$ Galois
cover of $W[u^{\pm 1}]$ and $p$ is odd, the isomorphism between
$\und J$ and $\und G'$ can also be shown directly using Remark
\ref{Piano}.)
\end{proof}

\subsubsection{}\label{sss4b11} We can now extend the explicit descriptions
of the group schemes $\Gg$ from the split cases above to the general
classical case. Recall $\frakO$ will denote either $\O(\calD)$ as in
(\ref{order}), or $W[v]$, $\frakR$, $\O(\calQ)$ as in the previous
sections. Then, we also denote by $Y$ the ``standard uniformizer" of
each of these algebras, i.e $X$ when $\frakO$ is $\O(\calD)$ or
$\O(\calQ)$, $v$ when $\frakO=W[v]$ or $\frakR=W_2[v]$ in the
unramified case, $v'$ when $\frakO=\frakR$ in the ramified case.

A $\frakO$-lattice chain in $\frakO[v^{\pm 1}]^n$ is a totally
ordered non-empty set $M_\bullet$ of left $\frakO$-submodules of
$\frakO[v^{\pm 1}]^n$ which are free of rank $n$, which is stable
under multiplication by $Y$ and $Y^{-1}$, and of the form
\begin{equation}
\cdots\subset Y M_0 \subset M_{r-1}\subset\cdots \subset M_1\subset
M_0\subset\cdots .
\end{equation}
with $M_i/M_{i+1}$, for all $i\in\Z$, free over $W$.

If $\frakO$ is one of $W[v]$, $\frakR$, $\O(\calQ)$ with form $h$ as
in the previous section, and $N$ a $\frakO$-lattice in
$\frakO[v^{\pm 1}]^n$ we can consider the dual $N^\vee=\{x\in
\frakO[v^{\pm 1}]^n\ |\ h(x, m)\in \frakO, \ \forall m\in N\}$ which
is also an $\frakO$-lattice. The $\frakO$-lattice chain $M_\bullet$
is called self-dual, when $N$ belongs to the lattice chain
$M_\bullet$, if and only if $N^\vee$ belongs to the lattice chain
$M_\bullet$. Then, for each index $i$, there is $j$ such that $h$
induces a prefect pairing
\begin{equation}\label{forms}
h: M_i\times M_j\to \frakO.
\end{equation}
If $M_\bullet$ is a (self-dual) $\frakO$-lattice chain, the base
change $M_\bullet\otimes_{W[v]}\O$, by $v\mapsto \varpi$, gives a
(self-dual) lattice chain as in \S \ref{ss4a}.

\begin{Remark}\label{genericLattice}
{\rm In fact, we can give such  self-dual $\frakO$-lattice chains by
the following construction: Let $M$ be $\O(\calD)^n$ or in general
$\frakO^n$ with the form as in \S \ref{exAlt}, \ref{exSymm},
\ref{exHerm}, etc. Consider the $F((v))$-vector space
$V=M\otimes_{W[v^{\pm 1}]}F((v))$; it is a free (left)
$\frakO\otimes_{W[v]}F((v))$-module and supports the perfect form
$h\otimes_{W[v]}F((v))$. Choose a   self-dual lattice chain
$L^\bullet=(L_i)_i$ of $\frakO\otimes_{W[v]}F[[v]]$-lattices in $V$
in the sense of \S \ref{ss4a} (cf. \cite{RapZinkBook}) (for the
equal characteristic dvr $F[[v]]$). Since $M\otimes_{W[v]}F[[v]]$ is
also a lattice in $V$, for each $i\in\Z$, there is $r_i\geq 0$ so
that $v^{r_i}M\otimes_{W[v]}F[[v]]\subset L_i\subset
v^{-r_i}M\otimes_{W[v]}F[[v]]$. The intersection $M_i:=L_i\cap
M[v^{-1}]$ is a finitely generated and reflexive (i.e equal to its
double dual) $W[v]$-module.
 Since $W[v]$ is regular Noetherian
of Krull dimension $2$ and $(v, p)$ has codimension $2$, it follows
that
  $M_i$ is  a finitely generated projective
module which then is actually $W[v]$-free (\cite{SeshadriPNAS}).  In
fact, locally free coherent sheaves on  $\Spec(W[v])-\{(v, p)\}$
uniquely extend to locally free coherent sheaves over $\Spec(W[v])$.
Similarly, module homomorphisms between two such sheaves
 uniquely extend.
  Using  this extension
property, we now see that $\frakO$-multiplication and the perfect
forms (considered as maps $L_i\to L_j^{\rm dual}$) extend to
$\{M_i\}$. Suppose  that $L_i\subset L_j$ are two $F[[v]]$-lattices
in the chain and consider the corresponding $W[v]$-lattices
$M_i\subset M_j$. Notice that $M_i/M_j$ has projective dimension $1$
as a $W[v]$-module and so it follows from the Auslander-Buchsbaum
theorem that $M_i/M_j$ has depth $1$. On the other hand, since
$M_i[v^{-1}]=M_j[v^{-1}]$, $M_i/M_j$ is supported along $v=0$. By
the above, all the associated primes of $M_i/M_j$ have height $1$,
therefore $M_i/M_j$ has no section supported over $(v, p)$. It now
follows that if the quotient $L_i/L_j$ is annihilated by $v$ and has
$F$-rank $d$, then
  $M_i/M_j$ is also annihilated by $v$ and is actually $W$-free of rank $d$.
Similarly, we see that there are $m_j\geq 0$ such that
$$
v^{m_j}\frakO^n\subset M_i\subset v^{-m_j}\frakO^n
$$
with $M_i/v^{m_j}\frakO^n$, $v^{-m_j}\frakO^n/M_i$ both $W$-free. We
can now show that all such $M_i$ are free $\frakO$-modules. Hence,
it also follows that $M_\bullet=\{M_i\}_i$ is a (self-dual)
$\frakO$-lattice chain in $M[u^{\pm 1}]$ in the sense above. }
\end{Remark}

Now consider the group scheme $\Gg'$ over $\Z_p[u]$ with $R$-valued
points the $\frakO\otimes_{\Z_p[u ]}R$-linear automorphisms of the
chain $M_\bullet\otimes_{\Z_p[u ]}R$ that respect the forms
$h\otimes_{\Z_p[u ]}R$ of (\ref{forms}). The arguments in
\cite[Appendix to Ch. 3]{RapZinkBook} show that $\Gg'$ is smooth. Base
changing by $\Z_p[u]\to \Z_p$ via $u\mapsto p$ or by $\Z_p[u]\to
\Q_p((u))$ produces smooth group schemes whose neutral component is
a Bruhat-Tits parahoric group scheme (see \S \ref{sss4a2}). Using \S
\ref{unique}, we can now see that the neutral component $\Gg$ of
$\Gg'$ is a group scheme obtained by Corollary \ref{application}.

\subsubsection{Variants}\label{exVariants}

Similarly, we can consider:

\begin{itemize}

\item{} anti-hermitian forms $\tilde h$ on $\frakR^n$ obtained by multiplying
the hermitian forms $h$ of \ref{exHerm} by either $\zeta$ (in the
unramified case), or by $v'$ (in the ramified cases).

\item{} $\epsilon$-hermitian forms $\tilde h$ on $\O(\calD)$ for the ``new involution" $d\mapsto d'$ on $\O(\calD)$
given by $d':=X^{-1}\cdot \bar d\cdot X$. Such forms can be obtained
by multiplying the $\epsilon$-hermitian forms $h$ for the main
involution of \S \ref{exQHermitian} by   $X$. Indeed, if $h$ is
$\epsilon$-hermitian for $d\mapsto \bar d$, then $ h\cdot X$ is
$(-\epsilon)$-hermitian for $d\mapsto d'$.

\end{itemize}

The automorphisms groups of these forms (after specialization to
$F$) do not produce additional isogeny classes of reductive groups.
However, as we will see later, considering these forms is useful in
constructing certain symplectic  embeddings and so these will appear
in our discussion of local models for Shimura varieties of PEL type.

\bigskip

\section{Loop groups and affine Grassmannians}\label{chapterLoop}
\setcounter{equation}{0}

In this chapter we define and show the (ind-)representability of the
various versions (``local" and ``global")
 of the affine Grassmannian that we will use. We start
 by showing a version of the descent lemma of Beauville-Laszlo
 for $\Gg$-torsors.

\subsection{A descent lemma}\label{affGrass}

We continue with the same notations, so that $\O$ is a discrete
valuation ring with fraction field $F$ and perfect residue field
$k$. Let $\Gg=\Spec(B) \to X=\AA^1_\O=\Spec(\O[u])$ be a smooth affine group
scheme over $\AA^1_\O$ with connected fibers. Suppose that $R$ is an
$\O$-algebra and denote by $r: \Spec(R)\to \Spec(\O[u])$ the
$R$-valued point of $\AA^1_\O$ given by $u\mapsto r\in R$. We will
identify the completion of $ R[u]$ along the ideal $(u-r)$ (which cuts out the graph of
$r$) with $R[[t]]$ using the local parameter $t=u-r$.

The following extends the descent lemma of Beauville-Laszlo \cite{BLdescente}.

\begin{lemma}\label{descentBL}  There is a $1$-$1$ correspondence
between   elements of $\Gg(R((t)))$ and triples $({\mathcal T},
\alpha, \beta)$, where ${\mathcal T}$ is a $\Gg$-torsor over $R[u]$
and $\alpha$, $\beta$ are trivializations of the torsors ${\mathcal
T}\otimes_{R[u]}R[t, t^{-1}]$ and ${\mathcal T}\otimes_{R[u]}R[[t]]$
respectively. The inverse of the correspondence associates to the
triple $({\mathcal T}, \alpha, \beta)$ the element
$(\alpha^{-1}\cdot \beta)(1)$.
 \end{lemma}

\begin{proof}
In \cite {BLdescente}, this is proven when $\Gg=\GL_n$. More
generally, \cite {BLdescente} shows how one can construct a
$t$-regular $R[u]$-module $M$ from a triple $(M, N, \phi)$ of a
$R[t,t^{-1}]$-module $M$, a $t$-regular $R[[t]]$-module $N$ and an
$R((t))$-isomorphism $\phi: R[[t]]\otimes_{R[u]} F\xrightarrow{\sim}
N\otimes_{R[[t]]}R((t))$. Starting from $g\in \Gg(R((t)))$, we can  apply this to
$M=B\otimes_{R[u]}R[t, t^{-1}]$, $N=B\otimes_{R[u]}R[[t]]$ and $\phi$
given by the (co)action of $g$, i.e $\phi$ is the composition
$$
  B\otimes_{\O[u]}R((t))\xrightarrow{ } (B\otimes_{\O[u]}R((t)))\otimes_{R((t))} (B\otimes_{\O[u]}R((t)))\xrightarrow{id\otimes g^*}  B\otimes_{\O[u]}R((t)) .
$$
In this, $g^*$ is the $R[u]$-algebra homomorphism $B\to R((t))$
corresponding to $g\in \Gg(R((t)))$. Denote by $C$ the corresponding
$R[u]$-module obtained using \cite {BLdescente}. Notice that $\phi$
is an $R((t))$-algebra homomorphism; this allows us to deduce that
$C$ is an $R[u]$-algebra.
Since $B$ is flat over $\O[u]$, by \cite {BLdescente} we   see that
$C$ is also flat over $R[u]$. In fact, since by construction,
$C\otimes_{R[u]} R[[t]]\simeq B\otimes_{\O[u]}R[[t]]$, the map
$\Spec(C)\to \Spec(R[u])$ is surjective and so $C$ is faithfully
flat over $R[u]$. For our choice of $\phi$, the algebra $C$ affords
a $B$-comodule structure $C\to B\otimes_{R[u]}C$ which base-changes
to the standard $B$-comodule structures   on
$B\otimes_{\O[u]}R[t, t^{-1}]$ and $B\otimes_{\O[u]}R[[t]]$. Now set
${\mathcal T}=\Spec(C)\to \Spec(R[u])$. By the above, $\mathcal T$
is a faithfully flat scheme with $\Gg$-action which is
$\Gg$-equivariantly isomorphic to $\Gg$ over $R[[t]]$ and $R[t,
t^{-1}]$. We would like to conclude that
 ${\mathcal T}:=\Spec(C)$ is the corresponding  torsor.
Consider the map $m: \Gg\times_{\Spec(\O[u])} {\mathcal T}\to
{\mathcal T}\times_{\Spec(R[u])}{\mathcal T}$ given by $(g,
\tau)\mapsto (g\cdot \tau, \tau)$. It is enough to show that   the
corresponding ring homomorphism
$$
m^*: C\otimes_{R[u]}C\xrightarrow{\ } B\otimes_{\O[u]}C
$$
is an isomorphism. Observe that $m^*$ is injective since
$m^*[t^{\pm 1}]$ is an isomorphism. Similarly, we can see that  $m^*$ is surjective using
\cite[Lemme 2] {BLdescente} since the base-change $R[u]\to R[t,
t^{-1}]\times R[[t]]$ is faithful and $m^*[t^{\pm 1}]$, $m^*[[t]]$ are isomorphisms.
\end{proof}

\subsection{Affine Grassmannians}

\subsubsection{The local affine Grassmannian.}\label{LocalaffGrass}

If $S=\Spec(R)$ is an affine $\O$-scheme we set  $D_S=\Spec(R[[u]])$
and  $D^*_S=D_S-\{0\}=\Spec(R((u)))$.  

 Let $\Gg \to X=\AA^1_\O=\Spec(\O[u])$ be a smooth affine group scheme with connected
 fibers. If $R$ is an $\O$-algebra, we set $L\Gg(R)=\Gg(R((u)))$ and $L^+\Gg(R)=\Gg(R[[u]])$.
Since $\Gg$ is affine, we can see that $L\Gg$, resp. $L^+\Gg$, is
represented by an ind-affine scheme, resp. affine scheme, over $\O$.
We also consider the quotient fpqc sheaf ${\rm
Gr}_{\Gg}:=L\Gg/L^+\Gg$ on $({\rm Sch}/\O)$
associated to $R\mapsto L\Gg(R)/L^+\Gg(R)$.

\begin{prop} \label{identify}  
 If $S$ is an affine scheme over $\O$, there is a natural identification
\begin{equation}\label{a}
 {\rm Gr}_{\Gg}(S)= \biggl\{\,\text{iso-classes of pairs } ( \mathcal T, \alpha) \biggm|
      \twolinestight{$\mathcal T$ a $\Gg$-torsor on $D_S$,}
   {$\alpha$ a trivialization of $\mathcal T |_{ D^*_S}$}\,\biggr\}\, .
\end{equation}
For any scheme $S$ over $\O$, there is a natural identification
\begin{equation}\label{fun1}
 {\rm Gr}_{\Gg}(S)= \biggl\{\,\text{iso-classes of pairs } ( \mathcal E, \beta) \biggm|
      \twolinestight{$\mathcal E$ a $\Gg$-torsor on $\AA^1_S$,}
   {$\beta$ a trivialization of $\mathcal E |_{ \AA^1_S\setminus\{u=0\}}$}\,\biggr\}\, .
\end{equation}
\end{prop}

\begin{proof}  The argument showing this can be found in
\cite{LaszloSorger}, see especially \cite[Prop. 3.10]{LaszloSorger}.
The crucial point is to observe that every $\Gg$-torsor ${\mathcal
T}$ over $R[[u]]$ can be trivialized over $R'[[u]]$ where $R\to R'$
is a faithfully flat extension. Indeed, ${\mathcal
T}\otimes_{R[[u]]}R$ has a section after such an extension $R\to
R'$; since ${\mathcal T}\to \Spec(R[[u]])$ is smooth this section
can be extended to a section over $R'[[u]]$. This shows the first
identification. The second identification now follows using the
descent lemma \ref{descentBL}.
\end{proof}

\subsubsection{}
Assume now in addition that $\Gg=\Gg_x \to X=\AA^1_\O=\Spec(\O[u])$
is a Bruhat-Tits group scheme in the sense of Theorem
\ref{grpschemeThm}.

\begin{prop}\label{indproper}
The sheaf ${\rm Gr}_{\Gg}$ is represented by an ind-projective
ind-scheme over $\O$.
\end{prop}

\begin{proof} For this we can appeal to the sketchy \cite{FaltingsLoops} (for the split case)
and to \cite{PappasRaTwisted} when the residue field $k$ is
algebraically closed. We give here a general proof by a different
argument.

We first show that ${\rm Gr}_{\Gg}$ is representable by an
ind-scheme of ind-finite type and separated over $X$. By Proposition
\ref{qaffine}, there is a closed group scheme immersion
$\Gg\hookrightarrow \GL_n$ such that the quotient $\GL_n/\Gg$ is
representable by a quasi-affine scheme over $\O[u]$. The argument in
\cite[4.5.1]{BeilinsonDrinfeld}   (or \cite[Appendix]{GaitsgoryInv},
cf. \cite{PappasRaTwisted}) now shows that the natural functor $
{\rm Gr}_{\Gg}\to {\rm Gr}_{\GL_n} $ is representable and is a
locally closed immersion. In fact, if the quotient $\GL_n/\Gg$ is
affine, this functor is a closed immersion. Now note that, as is
well-known (loc. cit.), the affine Grassmannian ${\rm Gr}_{\GL_n}$
is representable by an ind-scheme which is ind-projective over $\O$.
It remains to show that ${\rm Gr}_{\Gg}$ is ind-proper.

Assume first that $\und G=H\otimes_\Z \O[u^{\pm 1}]$ is split.
Consider an alcove $C$ whose closure contains $x$. If $y$ is in the
interior of the alcove $C$, then an argument as in \ref{unique}
shows that there is a group scheme homomorphism $\Gg_y\to \Gg_x$
which induces $\Gg_y[u^{-1}]=\Gg_x[u^{-1}]$. ($\Gg_y$ is a group
scheme corresponding to an Iwahori subgroup.) Hence, the morphism
${\rm Gr}_{\Gg_y}=L\Gg_y/L^+\Gg_y\to {\rm
Gr}_{\Gg_x}=L\Gg_x/L^+\Gg_x$ is surjective and it is enough to show
that ${\rm Gr}_{\Gg_y}$ is ind-proper. Now observe that the closure
of the alcove $C$ always contain a hyperspecial point $x_0$; then
$\Gg_{x_0}$ is reductive, $\Gg_{x_0}\simeq H\otimes_\Z\O[u]$. As in
the proof of Theorem \ref{grpschemeThm} the group scheme
homomorphism $\Gg_{y}\to \Gg_{x_0}$ identifies $\Gg_{y}$ with the
dilatation of $H\otimes_\Z\O[u]$ along a Borel subgroup $B$ of the
fiber $H$ over $u=0$. This implies that the fpqc sheaf associated to
$R\mapsto \Gg_{x_0}(R[[u]])/\Gg_y(R[[u]])$ is
representable by the smooth  projective homogeneous space $Y:=H/B$
over $\O$. Hence, the morphism ${\rm Gr}_{\Gg_y}=L\Gg_y/L^+\Gg_y\to
{\rm Gr}_{\Gg_{x_0}}=L\Gg_{x_0}/L^+\Gg_{x_0}$ is an fppf fibration
with fibers locally isomorphic to $Y$; in particular it is a
projective surjective morphism. (In fact, we note here that, as in
\cite{FaltingsLoops}, we can see that the quotient morphism
$L\Gg_y\to {\rm Gr}_{\Gg_y}=L\Gg_y/L^+\Gg_y$ is an $L^+\Gg_y$-torsor
which splits locally for the Zariski topology.) Now recall that  there is a representation $H\hookrightarrow\GL_n$ with $\GL_n/H$ affine, while
 $\Gg_{x_0}\simeq H\otimes_\Z\O[u]$.
As above, we see then that ${\rm Gr}_{\Gg_{x_0}}\hookrightarrow {\rm
Gr}_{\GL_n}$ is a closed immersion, and that ${\rm Gr}_{\Gg_{x_0}}$
is ind-projective and also ind-proper.

Next we consider the general case. It is enough to prove that ${\rm
Gr}_{\Gg}$ is ind-proper over $\Spec(\O)$ after base changing by a
finite unramified extension $\O'/\O$. Therefore, by replacing $\O$
by $\O'$ we may assume that $\und G$ is quasi-split and splits over
$\O[v^{\pm 1}]/\O[u^{\pm 1}]$. We now return to the notations of the
proof of Theorem \ref{grpschemeThm}. In particular, if $x$ is in
$\A(\und G_F, \und S_F, F)=\A(H, T_H, \tilde F)^\Gamma$, then
$\Gg_x$ is the neutral component of $({\rm
Res}_{\O[v]/\O[u]}\Hh_x)^{\gamma_0}$. By the argument in the last
part of that proof,  we can find a $\gamma_0$-stable affine alcove
$C$ in the apartment $\A(H_{\ti F}, T_H, \tilde F)$ such that $x$
belongs to the closure $\bar C$.
 Denote by $y$ the barycenter of $C$ which is then fixed by $\gamma_0$.
Then $\Hh_y$ is an Iwahori group scheme and $\overline\Hh_y^{\rm
red}=\calT$ is the split torus over $\O$.  An argument as in the
split case above, shows that it is enough to show that ${\rm
Gr}_{\Gg_y}$ is ind-proper. For simplicity, set $\Gg=\Gg_y$,
$\Hh=\Hh_y$. There is an exact sequence of pointed sets
$$
\Hh(R((v)))^{\gamma_0}/\Hh(R[[v]])^{\gamma_0}\hookrightarrow
(\Hh(R((v)))/\Hh(R[[v]])^{\gamma_0} \xrightarrow{\delta}
{\rH}^1(\Gamma, \Hh(R[[v]])).
$$
Now observe $(\Hh(R((v)))/\Hh(R[[v]])^{\gamma_0}={\rm
Gr}_{\Hh}(R)^{\gamma_0}={\rm Gr}_{\Hh}^{\gamma_0}(R)$, where ${\rm
Gr}_{\Hh}^{\gamma_0}$ is the closed ind-subscheme of ${\rm
Gr}_{\Hh}$ given by taking ${\gamma_0}$-fixed points.  The kernel of
$ \Hh(R[[v]]))\to \overline{\Hh}^{\rm red}(R)=\calT(R)$ is affine
pro-unipotent  and we  see that $\rH^1(\Gamma,
\Hh(R[[v]]))=\rH^1(\Gamma, \calT(R))$. Now,   consider the closed
subgroup scheme $Q_\calT$ of $\calT$ of elements $x$ that satisfy
the equation $N(x)=\prod_{i=0}^{e-1}\gamma_0^i(x)=1$. We can see
that the  sheaf   $R\mapsto \rH^1(\Gamma, \calT(R))$ is given by the
quotient $Q_\calT/\calT^{\gamma_0-1}$. The map $\delta$ is given as
follows: starting with $x\in   (\Hh(R((v)))/\Hh(R[[v]])^{\gamma_0}$
we can find $h\in \Hh(R((u)))$ such that $h\gamma_0(h)^{-1}$ is in
$\Hh(R[[v]]))$. We set $\delta(x)=\overline{h\gamma_0(h)^{-1}}$
which is well-defined in $Q_\calT/\calT^{\gamma_0-1}$.
 Using  Proposition \ref{locconstant}
we see that $Q_\calT/\calT^{\gamma_0-1}$ is  a finite \'etale
commutative group scheme $Q$ over $\O$ (of order that divides $e$).
The above exact sequence now gives that the sheaf associated to the
presheaf $R\mapsto \Hh(R((v)))^\gamma/\Hh(R[[v]])^{\gamma_0}$ is
represented by the fiber of the ind-scheme morphism $\delta: {\rm
Gr}_{\Hh}^{\gamma_0}\to Q$ over the identity section $\Spec(\O)\to
Q$. We conclude that the fpqc quotient
$L\Hh^{\gamma_0}/L^+\Hh^{\gamma_0}$ is represented by an ind-proper
ind-scheme over $\O$. To finish the proof recall that by
construction $\Gg$ is the neutral  component of $({\rm
Res}_{\O[v]/\O[u]}\Hh)^{\gamma_0}$. Using this,  Corollary
\ref{locconstant} and the fact that $\gamma_0$-fixed points of
affine pro-unipotent groups are connected, we see that the sheaf
associated to
$$
R\to ({\rm Res}_{\O[v]/\O[u]}\Hh)^{\gamma_0}(R[[u]])/\Gg(R[[u]])=
\Hh(R[[v]])^{\gamma_0}/\Gg(R[[u]])
$$
is represented by the finite \'etale commutative group scheme   of
connected components of $\calT =\overline{\Hh}^{\rm red}$.
Therefore, ${\rm Gr}_{\Gg}$  given by $R\mapsto
\Hh(R((v)))^{\gamma_0}/\Gg(R[[u]])= \Gg(R((u)))/\Gg(R[[u]])$ is
represented by a finite \'etale cover of
$L\Hh^{\gamma_0}/L^+\Hh^{\gamma_0}$. As such it is also an
ind-proper ind-scheme over $\O$.
 \end{proof}

\subsubsection{The global affine Grassmannian}\label{6b3}

 We continue with the same assumptions, but for a little while we allow $\Gg$ to
be any smooth affine group
 scheme over $X=\AA^1_\O$ with connected fibers.

Let $S\in ({\rm Sch}/X)$, with structure morphism $y: S\to X$. We
will denote by $\Gamma_y\subset X\times S$ the closed subscheme
given by the graph of $y$ and consider  the formal
completion of $X\times S$ along $\Gamma_y$.
Suppose that $S=\Spec(R)$ is affine. Then the above completion is an
affine formal scheme and following
  \cite[2.12]{BeilinsonDrinfeld} we can also consider the affine scheme $\hat\Gamma_y$
  given by the relative spectrum of the ring of regular functions on that completion.
  There is a natural closed immersion $\Gamma_y\to \hat\Gamma_y$ and we will
denote by $\hat\Gamma_y^\circ:=\hat\Gamma_y-\Gamma_y$ the complement of
the image. If $y: \Spec(R)\to X=\AA^1_\O$ is given by $u\mapsto y$,
we have $\Gamma_y\simeq \Spec(R[u]/(u-y))$, $\hat\Gamma_y\simeq
\Spec(R[[w]])$. When $y=0$, $\hat\Gamma_y=D_S$,
$\hat\Gamma_y^\circ=D^*_S$ as before. We can see directly that there is a morphism
$\hat\Gamma_y\to X\times S$ given by $R[u]\to R[[w]]$; $u\mapsto
w+y$. We will often  write
$\hat\Gamma_y= \Spec(R[[u-y]])$. Then
$\hat\Gamma_y^\circ=\Spec(R[[u-y]][(u-y)^{-1}])$.

\subsubsection{}\label{6b4} We will now consider various functors on $({\rm Sch}/X)$.
These will be fpqc sheaves on $X$ that can be described by giving
their values on affine schemes over $X$.

First consider the functor that associates to an $\O[u]$-algebra $R$
(given by $u\mapsto y$) the group
\begin{equation}\label{globloop1}
\L\Gg(R)=  \Gg(\hat\Gamma^\circ_y)=\Gg(R[[u-y]][(u-y)^{-1}]).
\end{equation}
Since $\Gg\to \Spec(\O[u])$ is smooth and affine, $\L\Gg$ is
represented by a formally smooth ind-scheme over $X$.

Next consider the functor that associates to an $\O[u]$-algebra $R$
the group
\begin{equation}\label{globloop2}
\L^+\Gg(R)=  \Gg(\hat\Gamma_y)=\Gg(R[[u-y]]).
\end{equation}
We can see that $\L^+\Gg$ is represented by a scheme over $X$ (not
of finite type) which is formally smooth.

Finally define the global affine Grassmannian of $\Gg$ over $X$ to
be the functor on $({\rm Sch}/X)$ given by
\begin{equation}\label{fun1}
 {\rm Gr}_{\Gg, X}(S)= \biggl\{\,\text{iso-classes of pairs } ( \mathcal E, \beta) \biggm|
      \twolinestight{$\mathcal E$ a $\Gg$-torsor on $X\times S$,}
   {$\beta$ a trivialization of $\mathcal E|_{ (X\times S)\setminus\Gamma_y}$}\,\biggr\}\, .
\end{equation}
Here and everywhere else the fiber products are over $\Spec(\O)$.

Similarly to Proposition \ref{identify}, the descent lemma
\ref{descentBL} implies that for $S=\Spec(R)$ the natural map given by restriction
along $\hat\Gamma_y\to X\times S$
\begin{equation}\label{equivfun}
 {\rm Gr}_{\Gg, X}(R)\xrightarrow{ \ } \biggl\{\,\text{iso-classes of pairs } ( \mathcal E, \beta) \biggm|
      \twolinestight{$\mathcal E$ a $\Gg$-torsor on $\hat\Gamma_y$,}
   {$\beta$ a trivialization of $\mathcal E |_{ \hat\Gamma_y^\circ}$}\,\biggr\}\,
\end{equation}
is a bijection for each $\O[u]$-algebra $R$. This provides with an
alternative description of ${\rm Gr}_{\Gg, X}$. Using this
description, we can see that $\L\Gg$, $\L^+\Gg$ act on $\Gr_{\Gg,
X}$ by changing the trivialization $\beta$. In fact, we have $\Gr_{\Gg, X}\simeq \L\Gg/\L^+\Gg$
but we are not going to use this.

\subsubsection{}\label{kappafibers}  Suppose now that  $\Gg$ is as in Theorem \ref{grpschemeThm}. Let $\kappa$ be either the fraction field $F$ or the residue field
$k$ of $\O$. Let $x: \Spec(\kappa)\to  X$, where $\kappa$ is  as
above and identify the completed local ring $\widehat{\mathcal O}_x$
of $X\times \Spec (\kappa)$ with $\kappa[[t]]$, using the local
parameter $t=u-x$. Let
\begin{equation}\label{basechangeG}
\Gg_{\kappa,x}:=\calG\times_{\Spec(\O[u])}\Spec(\kappa[[t]]).
\end{equation}

(i) Suppose that $x$ factors through $0: \Spec(\O)\to X$. Recall
that by Theorem \ref{grpschemeThm} the base change $\Gg_{\kappa,
0}$, can be identified with a Bruhat-Tits group scheme $P_\kappa
:=\P_{x_{\kappa((u))}}$ over the dvr $\kappa[[t]]=\kappa[[u]]$.

Let $L^+ P_\kappa$ be the affine group scheme over $\Spec(\kappa)$
representing the functor on $\kappa$-algebras
\begin{equation*}
R\mapsto L^+ P_\kappa (R) = P_\kappa ( R[[t]])\, ,
\end{equation*}
and $L P_\kappa$  the ind-group scheme over $\Spec (\kappa)$
representing the functor
\begin{equation*}
R\mapsto L P_\kappa (R) = P_\kappa \big(R((t))\big)=G_\kappa
\big(R((t))\big)\, .
\end{equation*}
Here $G_\kappa=P_\kappa[t^{-1}]$ (which is denoted by $\und
G_{\kappa((u))}$ in Chapter \ref{groupscheme}) is the connected
reductive group over $\kappa((t))$ which is obtained by base
changing $\Gg\to \Spec(\O[u])$ along $\O[u]\to \kappa((t))$,
$u\mapsto t$. As in Proposition \ref{indproper} we see that there is
an ind-proper ind-scheme $\Gr_{P_\kappa}$ over $\kappa$ which
represents the quotient $L P_\kappa/L^+ P_\kappa$ of fpqc-sheaves on
$\kappa$-schemes. By Proposition \ref{identify}, $\Gr_{P_\kappa}$ is
the ind-scheme representing 
\[
   R\mapsto \Gr_{P_\kappa} (R) = \biggl\{\,\text{iso-classes of pairs } (\mathcal E, \beta) \biggm|
   \twolinestight{$\mathcal E$ a $P_\kappa$-torsor on $\Spec R[[t]]$,}
   {$\beta$ a trivialization of $\mathcal E|_{\Spec R((t))}$}\,\biggr\} .
\]
The base change
$\Gr_{P_\kappa}\times_{\Spec(\kappa)}\Spec(\bar\kappa)$
  is an  affine flag variety as in \cite{PappasRaTwisted}.

(ii) Suppose  $x: \Spec(\kappa)\to X$ does not factor  through $0:
\Spec(\O)\to X$. Then by Theorem \ref{grpschemeThm} (i), the base
change $\Gg_{\kappa, x}$ is a reductive group scheme which is a form
of $H$. We can see that
\begin{equation}
\Gg_{\kappa, x}\times_{\Spec(\kappa)}\Spec(\kappa')\simeq
H\times_{\Spec(\O)}\Spec( \kappa'[[t]]).
\end{equation}
for a finite  $\kappa'/\kappa$.
As above, we also have the affine Grassmannian $\Gr_{\Gg_{\kappa,
x}}$ over $\kappa$; by the above, we can see that $\Gr_{\Gg_{\kappa,
x}}\times_{\Spec(\kappa)}\Spec(\bar\kappa)$  can be identified with
the usual affine Grassmannian ${\rm Gr}_H$  over $\bar\kappa$ for
the split reductive group $H$. The following observation now follows
from Proposition \ref{identify}:

\begin{prop} \label{globalGr}
Let $x: \Spec(\kappa)\to  X$, where $\kappa$ is either the residue
field $k$ of $\mathcal O$, or the fraction field $F$ of $\mathcal
O$, and identify the completed local ring $\widehat{\mathcal O}_x$
of $X\times \Spec (\kappa)$ with $\kappa[[t]]$, using the local
parameter $t=u-x$. Then  restricting $\Gg$-bundles from
$\O[u]\otimes_\O R$ to $R[[t]]$, $u\mapsto t+x$, induces an
isomorphism over $\Spec (\kappa)$,
\begin{equation*}
i_x^\ast \colon {\rm Gr}_{\Gg, X}\times_{X, x}
\Spec(\kappa)\xrightarrow{\sim} \Gr_{\Gg_{\kappa, x}}\, .
\end{equation*}
Here $\Gr_{\Gg_{\kappa, x}}=L\Gg_{\kappa, x}/L^+\Gg_{\kappa, x}$
denotes the affine Grassmannian over $\kappa$ as above; this is
isomorphic to either $\Gr_{P_\kappa}$ if $x$ maps to $0$, or to
$\Gr_H$ over $\bar \kappa$ otherwise.
 \end{prop}

 Notice here that at this point we only consider ${\rm Gr}_{\Gg, X}$ as a fpqc sheaf over $X$.
However, using the next proposition we will soon see that these are
actually isomorphisms of
 ind-schemes.
 Remark here that the above proposition combined with Proposition \ref{indproper}
 already shows that the
 fiber of ${\rm Gr}_{\Gg, X}$ over $x: \Spec(\kappa)\to X$ is represented by an
 ind-scheme which is ind-projective over $\Spec(\kappa)$.

\begin{prop}\label{indscheme} Suppose that $\Gg$ is as in Theorem \ref{grpschemeThm}.
The  functor ${\rm Gr}_{\Gg, X}$ on $({\rm Sch}/X)$ is representable
by an  ind-projective ind-scheme over $X$.
\end{prop}

\begin{proof}  We first show that ${\rm Gr}_{\Gg, X}$ is representable by an ind-scheme
 of ind-finite type and separated over $X$. This follows the corresponding argument in the proof of Proposition
\ref{indproper}. By Proposition \ref{qaffine}, there is a closed
group scheme immersion $\Gg\hookrightarrow \GL_n$ such that the
quotient $\GL_n/\Gg$ is representable by a quasi-affine scheme over
$\O[u]$. The argument in  \cite{BeilinsonDrinfeld} (or
\cite[Appendix]{GaitsgoryInv}) now shows that the natural functor $
{\rm Gr}_{\Gg, X}\to {\rm Gr}_{\GL_n, X} $ is representable and is a
locally closed immersion. In fact, if the quotient $\GL_n/\Gg$ is
affine, this functor is a closed immersion. Now note that ${\rm
Gr}_{\GL_n, X}$ is representable by an ind-scheme separated of
ind-finite type over $X$. This is well-known (see for example
\cite{BeilinsonDrinfeld}). In fact,  ${\rm Gr}_{\GL_n, X}$ is
ind-projective over $X$.

It remains to show that ${\rm Gr}_{\Gg, X}\to X$ is ind-proper. By
Propositions \ref{indproper} and \ref{globalGr} each fiber of
$\Gr_{\Gg, X}\to X$ is ind-proper. It is enough to show that the
base change by $\tilde X=\Spec(\tilde\O_0[v])\to X$ is ind-proper.
Notice that, since $\tilde X-\{0\}\to X-\{0\}$ is finite \'etale,
there is an isomorphism $\Gr_{\Gg, X}\times_X (\tilde X-\{0\})\simeq
\Gr_{H, X}\times_X (\tilde X-\{0\})=\Gr_H\times_{\O}(\tilde
X-\{0\})$ (cf. \cite[Lemma 3.3]{ZhuCoherence}, here again $H$ is the
split Chevalley form). Therefore, by Proposition \ref{indproper}
applied to $\Gr_H$, we see that
  the  restriction of $\Gr_{\Gg, X}\to X$ over $U=(\tilde X-\{0\})\otimes_\O F$ is ind-proper.
We can write this restriction as a limit $S_i$ of proper schemes
over $U$. In fact, using standard results on the structure of the
affine Grassmannians $\Gr_H$ over the field $F$
(\cite{GaitsgoryInv}, \cite{FaltingsLoops}, \cite{PappasRaTwisted})
we can assume that $S_i=\sqcup_j S_{ij}$ with $S_{ij}$ proper
schemes over $U$ with geometrically connected fibers. Denote by
$Y_{ij}$, resp. $Z_{ij}$, the Zariski closures of $S_{ij}$ in ${\rm
Gr}_{\Gg, X}\times_X \tilde X$, resp. ${\rm Gr}_{\GL_n, X}\times_X
\tilde X$. Since ${\rm Gr}_{\GL_n, X}\to X$ is ind-proper,
$Z_{ij}\to \tilde X$ is proper. Since ${\rm Gr}_{\Gg, X}\to {\rm
Gr}_{\GL_n, X}$ is a locally closed immersion, $Y_{ij}$ is open and
dense in $Z_{ij}$. Denote by bar fibers at a closed point of $\tilde
X$. It enough to show that we always have $\bar Y_{ij}=\bar Z_{ij}$.
Since all the fibers of $\Gr_{\Gg, X}\to X$ are ind-proper, $\bar
Y_{ij}$ is proper and so $\bar Y_{ij}$ is closed in $\bar Z_{ij}$.
By Zariski's main theorem applied to $Z_{ij}\to \tilde X$, we see
that $\bar Z_{ij}$ is connected and so $\bar Y_{ij}=\bar Z_{ij}$.
Hence, $Y_{ij}=Z_{ij}$ and $Y_{ij}\to \tilde X$ is proper. It
remains to see that each point of each fiber of $ {\rm Gr}_{\Gg,
X}\times_X \tilde X\to \tilde X$ belongs to some $Y_{ij}$. This
lifting property
  can be seen by the argument in the proof of Proposition \ref{Prop8.8}.
\end{proof}

\subsubsection{Specialization along $u=\varpi$.}
Now let us fix a uniformizer $\varpi$ of $\O$. We denote by $\varpi$
the section of $X=\AA^1_\O$ over $\O$ defined by $u\mapsto\varpi$.
Let $G$ be connected reductive over $F$, and split over a tamely
ramified extension $\ti F/F$ as in \ref{sss1a2}; let $\und G$ be
constructed from $G$ as in \S \ref{reductive group}. In addition, we
fix an isomorphism $\und G_F\simeq G$ from a rigidification of $G$
as explained in \ref{sss3a4}. This produces a group scheme
$\Gg:=\calG_x$ as in Corollary \ref{application}, which is
independent of the choice of the rigidification of $G$ up to
isomorphism.

Notice that there is an isomorphism
\begin{equation}\label{poweriso}
\tilde\O_0[v^{\pm 1}]\otimes_{\O[u^{\pm
1}]}F[[u-\varpi]]\xrightarrow{\sim} \tilde F[[z]]=\tilde
F[[u-\varpi]],
\end{equation}
given by $v\mapsto \ti\varpi \cdot (1+z)$  where $\ti\varpi^e=\varpi
$. Here $z$ maps to the power series
$(1+\frac{(u-\varpi)}{\varpi})^{1/e}-1$, where the $e$-th root is
expressed by using the standard binomial formula.  This isomorphism
  matches the action of $\Gamma$ on the left hand
side (coming from the cover $\O[u]\to \tilde\O_0[v]$ by base
change), with the  action on $\tilde F[[z]]$  given by the Galois
action on the coefficients $\tilde F$. Using this and the
construction of the group scheme $\und G$ in \S \ref{reductive
group} we obtain an isomorphism
\begin{equation}\label{iso6.9}
\Gg_{F, \varpi}\xrightarrow{\sim} G\times_{\Spec(F)}\Spec(
F[[u-\varpi]])
\end{equation}
well defined up to $G(F)$-conjugation.

Denote by ${\rm Gr}_{\Gg, \O}$ the fiber product
\begin{equation}
{\rm Gr}_{\Gg, \O}:={\rm Gr}_{\Gg, X}\times_{X, \varpi} \Spec(\O)\to
\Spec(\O).
\end{equation}
Using Proposition \ref{indscheme} we see that this is an
ind-projective ind-scheme
 over $\Spec(\O)$.

 Proposition \ref{globalGr} and the discussion in
the beginning of \S \ref{kappafibers} implies:

\begin{cor}\label{fibers}
1) The  generic fiber  ${\rm Gr}_{\Gg,
\O}\times_{\Spec(\O)}\Spec(F)$ is equivariantly isomorphic to the
affine Grassmannian $\Gr_{G, F}$ of $G$ over $\Spec(F)$.

2) The special fiber  ${\rm Gr}_{\Gg, \O}\times_{\Spec(\O)}\Spec(k)$
is equivariantly isomorphic to the affine Grassmannian $\Gr_{P_k}$
over $\Spec (k)$.\endproof
\end{cor}

\subsubsection{Notation}
Let $\calG$ be as in Corollary
\ref{application}. If $f: S=\Spec(R)\to X$ is a scheme morphism
given by $u\mapsto r$, we will write ${\rm Gr}_{\Gg, R, r}$ 
for the fiber product ${\rm Gr}_{\Gg, X}\times_{X, f} S$. 
If $R$ is an $\O$-algebra and
$\O[u]\to R$ is given by $u\mapsto \varpi$, we will simple write
${\rm Gr}_{\Gg, R}$ instead. This agrees with our use of the notation
${\rm Gr}_{\Gg, \O}$ above.

 \bigskip

\section{Local models}\label{ChLocal}

\setcounter{equation}{0}

Here we give our group-theoretic definition of local models. We also
explain how, in the examples of $\GL_n$, ${\rm GSp}_{2n}$ and a
minuscule coweight, it follows from \cite{GortzFlatGLn} and
\cite{GortzSymplectic} that these agree with the local models of
\cite{RapZinkBook}. This last result will be generalized in the next
chapter.

\subsection{Generalized local models}

\subsubsection{Cocharacters.}\label{coch} We continue with the above assumptions and notations.
Suppose now that $\{\mu\}$ is a geometric conjugacy class of one
parameter subgroups of $G$, defined over an algebraic closure
$\overline F$ of $F$ that contains the field $\tilde F$. Let $E$ be
the field of definition of $\{\mu\}$, a finite extension of $F$
contained in $\overline F$ (the reflex field of the pair $(G, \{
\mu\})$).

First observe that since $G$ is quasi-split over the maximal
unramified extension $\tilde F_0$ of $F$ in $\ti F$ we can find
(\cite[Lemma (1.1.3)]{KottTwisted}) a representative of $\{\mu\}$
defined over $E'=E\tilde F_0$, which factors $\mu: \Gm_{E'}\to
T_{E'}\to G_{E'}$, where $T$ is the maximal torus of $G$ given as in
\ref{sss1a3}.
 Notice that $\mu$ gives an
$E'[z,z^{-1}]$-valued point of $G_{E'}$, therefore an
$E'((z))$-valued point of $G_{E'}$, therefore an $E'$-valued point
of the loop group $LG $. By (\ref{iso6.9}) we have an isomorphism
$$
G(F((z)))\xrightarrow{\sim}
\Gg_{F,\varpi}(F((u-\varpi)))=\Gg_{F,\varpi}(F((t))).
$$
We denote by $s_\mu$ the corresponding $E'$-valued point in
$L\Gg_{F, \varpi}$.

\subsubsection{Schubert varieties in mixed characteristic.}\label{7a2}
 We would like to define
a projective scheme $M_{\Gg,  \mu }$ over $\O_E$ which we might view
as a generalized local model. Recall the definition of $s_\mu\in
L\Gg_{F, \varpi}(E')$ and consider the $L^+\Gg_{F, \varpi}$-orbit
$(L^+\Gg_{F, \varpi})_{E'}\cdot [s_\mu]$ of the corresponding point
$[s_\mu]$ in the affine Grassmannian $(L\Gg_{F, \pi}/L^+\Gg_{F,
\varpi})\times_F E'$. This orbit is contained in a projective
subvariety of $(L\Gg_{F, \varpi}/L^+\Gg_{F, \varpi})\times_F E'$ which by
Corollary \ref{fibers} (1) above can be identified with the generic
fiber of $\Gr_{\Gg, \O}\otimes_\O\O_{E'}\to \Spec(\O_{E'})$. Since
the conjugacy class of $\mu: \Gm_{E'}\to   G_{E'}$ is   defined over
$E$, the same is true for the orbit $(L^+\Gg_{F, \varpi})_{E'}\cdot
[s_\mu]$: There is an $E$-subvariety $X_\mu$ of  $(L\Gg_{F,
\varpi}/L^+\Gg_{F, \varpi})\times_F E$ such that
$X_\mu\times_EE'=(L^+\Gg_{F, \varpi})_{E'}\cdot [s_\mu]$.

\begin{Definition}
The generalized local model (or mixed characteristic Schubert
variety) $M_{\Gg, \mu}$ is the reduced scheme over $\Spec(\O_E)$
which underlies the Zariski closure of the orbit $X_\mu$ in the
ind-scheme $\Gr_{\Gg, \O_E}=\Gr_{\Gg,
\O}\times_{\Spec(\O)}\Spec(\O_E)$.
\end{Definition}

Since by Proposition \ref{indscheme}, $\Gr_{\Gg, \O_E}\to \Spec(\O_E)$ is ind-projective,
$M_{\Gg, \mu}$ is also projective over $\Spec(\O_E)$.
\smallskip

\subsection{Some examples}\label{ShimuraLocal}

\subsubsection{} \label{lattice}

{\sl The case of $\GL_{N}$.} Recall the notations of \S \ref{exGL}.
In particular, $\Gg$ is the group scheme over $\O[u]$ associated to
the lattice chain $\{W_i\}_i$.

Consider the functor ${\mathfrak L}$ on $({\rm Sch}/X)$ which to an
$X$-scheme $y: S\to X$, associates the set of isomorphism classes of
collections $(\calE_i, \psi_i,  \alpha_i)_{i\in \Z}$ where, for each
$i\in \Z$, $\E_i$  are  locally free coherent $\O_{X\times
S}$-sheaves on $X\times S$ of rank $N$, $\psi_i  : \E_{i+1}\to \E_{i
}$ are $\O_{X\times S}$-module homomorphisms, and $\alpha_i$ are
$\O_{{X\times S}-\Gamma_y}$-module isomorphisms $\alpha_i:
W_i\otimes_{\O[u]}\O_{{X\times S}-\Gamma_y}\xrightarrow{\sim}
\E_{i}\otimes_{\O_{X\times S}}\O_{{X\times S}-\Gamma_y} $
 that satisfy the
following conditions:

(a) the data are periodic of period $r$, $(\calE_i, \psi_i,
\alpha_i)=( \E_{i+r}, \psi_{i+r},  \alpha_{i+r})$, for all $i\in
\Z$,

(b) we have $\alpha_{i+1 }\cdot \psi_{i}  =\alpha_i\cdot \iota_i$,
for all $i\in \Z$,

(c) Each composition of $r$ successive $\psi_i$ is given by
multiplication by $u$, i.e
$$
 \prod_{k=0}^{r-1}\psi_{i-k}=u :\E_{i+r}=\E_{i}\to \E_{i},
$$
for all $i\in \Z$, and,

(d) for each $i\in \Z$, the cokernel $\E_{i}/\psi_{i}(\E_{i+1})$ is
a locally free $\O_S$-module of rank $r_i$.
\smallskip

We can see that ${\mathfrak L}$ is an fpqc  sheaf on $({\rm
Sch}/X)$.

When $S=\Spec(R)$ is affine, and $y: \Spec(R)\to X=\Spec(\O[u])$ is
given by $u\mapsto y$, we have $X\times
S-\Gamma_y=\Spec(R[u][(u-y)^{-1}])$. Since $u-y$ is not a zero
divisor in $R[u]$ we can use $\alpha_i$ to identify $\E_i$ with the
sheaf corresponding to a $R[u]$-locally free rank $N$ submodule
$E_i$ of $R[u][(u-y)^{-1}]^N$.
\smallskip

We can now show:

\begin{prop}\label{latGL}
There is a natural equivalence of functors ${\rm Gr}_{\Gg,
X}\xrightarrow{\sim} {\mathfrak L}$ where $\Gg$ is the group scheme
as above.
\end{prop}

\begin{proof}
Observe that a $\Gg$-torsor $\calT$ over $X\times S$ induces via
$\Gg\hookrightarrow \prod_{i=0}^r\GL(W_i)\to \GL(W_i)$ a
$\GL(W_i)$-torsor over $X\times S$. This amounts to giving a locally
free coherent $\O_{X\times S}$-sheaf $\E_i$ of rank $N$; since $\Gg$
respects the maps  $W_{i+1}\to W_{i}$, we obtain $\psi_i:
\E_{i+1}\to \E_{i}$. A $\Gg$-trivialization of $\calT$ over $X\times
S-\Gamma_y$ produces isomorphisms $\alpha_i$ as above. We extend
this data by periodicity; then (a), (b), (c), (d) are satisfied.
This gives the arrow ${\rm Gr}_{\Gg, X}\to {\mathfrak L}$. To show
that this is an equivalence, start with data $(\calE_i, \psi_i,
\alpha_i)_{i\in \Z}$ giving an element of ${\mathfrak L}(S)$. We
would like to show that these are produced by a $\Gg$-torsor $\calT$
with a trivialization over $X\times S-\Gamma_y$. It is enough to
assume that $S$ is affine, $S=\Spec(R)$. Since, $\Gg$ is the
subgroup of $\prod_{i\in \Z/r\Z}\GL(W_i)$ that respects $\iota_i$, we
can now see that it is enough to show the following:  Locally for
the Zariski topology on $R$, there are isomorphisms
$$
\lambda_i: W_{i}\otimes R[u]\xrightarrow {\sim} \E_i
$$
such that $\lambda_{i}\cdot \iota_{i}=\psi_i\cdot \lambda_{i+1}$.
This follows by an argument similar to the proof of \cite[Appendix to Ch. 3,
Prop. A.4]{RapZinkBook}.
\end{proof}

Now if $\mu:\Gm \to \GL_N$ is the minuscule coweight given by
$a\mapsto {\rm diag}(a^{(d)}, 1^{(N-d)})$ and $\O=W(\Ff_p)$, we can see, using
Proposition \ref{latGL}, that in this situation, the local models
$M_{\Gg, \mu}$ agree with the Rapoport-Zink local models for $\GL_N$
and $\mu$ considered in \cite{RapZinkBook}. Indeed, in this case, by
\cite{GortzFlatGLn}, the  local models of \cite{RapZinkBook} are
flat over $\O$ and so they agree with the $M_{\Gg, \mu}$ above.
 \smallskip

\subsubsection{} \label{latticeGSp}

 {\sl The case of ${\rm GSp}_{2n}$.} Recall the notations of \S \ref{exGSp}.
 In particular, $\Gg$ is the group
scheme over $\O[u]$ associated to the self-dual lattice chain
$\{W_i\}_i$.

Consider the functor ${\mathfrak {LSP}}$ on $({\rm Sch}/X)$ which to
an $X$-scheme $y: S\to X$, associates the set of isomorphism classes
of collections $(\calE_i, \psi_i,  \alpha_i, h_i)_{i\in \Z}$ where
$(\calE_i, \psi_i,  \alpha_i)_{i\in \Z}$ give an object of
$\mathfrak L$ and in addition
$$
h_i : \calE_i\times \calE_{-i-a}\to \O_{X\times S}
$$
are perfect $\O_{X\times S}$-bilinear forms that satisfy
\begin{itemize}

\item[a)] $h_{i}(\psi_i(x), y)=h_{i+1}(x, \psi_{-i-1-a}(y))$,  for $x\in \calE_{i+1}$, $y\in \calE_{-i-a}$.

\item[b)] There is $c\in \O_{{X\times S}-\Gamma_y}^*$, such that $h_i\cdot (\alpha_i, \alpha_{-i-a})=c\cdot h$
for all $i\in \Z$ (as forms $W_i\otimes\O_{{X\times
S}-\Gamma_y}\times W_{-i-a}\otimes\O_{{X\times S}-\Gamma_y}\to
\O_{{X\times S}-\Gamma_y}$).
\end{itemize}

A similar argument as above, (cf. \cite[Appendix to Ch. 3, Prop.
A. 21]{RapZinkBook}) now gives
\begin{prop}
There is a natural equivalence of functors ${\rm Gr}_{\Gg,
X}\xrightarrow{\sim} {\mathfrak {LSP}}$ where $\Gg$ is the
(symplectic) group scheme as above.\endproof
\end{prop}

Again as a result of the above, combined with the flatness result of
\cite{GortzSymplectic}, we can see that if $\mu:\Gm \to {\rm
GSp}_{2n}$ is the standard minuscule coweight given by $a\mapsto
{\rm diag}(a^{(n)}, 1^{(n)})$ and $\O=W(\Ff_p)$, then  the local models
$M_{\Gg, \mu}$ in this situation agree with the local models for
${\rm GSp}_{2n}$ considered in \cite{RapZinkBook}.
\smallskip

\subsubsection{} One can find a similar interpretation of ${\rm Gr}_{\Gg, X}$ as
moduli spaces of chains of bundles with additional structure given
by suitable forms in   more cases as in \S \ref{exClassical}, for
example when $G$ is an orthogonal group or a (ramified) unitary
group. We will leave the details to the reader. A corresponding
statement comparing the local models $M_{\Gg, \mu}$ with the local
models in the theory of PEL Shimura varieties (\cite{RapZinkBook},
\cite{GortzSymplectic}, \cite{GortzFlatGLn}, \cite{PappasRaI},
\cite{PappasRaII}, \cite{PappasRaIII}, \cite{PRS}) will be explained
in the next paragraph.

\bigskip

 \section{Shimura varieties and   local models} \label{Shimura}
\setcounter{equation}{0}

Here we discuss Shimura varieties and their integral models over
primes where the level subgroup is parahoric. We  conjecture that
there exist integral models that fit in a ``local model diagram" in
which the local model is given by our construction in the previous
chapter. We show this in most cases of Shimura varieties of PEL
type. We also explain how Theorem \ref{thmPEL}
of the introduction follows from our main structural results
on local models (which will be shown in the next section).

\subsection{The local model diagram}\label{8a}
Let $Sh_{\bold K} = Sh ({\bold G}, \{h\}, {\bold K})$ denote a
Shimura variety \cite{DeligneTravauxShimura} attached to the triple
consisting of  a {\sl connected} reductive group $\bold G$ over
$\mQ$, a family of Hodge structures $h$ and a compact open subgroup
$\bold K\subset \bold G(\mA_f)$. We fix a prime number $p$ and
assume that $\bold K$ factorizes as $\bold K = K^p\cdot K_p\subset
{\bold G}(\mA_f^p)\times {\bold G} (\mQ_p)$. We assume in addition
that $K=K_p$ is a parahoric subgroup of $\bold G (\mQ_p)$, i.e it
corresponds to the connected stabilizer of a point of the
Bruhat-Tits building of $\bold G\otimes_{\Q}\mQ_p$. We denote by
$\calP$ the corresponding Bruhat-Tits group scheme over $\Z_p$.

Let $\boE\subset\mC$ denote the reflex field of  $({\bold G},
\{h\})$, i.e. the field of definition of the geometric conjugacy
class of one-parameter subgroups $\{\mu\} = \{\mu_h\}$ attached to
$\{h\}$, cf.~\cite{DeligneTravauxShimura}. Then $\boE$ is a finite
extension of $\mQ$. Fixing an embedding
$\overline{\mQ}\to\overline{\mQ}_p$ determines a place $\wp$ of
$\boE$ above $p$. We denote by the same symbol the canonical model
of $Sh_{\bold K}$ over $\mE$ and its base change to $\mE_{\wp}$. For
simplicity, set $E=\boE_{\wp}$ and denote by $\O_E$ the ring of
integers of $E$ and by $k_E$ its residue field. It is then an
interesting problem to define a suitable model $\mathcal S_{\bold
K}$ of $Sh_{\bold K}$ over $\Spec(\O_E)$. Such a model should be
projective if $Sh_{\bold K}$ is (which is the case when $\bold
G_{\rm ad}$ is $\mQ$-anisotropic), and should always have manageable
singularities. In particular, it should be flat over $\Spec(\O_E)$,
and its local structure should only depend on the ``localized" group $G
= \bold G\otimes_{\mQ}\mQ_p$, the geometric conjugacy class
$\{\mu\}$ over $\overline{\mQ}_p$, and the parahoric subgroup $K =
K_p$ of $G (\mQ_p)$. Note that, due to the definition of a Shimura
variety, the conjugacy class $\{\mu\}$ is minuscule.

Suppose now in addition that the group $G$ splits over a tamely
ramified extension of $\Q_p$. 
We can then apply the constructions of the previous paragraphs to $G
$, $\O=\Z_p$, and a point of the building ${\mathcal B}(G ,
\mQ_p)$ that corresponds to $K \subset G(\mQ_p)$. By Theorem
\ref{grpschemeThm}, we obtain a smooth affine group scheme $\Gg \to
\Spec(\Z_p[u])$; the choice of $\{\mu \}$ allows us to give a
projective scheme $M_{\Gg  , \mu }\to \Spec(\O_E)$.  Let us set
$$
{\rm M}(G, \{\mu\})_K=M_{\Gg ,  \mu  }.
$$
  By its construction,
${\rm M}(G, \{\mu\})_K$ affords an action of the group scheme
$\Gg\otimes_{\Z_p[u],u\mapsto p}\O_E=\P\otimes_{\Z_p}\O_E$. The
conjecture  is   that there exists a model  $\mathcal S_{\bold K}$
of the Shimura variety over $\O_{E}$ whose singularities are
``described by the local model  ${\rm M}(G, \{\mu\})_K$''. More
precisely:

We conjecture that there is such a $\mathcal S_{\bold K}$ that affords
a {\it local model diagram}
\begin{equation}\label{locmoddiagram}
\xymatrix{
& {\widetilde{\mathcal S}}_{\bold K}\ar[ld]_{\pi}\ar[rd]^{\widetilde{\varphi}} & \\
{\quad\quad  \mathcal S_{\bold K} \quad\quad} & & {{\rm M}(G,
\{\mu\})_K}\, ,
 }
\end{equation}
 of $\mathcal O_{E}$-schemes, in which:
\begin{itemize}
\item $\pi$ is a torsor under the  group $\mathcal P_{\O_E}=\mathcal P\otimes_{\mZ_p}\mathcal O_{E}$,

\item $\widetilde{\varphi}$ is $\mathcal P_{\O_E}$-equivariant and smooth of relative dimension $\dim G$.
\end{itemize}

(Equivalently, using the language of algebraic stacks, there should
be a smooth morphism  of algebraic stacks
\begin{equation*}
\varphi:  {\mathcal S}_{\bold K} \to [{\rm M}(G, \{\mu\})_K/\mathcal
P_{\O_E}]\, .
\end{equation*}
of relative dimension $\dim G$ where in the brackets we have the
stack quotient. See also \cite{PRS}.) In fact, we conjecture that
there is such a diagram with $\widetilde{\varphi}$ being in addition surjective.

The existence of the local model diagram implies the following: Suppose $x$ is
a point of
 $\mathcal S_{\bold K} $ with values in the finite field $\mF_q$. By Lang's theorem
 the $\P_{\O_E}$-torsor $\pi$ splits over $x$, and so there is  $\ti x\in \widetilde{\mathcal S}_{\bold K}(\mF_q)$
 with image $y=\tilde\phi(\tilde x)\in {\rm M} (G, \{\mu\})_K
({\mF}_q)$  such that the henselizations of $ \mathcal S_{\bold K} $
at $x$ and of ${\rm M}(G,\{\mu\})_K$ at $y$ are isomorphic. The
$\mathcal P_{\O_E}$-orbit of $y$ in ${\rm M}(G,\{\mu\})_K$ is
well-defined.

\subsubsection{} Now suppose that  there is a closed group
scheme immersion $\rho: {\bold G}\hookrightarrow {\rm GSp}_{2n}$
such  that the composition of $\rho$ with $\mu$ is in the conjugacy
class of the standard minuscule cocharacter of ${\rm GSp}_{2n}$. For
typesetting simplicity, set $k=\overline{\bbF}_p=\bar k_E$. We will
also assume that there is a self-dual ``lattice" chain
$W_\bullet=\{W_i\}_{i\in \Z}$ in $\Z_p[u]^{2n}$ as in \S
\ref{latticeGSp}
 such that

\begin{itemize}
\item the homomorphism ${\bold G}\rightarrow {\rm GSp}_{2n}$
extends to a homomorphism $\Gg\rightarrow {\rm GSp}(W_\bullet)$,

\item  the homomorphism $\Gg\otimes_{\Z_p[u]}k[[u]]\rightarrow {\rm GSp}(W_\bullet)\otimes_{\Z_p[u]}k[[u]]$
is a locally closed immersion, the Zariski closure of
$\Gg\otimes_{\Z_p[u]}k((u))$ in ${\rm
GSp}(W_\bullet)\otimes_{\Z_p[u]}k[[u]]$ is a smooth group scheme
$P'_k$ over $k[[u]]$ and   $P'_k(k[[u]])$ stabilizes $x_{k((u))}$ in
the building of $\Gg(k((u)))$; then
$P_k:=\Gg\otimes_{\Z_p[u]}k[[u]]$ is the neutral component of
$P'_k$.
\end{itemize}

Under these assumptions, extending torsors via the homomorphism
$\Gg\to {\rm GSp}(W_\bullet)$   gives   ${\rm Gr}_{\Gg, \Z_p}\to
{\rm Gr}_{{\rm GSp}(W_\bullet), \Z_p}$. Restricting to ${\rm M}(G,
\{\mu\})_K\hookrightarrow {\rm Gr}_{\Gg, \Z_p}\otimes_{\O}\O_E$
gives a morphism of schemes
$$
\iota: {\rm M}(G, \{\mu\})_K\to {\rm M}({\rm
GSp}_{2n})_{W_\bullet}\otimes_{\Z_p}\O_E
$$
where ${\rm M}({\rm GSp}_{2n})_{W_\bullet}$ is the symplectic local
model as in \cite{GortzSymplectic} (cf. \ref{latticeGSp}).

\begin{prop}\label{embeddLoc}
Under the above assumptions,  $\iota: {\rm M}(G, \{\mu\})_K\to {\rm
M}({\rm GSp}_{2n})_{W_\bullet}\otimes_{\Z_p}\O_E $ is a closed
immersion.
\end{prop}

\begin{proof}
Recall that the generic fiber ${\rm M}(G, \{\mu\})_K\otimes_{\O_E}E$
of ${\rm M}(G, \{\mu\})_K$ is the flag variety of parabolics
 corresponding to $\{\mu\}$; the generic fiber of
${\rm M}({\rm GSp}_{2n})_{W_\bullet}$ is the Lagrangian Grassmannian
${\rm LGr}(n, 2n)$ of $n$-dimensional isotropic subspaces; our
assumption on $\{\rho\circ\mu\}$ implies that $\iota\otimes_{\O_E}E$ is
a closed immersion. We will now explain why, in this set-up, the
morphism on the special fibers $\iota\otimes_{\O_E}k_E$ is also a
closed immersion.

As above $P_k=\Gg\otimes_{\Z_p[u]}k[[u]]$ and $P'_k$ is the closure of
$P_k[u^{-1}]$ in ${\rm GSp}(W_\bullet)\otimes_{\Z_p[u]}k[[u]]$;
$P_k$ is a parahoric group scheme over $k[[u]]$. By our assumption,
$P'_k$ is a smooth affine group scheme over $k[[u]]$ and $P_k$ is the
neutral component of $P'_k$.
 Both ${\rm Gr}_{P_k}$ and $ {\rm Gr}_{P'_k}$
are ind-proper ind-schemes over $k$ and the natural morphism
$$
{\rm Gr}_{\Gg, k}={\rm Gr}_{P_k}\to {\rm Gr}_{P'_k}
$$
is  finite \'etale. In what follows, for simplicity, set $P=P_k$,
$P'=P'_k$. Consider the
 Kottwitz homomorphism $\kappa: P'(k((u)))=P(k((u)))\to \pi_1(P[u^{-1}])_I$ for the reductive group $P[u^{-1}]=P'[u^{-1}]$
over $k((u))$. By \cite{HainesRapoportAppendix}, since $P'(k[[u]])$
stabilizes $x_{k((u))}$, the intersection of the kernel ${\rm
ker}(\kappa)$ with $P'(k[[u]])$ is equal to $P(k[[u]])$.  By
\cite{PappasRaTwisted}, the homomorphism $\kappa$ induces a
bijection
$$
\pi_0(LP)\simeq \pi_0({\rm Gr}_{P})\xrightarrow{\sim}
\pi_1(P[u^{-1}])_I
$$
between the set of connected components of ${\rm Gr}_P$ and the
group  $\pi_1(P[u^{-1}])_I$. The above now imply that ${\rm
Gr}_{\Gg, k}={\rm Gr}_{P}\to {\rm Gr}_{P'}$ identifies each
connected component of ${\rm Gr}_{P}={\rm Gr}_{\Gg, k}$ with a
connected component of ${\rm Gr}_{P'}$.

Now $P'$ is a closed subgroup scheme of  $Q:={\rm
GSp}(W_\bullet)\otimes_{\O[u]}k[[u]]$. By \cite{Ana} and \cite[VI.
2.5]{RaynaudLNM119}, the quotient $Q/P'$ is quasi-projective over
$k[[u]]$.
Suppose that $A$ is an Artin local
$k$-algebra. Then $A((u))$ is a local ring and so each morphism
$\Spec(A((u)))\to Q/P'$ factors through an open affine subscheme of
$Q/P'$. Using this together with the argument of
\cite[Appendix]{GaitsgoryInv}, we can see that the fibered product
$\Spec(A)\times_{{\rm Gr}_Q}{\rm Gr}_{P'}$ is represented by the
closed subscheme of $\Spec(A)$ where the morphism $\Spec(A'((u)))\to
Q/P'$ obtained from a corresponding $\Spec(A((u)))\to Q/P'$ extends
to $\Spec(A'[[u]])\to Q/P'$. In particular, for any such $A$, ${\rm
Gr}_{P'}(A)\to {\rm Gr}_{Q}(A)$ is injective. Now let $ {\rm
Gr}_{P'}=\varinjlim_iY_i$, ${\rm Gr}_{Q}=\varinjlim_j Z_j$, with
$Y_i$, $Z_j$ proper closed subschemes and suppose $Y_i$ maps to
$Z_{j(i)}$. Applying the above, we see that $f_i: Y_i\to Z_{j(i)}$
is quasi-finite;  since $Y_i$ is proper, $f_i$ is also proper and
hence finite by Zariski's main theorem. Since $f_i(A)$ is injective
for all $A$ as above, we see that $f_i$ is a closed immersion.
 We conclude that
$$
{\rm Gr}_{P'}\to  {\rm Gr}_{Q}
$$
is a closed immersion. Now notice that Zariski's main theorem
implies that the special fiber ${\rm M}(G,
\{\mu\})_K\otimes_{\O_E}k$ of ${\rm M}(G, \{\mu\})_K$ is connected:
indeed, the generic fiber ${\rm M}(G, \{\mu\})_K\otimes_{\O_E}E $
over $E$ is geometrically connected and ${\rm M}(G, \{\mu\})_K\to
\Spec(\O_E)$ is proper by construction. Since each connected
component of ${\rm Gr}_{\Gg, k}={\rm Gr}_P$ identifies with a
connected component of ${\rm Gr}_{P'}$ and ${\rm M}(G,
\{\mu\})_K\otimes_{\O_E}k$ is connected we conclude from by above
that the morphism ${\rm M}(G, \{\mu\})_K\otimes_{\O_E}k\to  {\rm
Gr}_{Q}$ is a closed immersion. Therefore the morphism
$\iota\otimes_{\O_E}k$ is also a closed immersion.

We will now show that $\iota$ is a closed immersion. For simplicity,
set $\rm M={\rm M}(G, \{\mu\})_K$. Denote by $\iota(\rm M)$ the
closed scheme theoretic image of $\iota: {\rm M}\to {\rm M}({\rm
GSp}_{2n})_{W_\bullet}\otimes_{\O}\O_E$. (Since
$\iota\otimes_{\O_E}E$ is a closed immersion and $\rm M$ is
integral, $\iota(M)$ coincides with the Zariski closure of ${\rm
M}\otimes_{\O_E}E$ in ${\rm M}({\rm
GSp}_{2n})_{W_\bullet}\otimes_{\O}\O_E$.) We would like to show that
$\rm M=\iota(\rm M)$. Once again, by Zariski's main theorem
$\iota({\rm M})_k=\iota({\rm M})\otimes_{\O_E}k$ is connected. Using
the valuative criterion of properness and the fact that both $\rm M$
and $\iota(\rm M)$ are proper and flat over $\O_E$, we see that
 ${\rm M}_k\to \iota({\rm M})_k$ is surjective. Consider
 $\iota\otimes_{\O_E}k: {\rm M}_k\to \iota({\rm M})_k\hookrightarrow {\rm M}({\rm GSp}_{2n})_{W_\bullet}\otimes_{\O}k$;
this is a closed immersion, and therefore so is ${\rm M}_k\to
\iota({\rm M})_k$. The map ${\rm M}\to \iota({\rm M})$ is proper and
quasi-finite, hence finite. Let $A$ be the local ring of $\iota(\rm
M)$ at a closed point $\iota(x)$ of $\iota({\rm M})_k$ which is the
image of a closed point $x$ of ${\rm M}_k$. Denote by $B$ the local
ring of ${\rm M}$ at $x$, then $A\subset B$. Also since $x$ is the
unique point of $\rm M$ that maps to $\iota(x)$, $B$ is finitely
generated over $A$. Since $\iota\otimes_{\O_E}k$ is a closed
immersion,   $A/\varpi_E A$ surjects onto $B/\varpi_E B$ and hence
$\varpi_E\cdot B/A=(0)$. Applying Nakayama's lemma to the finitely
generated $A$-module $B/A$ we can conclude $A=B$. From this and the
above we deduce $\rm M=\iota({\rm M})$.
 \end{proof}
\smallskip

\subsection{The PEL case}\label{PELremark}

In this paragraph, we elaborate on the local models for Shimura
varieties of PEL type.   We will assume throughout that the prime
$p$ is odd.

We follow \cite[Chapter 6]{RapZinkBook} (see also  \cite{KottJAMS}):
Let ${\boB}$ be a finite dimensional semisimple algebra   over $\Q$
with a positive involution $*$. Then the center $\boF$ of $\boB$ is
a product of CM fields and totally real fields.
  Let $\boV$ be a finite dimensional $\Q$-vector space
of dimension $2n$ with a  perfect alternating $\Q$-bilinear form $(\
,\ ): \boV\times \boV\to \Q$. Assume that $\boV$ is equipped with a
$\boB$-module structure, such that
\begin{equation}\label{*form}
(bv, w)=(y, b^*w), \quad \forall v, w\in \boV, \quad b\in \boB.
\end{equation}
Set ${\boG}\subset {\rm Aut}_{\boB}(\boV)$ to be the closed
algebraic subgroup over $\Q$ such that
$$
{\boG}(\Q)=\{g\in {\rm Aut}_{\boB}(\boV)\ |\ (gv, gw)=c(g)(v, w),
\forall v, w\in \boV, c(g)\in \Q\}.
$$
Let $h:{\boS}:={\rm Res}_{\C/\R}\Gm_{\C}\to {\bold G}_{\R}$ be a
morphism that defines on ${\bold V}_{\R}$ a Hodge structure of type
$(1,0)$, $(0,1)$, such that $(v, h(\sqrt{-1})w)$ is a symmetric
positive bilinear form on $\boV_{\R}$. This gives a $\boB$-invariant
decomposition $\boV_\C=\boV_{0,\C}\oplus \boV_{1, \C}$ where $z\in
{\boS}$ acts on $\boV_{0,\C}$ by multiplication by $\bar z$ and on
$\boV_{1,\C}$ by multiplication by $z$. Then $({\boG}, \{h\})$
defines a Shimura variety of PEL type (cf.
\cite{DeligneTravauxShimura}, \cite{KottJAMS}).\footnote{Note   that
$\boG$ is not always connected and so the set-up  differs slightly
from the previous paragraph where it was assumed that $\boG$ is
connected.} The reflex field $\boE$ is the subfield of $\C$ that is
generated by the traces ${\rm Tr}_{\C}(b|\boV_{0,\C})$ for $b\in
\boB$. Using the isomorphism ${\bold S}_\C\simeq \C^*\times \C^*$,
$z\to (z, \bar z)$, we define $\mu: \Gm_\C\to {\boG}_\C$ as
$\mu(z)=h_\C(z, 1)$; the field $\boE$ is the field of definition of
the ${\boG}$-conjugacy class of $\mu$.  By definition, we have an
embedding $\rho: {\boG}\hookrightarrow {\rm GSp}(\boV, (\ ,\ ))={\rm
GSp}_{2n}$ and $\rho\circ \mu$ is conjugate to the standard minuscule coweight of
${\rm GSp}_{2n}$.

\subsubsection{} Now  set $G^\flat={\boG}_{\Q_p}$ and denote by
$G={\boG}_{\Q_p}^\circ$ the neutral component. Let ${\mathfrak
P}|(p)$ be a prime of $\boE$ and set  $E^\flat={\boE}_{\mathfrak
P}\subset \bar\Q_p$; we see that $\mu$ gives a unique  conjugacy class of coweights
$\Gm_{\bar\Q_p}\to G^\flat_{\bar \Q_p}$ which is defined over
$E^\flat$. Observe that each cocharacter $\mu: \Gm_{\bar\Q_p}\to
G^\flat_{\bar\Q_p}$ lands in the neutral component $G_{\bar\Q_p}$
and we can choose a representative $\mu:  \Gm_{\bar\Q_p}\to G_{\bar
\Q_p}$. Then the corresponding geometric $G_{\bar\Q_p}$-conjugacy
class is defined over $E$ which is a finite extension of $E^\flat$
(and can depend on our choice). We will also denote this conjugacy
class by $\{\mu\}$. As in \cite{RapZinkBook} we suppose that there exists an order $\O_{\bold B}$ of $B$ such
that $\O_B:=\O_{\bold B}\otimes\Z_p$ is a maximal order which is
stable under the involution $*$. Let $\{{\mathcal L}\}$ be a
self-dual multi-chain of $\O_{ B} $-lattices in $V={\bold V}_{\Q_p}$
with respect to $*$ and the alternating form $(\ ,\ )$ (in the sense
of \cite[Chapter 3]{RapZinkBook}). Consider the group scheme
$\P^\flat$ over $\Z_p$ whose $S$-valued points is the group of
$\O_B\otimes\O_S$-isomorphisms of the multi-chain
$\{\L\otimes\O_S\}$ that respect  the forms  up to (common)
similitude in $\O^*_S$. The generic fiber of $\P^\flat$ is
$G^\flat$.  It follows from \cite[Appendix to Ch. 3]{RapZinkBook}  that
$\P^\flat$ is  smooth over $\Z_p$.

  As we discussed in \ref{ss4a}, each such self-dual multichain
$\{{\mathcal L}\}$ gives a point $x(\L)$ in the Bruhat-Tits
building ${\mathcal B}(G, \breve\Q_p)$; the
group $\P^\flat(\breve\Z_p)\cap G(\breve\Q _p)$ is the
stabilizer subgroup of $x(\L)$ in $G(\breve\Q_p) $. By
\cite{BTII}, there is a unique affine smooth group scheme   $\calP'$
over $\Spec(\Z_p)$ with generic fiber $G$ such that
$\calP'(\breve\Z_p )$ is the stabilizer of $x(\L)$ in $G(\breve\Q_p)$. Then there is a group scheme embedding $\P'\hookrightarrow
\P^\flat$ which extends $G=(G^\flat)^\circ \hookrightarrow G^\flat$
and in fact, $\calP'$ is the Zariski closure of $G$ in
$\calP^\flat$. Finally,  the neutral component $\P:=(\P')^\circ$ is
the parahoric group scheme of $G$ associated to $x(\L)$.

\subsubsection{} In this paragraph we describe some constructions from \cite{RapZinkBook}.
The reader is referred to this work for more details.

We will consider $B:={\bold B}\otimes_{\Q}\Q_p$ as a central
$F:={\bold F}\otimes_{\Q}\Q_p$-algebra. For simplicity, we will
assume that the invariants of the involution $*$ on the center $F$
of $B$ are a field $F_0$. (The general case of $(B, *, V, (\ ,\ ))$
can be decomposed into a direct sum of cases with this property.)
There are two cases:

(A) The center of   $B$ is a product $Z=F\times F$, then $B\simeq
M_n(D)\times M_n(D^{\rm opp})$ and $(x, y)^*=(y, x)$. (Here $D$ is a
central division algebra over $F$. Note that $D=D^{\rm opp}$ as
sets.)

(B) The center of $B\simeq M_n(D)$ is a field $F$.
\smallskip

(Case A) We set $\O_B=M_n(\O_D)\times M_n({\O_D}^{\rm opp})$; denote
by $*$ the exchange involution on $\O_B$. Set
$U=\O_D^n\otimes_{\O_D}T$ (a left $M_n(\O_D)$-module) where $T\simeq
\O_D^m$ and $\tilde U={\rm Hom}_{\Z_p}(U,\Z_p)$ which is then also
naturally a left $M_n({\O_D}^{\rm opp})$-module. Set $W=U\oplus
\widetilde U $ which is a   left $\O_B$-module; then $W$ also
supports a unique perfect alternating $\Z_p$-bilinear form $(\ ,\ )$
for which $U$, $\widetilde U$ are isotropic and $((u, 0), (0, \tilde
u))=\tilde u(u)$ for all $u\in U$, $\tilde u\in \widetilde U$. In
this case, the lattice chain $\L_\bullet$ comes about as follows:
Choose an $\O_D$-lattice chain $\Gamma_\bullet$ in
$\O_D^m=\oplus_{i=1}^m\O_De_i$
$$
\Gamma_r=\Pi\cdot \Gamma_0\subset \Gamma_{r-1}\subset \cdots \subset
\Gamma_0=\O_D^m
$$
such that $\Gamma_{j}=\oplus_{j=1}^m \Pi^{a_j} \O_D $ for some
$1\geq a_j\geq 0$. Consider
$$
U_\bullet=\O_D^n\otimes_{\O_D}\Gamma_\bullet.
$$
Then $\L_\bullet=U_\bullet\oplus \widetilde U_\bullet$
for a  unique choice of $\Gamma_\bullet$ as above. 

(Case B) This can be split into three cases (B1), (B2), (B3) which
correspond to (II), (III), (IV) of \cite{RapZinkBook}, p. 135.

(B1) $B=M_n(F)$, $F=F_0$.

(B2) $B=M_n(F)$ and $F/F_0$ is a quadratic extension.

(B3) $B=M_n(D)$ where $D$ is a quaternion algebra over $F$ and
$F=F_0$.

Recall that, as in \cite{RapZinkBook}, we always assume that we have
a maximal order $\O_D$ such that $\O_B=M_n(\O_D)$ is stable under
the involution $*$. Here, we use $D$ to denote either the quaternion
algebra $D$ or $F$ depending on the case we are considering. As in
\cite[Appendix to Ch. 3]{RapZinkBook}, we see that there is a certain perfect
form $H:\O_D^n\times\O_D^n\to \O_D$ (the possible types are:
symmetric, alternating, hermitian, anti-hermitian, quaternionic
hermitian or antihermitian for the main or the new involution on
$D$)  on the (right) $\O_D$-module $U=\O_D^n$ such that the
involution $*$ satisfies
\begin{equation}
H(Au, v)=H(u, A^*v), \quad  u, v\in U , \ A\in M_n(\O_D).
\end{equation}
Here we identify $M_n(\O_D)$ with the right-$\O_D$-module
endomorphisms ${\rm End}_{-\O_D}(U)$ of $U$. We will denote the involution
of $\O_D$ that we are using by $d\mapsto \breve d$. Then $d\mapsto
\breve{d}$ can be trivial, conjugation, the main involution or the
new involution in the case of quaternion algebras. For simplicity,
we will also refer to all the possible types of forms as
$\epsilon$-hermitian for the involution $d\mapsto \breve d$. If this
involution is trivial, ``$1$-hermitian" means symmetric and
``$(-1)$-hermitian" means alternating.

Let $\vartheta_F$ be a generator of the different of $F/\Q$ such
that $\bar\vartheta_F=-\vartheta_F$ if $F\neq F_0$. Define $h:
\O_D\times\O_D\to \Z_p$ by $h={\rm Tr}_{F/\Q_p}(\vartheta^{-1}_FH)$
if $D=F$, and by $h={\rm Tr}_{F/\Q_p}(\vartheta^{-1}_F{\rm
Tr}^0(\Pi^{-1}H))$ if $D$ is quaternion, where ${\rm Tr}^0: D\to F$
is the reduced trace.

As explained in \cite[Appendix to Ch. 3]{RapZinkBook}, we can now employ
Morita equivalence and write $V=M_n(D)\otimes_DW$ with $W\simeq D^m$
a free left $D$-module. The alternating form $(\ ,\ )=\calE(\ ,\ )$
on $V$ can be written
\begin{equation}
\calE(u_1\otimes w_1, u_2\otimes w_2)=h(u_1, u_2\Psi(w_1, w_2))
\end{equation}
where $\Psi: W\times W\to D$ is an $\epsilon$-hermitian form on $W$.
The sign $\epsilon$ of $\Psi$ is the opposite of that of $H$. We can
also write
\begin{equation}
\L_i=M_n(\O_D)\otimes_{\O_D} N_i
\end{equation}
where $N_i$ are $\O_D$-lattices in $W$. The perfect forms $\calE:
\L_{i}\times\L_j\to \Z_p$ induce $\Psi: N_i\times N_j\to \O_D$ such
that
\begin{equation}
\calE(u_1\otimes n_1, u_2\otimes n_2)=h(u_1, u_2\Psi(n_1, n_2)).
\end{equation}
Then $\{N_i\}$ give a polarized chain of $\O_D$-lattices in $W$ for
the form $\Psi$ and the point here is that there is a uniquely
determined  polarized lattice chain $N_\bullet$ that produces the
polarized chain $\L_\bullet$ as above.

\subsubsection{} To proceed we assume that the group $G$ splits over a
tamely ramified extension of $\Q_p$. In particular, the prime $p$ is
at most tamely ramified in the center $F$ of $B$. Now let us explain
how we can extend the above construction over the base $\Z_p[u]$.

(Case A) We set $\O(\calB)=M_n(\O(\calD))\times M_n(\O(\calD)^{\rm
opp})$ with $\O(\calD)$ over $W[v]$ as given in (\ref{order});
denote again by $*$ the exchange involution on $\O(\calB)$. Recall
we view $W[v]$ as a $\Z_p[u]$-algebra via $u\mapsto v^e\cdot (p\cdot
\varpi^{-e})$. Set $\und U = \O(\calD)^n\otimes_{\O(\calD)}\und T$
(a left $M_n(\O(\calD))$-module) where $\und T\simeq \O(\calD)^m$
and   $ \widetilde {\und U} ={\rm Hom}_{\Z_p[u]}(\und U,\Z_p[u])$
which is then also naturally a  left $M_n(\O(\calD)^{\rm
opp})$-module. Set $\und W=\und U\oplus  \widetilde {\und U} $ which
is a   left $\O(\calB)$-module; then $\und W$ also supports a unique
perfect alternating $\Z_p[u]$-bilinear form $(\ ,\ )$ for which
$\und U$, $\widetilde {\und U}$ are isotropic and $((u, 0), (0,
\tilde u))=\tilde u(u)$ for all $u\in \und U$, $\tilde u\in
\widetilde {\und U}$. Now choose an $\O(\calD)$-lattice chain
$\und\Gamma_\bullet$  in $\O(\calD)^m=\oplus_{i=1}^m\O(\cal D)e_i$
$$
\und\Gamma_r=X\cdot \und \Gamma_0\subset \und\Gamma_{r-1}\subset
\cdots \subset \und \Gamma_0=\O(\calD)^m
$$
such that $\und\Gamma_{j}=\oplus_{j=1}^m X^{a_j} \O(\calD) $ for
some $1\geq a_j\geq 0$ which lifts the corresponding lattice chain
$\Gamma_j$. Consider
$$
\und U_\bullet=\O(\calD)^n\otimes_{\O(\calD)}\Gamma_\bullet.
$$
Then $\und W_\bullet=\und {U}_\bullet\oplus \widetilde {\und
{U}}_\bullet$ is a self-dual (polarized) chain
 (as in \S \ref{lattice} (case 2)) that lifts $\{\L_i\}_i$. Consider
$$
\Gg(R)={\rm Aut}_{\O(\calB)\otimes_{\Z_p[u]}R}(\{ \und
W_\bullet\otimes_{\Z_p[u]}R\}, (\ ,\ ))
$$
where the automorphisms of the chain are supposed to preserve the
form $( , )$ up to common similitude in $R^*$. As in \S
\ref{sss4b3}, \S \ref{sss4b11}, we can see that $\Gg$ is one of the
group schemes constructed in Theorem \ref{grpschemeThm}. By its
construction above,  the group scheme $\Gg$ is a closed subgroup
scheme of ${\rm GSp}(\und W_\bullet)$.
\smallskip

(Case B) Recall from \ref{sss4c4} our notation of $\frakO$ which
could be $W[v]$, $\frakR$ or $\O(\calQ)$. The algebra $\frakO$
specializes to $\O_D$ above under $v\mapsto \varpi_0$ (note here
that in general the ``base field" is $F_0$). Notice that the
involution $d\mapsto \breve d$ of $\O_D$ has a canonical extension
to an involution of $\frakO$; we will denote this involution by the
same symbol.
 Now to start our construction, first extend
the form $H$ to a perfect form $\und H: \frakO^n\times\frakO^n\to
\frakO$ of the type described in \S \ref{exClassical} such that
under $W[v]\to \O_{F_0}$, $v\mapsto \varpi_0$, we obtain a form
isomorphic to $H$. For example, suppose we are in case (B1). Then
depending on the type of the involution $*$, the form $H$ is either
symmetric or alternating. In the alternating case, all perfect forms
on $\O_F^n$ are isomorphic to the standard form and we can extend
this over $W[v]$ as in \S \ref{exAlt}. In the symmetric case, the
discriminant of the perfect form is in
$W^*/(W^*)^2\simeq\O_F^*/(\O_F^*)^2$ and the Hasse invariant is
trivial; this implies that there are only two cases to consider: the
split case in which $H$ is isomorphic to the standard form and the
quasi-split unramified case; these can be lifted as in the first
cases of \S \ref{exSymm}. The other cases can be dealt with by similar
arguments using \S \ref{exhaustive} and following the pattern
explained in \cite[Appendix to Ch. 3]{RapZinkBook}. We will leave the details
to the reader. Notice that since we are assuming that the form $H$
is perfect on $\O_D^n$ there are fewer cases to consider. On the
other hand, we now also have to allow the anti-hermitian forms and
also the quaternionic $\epsilon$-hermitian for the new involution as
in \S \ref{exVariants}.

Define now a perfect form $\und h: \frakO^n\times\frakO^n\to
\Z_p[u]$ as follows: We set $\und h={\rm
Tr}_{\frakO/\Z_p[u]}(v^{1-e}\cdot \und H)$ if $\frakO=W[v]$ or
$\frakO=\frakR=W_2[v]$ (unramified case), $\und h={\rm
Tr}_{\frakO/\Z_p[u]}(v'^{1-2e}\cdot \und H)$ if
$\frakO=\frakR=W[v']$ (ramified case), and finally $\und h={\rm
Tr}_{W[v]/\Z_p[u]}(v^{1-e}\cdot {\rm Tr}^0(X^{-1}\und H))$ if
$\frakO=\O(\calQ)$.

This construction allows us to extend the involution $*$ on
$\O_B=M_n(\O_D)$ to an involution on ${\O(\calB)}=M_n(\frakO)$ which
we will also denote by $*$. Indeed, we can set
\[
A^*=\und C\cdot {}^t\breve A\cdot \und C^{-1}
\]
where $\und C$ is the matrix in ${\rm GL}_n(\frakO)$ that gives the
perfect form $\und H$. This  satisfies $C=\ep\cdot {}^t\breve C$.

Using \S \ref{exhaustive} together with \S \ref{exVariants}, we
extend $\Psi$ above to an $\epsilon$-hermitian form $\und\Psi$ on
$\calU=\frakO[v^{\pm 1}]^m$ for the involution $d\mapsto \breve d $
of the type described in \S \ref{exClassical}. As before, the form
$\und\Psi$ has the opposite parity of that of $\und H$ (i.e if one
is $\epsilon$-hermitian the other is $(-\epsilon)$-hermitian).

The self-dual $\O_D$-chain $N_\bullet$ gives a point $x$ in the
building of the corresponding group over $\Q_p$. We can  assume that
$x$
  belongs to the apartment of the standard maximal split torus; then we can give a
  self-dual $\frakO$-lattice chain $\und N_\bullet$ in $\frakO[v^{\pm 1}]^m$
that extends the self-dual $\O_D$-chain $N_\bullet$. This can be
done explicitly on a case-by-case basis by appealing to the list of
cases in \cite[Appendix to Ch. 3]{RapZinkBook}. Alternatively, we can argue
as follows: As in \S \ref{sss3a3}, \S \ref{ss3c},
the identification of apartments induced by $u\mapsto p$, gives a
corresponding point $x_{\Q_p((u))}$ in the building of $\und
G_{\Q_p((u))}$ over $\Q_p((u))$; this  corresponds to a self-dual
$\Q_p[[u]]$-lattice chain. Now we can use the construction in Remark
\ref{genericLattice}   to obtain the desired $\frakO$-lattice chain
over $\Z_p[u]$.

Now consider the tensor products
$
\und M_i=\frakO^n\otimes_{\frakO}\und N_i.
$
These are free (left) $M_n(\frakO)$-modules; they are all contained
in $\und M_i[v^{-1}]=(\frakO^n\otimes_{\frakO}\frakO^m)[v^{-1}]$.
Define a $\Z_p[u]$-valued form $\und \calE$ on
$(\frakO^n\otimes_{\frakO}\frakO^m)[v^{-1}]$ by
\begin{equation}
\und \calE(u_1\otimes n_1, u_2\otimes n_2)=\und h(u_1, u_2\und
\Psi(n_1, n_2)).
\end{equation}
We can check that the parity condition above implies that the form
is alternating, i.e $\und\calE(x, y)=- \und\E(y, x)$. Also, the form
$\und\calE$ satisfies
\begin{equation}
\und\calE(b^*m_1, m_2)=\und\calE(m_1, bm_2)
\end{equation}
for $b\in M_n(\frakO)$, $m_1$, $m_2\in \und M_i[v^{-1}]$. (Notice
that the form $\und\Psi$ is uniquely determined by $\und\calE$.) As
a result of this construction, we have given in particular, a
self-dual chain $\{\und M_i\}_i$ of $M_n(\frakO)$-lattices for the
involution $*$ and the alternating form $\und\calE$ that specializes
to the self-dual chain of $M_n(\O_D)$-lattices $\{\calL\}$ of our
initial data.

Now consider  the group scheme $\Gg'$ of automorphisms of the
self-dual $\frakO$-lattice chain $(\und N_i)_i$ such that the
corresponding automorphism of $(\und M_i)_i$ respects the form
$\und\calE$ up to a common similitude as above. By \S \ref{sss4b11}
the connected component $\Gg$ of $\Gg'$ is an example of the group
schemes of Theorem \ref{grpschemeThm}. Our construction provides
a group scheme homomorphism
\begin{equation}
\und\rho: \Gg'\xrightarrow {\ }{\rm GSp}(\und M_\bullet).
\end{equation}
As in \cite[Appendix to Ch. 3]{RapZinkBook}, using Morita equivalence we can
see that this is a closed immersion. We can also see that $\und\rho$
extends the natural symplectic representation $\rho: \boG\to {\rm
GSp}$ obtained from the PEL data. It follows that the homomorphism
$\und\rho: \Gg\to {\rm GSp}(\und M_\bullet)$ satisfies the
assumptions of the previous section. (Here, we indeed have to allow
that $\Gg$ may not be closed in ${\rm GSp}(\und M_\bullet)$. An
example is when $G$ is a ramified unitary similitude group on an
even number of variables and $x$ a vertex corresponding to a single
self-dual lattice for the corresponding hermitian form. Then the
fiber over $u=0$ of the Zariski closure  of $\und G$ in ${\rm
GSp}(\und M_\bullet)$ has two connected components, see
\cite[1.3]{PappasRaIII}.)

Using the above together with Proposition \ref{embeddLoc}, we now
see that in the case of PEL Shimura varieties of \cite{RapZinkBook}
we can identify ${\rm M}(G, \{\mu\})_K=M_{\Gg, \mu}$ with the
Zariski closure of $G/P_{\mu}$ in the symplectic local model ${\rm
M}({\rm GSp}_{2n})_{M_\bullet}\otimes_{\O}\O_E$ of
\cite{GortzSymplectic} under the standard symplectic representation
$\rho: {\bold G}\to {\mathrm {GSp}}_{2n}$.\footnote{We emphasize
here that, in general, $\boG$ is not connected and that
$G={\boG}_{\Q_p}^\circ$.}

\subsubsection{} \label{sss8c4}
We will now explain how the work of Rapoport and Zink
(\cite{RapZinkBook}) combined with the above can be used to produce
integral models of PEL Shimura varieties that afford a diagram as in
(\ref{locmoddiagram}). Then  as a result of Theorem \ref{thm01},
these models satisfy favorable properties, c.f. Theorem
\ref{thmPEL}. Our explanation becomes more complicated when the
group ${\bold G}_{\Q_p}$ is not  connected, one reason being that
our theory of local models has been set up only for connected
groups. At first glance, the reader can assume that ${\bold
G}_{\Q_p}$ is connected; then everything simplifies considerably.

We will continue to use some of the notations and constructions of
\cite{RapZinkBook}. Starting from the PEL data ${\mathfrak
D}=({\boB}, \O_{\boB}, $*$, {\boV}, (\ ,\ ), h, \{\L \}, K^p)$ with
corresponding group $\boG$ and the choice of a prime $\frakP|(p)$ of
the reflex field $\boE$, Rapoport and Zink define a moduli functor
$\calA_{K^p}$ over $\O_{E^\flat}=\O_{{\boE}_{\frakP}}$ (see
\cite[Definition 6.9]{RapZinkBook}, here $E^\flat={\boE}_{\frakP}$
is the local reflex field). Here $K^p$ is a compact open subgroup of
${\bold G}({\mathbb A}^p_f)$. When $K^p$ is small enough, this
functor is representable by a quasi-projective scheme $\calA_{K^p}$
over $\Spec(\O_{E^\flat})$.

Recall $G^\flat={\boG}_{\Q_p}$, $G={\boG}^\circ_{\Q_p}$ (a connected
reductive group); as usual, we assume these split over a tamely
ramified extension of $\Q_p$. The Shimura data give a conjugacy
class of cocharacters $\mu: \Gm_{\bar\Q_p}\to G^\flat_{\bar\Q_p}$;
then $E^\flat$ is the field of definition  of this conjugacy class.
Denote by $K^\flat_p$ the stabilizer of the lattice chain $\{\L\}$
in $G^\flat(\Q_p)={\bold G}(\Q_p)$. Set  ${\bold K}^\flat=K^p\cdot
K^\flat_p$.
 Then the generic fiber $\calA_{K^p}\otimes_{\O_{E^\flat}}E^\flat$ contains
the Shimura variety   $Sh_{\bold K^\flat}\otimes_{\boE}E^\flat$ for
${\bold G}$, as a union of some of its connected components;
$\calA_{K^p}\otimes_{\O_{E^\flat}}E^\flat$ could also contain more
Shimura varieties, which  correspond to other forms of the group
${\bold G}$ (for example, when the Hasse principle fails c.f.
\cite{KottJAMS}). Set $K'_p=K^\flat_p\cap G(\Q_p)$ and denote by
$K=K_p$ the parahoric subgroup of $G(\Q_p)$ that corresponds to $x(
\L )$. Recall we denote by
 $\calP$  the corresponding (connected) smooth group scheme over $\Z_p$ and by $\calP'$ the  smooth group
 scheme over $\Z_p$ determined by the stabilizer of $x(\L)$ so that $K'_p=\calP'(\Z_p)$;
$\calP$ is the neutral component of $\calP'$. Then $K=K_p$ is a
normal subgroup of finite index in $K'_p$. (In most cases, we have
$K'_p=K_p$, $\calP'=\calP$.)  Also recall that we denote by
$\calP^\flat$ the smooth group scheme of $\O_B$-isomorphisms of the
polarized multichain $ \{\L \}$ up to common similitude; we have
$\calP^\flat\otimes_{\Z_p}\Q_p=G^\flat$.

By \cite{RapZinkBook}, we have a smooth morphism of algebraic stacks
over $\Spec(\O_{E^\flat})$
\begin{equation}\label{phiRZ}
\varphi:  \calA_{K^p}\to [{\rm M}^{\rm naive}/\mathcal
P^\flat_{\O_{E^\flat}}]\,
\end{equation}
where ${\rm M}^{\rm naive}$ is the ``naive" local model that
corresponds to our data (see loc. cit. Def. 3.27, where this is
denoted by ${\rm M}^{\rm loc}$). By its definition, ${\rm M}^{\rm
naive}$ is a closed subscheme of the symplectic local model ${\rm
M}({\rm GSp}_{2n})_{M_\bullet}\otimes_{\Z_p}\O_{E^\flat}$ as above.

The generic fiber  of ${\rm M}^{\rm naive}$ is a projective
homogeneous space over $E^\flat$ for the group $G^\flat$ (which is
not always connected). Observe that each cocharacter $\mu:
\Gm_{\bar\Q_p}\to  G^\flat_{\bar\Q_p}$ lands in the neutral
component $G_{\bar\Q_p}$. Let $\{\mu_i\}_{i}$ be a set of
representatives (up to $G(\bar\Q_p)$-conjugation) of the
cocharacters of $G_{\bar\Q_p}$ in the $G^\flat(\bar\Q_p)$-conjugacy
class of $\mu$; we can write $\mu_i=\tau_i\mu\tau_i^{-1}$ with
$\tau_i\in G^\flat(\bar\Q_p)$. Denote by $E_i$ the   field of
definition of the $G(\bar\Q_p)$-conjugacy class of $\mu_i$. Then
$E^\flat\subset E_i$. The generic fiber ${\rm M}(G,
\{\mu_i\})_K\otimes_{\O_{E_i}}E_i=G_{E_i}/P_{\mu_i}$ is a
homogeneous space for $G_{E_i}$. By the above and Proposition
\ref{embeddLoc}, we obtain   closed immersions
$$
\iota_i: {\rm M}(G, \{\mu_i\})_K\to {\rm M}^{\rm
naive}\otimes_{\O_{E^\flat}}\O_{E_i}
$$
of schemes over $\O_{E_i}$ which are equivariant for the action of
the group scheme $\calP_{\O_{E_i}}$. In fact, since
$\calP'_{\O_{E_i}}$ is a closed subgroup scheme of ${\rm
GSp}(M_\bullet)$, the action of $\calP_{\O_{E_i}}$ on ${\rm M}(G,
\{\mu_i\})_K$ extends to an action of $\calP'_{\O_{E_i}}$ such that
$\iota_i$ remains equivariant. Now let $\ti E$ be a Galois extension
of $\Q_p$ that splits $G$ and contains all the fields $E_i$. Then
the base change ${\rm M}^{\rm naive}\otimes_{\O_{E^\flat}}\ti E$ is
a disjoint union
$$
{\rm M}^{\rm naive}\otimes_{\O_{E^\flat}}\ti E=\bigsqcup_i ({\rm
M}(G, \{\mu_i\})_{K}\otimes_{\O_{E_i}}\ti E)=\bigsqcup_i
G_{E_i}/P_{\mu_i}\otimes_{E_i}\ti E.
$$
  We obtain a morphism
 $$
 \iota: N:= \bigsqcup_i {\rm M}(G, \{\mu_i\})_K\otimes_{\O_{E_i}}\O_{\ti E}\to {\rm M}^{\rm naive}\otimes_{\O_{E^\flat}}\O_{\ti E}
 $$
 of schemes over $\O_{\ti E}$. The scheme $N$ supports an action of
 $\P'_{\O_{\ti E}}$ while its generic fiber supports a compatible action of
 $G^\flat_{\ti E}$; we can see that the action of $\P'_{\O_{\ti E}}$ on $N$ extends to an action of $\P^\flat_{\O_{\ti E}}$
 such that $\iota$ is $\P^\flat_{\O_{\ti E}}$-equivariant. Both the source and target of $\iota$ (considered as schemes over $\O_{E^\flat}$) support an action of
 the Galois group $\Gamma={\rm Gal}(\ti E/E^\flat)$ and $\iota$ is $\Gamma$-equivariant (we can check these statements by looking at the generic fibers).

  In this more general situation, we define the local model to be the quotient
 \[
 {\rm M}^{\rm loc}=N/\Gamma.
 \]
  This is a flat  $\O_{E^\flat}$-scheme. The morphism $\iota$ gives
 $$
 \xi: {\rm M}^{\rm loc}=N/\Gamma\xrightarrow{ \ } ({\rm M}^{\rm naive}\otimes_{\O_{E^\flat}}\O_{\ti E})/\Gamma=
 {\rm M}^{\rm naive}
 $$
which we can check is $\P^\flat_{\O_{E^{\flat}}}$-equivariant and
gives an isomorphism between generic fibers.

\begin{Remark}
{\rm If ${\bold G}$ is connected, then $G=G^\flat$, there is only
one $i$ and $E=E_i=E^\flat$ is the local reflex field as above. In
that case, ${\rm M}^{\rm loc}={\rm M}(G, \{\mu\})_K$.}
\end{Remark}

Pulling back the smooth $\phi$ (\ref{phiRZ}) along $\xi$ gives a
cartesian diagram
\begin{equation}\label{pullback}
 \xymatrix{
      \calA'_{K^p} \ar[r] \ar[d] & [{\rm M}^{\rm loc} /\calP^\flat_{\mathcal O_{E^\flat}}] \ar[d] \\
\calA_{K^p} \ar[r] &
          [{\rm M}^{\rm naive}  /\cal\P^\flat_{\mathcal O_{E^\flat}}]
}
\end{equation}
 with both horizontal arrows smooth.

The schemes $ \calA'_{K^p}$ and $\calA_{K^p}$ have isomorphic
generic fibers. Using the arguments in  \cite[\S 8]{KottJAMS}
applied to $\calA'_{K^p}\otimes_{\O_{E^\flat}}E^\flat=
\calA_{K^p}\otimes_{\O_{E^\flat}}E^\flat$ we can now express this
scheme as a union of Shimura varieties $Sh^{(j)}_{\bold K^\flat}$
given by various forms ${\bold G}^{(j)}$ of the group $\bold
G$.\footnote{For this, we  need to allow in our formalism Shimura varieties
for some non-connected reductive groups.} (These forms satisfy ${\bold
G}^{(j)}({\Q_v})\simeq {\bold G}({\Q_v})$, for all places $v\neq
p$). Therefore, $\calA'_{K^p}$ is a flat  model over $\O_{E^\flat}$
for such a union of Shimura varieties.

 Since we assume that $p$ is odd, we have
$p\nmid|\pi_1(G_{\on{der}})|$. Theorem \ref{CMfiber} and Proposition \ref{normalgen} below
imply that the base changes ${\rm M}(G, \{\mu_i\})_K\otimes_{\O_{E_i}}\O_{\ti E}$
are normal. Since normality is preserved by taking invariants by a finite group, ${\rm M}^{\rm loc}=N/\Gamma$
is also normal. It follows that the strict henselizations of ${\rm M}^{\rm loc}$ at all closed points
are normal and therefore also integral. We can now conclude, using the smoothness of
$\phi:  \calA'_{K^p} \rightarrow [{\rm M}^{\rm loc}
/\calP^\flat_{\mathcal O_{E^\flat}}] $,  that the same is true for the strict henselizations of
$\calA'_{K^p}$. It follows that the
Zariski
closures of $Sh^{(j)}_{\bold K^\flat}$ in
$\calA'_{K^p}$ do not intersect; denote these (reduced) Zariski closures by $\calS^{(j)}_{\bold
K^\flat}$. They are integral models
of the Shimura varieties $Sh^{(j)}_{\bold K^\flat}$. Then, for each $j$, the morphism
\begin{equation}\label{calSj}
\phi:  \calS^{(j)}_{\bold
K^\flat} \rightarrow [{\rm M}^{\rm loc} /\calP^\flat_{\mathcal
O_{E^\flat}}],
\end{equation}
obtained by restricting $\phi$, is also smooth.
Hence, ${\rm M}^{\rm loc}$ is also a ``local model" for the integral models
$\calS^{(j)}_{\bold K^\flat}$ of the Shimura varieties $Sh^{(j)}_{\bold K^\flat}$.
In most cases, the structure of ${\rm M}^{\rm loc}$ and therefore also the local structure
of $\calS^{(j)}_{\bold K^\flat}$ can be understood using Theorem
\ref{thm01}. We will explain this below.

\begin{Remark}\label{sss8d5}
{\rm We can see as in \cite{KottJAMS} that, in the situation above, all the forms ${\bold G}^{(j)}$ that appear in the list
  satisfy ${\bold G}^{(j)}(\breve\Q_p)\simeq {\bold G}({\breve\Q_p})$.

(a) Suppose now that, in addition, at least one
of the following two conditions is satisfied:
\begin{itemize}
\item[(p1)]  the special fiber of $\P^\flat$ is connected, i.e $K^\flat_p$ is parahoric,
or
\item[(p2)]  the sheaf of connected components  $\P^\flat/(\P^\flat)^\circ$
is \'etale locally constant and   its generic fiber is isomorphic
(via the natural map) to $G^\flat/(G^\flat)^\circ$.
\end{itemize}
Then, the arguments of \cite[ Lemma 7.2 and \S 8]{KottJAMS} extend
to show that all the groups ${\bold G}^{(j)}$   also satisfy ${\bold G}^{(j)}(\Q_p)\simeq {\bold G}(\Q_p)$.
(Indeed, then one can apply Lang's theorem as in the proof of Lemma
7.2 loc. cit.  since the special fiber of $(\P^\flat)^\circ$ is
connected.)

(b) Assume that, in addition to (p1) or (p2), $\bold G$ satisfies the Hasse principle, i.e
$\rH^1(\Q, {\bold G})\to \prod_v\rH^1(\Q_v, {\bold G})$ is
injective. Then the argument of \cite[\S 8]{KottJAMS} shows that
there is only one Shimura variety for the group $\bold G^{(j)}=\bold
G$ that appears in the generic fiber of $\calA_{K_p}$. In this case, we just
take $\calS_{\bold K^\flat}=\calA'_{K^p}$.
}
\end{Remark}

\subsubsection{}\label{rem8c4} Here we explain how we can  combine
the local model diagram (\ref{calSj}) with Theorem \ref{thm01} to give results
on   integral models of PEL Shimura varieties. In particular, we show Theorem
\ref{thmPEL}.

 a) The set-up simplifies drastically if $G^\flat$ is
connected, i.e if $G=G^\flat$. This is the case when ${\bold G}$
does not have any orthogonal factors. Then ${\rm M}^{\rm loc}={\rm
M}(G, \{\mu\})_K$ ($\Gamma=\{1\}$) and $\xi=\iota$ is a closed immersion. Also, in
this case the derived group $G_{\rm der}$ is simply-connected.  We
can then use Theorem \ref{thm01} and the above directly   to obtain
information on the integral models
 $\calA'_{K^p}$ and ${\mathcal S}^{(j)}_{\bold K^\flat}$. Assume in addition that $K^\flat_p$ is
parahoric, i.e that the special fiber of $\P^\flat=\P'$ is
connected. Then the existence of the local model diagram morphism $\phi$ of (\ref{calSj})
together with the results in \S \ref{ss8d}, and in particular Theorem \ref{thm01}, implies, in a standard way, that the  integral models $\calA'_{K^p}$ and ${\mathcal S}^{(j)}_{\bold K^\flat}$ satisfy the conclusions of the statement of Theorem \ref{thmPEL}.
 (Notice here that in the statement of Theorem \ref{thmPEL}
 we refer somewhat ambiguously to the ``Shimura variety $Sh_{\bold K}$
defined by the PEL data $\mathfrak D$". If the Hasse principle is not satisfied,
$Sh_{\bold K}$ is by definition the generic fiber of $\calA'_{K^p}$ which is
a union of several Shimura varieties; by the above, ${\mathcal S}^{(j)}_{\bold K^\flat}$
give integral models as in
 Theorem \ref{thmPEL} for each one of them.)

 b) Continue to assume that $G=G^\flat$. In general, if $\calP'=\calP^\flat$ is not connected,   $K'_p/K_p$ is a non-trivial abelian group.
 The above allows us to understand the
 structure of integral models of Shimura varieties with $K'_p$-level structure.
With some extra work, one should  be able to extend this and also produce
local model diagrams for Shimura
varieties with $K_p$-level structure. We will not
discuss this in this paper.  See \cite[1.3]{PappasRaIII} for an example where a case of
$K_p$-level structure for the ramified unitary group is discussed.

c)  Let us now consider the general case, i.e allow $G\neq G^\flat$.
First of all, let us remark that the scheme ${\rm M}^{\rm loc}$ defined above coincides in
various special  cases with the ``corrected" local model as
considered in \cite{PappasRaI}, \cite{PappasRaII},
\cite{PappasRaIII}, \cite{SmithOrth1}, \cite{PRS}. Indeed, first
assume that $G=G^\flat$, $E=E^\flat$ as above. Then, for both our
definition above and the definition in these references, ${\rm
M}^{\rm loc}$ is simply the Zariski closure of $G_E/P_\mu$ in ${\rm
M}^{\rm naive}$. Now let us allow $G\neq G^\flat$;  such cases have not been
studied extensively before and one needs to work a bit harder.   Some (even) orthogonal similitude group cases where
$G\neq G^\flat$ have been considered in \cite{PappasRaIII},
\cite{SmithOrth1} (see \cite{PRS}). In the split orthogonal case,
$E=E^\flat=\Q_p$, and it makes sense to define the corrected local
model as the Zariski closure of $G^\flat/P_\mu=(G/P_{\mu_1})\sqcup
(G/P_{\mu_2})$ in ${\rm M}^{\rm naive}$; it follows from \cite[\S
8.2]{PappasRaIII} that this Zariski closure is the disjoint union of
the closures of $G/P_{\mu_1}$ and $G/P_{\mu_2}$. This coincides with
the description of ${\rm M}^{\rm loc}$ above and we have ${\rm
M}^{\rm loc}={\rm M}(G, \{\mu_1\})_K\sqcup {\rm M}(G, \{\mu_2\})_K$.
  This local model
was studied in the Iwahori case by Smithling in \cite{SmithOrth1}.
In general, the suggested correction of the naive local model in this orthogonal case
\cite{PappasRaIII} involves the so-called ``spin condition" whose
main virtue is that it also attempts to  give a moduli theoretic
interpretation. (The spin condition also makes sense  in the
quasi-split even orthogonal case (\cite[\S 8.2]{PappasRaIII},
\cite[\S 2.7]{PRS}). We conjecture that the corresponding local
model always agrees with our definition above.)

As an example of a result we can obtain when $G\neq G^\flat$,
let us suppose that $G^\flat $ is a split even orthogonal similitude
group as above and that $\L$ gives a maximal self-dual lattice chain
(the Iwahori case considered in \cite{SmithOrth1}). Then the special
fiber of $\P^\flat$ is connected.
By (\ref{calSj}) and the above, the local model  ${\rm
M}^{\rm loc}={\rm M}(G, \{\mu_1\})_K\sqcup {\rm M}(G, \{\mu_2\})_K$ describes the singularities of
corresponding integral models of  ``orthogonal" PEL Shimura varieties.
Theorem \ref{thm01} then also implies that these   integral models
are normal and have reduced special
fibers with geometric components which are normal and
Cohen-Macaulay. In the general case that $G\neq G^\flat$, one would need to study the quotient $N/\Gamma={\rm M}^{\rm loc}$
but we will leave this for another occasion.

d) For general Shimura varieties of abelian type, the existence of a
suitable integral model and a local model diagram as above is the
subject of joint work in progress by the first named
author and M. Kisin  \cite{K-P}.

\medskip

\section{The special fibers of the local models}\label{ss8d}
Here we show our main results on the structure of the local models,
including Theorem \ref{thm01} of the introduction.

\subsection{Affine Schubert varieties and the $\mu$-admissible set}\label{8d1}
Let us review some aspects of the theory of Schubert varieties in the (generalized)
affine flag varieties, following
\cite{FaltingsLoops,PappasRaTwisted}.
In this section, we assume
that $F'=k((u))$, $\O'=k[[u]]$, with $k$ algebraically closed, and let $G'$ be a
connected reductive group over $F'$, split over a tamely ramified
extension $\tilde{F}'$ of $F'$.

\subsubsection{}   Let $S'$
be a maximal $F'$-split torus of $G'$ and $T'=Z_{G'}(S')$ a maximal
torus. Let $I=\Gal(\tilde{F}'/F')$. Let $x\in\calB(G',F')$ be a
point in the building, which we assume lies in the apartment
$\A(G',S', F')$ and let $\calP'_{x}$ be the corresponding parahoric
group scheme over $k[[u]]$. Let $\Gr_{\calP'_{x}}$ be the  affine
Grassmannian. Then $L^+\calP'_{x}$ acts on $\Gr_{\calP'_{x}}$ via
left multiplication. The orbits then are parameterized by
$W^{\calP'_{x}}\setminus\widetilde{W}'/W^{\calP'_{x}}$, where
$\widetilde W'$ is the Iwahori-Weyl group of $G'$ and
$W^{\calP'_{x}}$ is the Weyl group of $\calP'_{x}\otimes k$. For
$w\in W^{\calP'_{x}}\setminus\widetilde{W}'/W^{\calP'_{x}}$, let
$S^{\calP'_x}_w\subset\Gr_{\calP'_x}$ denote the corresponding
Schubert variety, i.e. the closure of the $L^+\calP'_x$-orbit
through $w$. Then according to \cite{FaltingsLoops}, \cite[Thm.
8.4]{PappasRaTwisted}, if $\on{char} k \nmid
|\pi_1(G'_{\on{der}})|$, where $G'_{\on{der}}$ is the derived group
of $G'$, then $S^{\calP'_x}_w$ is normal, has rational singularities
and is Frobenius-split if $\on{char} k>0$.

Let us also recall the structure of the Iwahori-Weyl group
$\widetilde W'$ of $G'$. It is defined via the exact sequence
\begin{equation}\label{IwahoriWeyl}
1\to T'(\O')\to N'(F')\to \widetilde W'\to 1
\end{equation}
(where $T'(\O')$ is the group of $\O'$-valued points
of the unique Iwahori group scheme for the torus $T'$) and acts on
$\A(G',S', F')$ via affine transformations. In addition, there is a
short exact sequence
\begin{equation}\label{Iwahori-Weyl}
1\to \xcoch(T')_{I}\to\widetilde{W}'\to {W'_0}\to 1,
\end{equation}
where ${W'_0}=N'(F')/T'(F')$ is the relative Weyl group of $G'$ over
$F'$. In what follows, we use $t_\la$ to denote the translation
element in $\widetilde{W}'$ given by $\la\in\xcoch(T')_I$ from the
above map \eqref{Iwahori-Weyl}\footnote{\label{fn}Note that under
the sign convention of the Kottwitz homomorphism in
\cite{KottIsocrystalsII}, $t_\la$ acts on $\A(G',S', F')$ by
$v\mapsto v-\la$.}. But occasionally, if no confusion is likely to arise,
we will also use $\la$ itself to
denote this translation element.
A choice of a special vertex $v$ of $\A(G',S',F')$ gives a splitting
of the above exact sequence and then we can write $w=t_\la w_f$ for
$\la\in \xcoch(T')_I$ and $w_f\in {W'_0}$.

Let us choose a rational Borel subgroup $B'$ of $G'$ containing $T'$. This
determines a set of positive roots $\Phi^+=\Phi(G',S')^+$ for $G'$.
There is a natural map $\xcoch(T')_I\to\xcoch(S')_\bbR$. We define
\begin{equation}\label{plus}
\xcoch(T')_I^+=\{\la\mid (\la,a)\geq 0 \mbox{ for } a\in\Phi^+\}.
\end{equation}
Observe that the chosen special vertex $v$ and the rational Borel
$B'$ determine a unique alcove $C$ in $\A(G',S',F')$. Namely, we
identify $\A(G',S',F')$ with $\xcoch(S')_\bbR$ by $v$ and then $C$
is the unique alcove whose closure contains $v$, and is contained in
the finite Weyl chamber determined by $B'$.

Let $W'_{\on{aff}}$ be the affine Weyl group of $G'$, i.e. the
Iwahori-Weyl group of $G'_{\on{sc}}$, the simply-connected cover of
$G'_{\on{der}}$. This is a Coxeter group. One has
\begin{equation}\label{affine Weyl}
1 \to \xcoch(T'_{\on{sc}})_{I}\to\ W'_{\on{aff}}\to {W'_0}\to 1,
\end{equation}
where $T'_{\on{sc}}$ is the inverse image of $T'$ in $G'_{\on{sc}}$.
One can write $\widetilde{W}'=W'_{\on{aff}}\rtimes \Omega'$, where
$\Omega'$ is the subgroup of $\widetilde{W}'$ that fixes the chosen
alcove $C$. This gives $\widetilde{W}'$ a quasi Coxeter group
structure;  it makes sense to talk about the length of an element
$w\in\widetilde{W}'$ and there is a Bruhat order on
$\widetilde{W}'$. Namely, if we write $w_1=w'_1\tau_1,
w_2=w'_2\tau_2$ with $w'_i\in W'_{\on{aff}}, \tau_i\in\Omega'$, then
$\ell(w_i)=\ell(w'_i)$ and $w_1\leq w_2$ if and only if
$\tau_1=\tau_2$ and $w'_1\leq w'_2$.

\subsubsection{}  Now let us recall the definition of the \emph{$\mu$-admissible set}
in the Iwahori-Weyl group (cf. \cite{KoRaMinuscule}, see also
\cite{PRS}). We continue with the above notations. Let $\bar{W}'$ be the absolute Weyl group of $G'$, i.e.
the Weyl group for $(G'_{\tilde{F'}},T'_{\tilde{F'}})$. Let
$\mu:(\bbG_m)_{\tilde{F'}}\to G'\otimes_{F'}\tilde{F'}$ be a
geometrical conjugacy class of 1-parameter subgroups. It determines
a $\bar{W}'$-orbit in $\xcoch(T')$. One can associate to $\mu$ a
${W'_0}$-orbit $\Lambda$ in $\xcoch(T')_{I}$ as follows. Choose a
Borel subgroup of $G'$ containing $T'$, and defined over $F'$. This
gives a unique element in this $\bar{W}'$-orbit, still denoted by
$\mu$, which is dominant with respect to this Borel subgroup. Let
$\bar{\mu}$ be its image in $\xcoch(T')_I$, and let
$\Lambda={W'_0}\bar{\mu}$. It turns out that $\Lambda$ does not
depend on the choice of the rational Borel subgroup of $G'$, since
any two such $F'$-rational Borel subgroups that contain $T'$ are
conjugate to each other by an element in ${W'_0}$. For
$\mu\in\xcoch(T')$, define the admissible set
\begin{equation}\label{Adm}
\Adm(\mu)=\{w\in\widetilde{W}'\mid w\leq t_\la, \mbox{ for some }
\la\in\Lambda\},
\end{equation}
and more generally,
\begin{equation}\label{Adm-parahoric}
\Adm^{\calP'_x}(\mu)=W^{\calP'_x}\Adm(\mu)W^{\calP'_x}.
\end{equation}
If $\on{char} k\nmid |\pi_1(G'_{\on{der}})|$, let us define a
reduced closed subvariety of $\Gr_{\calP'_x}$ whose underlying set
is given by
\[\calA^{\calP'_x}(\mu)=\bigcup_{w\in \Adm(\mu)} L^+\calP'_x w L^+\calP'_x/L^+\calP'_x=\bigcup_{w\in W^{\calP'_x}\setminus \Adm^{\calP'_x}(\mu)/W^{\calP'_x}}S^{\calP'_x}_w.\]
In general, we can define a slightly different variety
$\calA^{\calP'_x}(\mu)^{\circ}$ as in \cite[Sect.
10]{PappasRaTwisted} (see also \cite[Sect. 2.2]{ZhuCoherence}),
which is isomorphic to $\calA^{\calP'_x}(\mu)$ when $\on{char}
k\nmid |\pi_1(G'_{\on{der}})|$. (Here, we write
$|\pi_1(G'_{\on{der}})|$ for the order of the algebraic fundamental
group of $G'_{\on{der}}(\overline {k((u))}^s)$). In fact, when
$\on{char} k\nmid |\pi_1(G'_{\on{der}})|$,
$\calA^{\calP'_x}(\mu)^{\circ}$ is the translation of
$\calA^{\calP'_x}(\mu)$ back to the ``neutral connected component"
of $\Gr_{\calP_x'}$ by a certain element of $G'(k((u)))$. The
variety $\calA^{\calP'_x}(\mu)^\circ$ depends only on $G'_{\on{ad}}$
and the image of $\mu$ under $G'\to G'_{\on{ad}}$. From its
definition (\emph{loc. cit.}), if we have the decomposition
$G'=G'_1\times G'_2$, $\mu=\mu_1+\mu_2$ and
$\calP'_x=(\calP'_x)_1\times(\calP'_x)_2$, then
\[\calA^{\calP'_x}(\mu)^\circ\simeq \calA^{(\calP'_x)_1}(\mu_1)^\circ\times \calA^{(\calP'_x)_2}(\mu_2)^\circ.\]

\subsection{Special fibers}\label{8d2}

Let us return to the local models $M_{\Gg, \mu}$ and show Theorem
\ref{thm01} of the introduction. We will assume throughout that
$p\nmid|\pi_1(G_{\on{der}})|$.

\subsubsection{} We denote by $\overline M_{\Gg, \mu}=M_{\Gg, \mu}\otimes_{\O_E}k_E$
the special fiber of $M_{\Gg,\mu}\to \Spec(\O_E)$ over the residue
field $k_E$ of $\O_E$.

\begin{thm}\label{CMfiber}
Suppose that $p\nmid|\pi_1(G_{\on{der}})|$.   Then the scheme $M_{\Gg,\mu}$ is normal.
In addition, the special fiber
$\overline M_{\Gg, \mu}$ is reduced, and each geometric irreducible
component of $\overline M_{\Gg, \mu}$   is normal, Cohen-Macaulay
and Frobenius split.
\end{thm}

  Notice that the set-up now is more general than in Theorem \ref{thm01},
since here $\mu$ is not necessarily minuscule.
Recall that by its construction, $\overline M_{\Gg, \mu}$ is a
closed subscheme of $\Gr_{P_{k_E}}$. Clearly, it is enough to prove
the theorem after base changing to $\breve{\O}_E$ with residue field
$\bar k$.   Then the second part of Theorem \ref{CMfiber} that refers to the special fiber
$\overline M_{\Gg, \mu}$  is a corollary of the
aforementioned results of \cite{FaltingsLoops, PappasRaTwisted}
combined  with  Theorem \ref{special fiber} below  which
gives a precise description of the geometric special fiber
$\overline M_{\Gg, \mu}\otimes_{k_E}\bar k$
 as a union of affine Schubert varieties.   We will first explain
 how this second part also implies the first part, i.e the normality
 of $M_{\Gg,\mu}$. This follows from Proposition \ref{normalgen} below
 together with the fact that the generic fiber of $M_{\Gg,\mu}$ is normal
 (since it is given by the single affine Schubert variety $X_\mu$ associated to $\mu$).

 \begin{prop}\label{normalgen}
 Suppose that $Y\to \Spec(\O)$ is a flat scheme of finite type
 with normal generic fiber and reduced special fiber. Then
 $Y$ is normal.
 \end{prop}

 \begin{proof} By Serre's criterion, it is enough to check
 that $Y$ satisfies properties (R1) and (S2). By assumption, these properties are satisfied
 by the generic fiber $Y\otimes_\O F$. Since the special fiber $\overline Y=Y\otimes_\O k$
 is a reduced scheme of finite type and $k$ is perfect, $\overline Y$  is generically smooth over $k$.
 It follows that $Y$ is regular in codimension $1$, i.e it satisfies (R1). It remains to show that
 $Y$ has depth $\geq 2$ at points of codimension $\geq 2$ which are supported on the special fiber. Since $Y$ is flat over $\Spec(\O)$,
 the uniformizer $\varpi$ provides the start of a regular sequence;   we can always obtain one additional element in this regular sequence
 since $\overline Y$ is reduced and hence it satisfies (S1).
 \end{proof}

 \subsubsection{} Here we show the second part of Theorem \ref{CMfiber}.
  In what follows, the notations are as in \S
\ref{kappafibers}. In particular, we have $P_{\bar k}=\calP_{x_{\bar
k((u))}}$; this is a parahoric group scheme  which is obtained by
base-changing $\Gg$ to $\bar k[[u]]$. (The corresponding reductive
group is $G'=G_{\bar k}=\Gg\times_X \Spec(\bar k((u)))$.)  Recall
that our construction of $\underline G$ over $\breve {\O}[u,
u^{-1}]$ in \S \ref{reductive group} produces an isomorphism between
the Iwahori-Weyl groups $\widetilde W$ of $G\otimes_F\breve{F}$ and
$\widetilde W'$ of $G'$. The constructions in \S \ref{8d1} above can
also be applied to $G\otimes_F\breve F$ and $\{\mu\}$ to produce a
$W_0$-orbit $\Lambda$ in $\xcoch(T)_I\subset \widetilde W$ and a
subset ${\rm Adm}(\mu)\subset \widetilde W$. Using this
identification, we will view $\Lambda$ and ${\rm Adm}(\mu)$ also as,
respectively, a   $W'_0$-orbit in $\xcoch(T')_I\subset \widetilde W'$, and a
subset ${\rm Adm}(\mu)\subset \widetilde W'$.

\begin{thm}\label{special fiber}
Suppose that $p\nmid|\pi_1(G_{\on{der}})|$. Then we have
\begin{equation*}
 \calA^{P_{\bar k}}(\mu)=\overline M_{\Gg, \mu}\otimes_{k_E}\bar k
 \end{equation*}
as closed subschemes of $\Gr_{P_{\bar k}}$.
\end{thm}

\begin{cor}\label{specialCM}
Suppose that $p\nmid|\pi_1(G_{\on{der}})|$ and that
 $x$ is special in $\calB(G_{\Fsm}, \Fsm)$. Then the
  scheme $M_{\Gg, \mu}$ is  Cohen-Macaulay and normal
  and the special fiber $\overline M_{\Gg, \mu}$ is geometrically
 irreducible and normal.
\end{cor}

\begin{proof}
It is enough to show that the geometric special fiber $\overline
M_{\Gg, \mu}\otimes_{k_E}\bar k$ is irreducible. Indeed, then by
Theorem \ref{CMfiber},  the special fiber $\overline M_{\Gg, \mu}$
is Cohen-Macaulay and we can conclude that
$M_{\Gg, \mu}$ is also Cohen-Macaulay. Normality follows as above.
The irreducibility   of $\overline M_{\Gg,
\mu}\otimes_{k_E}\bar k$ follows from Theorem \ref{special fiber}:
Indeed,
 when $x$ is special, $W^{\calP'_x}=W_0'$ and
$ W^{\calP'_x}\setminus \Adm^{\calP'_x}(\mu)/W^{\calP'_x}$ has only
one extreme element in the Bruhat order, namely the image of $t_\mu$
(cf. \cite[Prop. 6.15]{ZhuCoherence}).
\end{proof}

\begin{Remark}\label{normalCM}
{\rm a) Some special cases of Theorems \ref{CMfiber} and
\ref{special fiber} were known before (e.g   \cite{GortzFlatGLn},
\cite{GortzSymplectic} for $\GL_n$ and ${\rm GSp}_{2n}$).
 See \cite{PRS} for a survey of these previous results.
A description of the components of the special fiber of local models
in terms of the $\mu$-admissible set as in Theorem \ref{special
fiber}  was first suggested by Kottwitz and Rapoport
 in \cite{KoRaMinuscule}.

 b) We conjecture that when $p\nmid|\pi_1(G_{\on{der}})|$, $M_{\Gg, \mu}$
 is always  Cohen-Macaulay, even if $x$ is not special. This conjecture
has   recently been proven in the split case by X. He \cite{HeCM}.}
 \end{Remark}

\smallskip

Now let us discuss the proof  of Theorem \ref{special fiber}.

\begin{proof}
We can assume that $F=\breve{F}$ and  the residue field $k=\bar{k}$
is algebraically closed; then we are trying to show  $\calA^{P_{
k}}(\mu)=\overline M_{\Gg, \mu}$. The proof is divided into two
parts. The first part is to exhibit $\calA^{P_k}(\mu)$ as a closed
subscheme of $\overline M_{\Gg, \mu}$. Then we apply the coherence
conjecture of Rapoport and the first author (\cite{PappasRaTwisted})
shown in \cite{ZhuCoherence} to deduce the theorem.

\begin{prop}\label{Prop8.8}
 The scheme $\calA^{P_k}(\mu)$ is naturally a closed subscheme of $\overline M_{\Gg, \mu}$.
\end{prop}

\begin{proof}
Let $\la\in\Lambda$, where $\Lambda\subset\xcoch(T)_I$ is the
$W_0$-orbit in $\xcoch(T)_I$ associated to $\mu$ as before. Let
$t_\la\in\widetilde{W}$ be the corresponding element in the
Iwahori-Weyl group. These elements then are extreme elements in
$\Adm(\mu)$ under the Bruhat order of $\widetilde{W}$. To prove the
lemma, then it is enough to show that $S^{\calP_k}_{t_\la}\subset
\overline M_{\Gg, \mu}$ for any $t_\la\in\Lambda$. This turns out to
be a direct consequence of the following two lemmas.

\begin{lemma}\label{groupaction} The scheme $\overline M_{\Gg, \mu}$ is
invariant under the action of $L^+P_k$ on $\Gr_{P_k}$.
\end{lemma}

\begin{proof}  Recall from \S \ref{6b4}
that $\L^+\Gg$ acts  naturally on $\Gr_{\Gg,X}$.  Now let $x:\Spec
(\kappa)\to X$ be a point where $\kappa$ is either the residue field
of $\O$ or the fractional field of $\O$. Then by definition
$\L^+\Gg\times_X\Spec(\kappa)\simeq L^+\Gg_{\kappa,x}$, and the
induced action of $\L^+\Gg\times_X\Spec(\kappa)$ on
$\Gr_{\Gg,X}\times_X\Spec( \kappa)$ is just the local action of
$L^+\Gg_{\kappa,x}$ on $\Gr_{\Gg_{\kappa,x}}$. Let
$(\L^+\Gg)_\O:=\L^+\Gg\times_X\Spec( \O)$ where the section $\Spec
(\O)\to X$ is given by $u\mapsto \varpi$ as before. Then
$(\L^+\Gg)_\O$ acts on $\Gr_{\Gg,\O}$. Over $E$, this is just the
action of $(L^+\Gg_{F,\varpi})_E$ on
$(\Gr_{\Gg_{F,\varpi}})_E=\Gr_{\Gg_{F,\varpi}}\times_{\Spec(F)}\Spec(E)$.
Since $\overline M_{\Gg, \mu}$ is defined to be the Zariski closure
of $(L^+\Gg_{F,\varpi})_E\cdot [s_\mu]$, the special fiber carries
the natural action of $L^+P_k=L^+\Gg_{k,0}$. \end{proof}
\smallskip

To state the second Lemma, recall that $\Gr_{\Gg,\O}$ is ind-proper
(Proposition \ref{indproper}). For $\la\in\Lambda$, let
$\tilde{\la}\in\bar{W}\mu$ be a lift of $\la$. The $E$-point
$[s_{\tilde{\la}}]$ of $\Gr_{\Gg,\O}$ (cf. \S \ref{coch}) gives rise to
a unique $\O_E$-section of $\Gr_{\Gg,\O}$, still denoted by
$[s_{\tilde{\la}}]$. Let $0$ be the closed point of $\Spec(\O_E)$.

\begin{lemma}\label{sla}
We have $[s_{\tilde{\la}}](0)=t_\la$ in ${\rm Gr}_{P_k}$.
\end{lemma}

\begin{proof}
Let $\O[u]\to\O[v]$ be given by
$u\mapsto v^e$. Set
\[\calT=\Res_{\O[v]/\O[u]}(T_H\otimes\O[v])^\gamma,\]
and let $\calT^0$ be the neutral connected component of $\calT$. Then by
construction of the group scheme $\Gg$ in Theorem
\ref{grpschemeThm}, $\calT^0$ is a subgroup of $\Gg$. Let $\L\calT$,
$\L\calT^0$ be the global loop groups over $X=\AA^1_\O$, whose
definitions are similar to \eqref{globloop1}. Exactly as in the proof of \cite[Prop. 3.5]{ZhuCoherence}
we see that
 a cocharacter $\nu$ of $G$ defined over $E$ gives rise
to a map
\[s_{\nu,\O_E}:\Spec(\O_E)\to \L\calT^0\to \L\Gg,\]
such that: (i) the restriction of this map to
$\Spec(E)\to\L\calT^0\to\L\Gg$ is the point $s_\nu$ as in \S \ref{coch};
and (ii) $s_{\nu,\O_E}(0)\in \L\calT^0(k)=T(k((u)))$  maps to
$t_\nu$ under the Kottwitz homomorphism $T(k((u)))\to \xcoch(T)_I$.
Clearly, these two statements together imply the lemma.
\end{proof}

This now concludes the proof of Proposition \ref{Prop8.8}.
\end{proof}
\smallskip

Now we proceed with the second part of the proof of the theorem. To
apply the coherence conjecture, we need to construct a natural line
bundle on $\Gr_{\Gg,X}$. The construction is parallel to \cite[Sect.
4]{ZhuCoherence} where we also refer the reader for more
information. Let $\calV_0=\on{Lie}\Gg$ be the Lie algebra of $\Gg$.
By \cite{SeshadriPNAS}, this  is a  free $\O[u]$-module of rank
$\dim_F G$. Then the adjoint representation
$\on{Ad}:\Gg\to\on{GL}(\calV_0)$ gives rise to a morphism
\[\on{ad}: \Gr_{\Gg,X}\to\Gr_{\on{GL}(\calV_0),X}.\]
Over $\Gr_{\on{GL}(\calV_0),X}$, we have the determinant line bundle
$\calL_{\det}$, defined as usual (for example, see \cite[Sect.
4]{ZhuCoherence}). Let $\calL=\calL_{2c}=\on{ad}^*(\calL_{\det})$ be
the corresponding line bundle on $\Gr_{\Gg,X}$, which is ample. Let
us denote by $\calL_{ k}$  the restriction of $\calL $ to the
special fiber $\Gr_{\Gg,\O}\otimes_\O k\simeq \Gr_{P_k}$, and let
$\calL_{ \bar{F}}$ be the restriction of $\calL $ to the geometric
generic fiber
$\Gr_{\Gg,\O}\otimes_\O\bar{F}\simeq\Gr_{\Gg_{F,\pi}}\otimes_F\bar{F}\simeq\Gr_H\otimes_F\bar{F}$.

Let $M_{\Gg, \mu, \bar F}$ be the geometric  generic fiber of
$M_{\Gg, \mu}$. Then for $n\gg 0$,
\[\dim_{\bar{F}}\Ga(M_{\Gg, \mu, \bar F},  \calL_{ \bar{F}}^{\otimes n})=\dim_{k}\Ga(\overline M_{\Gg, \mu }, \calL_{ k}^{\otimes n})\geq \dim_k\Ga(\calA^{P_k}(\mu), \calL_{ k}^{\otimes n}),
\]
and the equality holds for $n\gg 0$ if and only if
$\calA^{P_k}(\mu)=\overline M_{\Gg, \mu}$. Now the coherence
conjecture of \cite{PappasRaTwisted} as proved in
\cite{ZhuCoherence} implies that
\begin{equation}\label{cohconj}
\dim_{\bar{F}}\Ga(M_{\Gg,  \mu ,\bar{F}}, \calL_{ \bar{F}}^{\otimes
n})= \dim_k\Ga(\calA^{P_k}(\mu), \calL_{ k}^{\otimes n})
\end{equation}
for any $n$ and this is enough to conclude the proof. For clarity,
we explain here how to deduce the equality (\ref{cohconj}) from the
results of \cite{ZhuCoherence}. First  we assume that $G_{\on{der}}$
is simply-connected. Then we can write
$G_{\on{der}}=\on{Res}_{F_1/F}G_1\times\cdots \times
\on{Res}_{F_m/F}G_m$ as the decomposition into simple factors, where
$G_i$ are almost simple, absolutely simple and simply-connected
groups defined over $F_i$. Let $\{\mu_i\}$ be the corresponding
conjugacy classes of cocharacters for $\Res_{F_i/F}(G_i)_{\on{ad}}$,
where $(G_i)_{\on{ad}}$ is the adjoint group of $G_i$, and let
$(P_k)_i$ be the corresponding parahoric group scheme over $F_i$.
Then
\[\calA^{P_k}(\mu)\simeq \calA^{P_k}(\mu)^{\circ} \simeq \prod_i \calA^{(P_k)_i}(\mu_i)^{\circ}\]
and under this isomorphism, $\calL_{ k}\simeq \calL_1\boxtimes
\cdots \boxtimes \calL_m$, where $\calL_i$ is ample on
$\calA^{(P_k)_i}(\mu_i)^{\circ}$ with central charge $2h^\vee_i$.
Here $h^\vee_i$ is the dual Coxeter number for
$G_i\otimes_{F_i}\bar{F}$, and the central charge is defined in
\cite[Sect. 10]{PappasRaTwisted} (also see \cite[Sect.
2.2]{ZhuCoherence}). On the other hand, $H_{\on{der}}=\prod
H_i^{m_i}$, where $H_i$ is the split form for $G_i$ and
$m_i=[F_i:F]$. Then $\mu_i=\mu_{i,1}+\cdots+\mu_{i,m_i}$ for
$\mu_{i,j}$ dominant coweights of $(H_i)_{\on{ad}}$. As usual, if
$\mu$ is a coweight of a split group $M$, we denote
$\overline{\Gr}_{M,\mu}$ the closure of the corresponding  $L^+M$
orbit in $\Gr_{M}$, which is indeed the same as $\calA^{M\otimes
k[[u]]}(\mu)$. We similarly denote $\calA^{M\otimes
k[[u]]}(\mu)^\circ$ by $\overline{\Gr}_{M,\mu}^\circ$. We have
\[
M_{\Gg,  \mu ,\bar{F}}\simeq \overline{\Gr}_{H,\mu}\simeq
\overline{\Gr}^\circ_{H_{\on{ad}},\mu}\simeq \prod
\overline{\Gr}^\circ_{(H_i)_{\on{ad}},\mu_{i,j}},
\]
and under this isomorphism, $\calL_{\bar{F}}\simeq
\boxtimes_i(\calL_{b,i}^{\otimes 2h^\vee_i})^{\boxtimes m_i}$, where
$\calL_{b,i}$ is the ample generator of the Picard group of each
connected component of $\Gr_{(H_i)_{\on{ad}}}$ (which is isomorphic
to $\Gr_{H_i}$). Now by \cite[Theorem 2, Proposition
6.2]{ZhuCoherence}, we have
\[
\dim_k\Ga(\calA^{(P_k)_i}(\mu_i)^{\circ}, \calL_i^{\otimes
n})=\prod_{j=1}^{m_i}\dim_{\bar{F}}\Ga(\overline{\Gr}_{(H_i)_{\on{ad}},\mu_{i,j}}^\circ,
\calL_{b,i}^{\otimes 2nh_i^\vee}).
\]
Combining this equality with the above gives (\ref{cohconj}) in this
case.

Next, we consider the general case when ${\rm
char}(k)\nmid|\pi_1(G_{\on{der}})|$. Let
\[1\to S\to \tilde{G}\to G\to 1\]
be a $z$-extension of $G$, i.e. $S$ is an induced central torus, and
$\tilde{G}_{\on{der}}$ is simply-connected. We can further assume
that $\tilde{G}$ is also tamely ramified. Then $\{\mu\}$ can be lifted
to a geometric conjugacy class $\{\tilde{\mu}\}$ of $\tilde{G}$. Let
us apply our construction to $\tilde{G}$ and the corresponding parahoric group
scheme to obtain the global
affine Grassmannian $\Gr_{\tilde{\Gg},\O}$ over $\O$. We can also
define $M_{\tilde{\Gg},\tilde{\mu}}$. Observe that over $\bar{F}$,
we have the natural map
$\Gr_{\tilde{\Gg},\bar{F}}\to\Gr_{\Gg,\bar{F}}$. Under this map, the
line bundle $\calL$ on $\Gr_{\Gg,\bar{F}}$ pulls back to the
corresponding line bundle on $\Gr_{\tilde{\Gg},\bar{F}}$ since the
adjoint representation of $\tilde{G}$ factors through $G$. In
addition, $M_{\tilde{\Gg},\tilde{\mu},\bar{F}}$ maps isomorphically
to $M_{\Gg,\mu,\bar{F}}$. Therefore,
\[\dim_{\bar{F}}\Ga(M_{\Gg, \mu, \bar F},  \calL_{ \bar{F}}^{\otimes n})=\dim_{\bar{F}}\Ga(M_{\tilde{\Gg}, \tilde{\mu}, \bar F},  \calL_{ \bar{F}}^{\otimes n}).\]
Likewise, we have $\Gr_{\tilde{P}_k}\to\Gr_{P_k}$ which maps
$\calA^{\tilde{P}_k}(\tilde{\mu})$ isomorphically to $\calA^{P_k}(\mu)$
(cf. \cite[Sect. 6]{PappasRaTwisted}); the corresponding line
bundles $\calL_{ k}$ respect the pullback under this map. Then
\[\dim_k\Ga(\calA^{P_k}(\mu),
\calL_{ k}^{\otimes n})=\dim_k\Ga(\calA^{\tilde{P}_k}(\tilde{\mu}),
\calL_{ k}^{\otimes n}).\] This allows us to reduce to the previous
case.
\end{proof}

\bigskip

\section{Nearby cycles and the conjecture of Kottwitz}

\setcounter{equation}{0}

In this chapter, we study the sheaves of nearby cycles of the local
models. We extend work of Gaitsgory and Haines-Ng\^o (see Theorem
\ref{comm  constraints}) and among other results, we show Theorems
\ref{thm02} (Kottwitz's conjecture) and \ref{thm03} of the
introduction.

\subsection{The nearby cycles}

\subsubsection{}\label{review of nearby cycles}
We begin by briefly recalling some general facts.  (For more
details, see for example \cite{IllusieMonodromy},
\cite{GortzHainesCrelle}.) Let $(S,s,\eta)$ be a Henselian trait,
i.e. $S$ is the spectrum of a Henselian discrete valuation ring, $s$
is the closed point of $S$ with residue field $k(s)$, and $\eta$ is
the generic point of $S$. For our purposes, we will always assume
that $k(s)$ is either finite or algebraically closed. Let
$\bar{\eta}$ be a geometric point over $\eta$ and $\bar{S}$ be the
normalization of $S$ in $\bar{\eta}$. Let $\bar{s}$ be the closed
point of $\bar{S}$. Let $I={\rm
ker}(\Gal(k(\bar{\eta})/k(\eta))\to\Gal(k(\bar{s})/k(s)))$ denote
the inertia group.

Let $\ell$ be a prime invertible in $\O_S$. For a scheme $X$,
separated and of finite type over $s$, there is the natural
category $\on{Sh}_c(X\times_s\eta,\overline{\bbQ}_\ell)$ whose
objects are constructible $\overline{\bbQ}_\ell$-sheaves on
$X_{\bar{s}}$, together with continuous actions of
$\Gal(k(\bar{\eta})/k(\eta))$, compatible with the action of
$\Gal(k(\bar{\eta})/k(\eta))$ on $X_{\bar{s}}$ via
$\Gal(k(\bar{\eta})/k(\eta))\to\Gal(\bar{s}/s)$ (see \cite[XIII,
1.2.4]{SGA7I-II}). The natural functor
$\on{Sh}_c(X)\to\on{Sh}_c(X\times_s\eta)$ is a full embedding with
essential image consisting of objects on which the inertia $I$ acts
trivially (SGA 7, XIII, 1.1.3). The ``bounded derived" category of
$\on{Sh}_c(X\times_s\eta,\overline{\bbQ}_\ell)$ is denoted by
$\on{D}_c^b(X\times_s\eta,\overline{\bbQ}_\ell)$\footnote{As usual,
this category is not the ``real" derived category of
$\on{Sh}_c(X\times_s\eta,\overline{\bbQ}_\ell)$, but is defined via
a limit process. See \cite[Footnote 2]{HainesNgoNearby}.}. The usual
perverse $t$-structure on
$\on{D}_c^b(X_{\bar{s}},\overline{\bbQ}_\ell)$ is naturally lifted
to $\on{D}_c^b(X\times_s\eta,\overline{\bbQ}_\ell)$, and we have the
corresponding category of perverse sheaves
$\on{Perv}(X\times_s\eta,\overline{\bbQ}_\ell)$. The natural functor
$\on{D}_c^b(X,\overline{\bbQ}_\ell)\to
\on{D}_c^b(X\times_s\eta,\overline{\bbQ}_\ell)$ is a full embedding,
and its essential image consists of objects in
$\on{D}_c^b(X\times_s\eta,\overline{\bbQ}_\ell)$ on which $I$ acts
trivially.

Recall that if $p:\frakX\to S$ is a  morphism of schemes which is separated and of finite type, then there is the
so-called nearby cycle functor
\[{\rm R}\Psi^\frakX: \on{D}_c^b(\frakX_\eta,\overline{\bbQ}_\ell)\to \on{D}_c^b(\frakX_s\times_s\eta,\overline{\bbQ}_\ell),\]
which restricts to an exact functor (\cite[Sect.
4]{IllusieMonodromy})
\[{\rm R}\Psi^\frakX: \on{Perv}(\frakX_\eta,\overline{\bbQ}_\ell)\to \on{Perv}(\frakX_s\times_s\eta,\overline{\bbQ}_\ell),\]
Let $f:\frakX\to\frakY$ be a morphism over $S$. There is   a
canonical natural transform $f_!R\Psi^{\frakX}\to R\Psi^{\frakY}
f_!$ which is an isomorphism if $f$ is proper. In addition, there is
  a canonical natural transform $f^*R\Psi^{\frakY}\to
R\Psi^{\frakX} f^*$, which is an isomorphism if $f$ is smooth.

We will also occasionally use the vanishing cycle functor
(\cite[XII]{SGA7I-II})
\[{\rm R}\Phi^\frakX: \on{D}_c^b(\frakX,\overline{\bbQ}_\ell)\to  \on{D}_c^b(\frakX_s\times_s\eta,\overline{\bbQ}_\ell),\]
which, roughly speaking, is defined via the distinguished triangle
\[\calF_{\bar{s}}\to {\rm R}\Psi^{\frakX}(\calF_\eta)\to {\rm R}\Phi^{\frakX}(\calF)\to. \]
A theorem of Gabber (cf. \cite[Sect. 4]{IllusieMonodromy}) says
that ${\rm R}\Phi^{\frakX}[-1]$ is also perverse exact with the
$t$-structure on $\on{D}_c^b(\frakX,\overline{\bbQ}_\ell)$ defined
as in loc. cit.

\begin{Remark}\label{indsheaf}
{\rm As explained in \cite[A. 2]{GaitsgoryInv}, if $X$ is an
ind-scheme, of ind-finite type over a field, then
$\on{D}_c^b(X,\overline{\bbQ}_\ell)$ is defined as the direct limit
of the corresponding category on finite dimensional closed
subschemes. As the push-forward along closed immersions is perverse
exact, this allows to also define a corresponding category
$\on{Perv}(X,\overline{\bbQ}_{\ell})$. Similarly, we then have
$\on{D}_c^b(X\times_s\eta,\overline{\bbQ}_\ell)$ and
$\on{Perv}(X\times_s\eta,\overline{\bbQ}_\ell)$. If $\frakX$ is
 of ind-finite type over $\O$, we can also define the nearby cycles
${\rm R}\Psi^\frakX$ and the results in the above discussion
appropriately extend to this case.

In what follows, without mentioning it explicitly, we will
understand that any category of sheaves on an ind-scheme is defined
as a direct limit as above. }
\end{Remark}

\subsubsection{} Let us now return to our set up, so
$S=\Spec(\O)$, where $\O$ is the ring of integers of a $p$-adic
field $F$ with residue field $k$. Let us fix a prime $\ell$, which
is invertible in $\O$.

Let $X_0=\Spec (\O)\to X=\AA^1_\O$ be the morphism given by
$u\mapsto 0$ and let $P=P_\O$ be the group scheme over $\O[[t]]$
given by
\[P =\Gg\times_X\Spec (\O[[t]]), \quad u\mapsto t.\]
Then for $\kappa$ either the fraction field $F$ of $\O$, or the
residue field $k$ of $\O$, $P \times_{\O[[t]]}\kappa[[t]]=P_\kappa$ is the
parahoric group scheme over $\kappa[[t]]$ associated to the point
$x_{\kappa((u))}$ in the building of $\calB(G_\kappa,\kappa((u)))$
as in \S \ref{3a}, see also \S \ref{kappafibers}. Let us also
consider
\[\Gr_{P}:=\Gr_{\Gg, \O, 0}=\Gr_{\Gg,X}\times_X X_0\]
given by this specialization along $u=0$. This is identified with
the local affine Grassmannian ${\rm Gr}_{\Gg}$ over $\Spec(\O)$
considered in \S \ref{LocalaffGrass}. The jet group $L^+P $ over
$\O$ is defined as follows: for every $\O$-algebra $R$,
\[L^+P (R)=P (R[[t]]).\]
Then $L^+P $ acts on $\Gr_{P }$.

Let $\Spec(\kappa)\to\Spec(\O)$ be a perfect field-valued point.
Then $\Gr_{P
}\times_{\Spec(\O)}\Spec(\kappa)=\Gr_{P_\kappa}=\Gr_{P,\kappa}$ is
the
 affine Grassmannian associated to $P_\kappa$, and when we
base change  the action of $L^+P $ on $\Gr_{P }$ under $\O\to
\kappa$, we obtain  the usual action of $L^+P_\kappa$ on
$\Gr_{P_\kappa}$.

For simplicity, we set $\breve{P}=P\otimes_\O\breve{\O}$, similarly
for the other (ind)-schemes. The $L^+\breve{P} $-orbits of
$\Gr_{\breve{P}}$ are parametrized by certain double cosets   in the
extended Weyl group $\widetilde{W}$. For $w\in\widetilde{W}$, let
$\mathring{S}_w$ be the corresponding orbit which is smooth over
$\breve{\O}$, and let $S_w$ be the corresponding Schubert scheme
over $\breve\O$. Recall that the group splits after an extension of
$\breve{\O}$ of degree prime to $p$. We can see that if
$\breve{\O}\to \kappa$ is as above, then there is a nilpotent
immersion $S_{w,\kappa}\to S_w\otimes_{\breve{\O}} \kappa$ where
$S_{w,\kappa}$ is the Schubert variety $S_{w,\kappa}$ in
$\Gr_{P_\kappa}$ corresponding to $w$. (This immersion is an
isomorphism if $p\nmid |\pi_1(G_{\on{der}})|$. Indeed, then the
Schubert  varieties $S_{w,\kappa}$ are   normal and the result
follows using \cite[Prop. 9.11]{PappasRaTwisted}.) We will often
identify $\overline{\bbQ}_\ell$-sheaves on $S_w\otimes_{\breve{\O}}
\kappa$ with the corresponding $\overline{\bbQ}_\ell$-sheaves on
$S_{w,\kappa}$.

We denote $\on{IC}_w$ to be the intersection cohomology sheaf on
$S_w$, i.e., the intermediate extension of
$\overline{\bbQ}_\ell[\dim S_{w}{+1}](\dim S_w/2)$ on
$\mathring{S}_w$. (Here $\dim S_w$ is the relative dimension over
$\breve\O$. For the definition of perverse sheaves on schemes over
$\O$, we refer to \cite[Sect. 4]{IllusieMonodromy}.) If $S_w$ is
defined over the discrete valuation ring $\O'$ with $\O\subset
\O'\subset\breve\O$, we will keep track  of the action of
$\Gal(\O'/\O)$ on $\on{IC}_w$, or equivalently, regard $\on{IC}_w$
also defined over $\O'$. For a $\kappa$-valued point of $\Spec(
\breve\O)$, the intersection cohomology sheaf on $S_{w,\kappa}$ is
denoted by $\on{IC}_{w,\kappa}$.

When $G=H\otimes_\O F$ with $H$ a split Chevalley group over $\O$
and $\calG=H\times_{\Spec(\O)} X$, then $\Gr_P$ is the affine
Grassmannian
  $\Gr_H$ over $\O$  and   $L^+P$ is
$L^+H$. The $L^+H$-orbits of $\Gr_H$ are parameterized by conjugacy
classes of one-parameter subgroups of $H$ and for
$\mu\in\xcoch(T_H)\subset\widetilde{W}$, we denote $S_\mu$ by
$\overline{\Gr}_\mu$. Similarly, we denote by $\on{IC}_\mu$   the
intersection cohomology sheaf on $\overline{\Gr}_\mu$.

\subsubsection{}

Let $M_{\Gg,\mu, E}$ denote the generic fiber of $M_{\Gg, \mu}$. If
$H$ is the split form of $G$, then $ M_{\Gg,\mu, \tilde F}=M_{\Gg,
\mu, E}\otimes_E\tilde{F}$ is a projective subvariety of
$\Gr_{\Gg,\O}\otimes_{\O} \tilde{F}\simeq\Gr_H\otimes_{\O}
\tilde{F}$ by Corollary \ref{fibers}. In general $M_{\Gg, \mu, E}$
is not smooth unless $\mu$ is minuscule. We denote $\F_\mu$ to be
the intersection cohomology sheaf on the generic fiber $M_{\Gg,\mu,
E}$. Then the pull back of $\F_\mu$ to $M_{\Gg,\mu, \tilde F}$ is
isomorphic to $\on{IC}_{\mu,\tilde{F}}$. Our goal  
is to establish a commutativity constraint for the nearby cycle
\[{\rm R}\Psi_\mu:={\rm R}\Psi^{ M_{\Gg, \mu}}(\F_\mu).\]

Recall that we denote by   $\ke$    the residue field of $\O_E$. We
first need

\begin{lemma}\label{equiv str}
The perverse sheaf\, ${\rm R}\Psi_\mu$ on the special fiber $
\overline M_{\Gg, \mu}=M_{\Gg, \mu}\otimes_{\O_E}\ke\subset\Gr_{P,
\ke }$ admits a natural $L^+P_{\ke}$-equivariant structure, i.e.,
${\rm R}\Psi_\mu$ admits a $L^+P_{\ke}\otimes_{\ke}
\bar{k}$-equivariant structure as perverse sheaves on $\Gr_{P,
\ke}\otimes_{\ke}\bar{k}$, which is compatible with the action of
$\Gal(\bar{F}/E)$ in an obvious sense (which will be clear from the
proof).
\end{lemma}

\begin{proof}
Let $\L^+_n\Gg$ be the $n$-th jet group of $\Gg$, i.e. the group
scheme over $X=\bbA^1_\O$, whose $R$-points classify pairs
$(y,\beta)$ with $y:\Spec(R)\to X$ and $\beta\in\Gg(\Ga_{y,n})$,
where $\Ga_{y,n}$ is the $n-$th nilpotent thickening of $\Ga_y$. In
other words,
\begin{equation}\label{nthjet}\L^+_n\Gg(R)=\Gg(R[u-y]/(u-y)^{n+1})\end{equation}
(cf. \eqref{globloop1}). It is clear that $\L^+_n\Gg$ is smooth over
$X$ and that the action of
$(\L^+\Gg)_{\O_E}:=\L^+\Gg\times_X\Spec(\O_E)$ on $M_{\Gg, \mu}$
factors through the action of
$(\L^+_n\Gg)_{\O_E}:=\L^+_n\Gg\times_{X}\Spec(\O_E)$ for some
sufficiently large $n$.

Let $m:(\L^+_n\Gg)_{\O_E}\times_{\O_E} M_{\Gg, \mu}\to M_{\Gg, \mu}$
be the above action.  Let $p:(\L^+_n\Gg)_{\O_E}\times_{\O_E} M_{\Gg,
\mu}\to M_{\Gg, \mu}$ be the natural projection. Then there is a
canonical isomorphism $m^*\F_\mu\xrightarrow{\sim} p^*\F_\mu$ as
sheaves on $(\L^+_n\Gg)_{E}\times_{E} M_{\Gg, \mu}$ since the
intersection cohomology sheaf $\F_\mu$ is naturally
$(\L^+\Gg)_{E}$-equivariant. By taking  nearby cycles, we have a
canonical isomorphism
\[
{\rm R}\Psi^{(\L^+_n\Gg)_{\O_E}\times_{\O_E} M_{\Gg,
\mu}}(m^*\F_\mu)\xrightarrow{\sim} {\rm
R}\Psi^{(\L^+_n\Gg)_{\O_E}\times_{\O_E} M_{\Gg, \mu}}(p^*\F_\mu),
\]
which is equivariant with respect to the action of
$\Gal(k(\bar{\eta})/k(\eta))$. Since both $m$ and $p$ are smooth
morphisms and taking nearby cycles commutes with smooth base change,
we have
\[
m^*{\rm R}\Psi^{ M_{\Gg, \mu}}(\F_\mu)\xrightarrow{\sim} p^*{\rm
R}\Psi^{ M_{\Gg, \mu}}(\F_\mu),
\]
compatible with the action of $\Gal(k(\bar{\eta})/k(\eta))$. The
cocycle condition of this isomorphism follows from the corresponding
cocycle condition for $m^*\F_\mu\xrightarrow{\sim} p^*\F_\mu$. This
then establishes the desired action and the lemma follows.
\end{proof}

\begin{Definition}We let
$\on{Perv}_{L^+P_k}(\Gr_{P_k}\times_kF,\overline{\bbQ}_\ell)$   be
the category whose objects are $(\calF,\theta)$, where $\calF\in
\on{Perv}(\Gr_{P_k}\times_kF,\overline{\bbQ}_\ell)$, and
$\theta:m^*\calF\xrightarrow{\sim} p^*\calF$ is a
$L^+P_{\bar{k}}$-equivariant structure on $\calF$, which is
compatible with the action of $\Gal(\bar{F}/F)$.
\end{Definition}

By the above Lemma, ${\rm R}\Psi_\mu={\rm R}\Psi^{ M_{\Gg,
\mu}}(\F_\mu)$ is an object of
$\on{Perv}_{L^+P_{\varkappa}}(\Gr_{P_{\varkappa}}\times_{\varkappa}E,\overline{\bbQ}_\ell)$
for $\varkappa=k_E$.

\subsubsection{}
Let $w\in\widetilde{W}$. Recall that for a chosen reduced expression
$\tilde{w}$ of $w$, there is the Demazure resolution
$D_{\tilde{w}}\to S_w$, where $D_{\tilde{w}}$ is smooth proper over
$\breve\O$, containing $\mathring{S}_w$ as a Zariski open subset,
and $D_{\tilde{w}}\setminus\mathring{S}_w$ is a divisor with normal
crossings relative to $\breve\O$.

\begin{lemma}\label{no monodromy}
Let $F\subset F'\subset \breve{F}$, $\O'$ be the normalization of
$\O$ in $F'$, and $k'$ be the residue field of $\O'$. Write for
simplicity  $P'=P\otimes_{\O}\O'$. Assume that $S_w$ is defined over
$\O'$. Let $\on{IC}_{w,F'}$ be the intersection cohomology sheaf on
$S_{w,F'}$. Then ${\rm R}\Psi^{\Gr_{P'}}(\on{IC}_{w,F'})$ is
isomorphic as an object of
$\on{Perv}(S_{w,k'}\times_{k'}F',\overline{\bbQ}_\ell)$ to the
intersection cohomology sheaf $\on{IC}_{w,k'}$ of $S_{w,k'}$ (recall
that we regard $\on{Perv}(S_{w,k'},\overline{\bbQ}_\ell)$ as a full
subcategory of
$\on{Perv}(S_{w,k'}\times_{k'}F',\overline{\bbQ}_\ell)$).
\end{lemma}
\begin{proof}The existence of $D_{\tilde{w}}$, together with the argument as in \cite[\S
5.2, \S 6.3]{HainesNgoNearby}, implies that the lemma holds for
$F'=\breve{F}$. Observe that we cannot apply the same argument
directly to the case $F'\subsetneqq \breve F$ because $D_{\tilde{w}}$ is
not necessarily defined over $\O'$. Instead, we argue as follows.
From the case $F'=\breve{F}$ above, we deduce that ${\rm
R}\Psi^{\Gr_{P'}}(\on{IC}_{w,F'})=\on{IC}_{w,k'}\otimes\calL$ for
some rank one local system on $S_{w,k'}$ coming from $\Spec(k')$. On
the other hand, $\mathring{S}_{w}$ is defined and is smooth over
$\O'$. Therefore, ${\rm
R}\Psi^{\Gr_{P'}}(\on{IC}_{w,F'})|_{\mathring{S}_{w,k'}}\simeq
\on{IC}|_{\mathring{S}_{w,k'}}=\overline{\bbQ}_\ell[\dim
S_{w,k'}]( (\dim S_{w,k'})/2)$. Hence, $\calL$ is trivial.
\end{proof}

\subsection{A commutativity constraint}\label{sectionConstraint}

\subsubsection{}
Let $\on{D}_{L^+P}(\Gr_{P },\overline{\bbQ}_\ell)$ be the bounded
$L^+P $-equivariant derived category of (constructible)
$\overline{\bbQ}_\ell$-sheaves on $\Gr_{P }$ in the sense of
Bernstein-Lunts \cite{BernsteinLunts}. Let us recall that
$\on{D}_{L^+P }(\Gr_{P },\overline{\bbQ}_\ell)$ is a monoidal
category with structure given by the ``convolution product" defined
by Lusztig (see \cite{GaitsgoryInv}, \cite{LusztigAst}). Namely, we
have the convolution diagram
\begin{equation}\label{Conv-diag}
\Gr_{P }\times\Gr_{P }\xleftarrow{\ q\ }LP \times\Gr_{P
}\xrightarrow{\ p\ }LP \times^{L^+P }\Gr_{P }=:\Gr_{P
}\tilde{\times}\Gr_{P }\xrightarrow{\ m\ } \Gr_{P },
\end{equation}
where $p,q$ are natural projections and $m$ is given by the left
multiplication of $LP $ on $\Gr_{P }$. Let $\calF_i\in \on{D}_{L^+P
}(\Gr_{P }, \overline{\Q}_l), i=1,2$, and let
$\calF_1\tilde{\times}\calF_2$ be the unique sheaf (up to a
canonical isomorphism) on $LP\times^{L^+P }\Gr_{P }$ such that
\begin{equation}\label{twp1}
p^*(\calF_1\tilde{\times}\calF_2)\simeq
q^*(\calF_1\boxtimes\calF_2).
\end{equation}
Then, by definition
\begin{equation}\label{conv prod}
\calF_1\star\calF_2=m_!(\calF_1\tilde{\times}\calF_2).
\end{equation}

If $\Spec(\kappa)\to\Spec(\O)$ is a field valued point, we have the
corresponding monoidal category
$\on{D}_{L^+P_\kappa}(\Gr_{P_\kappa},\overline{\bbQ}_\ell)$ defined
in the same manner. Let
$\on{Perv}_{L^+P_\kappa}(\Gr_{P_\kappa},\overline{\bbQ}_\ell)$ be
the core of the perverse $t$-structure.

Observe that if $\calF_1$,
$\calF_2\in\on{Perv}_{L^+P_k}(\Gr_{P_k}\times_kF,\overline{\bbQ}_\ell)$,
then there is a natural action of $\Gal(\bar{F}/F)$ on
$\calF_1\star\calF_2$.  (In fact, one can define the ``derived
category"
$\on{D}_{L^+P_k}(\Gr_{P_k}\times_k\eta,\overline{\bbQ}_\ell)$ so
that $\calF_1\star\calF_2$ will be an object in
$D_{L^+P_k}(\Gr_{P_k}\times_k\eta,\overline{\bbQ}_\ell)$. We will
not use this concept in the paper.)

\subsubsection{}

In general, if $\calF_1,\calF_2$ are perverse sheaves, it is not
always the case that $\calF_1\star\calF_2$ is perverse. However, the
main result of this subsection is

\begin{thm}\label{comm constraints}
There is a canonical isomorphism
\[
c_{\calF}:\on{IC}_{w,\bar{k}}\, \star\ {\rm
R}\Psi_\mu\xrightarrow{\sim} {\rm R}\Psi_\mu\star\on{IC}_{w,\bar{k}}
\]
of perverse sheaves on $\Gr_{P_{\bar{k}}}$. In addition, if
$S_{w,\bar{k}}$ is defined over $k'\supset k_E$, this isomorphism
respects the action of $\Gal(\bar F/E')$ on both sides, where
$E'=EF'\subset \bar{F}$, and $F'$ is the unique subfield in
$\breve{F}$ with residue field $k'$.
\end{thm}

\begin{Remark} {\rm In the case $G=\GL_n$ or $\on{GSp}_{2n}$,
and $x$ is in an alcove, i.e the parahoric  group is an Iwahori,
this is one of the main results of \cite{HainesNgoNearby} (loc. cit.
Proposition 22). }
\end{Remark}

The proof is given by a mixed characteristic version of the arguments in
\cite{GaitsgoryInv,ZhuCoherence}. We need a version of the
Beilinson-Drinfeld Grassmannian defined over $X=\AA^1_\O$. For a
scheme $y:S\to X$ we set,
\begin{equation*}\label{BD Grass}
\Gr^{\on{BD}}_{\Gg,X}(S)=\biggl\{\,\text{iso-classes of pairs } (
\mathcal E, \beta) \biggm|
      \twolinestight{${\mathcal E}$ a $\Gg$-torsor on $X\times S$,}
   {$\beta$ a trivialization of $\mathcal E |_{ (X\times S)\setminus\Gamma_y\cup (0\times S) }$}\,\biggr\}\, .
\end{equation*}
Here $X\times  S=X\times_{\Spec(\O)}S$.
To prove that $\Gr^{\on{BD}}_{\Gg,X}$  is indeed represented by an
ind-scheme, one proceeds as in the proof of Proposition
\ref{indscheme}: Namely, it is standard (\cite{BeilinsonDrinfeld}) that
$\Gr^{\on{BD}}_{\GL_n,X}$ is represented by an ind-scheme. The
general case follows from the fact that if $\Gg\to\GL_n$ is a closed
embedding such that $\GL_n/\Gg$ is quasi-affine, then
$\Gr^{\on{BD}}_{\Gg,X}\to\Gr^{\on{BD}}_{\GL_n,X}$ is a locally
closed embedding.

Let us describe $\Gr^{\on{BD}}_{\Gg,X}$ more explicitly. For this
purpose, set $\mathring{X}=\Spec(\O[u,u^{-1}])\hookrightarrow X$ and
let $X_0=\Spec(\O)\hookrightarrow X$ be given by $u\mapsto 0$ as
before.

The following isomorphisms are clear
\[\Gr^{\on{BD}}_{\Gg,X}|_{\mathring{X}}\simeq \Gr_{P }\times_{\Spec(\O)}\Gr_{\Gg,X}|_{\mathring{X}}\ ,\quad\quad \Gr^{\on{BD}}_{\Gg,X}|_{X_0}\simeq \Gr_{P }.\]

Let us set
$\Gr^{\on{BD}}_{\Gg,\O}:=\Gr^{\on{BD}}_{\Gg,X}\times_X\Spec(\O)$,
where $\varpi:\Spec(\O)\to X$ is given by $u\mapsto\varpi$. Observe
that on
$$
\Gr^{\on{BD}}_{\Gg,E'}
=\Gr^{\on{BD}}_{\Gg,\O}\times_{\Spec(\O)}{\Spec(E')}\simeq
\Gr_{P_{E'}}\times_{\Spec(E')}\Gr_{\calG, E'},
$$
 we can form
$\on{IC}_{w,E'}\boxtimes\, \F_\mu$ over
$S_{w,E'}\times_{\Spec(E')}M_{\Gg, \mu, E'}$. (Here and in what
follows, for simplicity, we write again $\F_\mu$ for the pull-back
of $\F_\mu$ to the base change over $E'$.)

Clearly, Theorem \ref{comm constraints} is a consequence of the
following.

\begin{prop}\label{aux}
We have canonical isomorphisms in
$\on{Perv}_{L^+P_{k'}}(\Gr_{P_{k'}}\times_{k'}E',\overline{\bbQ}_\ell)$:

(a)
  \ ${\rm R}\Psi^{\Gr_{\Gg,\O_{E'}}^{\on{BD}}} (\on{IC}_{w,E'}\boxtimes\,\F_\mu)\xrightarrow{\sim} {\rm R}\Psi_\mu\star\on{IC}_{w,k'}.$

(b)  \ ${\rm R}\Psi^{\Gr_{\Gg,\O_{E'}}^{\on{BD}}
}(\on{IC}_{w,E'}\boxtimes\,\F_\mu)\xrightarrow{\sim}\on{IC}_{w,k'}\star\,
{\rm R}\Psi_\mu$.
\end{prop}
\begin{proof} In order to simplify the notation, in this proof we will set
$\O'=\O_{E'}$.

We first prove (a). We define an ind-scheme over $X$ by attaching to
every morphism $y:S\to X$,
\begin{equation*}\label{Conv Grass}
\Gr^{\rm Conv}_{\Gg,X}(S)=\Biggl\{\,\text{iso-classes of  } (
\mathcal E, \mathcal E',\beta,\beta') \Biggm|
      \threelinestight{$\mathcal E$, $\mathcal E'$ are two $\Gg$-torsors on $X\times S$,}
   {$\beta$ a trivialization of $\mathcal E |_{ (X\times S)\setminus\Gamma_y}$,} {$\beta'$ an isomorphism $\mathcal E'|_{\mathring{X}\times S}\simeq \mathcal E|_{\mathring{X}\times S}$ }\Biggr\}\, .
\end{equation*}
Observe that there is an natural projection $p:\Gr^{\rm
Conv}_{\Gg,X}\to\Gr_{\Gg,X}$ by forgetting $(\calE',\beta')$, and a
natural map $m:\Gr^{\rm Conv}_{\Gg,X}\to\Gr^{\on{BD}}_{\Gg,X}$
sending $(\calE,\calE',\beta,\beta')$ to $(\calE',\beta\beta')$.

The map $p$ makes $\Gr^{\rm Conv}_{\Gg,X}$  a fibration over
$\Gr_{\calG,X}$ with fibers isomorphic to $\Gr_{P}$. To see this, we
define a fpqc sheaf $\widetilde{\Gr}_{\Gg,X}$ over $X$ by setting,
for $S$ affine,
\[
\widetilde{\Gr}_{\Gg,X}(S)=\biggl\{\,\text{iso-classes of  } (
\mathcal E,\beta,\beta') \biggm|
      \twolinestight{$\mathcal E$ a $\Gg$-torsor on $X\times S$,
   $\beta$ a trivialization of} {$\mathcal E |_{ (X\times S)\setminus\Gamma_y}$, $\beta'$ a trivialization of
   $\mathcal E|_{D_S}$ }\,\biggr\}\, .
\]
Then $L^+P $ acts on $\widetilde{\Gr}_{\Gg,X} $ by changing the
trivialization $\beta'$ and this makes $\widetilde{\Gr}_{\Gg,X}$ a
$L^+P $-torsor over $\Gr_{\Gg,X}$. In addition,
\[
\Gr^{\rm Conv}_{\Gg,X}=\widetilde{\Gr}_{\Gg,X}\times^{L^+P}\Gr_{P}.
\]
The map $m$ can be described as follows. We have
\[
m|_{\mathring{X}}:\Gr^{\rm
Conv}_{\Gg,X}|_{\mathring{X}}\simeq\Gr^{\on{BD}}_{\Gg,X}|_{\mathring{X}},\quad\quad
m|_{X_0}:\Gr_{P }\,\widetilde{\times}\, \Gr_{P }\to\Gr_{P },
\]
where $\Gr_{P }\, \widetilde{\times}\, \Gr_{P }\to\Gr_{P }$ is
defined in \eqref{twp1}.

By specialization along $u\mapsto \varpi$, we obtain  corresponding
ind-schemes over $\Spec(\O)$. Let us  further base change all the
ind-schemes along $\O\to\O'=\O_{E'}$. In particular, we denote
$\Gr^{\rm Conv}_{\Gg,X}\times_X\Spec(\O')$ by $\Gr^{\rm
Conv}_{\Gg,\O'}$. Similarly, we write $\Gr^{\rm Conv}_{\Gg, E'}$,
$\widetilde{\Gr}_{\Gg, E'}$, etc. for the base change to $\Spec(E')$
also given by $u\mapsto \varpi$.

Regard $\on{IC}_{w,E'}\boxtimes\,\F_\mu$ as a sheaf on
$\Gr_{\Gg,E'}^{\rm Conv} \simeq \Gr_{P_{E'}}\times_{\Spec(E')}
\Gr_{\Gg, E'}$. Since taking nearby cycles commutes with proper
push-forward, to show (a) it will be enough to show that there is a
canonical isomorphism
\begin{equation}\label{ser3}
{\rm R}\Psi^{\Gr^{\rm Conv}_{\Gg,\O'}}(\on{IC}_{w,E'}\boxtimes\,
\F_\mu)\xrightarrow{\sim}{\rm
R}\Psi_\mu\,\widetilde{\times}\,\on{IC}_{w,k'}
\end{equation}
of sheaves on $\Gr_{P_{k'}}\,\widetilde\,{\times}\Gr_{P_{k'}}$;
 here $R\Psi_\mu\,\widetilde{\times}\,\on{IC}_{w,k'}$ is the twisted
product defined as in \eqref{twp1}.

Let $L^+_nP $ be the $n$-th jet group of $P $ whose definition is
similar to $\L^+_n\Gg$ in \eqref{nthjet}. (In fact, $L^+_nP
=\L^+_n\Gg\times_XX_0$.) Choose $n$ sufficiently large so that the
action of $L^+P $ on $S_w$ factors through $L^+_nP $. Let $\Gr_{\Gg,
n, X}$ be the $L^+_nP $-torsor over $\Gr_{\Gg,X}$ that classifies
$(\calE, \beta, \beta')$ where $(\calE, \beta)$ are as in the
definition of $\Gr_{\Gg,X}$ and $\beta'$ is a trivialization of the
restriction of $\calE$ over the $n$-th infinitesimal neighborhood
$X_n$ of $X_0\subset X$. Set
$$
\Gr_{\Gg, n, \O_{E'}}=\Gr_{\calG, n, X}\times_X\Spec(\O'),\quad
 \Gr_{\Gg, n,  E'}=\Gr_{\calG, n, X}\times_X\Spec( E' ).
 $$
Then $\F_\mu\,\widetilde{\times}\,\on{IC}_{w,E'}$ is supported on
\[
 \widetilde{\Gr}_{\Gg,E'} \times^{L^+P }S_w\simeq \Gr_{\Gg,n, E'} \times^{L^+_nP }S_w\subset\Gr^{\rm Conv}_{\Gg,E'} .
\]
Observe that over $E'$, it makes sense to talk about
$\F_\mu\twprod\on{IC}_{w,E'}$ (as defined via \eqref{twp1}), which
is canonically isomorphic to $\on{IC}_{w,E'}\boxtimes\, \F_\mu$.
Therefore, \eqref{ser3} is equivalent to
\begin{equation}\label{ser4}
{\rm R}\Psi^{\Gr^{\rm
Conv}_{\Gg,\O'}}(\F_\mu\twprod\on{IC}_{w,E'})\simeq {\rm
R}\Psi_\mu\twprod \on{IC}_{w,k'}.
\end{equation}
Let us denote the pullback of $\F_\mu$ to $\Gr_{\Gg, n, E'} $ by
$\widetilde{\F}_\mu$. Since $\Gr_{\Gg,n, X }\to \Gr_{\Gg,X}$ is
smooth, ${\rm R}\Psi^{\Gr_{\calG, n,\O'}} (\widetilde{\F}_\mu)$ is
canonically isomorphic to the pullback of ${\rm R}\Psi_\mu$, and by
Lemma \ref{no monodromy},
\[
{\rm R}\Psi^{\Gr_{\calG, n, \O'} \times_{ \O' }  S_{w,
\O'}}(\widetilde{\F}_\mu\boxtimes\on{IC}_{w,E'})\simeq {\rm
R}\Psi^{\Gr_{\calG, n,
\O'}}(\widetilde{\F}_\mu)\boxtimes\on{IC}_{w,k'}.
\]
Observe that both sides are in fact $L^+_nP_{k'}$-equivariant
perverse sheaves, and the isomorphism respects to the equivariant
structure (by the similar argument as in the proof of Lemma
\ref{equiv str}). We thus have \eqref{ser4} and therefore have
finished the proof of (a).

Next we prove (b), which is similar. There is another convolution
affine Grassmannian $\Gr_{\Gg,X}^{\rm Conv'}$, which represents the
functor that associates every $X$-scheme $y:S\to X$,
\begin{equation*}
\Gr^{\rm Conv'}_{\Gg,X}(S)=\Biggl\{\,\text{iso-classes of  } (
\mathcal E, \mathcal E',\beta,\beta') \Biggm|
      \threelinestight{$\mathcal E, \mathcal E'$ are two $\Gg$-torsors on $X\times S$,}
  { $\beta$ a trivialization of $\mathcal E |_{  \mathring{X}\times S }$,} {$\beta'$ an isomorphism $\mathcal E'|_{X\times S\setminus\Ga_y}\simeq \mathcal E|_{X\times S\setminus\Ga_y}$}\,\Biggr\}\, .
\end{equation*}

Clearly, we have $m':\Gr^{\rm Conv'}_{\Gg,
X}\to\Gr_{\Gg,X}^{\on{BD}}$ by sending
$(y,\calE,\calE',\beta,\beta')$ to $(y,\calE',\beta\beta')$. This is
an isomorphism over $\mathring{X}$, and $m'|_{X_0}$ is  again
 the local convolution diagram
\[
m:\Gr_{P}\twprod \Gr_{P}\to\Gr_{P}.
\]
Again, regard $\on{IC}_{w,E'}\boxtimes\,\F_\mu$ as a sheaf on
$\Gr^{\rm Conv'}_{\calG,E'} \simeq \Gr_{P_{E'}}\times_{E'}
\Gr_{\Gg,E'} $.  Again, as nearby cycles commute with proper
push-forward, it is enough to prove that as sheaves on
$\Gr_{P_{k'}}\twprod\Gr_{P_{k'}}$,
\begin{equation}\label{ser1}
{\rm R}\Psi^{\Gr_{\Gg,\O'}^{\rm Conv'} }(\on{IC}_{w,E'}\boxtimes\,
\F_\mu)\xrightarrow{\sim} \on{IC}_{w,k'}\twprod R\Psi_\mu.
\end{equation}
Recall $\calL^+_n\calG$ is the $n$-th jet group of $\calG$. This is
smooth over $X$ and the action of
$(\L^+\Gg)_{\O_E}=\L^+\Gg\times_X\Spec(\O_E)$ on $M_{\Gg, \mu}$
factors through   $(\calL^+_n\calG)_{\O_E}$ for some sufficiently
large $n$.

Let us define the $\calL^+_n\calG$-torsor $\calQ_n$ over $\Gr_{P
}\times_{\Spec(\O)} X$ as follows. Its $S$-points are quadruples
$(y,\calE,\beta,\beta')$, where $y:S\to X$, $(\calE,\beta)$ are as
in the definition of $\Gr_{P}$ (and therefore $\beta$ is a
trivialization of $\calE$ on $\mathring{X}\times S$), and $\beta'$
is a trivialization of $\calE$ over $\Ga_{y,n}$, the $n$-th
nilpotent thickening of the graph $\Ga_y$ of $y$. Then we have the
twisted product
 \[
 \calT_n:=(\calQ_n\times_X\Spec(\O'))\times^{(\calL^+_n\calG)_{\O'} }M_{\Gg, \mu, \O'}\subset\Gr_{\calG,\O'}^{\rm Conv'} .\]
Over $E'$, we can form the twisted product
$\on{IC}_{w,E'}\twprod\F_\mu$ on $\calT_n$ as in \eqref{twp1}, which
is canonically isomorphic to $\on{IC}_{w,E'}\boxtimes\,\F_\mu$. By
the same argument as in the proof of (a) (i.e. by pulling back
everything to $(\calQ_n\times_X\Spec(\O'))\times M_{\Gg, \mu, \O'}$)
 we have
\[{\rm R}\Psi^{ \calT_n}
 (\on{IC}_{w,E'}\twprod\F_\mu)\xrightarrow{\sim}
\on{IC}_{w,k'}\twprod{\rm R}\Psi_\mu.\] Therefore, \eqref{ser1}
holds and this completes the proof of the proposition.
 \end{proof}
\begin{Remark}
{\rm Observe that in \cite{GaitsgoryInv}, the proof of the second
statement of the proposition is considerably more difficult than the
proof of the first one. Indeed, to prove the second statement,
Gaitsgory used the fact that every $H$-torsor over
$\bbA^1_{\bar{k}}$ admits a reduction to a Borel subgroup, whose
counterpart for the general parahoric group schemes $\Gg$ over
$\bbA^1_\O$ has not been documented. Instead, our approach exploits
the extra flexibility provided by the use of the group schemes
$\Gg$ and treats both cases in a parallel way.}
\end{Remark}

\subsection{The monodromy of the nearby cycles}
\subsubsection{}
Let $M_{\Gg, \mu}$ be the generalized local model defined  as
before. This is a flat projective $\O_E$-scheme. To study the
  action of the inertia group $ I_E={\rm ker}({\rm Gal}(\bar F/E)\to {\rm Gal}(\bar k/k_E))$
  (``monodromy") on the nearby cycles, it is convenient to first consider the restriction of this action to
  the inertia $I_{\ti F}$, where, as we recall (see \S \ref{sss1a2}), $\tilde F$ splits the group $G$.

\begin{thm}\label{monodromy}
The action of the inertia $I_{\ti F}$ on the nearby cycles ${{\rm
R}\Psi}_\mu$ is unipotent.
\end{thm}

\begin{proof}  Consider
the base change
\[
\widetilde M_{\Gg, \mu}:= M_{\Gg, \mu}\otimes_{\O_E}\O_{\tilde{F}}.
\]
Let us set $\widetilde{{\rm R}\Psi}_\mu:={\rm R}\Psi^{\widetilde
M_{\Gg,\mu}}(\F_\mu)$. Recall that $\widetilde{{\rm R}\Psi}_\mu$ is
canonically isomorphic to  the sheaf ${\rm R}{\Psi}_\mu$ on
$\overline M_{\Gg, \mu}\otimes_{k_E}\bar k$ but with its ${\rm
Gal}(\bar F/F)$-action restricted to ${\rm Gal}(\bar F/\ti F)$.

To study the inertia action we can base change to $\breve\O$. Let
$(S,s,\eta)$ be a strictly Henselian trait, i.e. $k(s)$ is separably
closed. Suppose that $\frakX\to S$ is a separated scheme of finite
type. Let us recall that the nearby cycle ${\rm R}\Psi^{\frakX}$ can
be canonically decomposed as
\[
{\rm R}\Psi^{\frakX}=({\rm R}\Psi^{\frakX})^{\rm un}\oplus ({\rm
R}\Psi^{\frakX})^{\rm non-un},
\]
where $({\rm R}\Psi^{\frakX})^{\rm un}$ is the unipotent part
  and $({\rm R}\Psi^{\frakX})^{\rm non-un}$ is the non-unipotent part
  (see \cite[\S 5]{GortzHainesCrelle}).
If $f:\frakX\to\frakY$ is a proper morphism over $S$, then the
isomorphism $f_*{\rm R}\Psi^{\frakX}\simeq {\rm R}\Psi^{\frakY}f_*$
respects this decomposition. We refer to \cite{GortzHainesCrelle}
for the details. From this discussion, by using the standard
technique of using an Iwahori subgroup contained in our given
parahoric, we see that Theorem \ref{monodromy} follows from

\begin{prop}\label{Iwahori case}
Theorem \ref{monodromy} holds in the case $x\in A(G,A,F)$ lies in
the alcove, i.e. when $P_k$ is an Iwahori group scheme.
\end{prop}

 Proposition \ref{Iwahori case} will be deduced as a consequence of the general theory of
central sheaves on  affine flag varieties together with the
following lemma.

\begin{lemma}\label{triv}
Under the above assumptions and notations, the  action of
$\Gal(\bar{F}/\tilde{F}\breve{F})$ on the cohomology groups $\rH^*(
M_{\Gg, \mu,\tilde{F}}, (\F_\mu)_{\tilde{F}})$ is trivial.
\end{lemma}
\begin{proof}
Recall that if $H$ is the split form of $G$, then $M_{\Gg, \mu,
\tilde F}\simeq \overline{\Gr}_{H, \mu,\tilde{F}}$ and
$(\F_\mu)_{\tilde{F}}\simeq\on{IC}_{\mu,\tilde{F}}$. As in Lemma
\ref{no monodromy}, the nearby cycle ${\rm
R}\Psi^{\overline{\Gr}_{H,
\mu}\otimes_\O\O_{\tilde{F}}}(\on{IC}_{\mu,\tilde{F}})$
 has trivial inertia action. This and proper base change implies the lemma.
\end{proof}
\smallskip

Now, we prove Proposition \ref{Iwahori case}. It is enough to show
the statement for $ \widetilde{{\rm R}\Psi}_\mu$. Decompose
\[
\widetilde{{\rm R}\Psi}_\mu=(\widetilde{{\rm R}\Psi}_\mu)^{\rm
un}\oplus (\widetilde{{\rm R}\Psi}_\mu)^{\rm non-un}.
\]
By taking   cohomology  and using the above lemma, we obtain
$\rH^*(\Gr_{P_k}, (\widetilde{{\rm R}\Psi}_\mu)^{\rm non-un})=(0). $
We claim that this already implies that $(\widetilde{{\rm
R}\Psi}_\mu)^{\rm non-un}=(0)$. Indeed, recall that by Theorem
\ref{comm constraints}, we have isomorphisms
\[
\widetilde{{\rm R}\Psi}_\mu\star\on{IC}_{w,\bar{k}} \simeq {\rm
R}\Psi^{\Gr^{\on{BD}}_{\Gg,\O_{\tilde F}}
}(\on{IC}_{w,E'}\boxtimes\,\calF_\mu)\simeq \on{IC}_{w,E'}\star\,
\widetilde{{\rm R}\Psi}_\mu
\]
compatible with the action of the inertia group
$\Gal(\bar{F}/\tilde{F}\breve{F})$. Since the inertia action on
$\on{IC}_{w,\bar{k}}={\rm
R}\Psi^{\Gr_{\breve{P}}}(\on{IC}_{w,\breve{F}})$ is trivial, we
obtain  isomorphisms of perverse sheaves
\begin{eqnarray*}
(\widetilde{{\rm R}\Psi}_\mu)^{\rm non-un}\star\on{IC}_{w,\bar{k}}\simeq   (\widetilde{{\rm R}\Psi}_\mu\star\on{IC}_{w,\bar{k}})^{\rm non-un}  \phantom{\ \ \ \ \ } \\
\phantom{\ \ \ \ } \simeq    (\on{IC}_{w,\bar{k}}\star\,
\widetilde{{\rm R}\Psi}_\mu)^{\rm non-un}
\simeq\on{IC}_{w,\bar{k}}\star\, (\widetilde{{\rm R}\Psi}_\mu)^{\rm
non-un}.
\end{eqnarray*}
(\cite[\S 5.3]{GortzHainesCrelle}). In other words,
$(\widetilde{{\rm R}\Psi}_\mu)^{\rm non-un}$ is a \emph{central
sheaf} in $ \on{Perv}_{L^+P_{\bar{k}}}(\Gr_{P_{\bar{k}}}, \overline
\Q_l)$ (cf. \cite{ArBezru, ZhuCoherence}). In loc. cit., it is
proven that central sheaves admit filtrations by the so-called
Wakimoto sheaves on $\Gr_{P_{\bar{k}}}$. Hence, by \cite[Corollary
7.9]{ZhuCoherence}, $\rH^*(\Gr_{P_{\bar{k}}}, (\widetilde{{\rm
R}\Psi}_\mu)^{\rm non-un})=(0)$ implies that $(\widetilde{{\rm
R}\Psi}_\mu)^{\rm non-un}=(0)$. This concludes the proof of the
Proposition and hence also of Theorem \ref{monodromy}.
\end{proof}

\subsubsection{}
Recall that we say that $x\in \calB(G, F)$ is a very special vertex, if
it is special in the sense of Bruhat-Tits in both $\B(G, F)$ and
$\B(G_{\breve{F}},\breve{F})$. 
If there exists such a vertex then the group $G$ is quasi-split over $F$
 (see \cite[\S 6]{ZhuSatake}).

\begin{prop}\label{mono triv}
Assume that $x$ is a very special vertex. Then the action of the inertia
$I_{\ti F}$    on $ {{\rm R}\Psi_\mu}$ is trivial. \end{prop}

\begin{proof}
Observe that when $x$ is a very special vertex, the hypercohomology
functor
\[
\rH^*:\on{Perv}_{L^+P_{\bar{k}}}(\Gr_{P_{\bar{k}}},\overline{\bbQ}_\ell)\to\on{Vect}_{\overline{\bbQ}_\ell}
\]
is a faithful functor. This follows from the semisimplicity of
$\on{Perv}_{L^+P_{\bar{k}}}(\Gr_{P_{\bar{k}}},\overline{\bbQ}_\ell)$,
as is shown in \cite{MirkVilonen, ZhuSatake}. Then the proposition
is a direct corollary of Lemma \ref{triv}.
\end{proof}

\begin{Remark}{\rm
 When $G=\Res_{K/F}\GL_n$, where $K$ is a finite extension of $F$
with Galois hull $\tilde{F}$ in $\bar{F}$, and $\mu$ is a Shimura
(minuscule) cocharacter,  Proposition \ref{mono triv} is shown in \cite{PappasRaI}. When $G$ is a ramified unitary similitude group and $\mu$ is a
Shimura cocharacter, it is shown in \cite{ZhuSatake}.
In these two cases, one can also obtain an explicit description of the monodromy
action of $ I_E$ on ${\rm R}\Psi_\mu$.  See \cite[Sect. 7]{PappasRaI} (also   Example \ref{exPRJAG} below)
and \cite[Theorem 6.2]{ZhuSatake} respectively.
We will return to
the subject of the  monodromy action on ${\rm R}\Psi_\mu$
in \S \ref{quasisplit}. There we will give a uniform but less
explicit description, in all cases that $G$ is quasi-split and $x$ is very special.}
\end{Remark}

\subsection{The semi-simple trace}\label{sstraceSect}
\subsubsection{} As is explained in \cite[\S 3.1]{HainesNgoNearby}, for any $X$ over $s$,
and $\bbF_q\supset k(s)$, there is a map
\[\tau^{\on{ss}}: \on{D}_c^b(X\times_s\eta,\overline{\bbQ}_\ell)\to \on{Func}(X(\bbF_{q}),\overline{\bbQ}_\ell),\]
called the semi-simple trace. This notion is due to Rapoport. Let us
briefly recall its definition and refer to \cite{HainesNgoNearby}
for details.

Let $V$ be an $\ell$-adic representation of the Galois group ${\rm Gal}(k(\bar\eta)/k(\eta))$. An admissible
filtration of $V$ is an increasing filtration $F_\bullet V$, stable
under the action of ${\rm Gal}(k(\bar\eta)/k(\eta))$ and such that, for all $i$, the action of the inertia
$I$ on $V_i/V_{i-1}$ factors through a finite quotient. For an
admissible filtration of $V$, one defines
\[\Tr^{\on{ss}}(\sigma, V)=\sum_i\Tr(\sigma, (\on{gr}^F_iV)^I),\]
where $\sigma$ is the geometric Frobenius element. It is well-known
that  admissible filtrations always exist and that the semi-simple
trace $\Tr^{\on{ss}}(\sigma, V)$ does not depend on the choice of
admissible filtration.

Next, if $\calF\in\on{Sh}_c(X\times_s\eta,\overline{\bbQ}_\ell)$,
then for every $x\in X(\bbF_q)$, we define
\[\tau_\calF^{\on{ss}}(x)=\Tr^{\on{ss}}(\sigma_x,\calF_{\bar{x}}),\]
where $\bar{x}$ is a geometric point over $x$ and $\sigma_x$ is the
geometric Frobenius of $\Gal(k(\bar{x})/k(x))$. In general, if $\calF\in
\on{D}_c^b(X\times_s\eta,\overline{\bbQ}_\ell)$, then we set
\[
\tau_\calF^{\on{ss}}(x)=\sum_i(-1)^i\tau_{\calH^i\calF}^{\on{ss}}(x).
\]

It is known that taking  the semi-simple trace gives an analogue of
 Grothendieck's usual sheaf-function dictionary. Namely, we can see that
$\tau^{\on{ss}}$ factors through the Grothendieck group of
$\on{D}_c^b(X\times_s\eta,\overline{\bbQ}_\ell)$. Also, if $f:X\to
Y$ is a morphism over $s$, then
\begin{equation}\label{sstrF}
f^*\tau^{\on{ss}}_\calF=\tau^{\on{ss}}_{f^*\calF},\quad
f_!\tau^{\on{ss}}_\calF=\tau^{\on{ss}}_{f_!\calF}.
\end{equation}

\subsubsection{}

Now, consider $x\in {\mathcal B}(G, F)$ and let $\calP_x$ be the
parahoric group scheme as before. Let $G'=\calG\times_X
\Spec(k((u)))$ be the corresponding reductive group over $k((u))$
and suppose $\calP_{x_{k((u))}}=\calG\times_X\Spec(k[[u]])$ is a
corresponding parahoric group scheme over $k[[u]]$. Let
$\bbF_q\supset k$ and set $P'_q:=\calP_{x_{k((u))}}(\bbF_q[[u]])$.
Let $\calH_q(G',P'):=\calH(G'(\bbF_q((u))),P'_q)$ be the Hecke
algebra of bi-$P'_q$-invariant, compactly supported locally constant
$\overline{\Q}_l$-valued functions on $G'(\bbF_q((u)))$ under
convolution. This is an associative algebra (usually
non-commutative). Let $\calZ(\calH_q(G',P'))$ denote its center.

 Since ${\rm R}\Psi_\mu$
is an object in
$\on{Perv}_{L^+P_{k_E}}(\Gr_{P_k}\times_{k_E}E,\overline{\bbQ}_\ell)$,
for every $\bbF_q\supset k_E$, we can consider $\tau_{{\rm
R}\Psi_\mu}^{\on{ss}}\in \calH_q(G',P')$.  As in
\cite{HainesNgoNearby}, we can see that Theorem \ref{comm
constraints} implies

\begin{thm} \label{thm9.13} With the above assumptions and notations,
 $\tau^{\on{ss}}_{{\rm
R}\Psi_\mu}\in\calZ(\calH_q(G',P'))$.\ \endproof
\end{thm}

In what follows, we explain how, in some cases, we can determine (or
characterize)   the central function $\tau^{\rm ss}_{{\rm
R}\Psi_\mu}$.

\subsubsection{Unramified groups}
Here, we assume that $G$ is unramified. This means that the group
$G$ is quasi-split and it splits over $\breve{F}$.
\smallskip

a) First assume that $x$ is hyperspecial. Then $\Gr_{P_{\bar{k}}}$
is isomorphic to the usual affine Grassmannian $\Gr_H\otimes
\bar{k}$ of $H$. In this case, the geometric fiber
$\overline{M}_{\Gg, \mu }\otimes_{k_E}\bar{k}$ is isomorphic (up to
nilpotents) to the Schubert variety
$\overline{\on{Gr}}_{\mu,\bar{k}}\subset \Gr_H\otimes\bar k$
corresponding to $\mu$. (If $p\nmid|\pi_1(G_{\on{der}})|$, these are
isomorphic by Theorem \ref{special fiber}.) Let us denote the
intersection cohomology sheaf on $\overline{M}_{\Gg, \mu }$ by
$\bar{\F}_\mu$ so that $(\bar{\F}_\mu
)_{\bar{k}}\simeq\on{IC}_{\mu,\bar{k}}$. By Proposition \ref{mono
triv}, we know that ${\rm R}\Psi_\mu$ is a
Weil sheaf on $\overline{\Gr}_{\mu,\bar{k}}$. 
By Lemma \ref{groupaction} and Lemma \ref{sla},
$(\bar{\F}_\mu)_{\bar{k}}$ is a subquotient of $({\rm
R}\Psi_\mu)_{\bar{k}}$ as Weil sheaves and when forgetting the Weil
structure over $k_E$, it is a direct summand since
$\on{Perv}_{L^+P_{\bar{k}}}(\Gr_{P_{\bar{k}}},
\overline{\Q}_{\ell})$ is a semi-simple category. Then it follows
from
\[\dim_{\overline{\Q}_{\ell}} \rH^*(\on{IC}_{\mu})=\dim_{\overline{\Q}_{\ell}} \rH^*(\F_\mu)=\dim_{\overline{\Q}_{\ell}} \rH^*({\rm R}\Psi_\mu),\]
that we have

\begin{prop}\label{propWeil}
${\rm R}\Psi_\mu\simeq \bar{\F}_{\mu}$ as Weil sheaves.\endproof
\end{prop}

 If $\mu$ is a minuscule coweight and  $x$ is 
hyperspecial, then $\overline{ M }_{\Gg, \mu }$ is smooth. If 
we set
$d=\dim \overline{ M }_{\Gg, \mu }=2(\rho, \mu)$ then, in particular,
we have
\begin{equation}
{\rm R}\Psi_\mu\simeq \overline{\Q}_{\ell}[d](d/2)
\end{equation}
 (the constant sheaf on
$\overline{ M }_{\Gg, \mu }$ up to cohomological shift and Tate
twist).

For simplicity, write $K=P'_q$ and denote $\calH_q(G',P')$ by
$\on{Sph}_q$ (the notation stands for the spherical Hecke algebra)
which is then already a commutative algebra. It is well-known that
the function
\[
A_\mu:=\tau^{\on{ss}}_{\bar{\F}_{\mu}}\in\on{Sph}_q
\]
is given by the Lusztig-Kato polynomial (at least  when $G$ is
split). Of course, in the   case that $\mu$ is minuscule,
$A_\mu=(-1)^{2(\rho,\mu)}q^{(\rho,\mu)}1_{Ks_\mu K}$, where
$1_{Ks_\mu K}$ is the characteristic function of the double coset
$Ks_\mu K$.  
\smallskip

b) Next we let $x$ be a general point in ${\mathcal B}(G, F)$. We
can then obtain the following statement which was conjectured by
Kottwitz and previously proven in the cases $G= \GL_n$ and
$\on{GSp}_{2n}$ by Haines and Ng\^o (\cite{HainesNgoNearby}, see
also \cite{RostamiThesis}).

\begin{thm} \label{ThmKott}(Kottwitz's conjecture)
Assume that $G$ is unramified. Then with the notations above,
$\tau^{\on{ss}}_{{\rm R}\Psi_\mu}$ is the unique element in the
center $\calZ(\calH_q(G',P'))$, whose image under the Bernstein
isomorphism $\calZ(\calH_q(G',P'))\xrightarrow{\sim} \on{Sph}_q$ is
$A_\mu$.
\end{thm}

\begin{proof} See \cite[Theorem 3.1.1]{HainesBC} for the
Bernstein isomorphism in this case. We first show the result when
$x$ is in an alcove, i.e the parahoric subgroup is an Iwahori. Then
the proof  follows, exactly as is explained in
\cite{HainesNgoNearby},
 from Theorem \ref{thm9.13} and Proposition \ref{propWeil},
by using (\ref{sstrF}) and the fact that taking nearby cycles
commutes with proper push-forward.  We refer the reader to loc. cit.
for more details.
 The general parahoric case now follows similarly by first finding an Iwahori subgroup
contained in the parahoric and then using the compatibility of the
Bernstein isomorphism with change in parahoric subgroup
(\cite[3.3.1]{HainesBC}).
\end{proof}

\begin{Remark}
{\rm Assume that $x$ lies in an alcove, i.e the parahoric subgroup
is an Iwahori, and that $\mu$ is minuscule. Using  Bernstein's
presentation of the Iwahori-Hecke algebras (again at least for $G$
split), it is possible to give
 explicit formulas for the functions $\tau^{\on{ss}}_{R\Psi_\mu}$. See
  \cite{HainesNgoNearby}  and \cite{HainesTest} for such formulas and various other
  related results.}
\end{Remark}

\subsubsection{Quasi-split groups; the special vertex case}\label{quasisplit}

Now, we assume that $G$ is quasi-split (but can be split only after
a ramified extension). We will restrict here to the case that $x$ is
a very special vertex in $\calB(G, F)$, i.e it is special and remains
special in $\calB(G_{\breve F}, \breve F)$ in the sense of
Bruhat-Tits. As in this case the special fiber
$\overline{M}_{\Gg,\mu}$ is not necessarily smooth, the nearby cycle
${\rm R}\Psi_\mu$ can be complicated. Our goal is to determine the
semi-simple  trace $\tau^{\on{ss}}_{{\rm R}\Psi_\mu}$ but before
that we need to understand the monodromy action on ${\rm
R}\Psi_\mu$.

We will start by describing $\widetilde{{\rm R}\Psi}_\mu$
explicitly.
 By Proposition \ref{mono triv}, $\widetilde{{\rm R}\Psi}_\mu$ is a Weil sheaf
on $\Gr_{P_{\bar{k}}}$.

Let us briefly recall the (ramified) geometric Satake correspondence
as established in \cite{MirkVilonen, ZhuSatake}. Let $H^\vee$ be the
dual group of the split form $H$ of $G$ over $\overline{\bbQ}_\ell$
in the sense of Langlands (i.e. the root datum of $H^\vee$ is dual
to the root datum of $H$). Then the Galois group $\Ga=\Gal(\tilde
F/F)$ (and therefore $I=I_F$) acts on $H^\vee$ via pinned
automorphisms (after choosing some pinning). Recall the notation
$G'=\Gg\times_X\Spec(k((u)))$ and $P_k=\Gg\times_X\Spec(k[[u]])$;
our constructions allow us to identify the actions of $I_F$ and
$I_{k((u))}$ on $H^\vee$ obtained from $G$ and $G'$ respectively:
Indeed, the action of $I_F$ (resp. $I_{k((u))}$) on $H^\vee$ 
is determined by $I_F\to {\rm Out}(G)$ (resp. $I_{k((u))}\to {\rm Out}(G')$).
These homomorphisms factor through the corresponding tame inertia quotients
which as in \S \ref{ssCovers} are identified with the fundamental group of $\breve\O[u^{\pm 1}]$. However, by our construction, ${\rm Out}(G)={\rm Out}(\und G)={\rm Out}(G')$,
and both $I_F\to {\rm Out}(G)$ and $I_{k((u))}\to {\rm Out}(G')$ are obtained as 
specializations of $\pi_1(\Spec(\breve\O[u^{\pm 1}]), \Spec(\bar F))\to {\rm Out}(\und G)$.

The geometric Satake correspondence of \cite{MirkVilonen, ZhuSatake}
is an  equivalence of tensor categories
\[\calS_{\bar{k}}:\on{Rep}((H^\vee)^I)\xrightarrow{  \sim  }\on{Perv}_{L^+P_{\bar{k}}}(\Gr_{P_{\bar{k}}},\overline{\bbQ}_\ell),\]
such that the composition $\rH^*\circ\calS_{\bar{k}}$ is isomorphic
to the natural forgetful fiber functor $\on{Rep}((H^\vee)^I)\to {\rm
Vect}( \overline\Q_l)$. On the other hand,
$\Gr_{\Gg,\O}\otimes_{\O}\tilde{F}$ is just $\Gr_{H, \tilde{F}}$ for
the split group $H$ and by \cite{MirkVilonen}  there is a full
embedding
\[
\calS_{\tilde{F}}:\on{Rep}(H^\vee)\to\on{Perv}_{(\calL^+\calG)_{\tilde{F}}}(\Gr_{\Gg,\O}\otimes_{\O}\tilde{F},\overline{\bbQ}_\ell).
\]
This maps $V_\mu$, the highest weight representation of $H^\vee$
corresponding to $\mu\in\xcoch(T)$, to $\on{IC}_{\mu,\tilde{F}}$, so
that $\calS_{\tilde F}(V_\mu)=\on{IC}_{\mu,\tilde F}$. Again, the
composition   $\rH^*\circ \calS_{\tilde F}$  is isomorphic to the
forgetful functor.

\begin{thm}\label{thm9.19}
The composition
\begin{multline*} (\calS_{\bar{k}})^{-1}\circ{\rm R}\Psi^{\Gr_{\Gg,\O}\otimes\O_{\tilde{F}}}\circ\calS_{\tilde{F}}:\\
\on{Rep}(H^\vee)\to\on{Perv}_{(\calL^+\calG)_{\tilde{F}}}(\Gr_{\Gg,\O}\otimes_\O\tilde{F},\overline{\bbQ}_\ell)
\to
\on{Perv}_{L^+P_{\bar{k}}}(\Gr_{P_{\bar{k}}},\overline{\bbQ}_\ell)\to\on{Rep}((H^\vee)^I)
\end{multline*}
is isomorphic to the natural restriction functor
\[\on{Res}:\on{Rep}(H^\vee)\to\on{Rep}((H^\vee)^I).\]
In particular, $\widetilde{{\rm R}\Psi}_\mu\simeq
\calS_{\bar{k}}(\on{Res}(V_\mu))$.
\end{thm}
\begin{proof}Since this is a statement about geometric sheaves, we
can assume that the residue field of $\tilde{F}$ is algebraically
closed (i.e. we replace $\tilde{F}$ by $\tilde{F}\breve{F}$ in
$\bar{F}$). In addition, we can ignore the Tate twist in what
follows.

Consider the ramified cover $\bbA^1_{\O_{\tilde{F}}}\to\bbA^1_\O$
given by $\O[u]\to\O_{\tilde{F}}[v], u\mapsto v^e$. Let us denote
the base change
\[\widetilde{\Gr}_{\Gg,X}:=\Gr_{\Gg,X}\times_{\bbA^1_\O}\bbA^1_{\O_{\tilde{F}}}.\]
We will consider the sub ind-schemes of $\widetilde{\Gr}_{\Gg,X}$
given by specializing $\O_{\tilde{F}}[v]$ along different
directions.
\begin{enumerate}
\item The map $\O_{\tilde{F}}[v]\to\tilde{F}$, $v\mapsto \tilde{\varpi}$
gives rise to a closed embedding
$i_{v=\tilde{\varpi}}:\Gr_{\Gg,\O}\otimes\O_{\tilde{F}}\to\widetilde{\Gr}_{\Gg,X}$.
We denote the open embedding
$\Gr_{\Gg,\O}\otimes\tilde{F}\to\Gr_{\Gg,\O}\otimes\O_{\tilde{F}}$
by $j_{v=\tilde{\varpi}}$.

\item The map $\O_{\tilde{F}}[v]\to\tilde{F}[v]$ gives rise to
$j_{\tilde{\varpi}\neq
0}:\widetilde{\Gr}_{\Gg,\bbA^1_{\tilde{F}}}\to
\widetilde{\Gr}_{\Gg,X}$. We denote
$\widetilde{\Gr}_{\Gg,\bbG_{m\tilde{F}}}\to
\widetilde{\Gr}_{\Gg,\bbA^1_{\tilde{F}}}$ by $j_{\tilde{F}}$.

\item The map $\O_{\tilde{F}}[v]\to\bar{k}[v]$,
$\tilde{\varpi}\mapsto 0$ gives rise to
$i_{\tilde{\varpi}=0}:\widetilde{\Gr}_{\Gg,\bbA^1_{\bar{k}}}\to
\widetilde{\Gr}_{\Gg,X}$. By its very definition, this is the
ind-scheme over $\bbA^1_{\bar{k}}$ considered in \cite{ZhuSatake}
(in \emph{loc. cit.}, it is denoted by $\widetilde{\Gr}_\Gg$). We
denote $\widetilde{\Gr}_{\Gg,\bbG_{m\bar{k}}}\to
\widetilde{\Gr}_{\Gg,\bbA^1_{\bar{k}}}$ by $j_{\bar{k}}$.

\item The map $\O_{\tilde{F}}[v]\to\O_{\tilde{F}}$, $v\mapsto 0$, gives rise
to $i_{v=0}:\Gr_{P_{\O_{\tilde{F}}}}\to\widetilde{\Gr}_{\Gg,X}$, and
its open complement is denoted by $j_{v\neq 0}:
\widetilde{\Gr}_{\Gg,\bbG_m}\to \widetilde{\Gr}_{\Gg,X}$.

\item The map $\O_{\tilde{F}}[v]\to\bar{k}$, $v\mapsto 0$, $\varpi\mapsto 0$ gives rise
to $i:\Gr_{P_{\bar{k}}}\to \widetilde{\Gr}_{\Gg,X}$.
\end{enumerate}

Since the reductive group scheme $\underline{G}$ constructed in
Sect. \ref{reductive group}  splits after the base change $\O[u^{\pm
1}]\to \O_{\tilde F}[v^{\pm 1}]$, we have
\[
\widetilde{\Gr}_{\Gg,\bbG_m}\simeq (\Gr_{H}\times
\bbG_{m})\otimes_\O\O_{\tilde{F}}.
\]
Then for $\mu\in\xcoch(T)$, we can regard
$\on{IC}_{\mu}\boxtimes\,\overline{\bbQ}_\ell[1]$, which is a
perverse sheaf on $\Gr_{H}\times \bbG_{m}$, as a perverse sheaf on
$\widetilde{\Gr}_{\Gg,\bbG_{m}}$.
Similarly, we have
$\on{IC}_{\mu,\kappa}\boxtimes\,\overline{\bbQ}_\ell[1]$ over
$\widetilde{\Gr}_{\Gg,\bbG_{m\kappa}}$ 
for $\kappa=\bar{k}$ or $\tilde{F}$.

By the construction of $\calS_{\bar k}$ in \cite{ZhuSatake}, there
is a canonical isomorphism
\[
{\rm
R}\Psi^{\widetilde{\Gr}_{\Gg,\bbA^1_{\bar{k}}}}(\on{IC}_{\mu,\bar{k}}\boxtimes\,\overline{\bbQ}_\ell[1])\simeq
\calS_{\bar{k}}(\on{Res}V_\mu).
\]
Therefore, the theorem will follow if we can construct a canonical
isomorphism
\begin{equation}\label{step1}
{\rm
R}\Psi^{\Gr_{\Gg,\O}\otimes\O_{\tilde{F}}}(\calS_{\tilde{F}}(V_\mu))\simeq
{\rm
R}\Psi^{\widetilde{\Gr}_{\Gg,\bbA^1_{\bar{k}}}}(\on{IC}_{\mu,\bar{k}}\boxtimes\,\overline{\bbQ}_\ell[1]).
\end{equation}

In what follows, we will make use of the following standard lemma.
(See also \cite[Corollary 2.3]{ZhuSatake}).
\begin{lemma}\label{another}
Let $p: \frakX\to S$ be as in \S \ref{review of nearby cycles} with $S$ a strictly henselian trait.
Let $\calF$ be a perverse sheaf on $\frakX_\eta$. If the inertia
action on ${\rm R}\Psi^\frakX(\calF)$ is trivial, then
\[
{\rm R}\Psi^\frakX(\calF)\simeq {^p}\rH^{0}i^*j_{*}\calF\simeq
{^p}\rH^{1}i^*j_*\calF\simeq i^*j_{!*}\calF,
\]
where ${^p}\rH^*$ stands for  perverse cohomology.
\end{lemma}
\begin{proof} Here, as usual, we denote by $i$ (resp. $j$) the natural closed (resp. open) immersion of the 
closed (resp. generic) fiber of $p: \frak X\to S$ into $\frakX$. Since the inertia action is trivial, from the distinguished triangle
\begin{equation}\label{distinguish}
i^*j_*\calF\to {\rm R}\Psi^\frakX(\calF)\stackrel{0}{\to} {\rm
R}\Psi^\frakX(\calF)\to \ ,
\end{equation}
(see for example \cite[(3.6.2)]{IllusieMonodromy}) we obtain that
$i^*j_*\calF$ is supported in perverse cohomological degree $0$ and
$1$, and both cohomology sheaves are isomorphic to ${\rm
R}\Psi^\frakX(\calF)$. But $i^*j_{!*}\calF={^p\rH}^{1}i^*j_*\calF$.
The lemma follows.
\end{proof}

Using this we see that we can reduce \eqref{step1} to proving an
isomorphism
\begin{equation}\label{step2}
{^p}\rH^{0}i^*j_{\bar{k}*}(\on{IC}_{\mu,\bar{k}}\boxtimes\,\overline{\bbQ}_\ell)\xrightarrow{\sim}{^p}\rH^{0}i^*j_{v=\tilde{\varpi}*}\calS_{\tilde{F}}(V_\mu).
\end{equation}
Observe that there is a natural map $i^*_{v=\tilde{\varpi}}j_{v\neq
0*}(\on{IC}_{\mu}\boxtimes\,\overline{\bbQ}_\ell)[-1]\to
j_{v=\tilde{\varpi}*}\calS_{\tilde{F}}(V_\mu)$, which is the
adjunction of $j_{v=\tilde{\varpi}}^*i^*_{v=\tilde{\varpi}}j_{v\neq
0*}(\on{IC}_{\mu}\boxtimes\,\overline{\bbQ}_\ell)[-1]\xrightarrow{\sim
} \calS_{\tilde{F}}(V_\mu)$. Similarly, we have a map
\begin{equation}
\label{stronger} i^*_{\tilde{\varpi}=0}j_{v\neq
0*}(\on{IC}_{\mu}\boxtimes\,\overline{\bbQ}_\ell)[-1]\to
j_{\bar{k}*}(\on{IC}_{\mu,\bar{k}}\boxtimes\,\overline{\bbQ}_\ell).
\end{equation}
The key observation, which we will prove later, is that
\begin{lemma}\label{stronger1}
The map \eqref{stronger} is an isomorphism.
\end{lemma}

From this, we obtain the following correspondence
\begin{equation}\label{adjunction}
{^p}\rH^0i^*j_{\bar{k}*}(\on{IC}_{\mu,\bar{k}}\boxtimes\,\bbQ_\ell)\leftarrow
{^p}\rH^0i^*j_{v\neq
0*}(\on{IC}_{\mu}\boxtimes\,\overline{\bbQ}_\ell)[-1]\to
{^p}\rH^0i^*j_{v=\tilde{\varpi}*}\calS_{\tilde{F}}(V_\mu),
\end{equation}
with the first arrow an isomorphism. Therefore, we can invert the
first arrow of \eqref{adjunction} and obtain an arrow as in
\eqref{step2}. To show that this gives an isomorphism, it is enough
to show that it induces an isomorphism on cohomology since as we
mentioned before,
$\rH^*:\on{Perv}_{L^+P_{\bar{k}}}(\Gr_{P_{\bar{k}}},\overline{\bbQ}_\ell)\to
{\rm Vect}(\overline \Q_\ell)$ is faithful.

Let $f:\widetilde{\Gr}_{\Gg,X}\to\bbA^1_{\O_{\tilde{F}}}$ be the
structure map, which is ind-proper.  Therefore the (derived)
push-forward $f_*$ commutes with any pullback and pushforward. Since
over $(\bbG_m)_{\O_{\tilde F}}\subset \bbA^1_{\O_{\tilde{F}}}$,
$f_*(\on{IC}_{\mu}\boxtimes\, \overline{\bbQ}_\ell)$ is just
constant, with stalks isomorphic to $V_\mu[1]$, it is immediately
seen that
\[
f_*i^*j_{\bar{k}*}(\on{IC}_{\mu,\bar{k}}\boxtimes\,\bbQ_\ell)\simeq
f_*i^*j_{v\neq
0*}(\on{IC}_{\mu}\boxtimes\,\overline{\bbQ}_\ell)[-1]\simeq f_*
i^*j_{v=\tilde{\varpi}*}\calS_{\tilde{F}}(V_\mu)\simeq V_\mu.
\]
Observe that the triangle \eqref{distinguish} implies that
$i^*j_{\bar{k}*}(\on{IC}_{\mu,\bar{k}}\boxtimes\,\bbQ_\ell)={^p}\rH^0(i^*j_{\bar{k}*}(\on{IC}_{\mu,\bar{k}}\boxtimes\,\bbQ_\ell))+{^p}\rH^1(i^*j_{\bar{k}*}(\on{IC}_{\mu,\bar{k}}\boxtimes\,\bbQ_\ell))[-1]$
as objects in
$\on{D}_{L^+P_{\bar{k}}}(\Gr_{P_{\bar{k}}},\overline{\bbQ}_\ell)$
(and similarly for
$i^*j_{v=\tilde{\varpi}*}\calS_{\tilde{F}}(V_\mu))$. Therefore the
isomorphism
$\rH^*(i^*j_{\bar{k}*}(\on{IC}_{\mu,\bar{k}}\boxtimes\bbQ_\ell))\simeq
\rH^*(i^*j_{v=\tilde{\varpi}*}\calS_{\tilde{F}}(V_\mu))$ implies
that the natural map \eqref{adjunction} induces the isomorphism
\[\rH^*({^p}\rH^0i^*j_{\bar{k}*}(\on{IC}_{\mu,\bar{k}}\boxtimes\bbQ_\ell))\simeq
\rH^*({^p}\rH^0i^*j_{v=\tilde{\varpi}*}\calS_{\tilde{F}}(V_\mu)),\]
and the theorem follows.

It remains to prove Lemma \ref{stronger1}. This will follow if we
can show:
\begin{lemma}\label{step3}
Consider the natural structure map
$\widetilde{\Gr}_{\Gg,X}\to\O_{\tilde{F}}$. Then:

(i) The vanishing cycle ${\rm
R}\Phi^{\widetilde{\Gr}_{\Gg,X}}(j_{v\neq
0*}(\on{IC}_\mu\boxtimes\overline{\bbQ}_\ell))$ is trivial.

(ii) The natural map
\[{\rm R}\Psi^{\widetilde{\Gr}_{\Gg,X}}(j_{\tilde{F}*}(\on{IC}_{\mu,\tilde{F}}\boxtimes\,\overline{\bbQ}_\ell))\to j_{\bar{k}*}{\rm R}\Psi^{\widetilde{\Gr}_{\Gg,\bbG_{m}}}(\on{IC}_{\mu,\tilde{F}}\boxtimes\,\overline{\bbQ}_\ell)\simeq j_{\bar{k}*}(\on{IC}_{\mu,\bar{k}}\boxtimes\,\overline{\bbQ}_\ell)\]
is an isomorphism.
\end{lemma}

Indeed, the natural map \eqref{stronger} factors as
\[i_{\tilde{\varpi}=0}^*j_{v\neq 0*}(\on{IC}_\mu\boxtimes\,\overline{\bbQ}_\ell)[-1]\to{\rm R}\Psi^{\widetilde{\Gr}_{\Gg,X}}(j_{\tilde{F}*}(\on{IC}_{\mu,\tilde{F}}\boxtimes\,\overline{\bbQ}_\ell))\to j_{\bar{k}*}(\on{IC}_{\mu,\bar{k}}\boxtimes\,\overline{\bbQ}_\ell).\]
(To see these two maps coincide, apply the adjunction between
$i_{\tilde{\varpi}=0}^*$ and $i_{\tilde{\varpi}=0*}$.) Now by Part
(i) of the lemma, the first arrow is an isomorphism, and by Part
(ii) of the lemma, the second arrow is an isomorphism. This shows
that Lemma \ref{step3} implies Lemma \ref{stronger1}.

Finally, we prove Lemma \ref{step3}. This statement can be regarded
as a global analogue of Lemma \ref{no monodromy}. Since there is no
resolution of $\widetilde{\Gr}_{\Gg,X}$ satisfying the conditions in
\cite[\S 5.2]{HainesNgoNearby}, we need a different argument. First,
we prove (i). According to Gabber's theorem, since $j_{v\neq
0*}(\on{IC}_\mu\boxtimes\,\overline{\bbQ}_\ell)[1]$ is perverse,
${\rm R}\Phi^{\widetilde{\Gr}_{\Gg,X}}(j_{v\neq
0*}(\on{IC}_\mu\boxtimes\,\overline{\bbQ}_\ell))$ is a perverse
sheaf. On the other hand, by \cite[\S 5.2]{HainesNgoNearby},
\[
j_{\bar{k}}^*{\rm R}\Phi^{\widetilde{\Gr}_{\Gg,X}}(j_{v\neq
0*}(\on{IC}_\mu\boxtimes\,\overline{\bbQ}_\ell))={\rm
R}\Phi^{\widetilde{\Gr}_{\Gg,\bbG_m}}(j_{v\neq 0}^*j_{v\neq
0*}(\on{IC}_\mu\boxtimes\,\overline{\bbQ}_\ell))=(0).
\]
Therefore, ${\rm R}\Phi^{\widetilde{\Gr}_{\Gg,X}}(j_{v\neq
0*}(\on{IC}_\mu\boxtimes\overline{\bbQ}_\ell))$ is a perverse sheaf
on $\Gr_{P_{\bar{k}}}$, which is $L^+P_{\bar{k}}$-equivariant by the
same argument as in Lemma \ref{equiv str}. To show it is trivial, it
is enough to prove that its cohomology vanishes. To do this we can
apply $f_*$ with
$f:\widetilde{\Gr}_{\Gg,X}\to\bbA^1_{\O_{\tilde{F}}}$ as above, and
we are done.

To prove (ii), as the nearby cycle functor commutes with Verdier
duality (cf. \cite[Theorem 4.2]{IllusieMonodromy}), it is enough to
prove the dual statement that the natural map
\[
j_{\bar{k}!}(\on{IC}_{\mu,\bar{k}}\boxtimes\overline{\bbQ}_\ell)\to
{\rm
R}\Psi^{\widetilde{\Gr}_{\Gg,X}}j_{\tilde{F}!}(\on{IC}_{\mu,\tilde{F}}\boxtimes\,\overline{\bbQ}_\ell)
\]
is an isomorphism. But the vanishing cycle ${\rm
R}\Phi^{\widetilde{\Gr}_{\Gg,X}}(j_{v\neq
0!}(\on{IC}_\mu\boxtimes\,\overline{\bbQ}_\ell))$ is trivial   by a similar
argument
 as above, and therefore
\begin{multline*}
\ \ \ \ \ \ \ \ \ \ {\rm
R}\Psi^{\widetilde{\Gr}_{\Gg,X}}j_{\tilde{F}!}(\on{IC}_{\mu,\tilde{F}}\boxtimes\,\overline{\bbQ}_\ell)\simeq
i_{\tilde{\varpi}=0}^*j_{v\neq
0!}(\on{IC}_\mu\boxtimes\overline{\bbQ}_\ell)[-1]\simeq \\
\simeq
j_{\bar{k}!}i_{\tilde{\varpi}=0}^*(\on{IC}_\mu\boxtimes\overline{\bbQ}_\ell)[-1]\simeq
j_{\bar{k}!}(\on{IC}_{\mu,\bar{k}}\boxtimes\overline{\bbQ}_\ell).\ \
\ \ \ \ \ \ \
\end{multline*}
This completes the proof of Lemma \ref{step3} and the theorem.
\end{proof}

\begin{Remark}\label{OrzoRemark}
{\rm It is possible that the theory of nearby cycles over a higher dimensional base, 
see for example \cite{Lau}, could provide an different approach 
to the above arguments.}
\end{Remark}

Let us use Theorem \ref{thm9.19} to  determine the inertia action on
${\rm R}\Psi_\mu$. Let $H^\vee\rtimes\Gal(\bar F/E)$ be the
Langlands dual group of $G_E$. Then according to \cite{RiZh},
$\rH^*(\F_\mu)$ carries a canonical action of $H^\vee\rtimes
\Gal(\bar F/E)$, that factors through $H^\vee\rtimes \Gal(\tilde
F/E)$, whose underlying representation of $H^\vee$ is just $V_\mu$
and whose underlying representation of $\Gal(\tilde{F}/E)$ is (a
certain twist of) the natural action of $\Gal(\tilde{F}/E)$ on
$\rH^*(\F_\mu)$. In addition, when we restrict this action to $I_E$,
the twist disappears. Therefore, $V_\mu$ is   naturally a
$(H^\vee)\rtimes I_E$-module. Since nearby cycles commute with
proper push-forward, we have $\rH^*({\rm R}\Psi_\mu)\simeq
\on{Res}(V_\mu)$ as a representation of $(H^\vee)^{I_F}\times I_E$.
Therefore, we obtain that
\begin{thm}\label{thm9.23}
Under the isomorphism
$\calS_{\bar{k}}(\on{Res}(V_\mu))\xrightarrow{\sim} {\rm R}\Psi_\mu$
obtained from Theorem \ref{thm9.19}, the action of the inertia $I_E$
on ${\rm R}\Psi_\mu$ corresponds to the action of $I_E\subset
(H^\vee)^{I_F}\times I_E$ on $\on{Res}(V_\mu)$.
\end{thm}

\begin{example}\label{exPRJAG}
{\rm Let us consider the following example,  which was discussed in
\cite{PappasRaI} (note that here we use different notation). Let
$G=\on{Res}_{K/F}\GL_n$ and $\calP_x=\on{Res}_{\O_{K}/\O_F}\GL_n$,
where $K/F$ is a totally {\sl tamely} ramified extension of degree
$d$.
In this case, we have $\Gr_{P_k}\simeq\Gr_{\GL_n}$, and
$\Gr_{\Gg,\O}\otimes_{\O}\tilde{F}\simeq(\Gr_{\GL_n})^d$. In
addition, $H^\vee=\GL_n^d$, and $(H^\vee)^I=\GL_n$, embedded in
$H^\vee$ via the diagonal embedding.

Let $\tilde{F}$ be the Galois closure of $K/F$.  Choose an ordering
of the set of embeddings $\phi: \tilde F\to \bar F$ which allows us
to identify ${\rm Gal}(\tilde F/F)$ with a subgroup of the symmetric
group $S_d$. Let $\mu$ be a minuscule coweight of $G$. We can write
$\mu=(\mu_1,\mu_2,\ldots,\mu_d)$ with $\mu_i :\Gm_{\tilde F}\to
(\GL_n)_{\tilde F}$; if $E$ is the reflex field of $\mu$, the Galois
group ${\rm Gal}(\tilde F/E)$ is the subgroup of those $\sigma\in
S_d$ with $\mu_{\sigma(i)}=\mu_i$. In this case,
\[\on{Res}(V_\mu)=V_{\mu_1}\otimes\cdots\otimes V_{\mu_d}.\]
as a representation of $\GL_n=(H^\vee)^I$. Therefore, Theorem
\ref{thm9.19} gives
\[\widetilde{{\rm R}\Psi}_\mu\simeq \on{IC}_{\mu_1}\star\on{IC}_{\mu_2}\star\cdots\star\on{IC}_{\mu_d}.\]
Indeed, this has been proved in \cite[Theorem 7.1]{PappasRaI}.
 As the action of
$\Gal(\tilde{F}/E)$ on $V_{\mu_1}\otimes\cdots \otimes V_{\mu_d}$
commutes with the action of $(H^\vee)^{I_F}$, we can write
\[V_{\mu_1}\otimes\cdots \otimes V_{\mu_d}=\bigoplus_{\lambda\leq\mu_1+\cdots+\mu_d} M_\lambda\otimes V_\lambda,\]
where each $M_\la$ is a representation of $\Gal(\tilde{F}/E)$.
Therefore, Theorem \ref{thm9.23} now implies that
\[{\rm R}\Psi_\mu\simeq  \bigoplus_{\lambda\leq\mu_1+\cdots+\mu_d} M_\lambda\otimes\on{IC}_\lambda,\]
and the action of $\Gal(\tilde{F}/E)$ on ${\rm R}\Psi_\mu$ is
through the action on each $M_\lambda$. This decomposition was
conjectured in \cite[Remark 7.4]{PappasRaI}. Actually, it seems that
our techniques can be extended to obtain this even when $K/F$ is
wildly ramified
 but we prefer to leave this for another time.}
\end{example}

\subsubsection{}
Finally, let us determine the (semi-simple) trace
$\tau^{\on{ss}}_{{\rm R}\Psi_\mu}$ in the case that $G$ is
quasi-split (but not unramified) and $x$ is very special.

 We first recall  the following extension of the
geometric Satake Langlands isomorphism (cf. \cite[Theorem
0.2]{ZhuSatake}). Recall $P_k=\calP_{x_{k((u))}}$, a parahoric group
scheme over $k[[u]]$. The affine flag variety $\Gr_{P_k}$ is defined
over $k$. Let $\frakP_x^0$ be the category of $L^+P_k$-equivariant,
semi-simple perverse sheaves on $\Gr_{P_k}$, pure of weight zero.
Then there is an equivalence
\[
\calS_k:
\on{Rep}((H^\vee)^I\rtimes\Gal(\bar{k}/k))\xrightarrow{\sim}
\frakP_x^0,
\]
where we consider $(H^\vee)^I\rtimes\Gal(\bar{k}/k)$ as a
pro-algebraic group over $\overline{\bbQ}_\ell$ and
$\on{Rep}((H^\vee)^I\rtimes\Gal(\bar{k}/k))$ is its category of
algebraic representations. Let $\bar{\mu}\in(\xcoch(T)_I)^\sigma$ so
that the corresponding Schubert variety in $\Gr_{P_k}$ is defined
over $k$. Let $\on{IC}_{\bar{\mu}}$ be its intersection cohomology
sheaf and $A_{\bar{\mu}}$ be the function obtained by taking the
Frobenius trace on the stalks of $\on{IC}_{\bar{\mu}}$. Let
$W_{\bar{\mu}}=\rH^*(\on{IC}_{\bar\mu})$. Via the above geometric
Satake isomorphism, $W_{\bar{\mu}}$ is  naturally a representation
of $(H^\vee)^I\rtimes\Gal(\bar{k}/k)$.

Recall that in \cite{ZhuSatake}, we call a smooth irreducible
representation of  $G'=\Gg\times_X\Spec(k((u)))$ over $F'=k((u))$ to
be ``unramified" if it has a vector fixed by
$P'_k=\calP_{x_{k((u))}}(k[[u]])$.   To each such  representation,
we attach a Langlands parameter
\[\on{Sat}(\pi): W_{F'}\to {^L}G= H^\vee\rtimes \Gal(\bar{F}'^s/F'),\]
where $W_{F'}$ is the Weil group of $F'$, such that
$\on{Sat}(\pi)(\gamma)= (1,\gamma)$ for $\gamma\in I_{F'}$. Let
$\Phi$ be a lift of the Frobenius to $W_{F'}$. Then as is shown in
\emph{loc. cit.}, up to conjugation, we can assume that
$\on{Sat}(\pi)(\Phi)\in (H^\vee)^{I_{F'}}\rtimes
\Gal(\bar{F}'^s/F')$ and is uniquely determined by this image up to
$(H^{\vee})^{I_{F'}}$-conjugacy. Now the Langlands parameter
$\on{Sat}(\pi)$ is characterized by the following identity.
\[
\on{tr}(\pi(A_{\bar{\mu}}))=\on{tr}((\on{Sat}(\pi)(\Phi),
W_{\bar{\mu}}).
\]

Now, since by Proposition \ref{mono triv} the inertia action on
${\rm R}\Psi_\mu$ factors through a finite quotient, we have
\[
\tau^{\on{ss}}_{{\rm R}\Psi_\mu}(x)=\on{tr}^{\on{ss}}(\sigma_x,
({\rm R}\Psi_\mu)_{\bar x})=\on{tr}(\sigma_x, ({\rm
R}\Psi_\mu)_{\bar x}^I) .
\]
We determine ${\rm R}\Psi_\mu^I$ as a Weil sheaf on
$\Gr_{P_{\bar{k}}}$. As nearby cycles commute with proper
pushforward, the cohomology $\rH^*({\rm
R}\Psi_\mu^I)=\rH^*(\calF_\mu)^I$ is pure of  weight zero.
Therefore, ${\rm R}\Psi_\mu^I$ is just the direct sum of
intersection cohomology sheaves on $\Gr_{P_{\bar{k}}}$, equipped
with the natural Weil structure. By the above discussion and the
results of the previous section, we obtain that
$\tau^{\on{ss}}_{R\Psi_\mu}\in \calH_q(G',K')$ is the unique
function such that
\[
\on{tr}(\tau^{\on{ss}}_{R\Psi_\mu}(\pi))=\on{tr}(\on{Sat}(\pi)(\sigma),V_\mu^I).
\]
This agrees with the prediction of Haines and Kottwitz.

\bigskip
\bigskip

\section{Appendix: Homogeneous spaces}

\setcounter{equation}{0}

In this section, we study the representability of certain quotients
of group schemes over a two-dimensional base. The results are used
in  Chapter \ref{chapterLoop} to show the ind-representability of the global
affine Grassmannian ${\rm Gr}_{\Gg, X}$.

\subsection{}We assume that $A$ is an excellent Noetherian regular ring
of Krull dimension $2$. If $G$, $H$ are smooth group schemes over
$A$ and $H\hookrightarrow G$ is a closed group scheme immersion then
by Artin's theorem (\cite{ArtinVersal}), the fppf quotient $G/H$ is
represented by an algebraic space which is separated of finite
presentation and in fact smooth over $A$.

\begin{lemma}\label{triple}
Suppose that $G_1\hookrightarrow G_2\hookrightarrow G_3$ are closed
group scheme immersions and $G_1$, $G_2$, $G_3$ are smooth   over
$A$. The natural morphism
$$
G_3/G_1\to G_3/G_2
$$
is an fppf fibration with fibers isomorphic to $G_2/G_1$. Suppose
that $G_2/G_1$ is quasi-affine (affine). If  $G_3/G_2$ is a scheme,
then so is $G_3/G_1$. If in addition $G_3/G_2$ is quasi-affine
(resp. affine), then so is $G_3/G_1$.
\end{lemma}

\begin{proof}
The  first statement follows from the fact that fppf descent is
effective for quasi-affine schemes. In fact, by our assumption, $
G_3/G_1\to G_3/G_2 $ is a quasi-affine morphism. If $G_3/G_2$ is
quasi-affine, then $G_3/G_1$ is also quasi-affine (affine) since its
structure morphism to $A$ is a composition of quasi-affine (resp.
affine) morphisms and as such is also quasi-affine (resp. affine).
\end{proof}

\begin{prop}\label{linearrep}
Let $\Gg, \Hh \to S=\Spec(A)$ be two smooth affine group schemes
with connected fibers. Assume that $\Hh$ is a closed subgroup scheme
of $\Gg$.

Set $\Gg=\Spec(B)$, so that $B$ is an $A$-Hopf algebra.
Then there is a free finitely generated $A$-module $M=A^n$ with
$\Gg$-action  (i.e a $B$-comodule)  and a projective $A$-submodule
$W\subset M$  which is a locally a direct summand,
 such that:

i) There is a  $\Gg$-equivariant surjection ${\rm Sym}^\bullet_A(M)
\to B $ and the $\Gg$-action on $M$ gives a group scheme
homomorphism $\rho: \Gg\hookrightarrow  {\rm GL}(M)$ which is a
closed immersion.

ii) The representation $\rho$ identifies $\Hh$ with the subgroup
scheme of $\Gg$ that stabilizes $W$.

 \end{prop}

\begin{proof} Write $\Gg=\Spec(B)$, $\Hh=\Spec(B')$ and
let $p: B \to B'$ be the ring homomorphism that corresponds to
$\Hh\subset \Gg$. Observe that $B$, $B'$ are Hopf algebras over $A$.
We will often refer to the $B$-comodules for the Hopf algebra $B$ as
``modules with $\Gg$-action". Since $\Gg$, $\Hh$ are smooth with
connected geometrical fibers  both $B$ and $B'$ are projective
$A$-modules by Raynaud-Gruson \cite[Proposition
3.3.1]{RaynaudGruson}. We start with two lemmas.

\begin{lemma}\label{proj1} Let $P$ be a projective $A$-module and $N\subset P$ be
a finitely generated $A$-submodule. Then the $A$-torsion submodule
of $P/N$ is finitely generated.\endproof
\end{lemma}

\begin{proof}As $P$ is a direct summand of a free $A$-module, we can
assume that $P=A^I$ itself is free, with a basis $\{e_{i}; i\in
I\}$. Let $n_1,\ldots,n_t$ be a set of generators of $N$, and write
$n_i=\sum a_{ij}e_j$. Then $J=\{j\in I\mid \exists\ i \mbox{ such
that } a_{ij}\neq 0\}$ is a finite set and $A^I/N=A^J/N\oplus
A^{I-J}$. The conclusion follows.
\end{proof}

\begin{lemma}\label{proj2}Let $P$ be a projective $A$-module and suppose that $N\subset P$ is a
finitely generated $A$-submodule. If $P/N$ is torsion free, then $N$
is a projective $A$-module.
\end{lemma}

\begin{proof}
Observe that $N$ is $A$-torsion free. Consider the double dual
$N^{\vee\vee}$. We have $N\subset N^{\vee\vee}\subset P$, and
$N^{\vee\vee}/N\subset P/N$ is torsion. Therefore, $N=N^{\vee\vee}$,
which is projective by our assumption that $A$ is regular and
two-dimensional.
\end{proof}
\smallskip

Now let us prove the proposition. Observe first \cite[Cor.
3.2]{ThomasonEqRes} and its proof imply (i) of the proposition (In
other words, \cite{ThomasonEqRes} implies that such a $\Gg$ is
linear, i.e a closed subgroup scheme of ${\rm GL}_n$.) To obtain the
proof of the whole proposition we have to refine this construction
from \cite{ThomasonEqRes} to account for the subgroup scheme $\Hh$.

Let $V\subset B$ be a finitely generated $A$-submodule with $\Gg$
action that contains both a set of generators of $I=\ker (B\to B')$
and a set of generators of $B$ as an $A$-algebra. Let $p(V)$ be the
image of $V$ under $p:B\to B'$. The following diagram is a
commutative diagram of  modules with $\Gg$-action
\[\begin{CD}
V@>{\rm coact}>> V\otimes_A B@>p\otimes 1>> p(V)\otimes_A B\\
@VVV@VVV@VVV\\
B@>{\rm comult}>>B\otimes_A B@>p\otimes 1>>B'\otimes_A B
\end{CD}\]
where $\Gg$ acts on $V\otimes_A B$, $p(V)\otimes_A B$, $B\otimes_A
B$, $B'\otimes_A B$ via the actions on the second factors.

The image of $N:=(p\otimes 1)\cdot {\rm coact}(V)$ in $p(V)\otimes_A
B$ is a finite $A$-submodule with $\Gg$-action. Let $\epsilon:B\to
A$ be the unit map which splits the natural $A\subset B$. Let $M$ be
the image of $N$ under $B'\otimes B\stackrel{1\otimes\epsilon}{\to}
B'$. Observe that $M$ is a finite $A$-module, but is not necessarily
$\Gg$-stable. By \cite[Prop. 2]{SerreGroIHES}, we can choose a
finite $\Gg$-stable $A$-module $\tilde{M}$ in $B'$ containing $M$.
By Lemma \ref{proj1}, we can enlarge $\tilde{M}$ if necessary to
assume that $B'/\tilde{M}$ is torsion free (so $\tilde{M}$ is
projective over $A$
 by the Lemma \ref{proj2}). We regard $\tilde{M}$ as a $\Gg$-stable submodule of
$B'\otimes_AB$ via $\tilde{M}\subset B'=B'\otimes_AA\subset
B'\otimes_AB$ (this is indeed a $\Gg$-stable submodule since the
inclusion $A\subset B$ is $\Gg$-equivariant). Let
$\tilde{N}=\tilde{M}+N$. Then $\tilde{N}$ is a finite $\Gg$-stable
submodule of $B'\otimes_AB$, and under the map $1\otimes\epsilon:
B'\otimes_A B\to B'$, $(1\otimes \epsilon)(\tilde{N})=\tilde{M}$.
Observe that the torsion submodule
 $t((B'\otimes_AB)/\tilde{N})\subset (B'\otimes_AB)/\tilde{N}$ is a module with $\Gg$-action
 and maps to zero under $(B'\otimes_AB)/\tilde{N}\stackrel{1\otimes\epsilon}{\to} B'/\tilde{M}$. Let
$\tilde{N}'$ be the preimage of $t((B'\otimes_AB)/\tilde{N})$ under
$B'\otimes_AB\to (B'\otimes_AB)/\tilde{N}$. From Lemma \ref{proj1},
$\tilde{N}'$ is finite $\Gg$-stable $A$-module, and $(1\otimes
\epsilon)(\tilde{N}')=\tilde{M}$. In addition, $\tilde{N}'$ is
locally free since $(B'\otimes_AB)/\tilde{N}'$ is torsion free.

Let $\tilde{V}$ be the $\Gg$-stable $A$-submodule of $B$ given by
the fiber product
\[\begin{CD}\tilde{V}@>>> \tilde{N}'\\
@VVV@VVV\\
B@>(p\otimes 1)\cdot{\rm comult}>> B'\otimes_AB.
\end{CD}\]
Observe that $(p\otimes 1)\cdot {\rm comult}: B' \to B'\otimes_AB$
is injective. Therefore, $\tilde{V}$ is an $A$-submodule of
$\tilde{N}'$ and therefore it is finitely generated over $A$. In
addition,  $\tilde{N}'\supset N$, $V\subset \tilde{V}$. Since
$B/\tilde{V}\hookrightarrow (B'\otimes_AB)/\tilde{N}'$ is torsion
free, $\tilde{V}$ is projective. Observe that
$$
B \xrightarrow {(p\otimes 1)\cdot {\rm comult} }
B'\otimes_AB\stackrel{1\otimes\epsilon}{\to} B'
$$ is just the projection $p$.
Therefore, $p(\tilde{V})=\tilde{M}$.

Therefore, we obtain the following commutative diagram
\[\begin{CD}
0@>>>W@>>>\tilde{V}@>>>\tilde{M}@>>>0\\
@.@VVV@VVV@VVV@.\\
0@>>> I@>>>B@>>>B'@>>>0
\end{CD}\]
with the first row finitely generated projective $A$-modules. Notice
that  $\tilde{V}\supset V$ contains a set of generators of the
$B$-ideal $I$ and a set of $A$-algebra generators of $B$. Hence, we
obtain a closed immersion $\Gg\xrightarrow {}{\rm GL}(\tilde V)$ of
group schemes and  we can see that $\Hh$ can be identified with the
closed subgroup scheme of $\Gg$ that preserves the direct summand
$W\subset M:=\tilde V$. By replacing  $M$ by $M\oplus M'$ and $W$ by
$W\oplus M'$ where $M'$ is a finitely generated projective
$A$-module with trivial $\Gg$-action, we can assume that $M$ is
$A$-free as desired.
\end{proof}

\begin{cor} \label{qproj} Suppose that  $\Hh\subset \Gg$ are as in Proposition \ref{linearrep}.
Then the fppf quotient $\Gg/\Hh$ is representable by a
quasi-projective scheme over $A$.
\end{cor}

\begin{proof}
By Artin's theorem the fppf quotient $\Gg/\Hh$ is represented by an
algebraic space over $A$. The algebraic space  $\Gg/\Hh$ is
separated of finite type and even smooth over $A$, the quotient
$\Gg\to \Gg/\Hh$ is also smooth.  Take $M$ and $W$ as in Proposition
\ref{linearrep} and set $P:=\wedge^{\on{rk}W}M$  and
$L=\wedge^{\on{rk}W}W\subset \wedge^{\on{rk}W}M=A^r$, where $r={{\rm
rank}(M)\choose {\rm rank}(W)}$. Then, $\Hh$ is the stabilizer of
$[L]$ in ${\rm Proj}( \wedge^{\on{rk}W}M)={\mathbb P}^{r-1}_A$. We
obtain a morphism $f: \Gg\to  {\mathbb P}^{r-1}_A$. This  gives a
monomorphism $\bar f: \Gg/\Hh\to  {\mathbb P}^{r-1}_A$ which is a
separated quasi-finite morphism of algebraic spaces. By
\cite[6.15]{Knutson}, $\Gg/\Hh$ is a scheme and we can now apply
Zariski's main theorem to
 $\bar f$. We obtain that $\bar f$ is a composition of an open immersion
with a finite morphism and we can conclude that $\Gg/\Hh$ is
quasi-projective. (See \cite[proof of Thm. 2.3.1]{ConradNotes} for a
similar argument).
 \end{proof}
\smallskip

\begin{Remark}
{\rm General homogeneous spaces over Dedekind rings are schemes
\cite{Ana}, but this is not always the case when the base is a
Noetherian regular ring of dimension $2$; see
\cite[X]{RaynaudLNM119}. In loc. cit. Raynaud asks if $\Gg/\Hh$ is a
scheme when both $\Gg$ and $\Hh$ are smooth and affine over a normal
base and $\Hh$ has connected fibers. The above Corollary gives a partial
answer to this question.}
\end{Remark}

\begin{cor}\label{qaffine}
Suppose  that $ \Gg$ is a smooth affine group scheme with connected
fibers over $A$. Then there $n\geq 1$ and a closed subgroup scheme embedding
$\Gg\hookrightarrow {\rm GL}_n$,   such that the
fppf quotient ${\rm GL}_n/\Gg$ is represented by a smooth
quasi-affine scheme over $A$.
\end{cor}
\begin{proof}
By Proposition \ref{linearrep} applied to $\Gg$ and $\Hh=\{e\}$, we
see that there is a closed subgroup scheme embedding $\rho:
\Gg\hookrightarrow {\rm GL}_m$ (this follows also directly from
\cite{ThomasonEqRes}). Now apply Corollary \ref{qproj} and its proof
to the pair of the group ${\rm GL}_m$ with its closed subgroup
scheme $\Gg$. We obtain a ${\rm GL}_m$-representation $\rho':
\GL_m\to \GL_r=\GL(M)$ that induces a locally closed embedding
$\GL_m/\Gg\hookrightarrow {\mathbb P}^{r-1}$. Denote by $\chi:
\Gg\to \Gm={\rm Aut}_A(L)$   the character giving the action of
$\Gg$ on the $A$-line $L$ (as in the proof of Corollary \ref{qproj})
and consider $\Gg\to \GL_m\times \Gm$ given by $g\mapsto (\rho'(g),
\chi^{-1}(g))$. Consider the quotient $(\GL_m\times \Gm)/\Gg$; to
prove it is quasi-affine it is enough to reduce to the case that $A$
is local. Then $L$ is free, $L=A\cdot v$, and $\Gg$ is the subgroup
scheme of $\GL_m\times \Gm$ (acting by $(g, a)\cdot m=a\rho'(g)(m)$)
that fixes $v$. This gives a quasi-finite separated monomorphism
$(\GL_m\times \Gm)/\Gg\rightarrow {\mathbb A}^r$ and so by arguing
as in  the proof of Corollary \ref{qproj} we see that $(\GL_m\times
\Gm)/\Gg$ is quasi-affine. Consider now the standard diagonal block
embedding $\GL_m\times \Gm\hookrightarrow \GL_{m+1}$. The quotient
$\GL_{m+1}/(\GL_m\times \Gm)$ is affine and we can conclude using
Lemma \ref{triple}.
\end{proof}

 \bibliographystyle{hamsplain}

\bibliography{ShimuraBiblioPZ}

\end{document}